\documentclass[reqno,10pt]{amsart}%
\usepackage{amsmath}
\usepackage{graphicx}
\usepackage{array}
\usepackage{epsfig}
\usepackage[all,cmtip]{xy}
\usepackage{float}
\usepackage{xcolor}
\usepackage{hyperref}
\hypersetup{colorlinks=true, linkcolor={blue!80!black},citecolor={brown!90!black},urlcolor={blue!80!black}}
\usepackage{amsfonts}
\usepackage{amssymb}
\usepackage{nicefrac} 
\providecommand{\U}[1]{\protect \rule{.1in}{.1in}}
\newcommand{\sdfrac}[2]{\mbox{\small$\displaystyle\frac{#1}{#2}$}}
\newtheorem{theorem}{Theorem}[section]
\newtheorem{corollary}[theorem]{Corollary}

\newtheorem{lemma}[theorem]{Lemma}
\newtheorem{proposition}[theorem]{Proposition}

\theoremstyle{definition}
\newtheorem{definition}[theorem]{Definition}

\newtheorem{example}[theorem]{Example}
\newtheorem{examples}[theorem]{Examples}
\newtheorem{remark}[theorem]{Remark}
\newtheorem{note}[theorem]{Note}
\numberwithin{equation}{section}
\begin{document}
\author[D.I. Dais]{Dimitrios I. Dais}
\address{University of Crete, Department of Mathematics and Applied Mathematics,
Division Algebra and Geometry, Voutes Campus, GR-70013, Heraklion, Crete, Greece}
\email{ddais@uoc.gr}
\title[On the Twelve-Point Theorem for $\ell$-Reflexive Polygons]{On the Twelve-Point Theorem \\ for $\ell$-Reflexive Polygons}

\begin{abstract}
It is known that, adding the number of lattice points lying on the boundary of
a reflexive polygon and the number of lattice points lying on the boundary of
its polar, always yields $12.$ Generalising appropriately the notion of
reflexivity, one shows that this remains true for \textquotedblleft$\ell
$-reflexive polygons\textquotedblright . In particular, there exist
(for this reason) infinitely many (lattice inequivalent) lattice polygons with
the same property. The first proof of this fact is due to Kasprzyk and Nill
\cite{KaNi}. The present paper contains a second proof (which uses tools only from toric geometry) as well as the description of complementary properties of these polygons and of the invariants of the corresponding toric log del Pezzo surfaces.

\end{abstract}
\maketitle



\section{Introduction\label{INTRO}}

\noindent The purpose of this paper is to give a second proof of the
so-called \textquotedblleft Twelve-Point Theorem\textquotedblright \ for
\textquotedblleft$\ell$-reflexive polygons\textquotedblright \ (see below
Theorem \ref{G12PTTHM}), to explain where $12$ comes from by taking a slightly different approach, and to provide some additional consequences of it from the point of view of toric geometry.\medskip

\noindent$\bullet$ \textbf{Polygons}. Let $P\subset \mathbb{R}^{2}$ be a
(convex) \textit{polygon}, i.e., the convex hull conv$(A)$ of a finite set
$A\subset \mathbb{R}^{2}$ of at least 3 non-collinear points. We denote by
Vert$(P)$ and Edg$(P)$ the set of its vertices and the set of its edges, respectively, and
by $\partial P$ and int$(P)$ its boundary and its interior, respectively. If
the origin $\mathbf{0}\in \mathbb{R}^{2}$ belongs to int$(P),$ then its
\textit{polar polygon} is defined to be
$
P^{\circ}:=\left \{  \left.  \mathbf{x}\in \mathbb{R}^{2}\right \vert
\left \langle \mathbf{x},\mathbf{y}\right \rangle \geq-1,\  \forall \mathbf{y}\in
P\right \}  ,
$
where
$
\left \langle \mathbf{x},\mathbf{y}\right \rangle :=x_{1}y_{1}+x_{2}%
y_{2},\  \text{for }\mathbf{x}=\tbinom{x_{1}}{x_{2}}\in \mathbb{R}^{2}\text{ and
}\mathbf{y}=\tbinom{y_{1}}{y_{2}}\in \mathbb{R}^{2},
$
is nothing but the usual inner product. Since $\mathbf{0}\in$
int$(P^{\circ})$ and $\left(  P^{\circ}\right)  ^{\circ}=P,$ the polarity
induces bijections%
\begin{equation}
\text{Vert}(P)\ni \mathbf{v}\longmapsto \left \{  \left.  \mathbf{x}\in P^{\circ
}\right \vert \left \langle \mathbf{x},\mathbf{v}\right \rangle =-1\right \}
\in \text{Edg}(P^{\circ}),\  \label{POLARITY1}%
\end{equation}
and%
\begin{equation}
\text{Edg}(P)\ni F\longmapsto \left \{  \left.  \mathbf{x}\in P^{\circ
}\right \vert \left \langle \mathbf{x},\mathbf{v}\right \rangle =\left \langle
\mathbf{x},\mathbf{v}^{\prime}\right \rangle =-1\right \}  \in \text{Vert}%
(P^{\circ}), \label{POLARITY2}%
\end{equation}
with $\mathbf{v},\mathbf{v}^{\prime}$ denoting the vertices of $F.$ \medskip

\noindent$\bullet$ \textbf{Lattices}. Since we shall deal with a special sort
of lattice polygons, we first recall some basic properties of lattices (cf.
\cite[Ch. 1, \S 3]{G-L}). Let $\left \Vert \mathbf{x}\right \Vert :=\left \langle
\mathbf{x},\mathbf{x}\right \rangle ^{\frac{1}{2}}$ denote the euclidean norm
of any $\mathbf{x}\in \mathbb{R}^{2}$.

\begin{proposition}
\label{LATTPROP}For any nonempty subset $N$ of $\mathbb{R}^{2}$ the following
conditions are equivalent\emph{:\smallskip \ }\newline \emph{(i)} $N$ is a
discrete subgroup of the additive group $\mathbb{R}^{2}$ \emph{(i.e.,}
$\mathbf{n}-\mathbf{n}^{\prime}\in N$ for all $\mathbf{n},\mathbf{n}^{\prime
}\in N,$ and for every $\mathbf{n}\in N$ there is an $\varepsilon \in
\mathbb{R}_{>0}$, s.t. $\mathbf{B}_{\varepsilon}(\mathbf{n})\cap
N=\{ \mathbf{n}\},$ where $\mathbf{B}_{\varepsilon}(\mathbf{n}):=\left \{
\left.  \mathbf{x}\in \mathbb{R}^{2}\right \vert \left \Vert \mathbf{x}%
-\mathbf{n}\right \Vert \leq \varepsilon \right \}  $\emph{),} and $N$ spans the
entire $\mathbb{R}^{2}$ as $\mathbb{R}$-vector space$.\smallskip$
\newline \emph{(ii)} There exists a set $\left \{  \mathbf{b}_{1},\mathbf{b}%
_{2}\right \}  $ of two $\mathbb{R}$-linear independent vectors $\mathbf{b}%
_{1},\mathbf{b}_{2}\in \mathbb{R}^{2}$ such that
\[
N=\left \{  k_{1}\left.  \mathbf{b}_{1}+k_{2}\mathbf{b}_{2}\right \vert
k_{1},k_{2}\in \mathbb{Z}\right \}  .
\]

\end{proposition}

\begin{definition}
\label{DEFLATTICE}A \textit{lattice} in $\mathbb{R}^{2}$ is a nonempty subset
$N$ of $\mathbb{R}^{2}$ which satisfies the conditions of Proposition
\ref{LATTPROP}. A set $\left \{  \mathbf{b}_{1},\mathbf{b}_{2}\right \}  $ as in
\ref{LATTPROP} (ii) is said to be a ($\mathbb{Z}$-)\textit{basis} of $N.$ ($N$
itself can be viewed as a free abelian group ($\mathbb{Z}$-module) of rank $2$
generated by $\left \{  \mathbf{b}_{1},\mathbf{b}_{2}\right \}  .$) If
$\mathbf{b}_{1}=\tbinom{\mathfrak{b}_{11}}{\mathfrak{b}_{21}}$ and
$\mathbf{b}_{2}=\tbinom{\mathfrak{b}_{12}}{\mathfrak{b}_{22}},$ we say that
$\mathcal{B}=\left(
\begin{smallmatrix}
\mathfrak{b}_{11} & \mathfrak{b}_{12}\\
\mathfrak{b}_{21} & \mathfrak{b}_{22}%
\end{smallmatrix}
\right)  $ is the corresponding \textit{basis matrix} of $N.$
\end{definition}

\begin{proposition}
\label{Basicness}If $N\subset \mathbb{R}^{2}$ is a lattice and $\mathbf{b}%
_{1},\mathbf{b}_{2}\in N$ are two $\mathbb{R}$-linear independent vectors,
then $\left \{  \mathbf{b}_{1},\mathbf{b}_{2}\right \}  $ is a basis of $N$ if
and only if \emph{conv}$(\left \{  \mathbf{0},\mathbf{b}_{1},\mathbf{b}%
_{2}\right \}  )\cap N=\left \{  \mathbf{0},\mathbf{b}_{1},\mathbf{b}%
_{2}\right \}  .$
\end{proposition}

\begin{proposition}
\label{BASECHANGE}Let $N\subset \mathbb{R}^{2}$ be a lattice with $\mathcal{B}$
as a basis matrix. Then a matrix $\mathcal{B}^{\prime}\in$ \emph{GL}%
$_{2}(\mathbb{R})$ is a \textit{basis matrix} of $N$ $\Leftrightarrow
\exists \mathcal{A}\in$ \emph{GL}$_{2}(\mathbb{Z}):$ $\mathcal{B}^{\prime
}=\mathcal{BA}.$
\end{proposition}

\begin{definition}
Let $N\subset \mathbb{R}^{2}$ be a lattice with $\mathcal{B}$ as a basis
matrix. The \textit{determinant} of $N$ is defined to be $\det(N):=\left \vert
\det(\mathcal{B})\right \vert .$ (By Proposition \ref{BASECHANGE}, $\det(N)$
does not depend on the particular choice of $\mathcal{B},$ because
$\det(\mathcal{A})\in \left \{  \pm1\right \}  $ for all $\mathcal{A}\in$
GL$_{2}(\mathbb{Z}).$)
\end{definition}

\begin{proposition}
\label{parallepiped}Let $N^{\prime}$ be a sublattice of a lattice
$N\subset \mathbb{R}^{2}.$ Suppose that $\left \{  \mathbf{b}_{1}^{\prime
},\mathbf{b}_{2}^{\prime}\right \}  $ and $\left \{  \mathbf{b}_{1}%
,\mathbf{b}_{2}\right \}  $ are bases of $N^{\prime}$ and $N,$ respectively,
and%
\[
\mathbf{b}_{1}^{\prime}=u_{11}\mathbf{b}_{1}+u_{12}\mathbf{b}_{2}%
,\  \  \mathbf{b}_{2}^{\prime}=u_{21}\mathbf{b}_{1}+u_{22}\mathbf{b}_{2},
\]
are the expressions of $\mathbf{b}_{1}^{\prime},\mathbf{b}_{2}^{\prime}$ as
integer linear combinations of $\mathbf{b}_{1},\mathbf{b}_{2}.$ Then the
number of points of $N$ which belong to the half-open parallelepiped
$\Pi:=\left \{  \xi_{1}\mathbf{b}_{1}^{\prime}+\xi_{2}\mathbf{b}_{2}^{\prime
}\left \vert \xi_{1},\xi_{2}\in \left[  0,1\right)  \right.  \right \}  $ equals%
\[
\sharp(\Pi \cap N)=\left \vert \det(\mathcal{U})\right \vert =\left \vert
N:N^{\prime}\right \vert =\frac{\det(N^{\prime})}{\det(N)},
\]
where $\mathcal{U}:=\left(  u_{ij}\right)  _{1\leq i,j\leq2}$ and $\left \vert
N:N^{\prime}\right \vert $ the index of (the subgroup) $N^{\prime}$ in $N.$
\end{proposition}

\begin{note}
(i) $\mathbb{Z}^{2}:=\left \{  \left.  \tbinom{\lambda_{1}}{\lambda_{2}%
}\right \vert \lambda_{1},\lambda_{2}\in \mathbb{Z}\right \}  $ is the
\textit{standard} (rectangular) \textit{lattice} in $\mathbb{R}^{2}$ having
$\left(
\begin{smallmatrix}
1 & 0\\
0 & 1
\end{smallmatrix}
\right)  $ as basis matrix (and determinant $=1$).\smallskip \  \newline(ii) The
automorphism group of the $\mathbb{R}$-vector space $\mathbb{R}^{2}$ is Aut$(\mathbb{R}%
^{2}):= \text{GL}(\mathbb{R}^{2})=\left \{  \left.  \Phi_{\mathcal{A}}\right \vert \mathcal{A}\in
\text{GL}_{2}(\mathbb{R})\right \}  ,$ where
\[
\fbox{$%
\begin{array}
[c]{ccc}
& \mathbb{R}^{2}\ni \tbinom{x_{1}}{x_{2}}\longmapsto \Phi_{\mathcal{A}}\left(
\tbinom{x_{1}}{x_{2}}\right)  :=\mathcal{A}\tbinom{x_{1}}{x_{2}}\in
\mathbb{R}^{2}, &
\end{array}
$}%
\]
and
\[
\text{Aut}_{\mathbb{Z}^{2}}(\mathbb{R}^{2}):=\left \{  \Psi \in \text{Aut}%
(\mathbb{R}^{2})\left \vert \Psi(\mathbb{Z}^{2})=\mathbb{Z}^{2}\right.
\right \}  =\left \{  \Phi_{\mathcal{A}}\left \vert \mathcal{A}\in \text{GL}%
_{2}(\mathbb{Z})\right.  \right \}  .
\]
(iii) If $N\subset \mathbb{R}^{2}$ is an arbitrary lattice with $\mathcal{B}$
as a basis matrix, then $N=\Phi_{\mathcal{B}}(\mathbb{Z}^{2}),$ and the
subgroup%
\[
\text{Aut}_{N}(\mathbb{R}^{2}):=\left \{  \Psi \in \text{Aut}(\mathbb{R}%
^{2})\left \vert \Psi(N)=N\right.  \right \}  =\left \{  \Phi_{\mathcal{BAB}%
^{-1}}\left \vert \mathcal{A}\in \text{GL}_{2}(\mathbb{Z})\right.  \right \}
\]
of Aut$(\mathbb{R}^{2})$ consists of the so-called \textit{unimodular }%
$N$-\textit{transformations.}
\end{note}

\begin{definition}
Assume that $N\subset \mathbb{R}^{2}$ is a lattice with $\mathcal{B}$ as a
basis matrix. Identifying the dual $\mathbb{Z}$-module $M:=$ Hom$_{\mathbb{Z}%
}(N,\mathbb{Z})$ with $\left \{  \left.  \mathbf{x}\in \mathbb{R}^{2}\right \vert
\left \langle \mathbf{x},\mathbf{n}\right \rangle \in \mathbb{Z},\forall
\mathbf{n}\in N\right \}  $ we embed it as a lattice in $\mathbb{R}^{2}$ (and
call it \textit{dual lattice}) having $(\mathcal{B}^{\intercal})^{-1}$ as
basis matrix, and determinant $\det(M)=(\det(N))^{-1}.$ (The standard lattice
$\mathbb{Z}^{2}$ is self-dual.)
\end{definition}

\noindent{}$\bullet$ \textbf{Lattice polygons}. A \textit{lattice polygon}
$P\subset \mathbb{R}^{2}$ \textit{w.r.t. a lattice} $N\subset \mathbb{R}^{2}$
(or an $N$-\textit{polygon}, for short) is a polygon with Vert$(P)\subset N.$
Let POL$(N)$ be the set of all $N$-polygons. For $P\in$ POL$(N)$ the number
$\sharp \left(  P\cap N\right)  $ is given by \textit{Pick's formula
\cite{Pick}}$:$
\begin{equation}
\sharp \left(  P\cap N\right)  =\text{area}_{N}(P)+\tfrac{1}{2}\sharp \left(
\partial P\cap N\right)  +1 \label{PICK}.
\end{equation}
Here, area$_{N}(P)$ denotes the Lebesgue measure on $\mathbb{R}^{2}$ normalised
w.r.t. $N,$ so that half-open parallelepiped determined by the members of a
basis of $N$ has area $1.$ (In fact, in terms of the \textquotedblleft
usual\textquotedblright \ area, this equals $\tfrac{\text{area}(P)}{\det(N)}$).
If $kP:=\left \{  k\left.  \mathbf{x}\right \vert \mathbf{x}\in
P\right \}  \  \left(  k\in \mathbb{Z}_{\geq0}\right)  $  denotes the $k$-th dilation of
$P,$ it is known that the \textit{Ehrhart polynomial}
\[
\text{Ehr}_{N}(P;k):=\sharp \left(  kP\cap N\right)  \in \mathbb{Q}\left[
k\right]
\]
of $P$ (w.r.t. $N$) equals
\begin{equation}
\text{Ehr}_{N}(P;k)=\text{area}_{N}(P)k^{2}+\tfrac{1}{2}\sharp \left(  \partial
P\cap N\right)  k+1 \label{Ehrhartpol}%
\end{equation}
and that Ehr$_{N}(P;k)=$ Ehr$_{N}(\Upsilon(P);k)$ for all affine integral
transformations $\Upsilon$ of $\mathbb{R}^{2}$ w.r.t. $N$ (which are composed
of unimodular\textit{ }$N$-transformations and $N$-translations). Furthermore,
according to the reciprocity law \cite[p. 50]{Ehrhart} for Ehr$_{N}%
^{\text{int}}(P;k):=\sharp \left(\text{int}(kP)\cap N\right)  ,$%
\begin{equation}
\text{Ehr}_{N}^{\text{int}}(P;k)=\text{Ehr}_{N}(P;-k)=\text{area}_{N}%
(P)k^{2}-\tfrac{1}{2}\sharp \left(  \partial P\cap N\right)  k+1.
\label{Ehrhartint}%
\end{equation}

\noindent$\bullet$ \textbf{The equivalence relation }\textquotedblleft%
$\backsim_{N}$\textquotedblright. On the set POL$_{\mathbf{0}}(N):=\left \{
\left.  P\in \text{POL}(N)\right \vert \mathbf{0}\in \text{int}(P)\right \}  $ we
define the equivalence relation:%
\[
P_{1}\backsim_{N}P_{2}\underset{\text{def}}{\Longleftrightarrow}\exists \Psi
\in \text{Aut}_{N}(\mathbb{R}^{2}):\Psi \left(  P_{1}\right)  =P_{2}.
\]
If $P_{1}\backsim_{N}P_{2},$ we say that $P_{1}$ and $P_{2}$ are
\textit{equivalent up to $N$-umimodular transformation}. If $P\in$
POL$_{\mathbf{0}}(N),$ we denote by $\left[  P\right]  _{N}:=\left \{  \left.
R\in \text{POL}_{\mathbf{0}}(N)\right \vert R\backsim_{N}P\right \}  $ its
equivalence class.

\begin{definition}
If $P\in$ POL$_{\mathbf{0}}(N),$ then for a fixed basis matrix $\mathcal{B}%
\in$ GL$_{2}(\mathbb{R})$ of $N$ we have $N=\Phi_{\mathcal{B}}(\mathbb{Z}%
^{2})$ with $\Phi_{\mathcal{B}}\in$ Aut$(\mathbb{R}^{2}).$ Thus, we may define
the polygon
\[
P^{\text{st}}:=\Phi_{\mathcal{B}^{-1}}(P)\in \text{POL}_{\mathbf{0}}%
(\mathbb{Z}^{2}).
\]
$P^{\text{st}}$ will be called the \textit{standard model} of $P$ w.r.t.
$\mathcal{B}$. By Proposition \ref{BASECHANGE}, $[P^{\text{st}}]_{\mathbb{Z}%
^{2}}$ does not depend on the particular choice of $\mathcal{B}.$
\end{definition}

\noindent{}If the induced bijection
$
\text{POL}_{\mathbf{0}}(N)/\backsim_{N}\, \ni \left[  P\right]  _{N}%
\longmapsto \lbrack P^{\text{st}}]_{\mathbb{Z}^{2}}\in \, \text{POL}_{\mathbf{0}%
}(\mathbb{Z}^{2})/\backsim_{\, \mathbb{Z}^{2}}%
$
is taken into account, it is sometimes convenient to work with the equivalence class of
$P^{\text{st}}$ instead of that of $P$ (and with the standard lattice
$\mathbb{Z}^{2}$ instead of $N$), e.g., when we draw figures, when we
construct certain polygon classification lists etc. It is worth mentioning
that
\[
\text{Ehr}_{N}(P;k)=\text{Ehr}_{\mathbb{Z}^{2}}(P^{\text{st}};k)\  \text{ for
all }k\in \mathbb{Z}_{\geq0},
\]
because $\sharp \left(  \partial P\cap N\right)  =\sharp(\partial P^{\text{st}%
}\cap \mathbb{Z}^{2})$ and area$_{N}(P)=$ area$_{\mathbb{Z}^{2}}(P^{\text{st}%
}).\medskip$

\noindent{}$\bullet$ \textbf{LDP-polygons}. Let $N\subset \mathbb{R}^{2}$ be a
lattice. A point $\mathbf{n}\in N\mathbb{r}\{ \mathbf{0}\}$ is said to be
\textit{primitive} (w.r.t. $N$) if%
\[
\text{conv}(\{ \mathbf{0},\mathbf{n}\})\cap N=\{ \mathbf{0},\mathbf{n}\}.
\]

\begin{definition}
\label{DEFLOCIND}(i) A polygon $Q\in$ POL$_{\mathbf{0}}(N)$ is called
\textit{LDP-polygon} if it has primitive vertices.\smallskip \  \newline(ii) Let
$Q\subset \mathbb{R}^{2}$ be an LDP-polygon (w.r.t. $N$) and $M:=$
Hom$_{\mathbb{Z}}(N,\mathbb{Z})$ be the dual of our reference lattice $N.$ For
$F\in$ Edg$(Q)$ we denote by $\boldsymbol{\eta}_{F}\in M$ the unique primitive
lattice point which defines an \textit{inward-pointing }normal of $F.$ The
affine hull of $F$ is of the form $\left \{  \left.  \mathbf{y}\in
\mathbb{R}^{2}\right \vert \left \langle -\boldsymbol{\eta}_{F},\mathbf{y}%
\right \rangle =l_{F}\right \}  ,$ for some positive integer $l_{F}.$ This
$l_{F}$ is nothing but the integral distance between $\mathbf{0}$ and $F,$ the
so-called \textit{local index} of $F$ (w.r.t. $Q$). The \textit{index} $\ell$
of $Q$ is defined to be the positive integer $\ell:=$ lcm$\left \{  \left.
l_{F}\right \vert F\in \text{Edg}(Q)\right \}  .$ It is easy to prove that%
\begin{equation}
\fbox{$%
\begin{array}
[c]{ccc}
& \ell=\min \left \{  k\in \mathbb{Z}_{>0}\left \vert \text{Vert}(kQ^{\circ
})\subset M\right.  \right \}  . &
\end{array}
$} \label{LEQUALITY}%
\end{equation}
(Note that the $M$-polygon $\ell Q^{\circ}$ is not necessarily an LDP-polygon
w.r.t. $M.$)
\end{definition}

\noindent{}For every positive integer $\ell$ we define 
\[
\text{LDP}(\ell;N):=\left \{  \left[  Q\right]  _{N}\left \vert Q\in
\text{POL}_{\mathbf{0}}(N)\  \text{is an LDP-polygon of index }\ell \right.
\right \}  .
\]

\begin{theorem}
$\sharp($\emph{LDP}$(\ell;N))<\infty$ for all $\ell \geq1.$
\end{theorem}

\noindent{}This can be derived by using (more general) results of Hensley
\cite{Hensley}, and Lagarias \& Ziegler \cite{Lag-Zieg}. LDP-polygons are of
particular interest because their $\backsim_{N}$-classes parametrise the
isomorphism classes of toric log del Pezzo surfaces. (See below
\S \ref{TORICLDPSURF}.) LDP-triangles of index $\leq3$ have been classified
(up to unimodular transformation) in \cite[\S 6]{Dais1} and \cite{Dais2}. More
recently, this classification has been extended considerably in \cite{KKN} via
a certain algorithm, by means of which it is possible to produce the
LDP-polygons of given index $\ell$ (up to unimodular transformation) by fixing
a \textquotedblleft special\textquotedblright \ edge and performing a prescribed
successive addition of vertices. Of course, their cardinality grows rapidly as we increase indices! The classification is complete for $\ell \leq17.$

\begin{theorem}
[\cite{KKN}]\label{KKNTHM}The values of the enumerating function $\ell
\mapsto \sharp($\emph{LDP}$(\ell;N))$ for $\ell \leq17$ are those given in the
following tables:
\[%
\begin{tabular}
[c]{|c||c|c|c|c|c|c|c|c|c|}\hline
$\ell$ & $1$ & $2$ & $3$ & $4$ & $5$ & $6$ & $7$ & $8$ & $9$\\ \hline \hline
$\sharp($\emph{LDP}$(\ell;N))$ & $16$ & $30$ & $99$ & $91$ & $250$ & $379$ &
$429$ & $307$ & $690$\\ \hline
\end{tabular}
\  \  \
\]%
\[%
\begin{tabular}
[c]{|c||c|c|c|c|c|c|c|c|}\hline
$\ell$ & $10$ & $11$ & $12$ & $13$ & $14$ & $15$ & $16$ & $17$\\ \hline \hline
$\sharp($\emph{LDP}$(\ell;N))$ & $916$ & $939$ & $1279$ & $1142$ & $1545$ &
$4312$ & $1030$ & $1892$\\ \hline
\end{tabular}
\  \  \
\]

\end{theorem}

\noindent Useful details for the structure of each of these $16+30+\cdots
+1892=15346$ LDP-polygons (vertices of representatives of standard models
w.r.t. a suitable coordinate system, interior and boundary lattice points,
area, local indices, Ehrhart and Hilbert series etc.) are included in the
database \cite{Br-Kas}.\medskip

\noindent $\bullet$ \textbf{Reflexive polygons}. Firstly we focus on
LDP-polygons of index $1.$

\begin{proposition}
\label{CONDREFL}Let $Q\subset \mathbb{R}^{2}$ be an LDP-polygon of index $\ell$
\emph{(}w.r.t. $N$\emph{)}. Then the following conditions are
equivalent\emph{:\smallskip} \newline \emph{(i)} The polar polygon $Q^{\circ}$
of $Q$ is an $M$-polygon, where $M:=$ \emph{Hom}$_{\mathbb{Z}}(N,\mathbb{Z}%
).\smallskip$\newline \emph{(ii)} $\ell=1.\smallskip$ \newline \emph{(iii)}
$l_{F}=1$ for all $F\in$ \emph{Edg}$(Q).\smallskip$ \newline \emph{(iv)}
\emph{int}$(Q)\cap N=\{ \mathbf{0}\}.\smallskip$\newline$\emph{(v)}$
\emph{Ehr}$_{N}(Q;k)=$ \emph{Ehr}$_{N}^{\text{\emph{int}}}(Q;k+1)$ for all
$k\in \mathbb{Z}_{\geq0}.$
\end{proposition}

\begin{proof}
By the definition of $\ell$ and by (\ref{LEQUALITY}), the equivalences
(i)$\Leftrightarrow$(ii)$\Leftrightarrow$(iii) are obvious. For
(v)$\Leftrightarrow$(i) see Hibi \cite[\S 35]{Hibi-Book}, \cite{Hibi}. The
implication (v)$\Rightarrow$(iv) is apparently true (by applying (v) for
$k=0$). It suffices to show the validity of implication (iv)$\Rightarrow$(ii).
Assuming that int$(Q)\cap N=\{ \mathbf{0}\},$ (\ref{Ehrhartint}) gives%
\begin{equation}
1=\sharp \left(  \text{int}(Q)\cap N\right)  =\text{Ehr}_{N}^{\text{int}%
}(Q;1)=\text{area}_{N}(Q)-\tfrac{1}{2}\sharp \left(  \partial Q\cap N\right)
+1\Rightarrow2\text{area}_{N}(Q)=\sharp \left(  \partial Q\cap N\right)  .
\label{1Equality}%
\end{equation}
Let $F$ be an edge of $Q$ having $\mathbf{n},\mathbf{n}^{\prime}$ as vertices.
We observe that $l_{F}=\frac{1}{\sharp \left(  F\cap N\right)  -1}\cdot
\frac{\left \vert \det(\mathbf{n},\mathbf{n}^{\prime})\right \vert }{\det(N)}.$
Next, we subdivide the triangle $T_{F}:=$ cov$(\left \{  \mathbf{0}%
,\mathbf{n},\mathbf{n}^{\prime}\right \}  )$ into $\sharp \left(  F\cap
N\right)  -1$ subtriangles, each of which having $\mathbf{0}$ and two
consecutive lattice points of $F$ as vertices. Obviously,
\[
2\, \text{area}_{N}(Q)=2\sum_{F\in \text{ Edg}(Q)}\text{area}_{N}(T_{F}%
)=\sum_{F\in \text{ Edg}(Q)}\left(  \sharp \left(  F\cap N\right)  -1\right)
l_{F}.
\]
If there were an edge with local index $>1$ (w.r.t. $Q$), we would have
\[
2\, \text{area}_{N}(Q)=\sum_{F\in \text{ Edg}(Q)}\left(  \sharp \left(  F\cap
N\right)  -1\right)  l_{F}>\sum_{F\in \text{ Edg}(Q)}\left(  \sharp \left(
F\cap N\right)  -1\right)  =\sharp \left(  \partial Q\cap N\right)  ,
\]
contradicting (\ref{1Equality}). Thus, $l_{F}=1$ for all $F\in$ Edg$(Q).$
\end{proof}

\begin{definition}
\label{DEFREFLPOL}An LDP-polygon $Q\subset \mathbb{R}^{2}$ (w.r.t. $N$) is
called \textit{reflexive polygon }(and $(Q,N)$ \textit{reflexive pair}) if it
satisfies the conditions of Proposition \ref{CONDREFL}.
\end{definition}

\begin{note}
(i) If $(Q,N)$ is a reflexive pair\textit{ }and\textit{ }$M$ the dual of $N,$
then $(Q^{\circ},M)$ is again a reflexive pair. (See \cite[Theorem 4.1.6, p.
510]{Batyrev2}.)\smallskip \  \newline(ii) The notion of reflexivity is
extendable to lattice polytopes of \textit{any} dimension $\geq3$ via
conditions \ref{CONDREFL} (i), (iii) and (v) (which remain equivalent). It
was introduced by Batyrev in \cite[\S 4]{Batyrev2}. Condition \ref{CONDREFL}
(iv) is necessary (for a lattice polytope of dimension $\geq3$ to be
reflexive) but is not sufficient: There are several lattice polytopes of
dimension $\geq3$ which have the origin as the only interior lattice point
without being reflexive. (Reflexive polytopes play a pivotal role in the
so-called \textquotedblleft combinatorial mirror symmetry\textquotedblright;
cf. \cite[Chapters 3 and 4]{Cox-Katz} and \cite{Batyrev-Nill}). On the other
hand, in dimension $2$ we meet nice lattice point enumerator identities like
(\ref{12PTFORMULA}).
\end{note}

\begin{theorem}
[Twelve-Point Theorem]\label{12PTTHM}If $(Q,N)$ is a \textit{reflexive pair
and }$M:=$ \emph{Hom}$_{\mathbb{Z}}(N,\mathbb{Z}),$ then%
\begin{equation}
\fbox{$%
\begin{array}
[c]{ccc}
& \sharp \left(  \partial Q\cap N\right)  +\sharp(\partial Q^{\circ}\cap
M)=12. &
\end{array}
$} \label{12PTFORMULA}%
\end{equation}

\end{theorem}

\noindent{}One proof consists in case-by-case verification of
(\ref{12PTFORMULA}) by passing through the explicit classification of
reflexive polygons, i.e., by the so-called \textquotedblleft exhaustion
method\textquotedblright.

\begin{theorem}
[Classification of reflexive polygons]\label{CLASSIFRP}Let $(Q,N)$ be a
\textit{reflexive pair. Then }$Q$ has at most $6$ vertices, and a
representative of the equivalence class of its standard model (w.r.t. any
basis matrix of $N$) is exactly one of the sixteen $\mathbb{Z}^{2}$-polygons
$\mathcal{Q}_{1},\ldots,\mathcal{Q}_{16}$ illustrated in Figure  \ref{Fig.0}, whose vertices (in an anticlockwise
order) are given in the second columns of the tables:\bigskip
\setlength\extrarowheight{3pt}
{\small
\begin{equation}
\begin{tabular}{|c||c|c|}
\hline
$i$ & \emph{vertices of }$\mathcal{Q}_{i}$ & $\sharp (\partial \mathcal{Q}%
_{i}\cap \mathbb{Z}^{2})$ \\ \hline\hline
$1$ & $%
\begin{array}{c}
\tbinom{1}{0},\tbinom{0}{1},\tbinom{-1}{-1}\smallskip %
\end{array}%
$ & $3$ \\ \hline
$2$ & $%
\begin{array}{c}
\tbinom{1}{0},\tbinom{0}{1},\tbinom{-2}{-1} \smallskip%
\end{array}%
$ & $4$ \\ \hline
$3$ & $%
\begin{array}{c}
\tbinom{1}{0},\tbinom{0}{1},\tbinom{-1}{0},\tbinom{0}{-1} \smallskip%
\end{array}%
$ & $4$ \\ \hline
$4$ & $%
\begin{array}{c}
\tbinom{1}{0},\tbinom{0}{1},\tbinom{-1}{0},\tbinom{-1}{-1} \smallskip%
\end{array}%
$ & $4$ \\ \hline
$5$ & $%
\begin{array}{c}
\tbinom{1}{0},\tbinom{0}{1},\tbinom{-1}{1},\tbinom{-1}{-1} \smallskip%
\end{array}%
$ & $5$ \\ \hline
$6$ & $%
\begin{array}{c}
\tbinom{1}{0},\tbinom{0}{1},\tbinom{-1}{0},\tbinom{-1}{-1},\tbinom{0}{-1} \smallskip%
\end{array}%
$ & $5$ \\ \hline
$7$ & $%
\begin{array}{c}
\tbinom{1}{-1},\tbinom{-1}{2},\tbinom{-1}{-1} \smallskip%
\end{array}%
$ & $6$ \\ \hline
$8$ & $%
\begin{array}{c}
\tbinom{1}{0},\tbinom{-1}{1},\tbinom{-1}{-1},\tbinom{1}{-1}\smallskip%
\end{array}%
$ & $6$ \\ \hline
\end{tabular}%
\begin{array}{c}
\ 
\end{array}%
\begin{tabular}{|c||c|c|}
\hline
$i$ & \emph{vertices of }$\mathcal{Q}_{i}$ & $\sharp (\partial \mathcal{Q}%
_{i}\cap \mathbb{Z}^{2})$ \\ \hline\hline
$9$ & $%
\begin{array}{c}
\tbinom{1}{0},\tbinom{0}{1},\tbinom{-1}{0},\tbinom{-1}{-1},\tbinom{1}{-1}\smallskip%
\end{array}%
$ & $6$ \\ \hline
$10$ & $%
\begin{array}{c}
\tbinom{1}{0},\tbinom{1}{1},\tbinom{0}{1},\tbinom{-1}{0},\tbinom{-1}{-1},%
\tbinom{0}{-1}\smallskip%
\end{array}%
$ & $6$ \\ \hline
$11$ & $%
\begin{array}{c}
\tbinom{-1}{-1},\tbinom{1}{-1},\tbinom{1}{0},\tbinom{0}{1},\tbinom{-1}{1}\smallskip%
\end{array}%
$ & $7$ \\ \hline
$12$ & $%
\begin{array}{c}
\tbinom{-1}{-1},\tbinom{0}{-1},\tbinom{1}{0},\tbinom{-1}{2}\smallskip%
\end{array}%
$ & $7$ \\ \hline
$13$ & $%
\begin{array}{c}
\tbinom{-1}{-1},\tbinom{1}{-1},\tbinom{1}{0},\tbinom{-1}{2}\smallskip%
\end{array}%
$ & $8$ \\ \hline
$14$ & $%
\begin{array}{c}
\tbinom{-1}{-1},\tbinom{1}{-1},\tbinom{1}{1},\tbinom{-1}{1}\smallskip%
\end{array}%
$ & $8$ \\ \hline
$15$ & $%
\begin{array}{c}
\tbinom{-1}{-1},\tbinom{1}{-1},\tbinom{-1}{3}\smallskip%
\end{array}%
$ & $8$ \\ \hline
$16$ & $%
\begin{array}{c}
\tbinom{-1}{-1},\tbinom{2}{-1},\tbinom{-1}{2}\smallskip%
\end{array}%
$ & $9$ \\ \hline
\end{tabular}%
\ \   \label{tablesofRFP}
\end{equation}}

\setlength\extrarowheight{-3pt}

\end{theorem}
\newpage

\begin{figure}[th]
	\includegraphics[height=11cm, width=8cm]{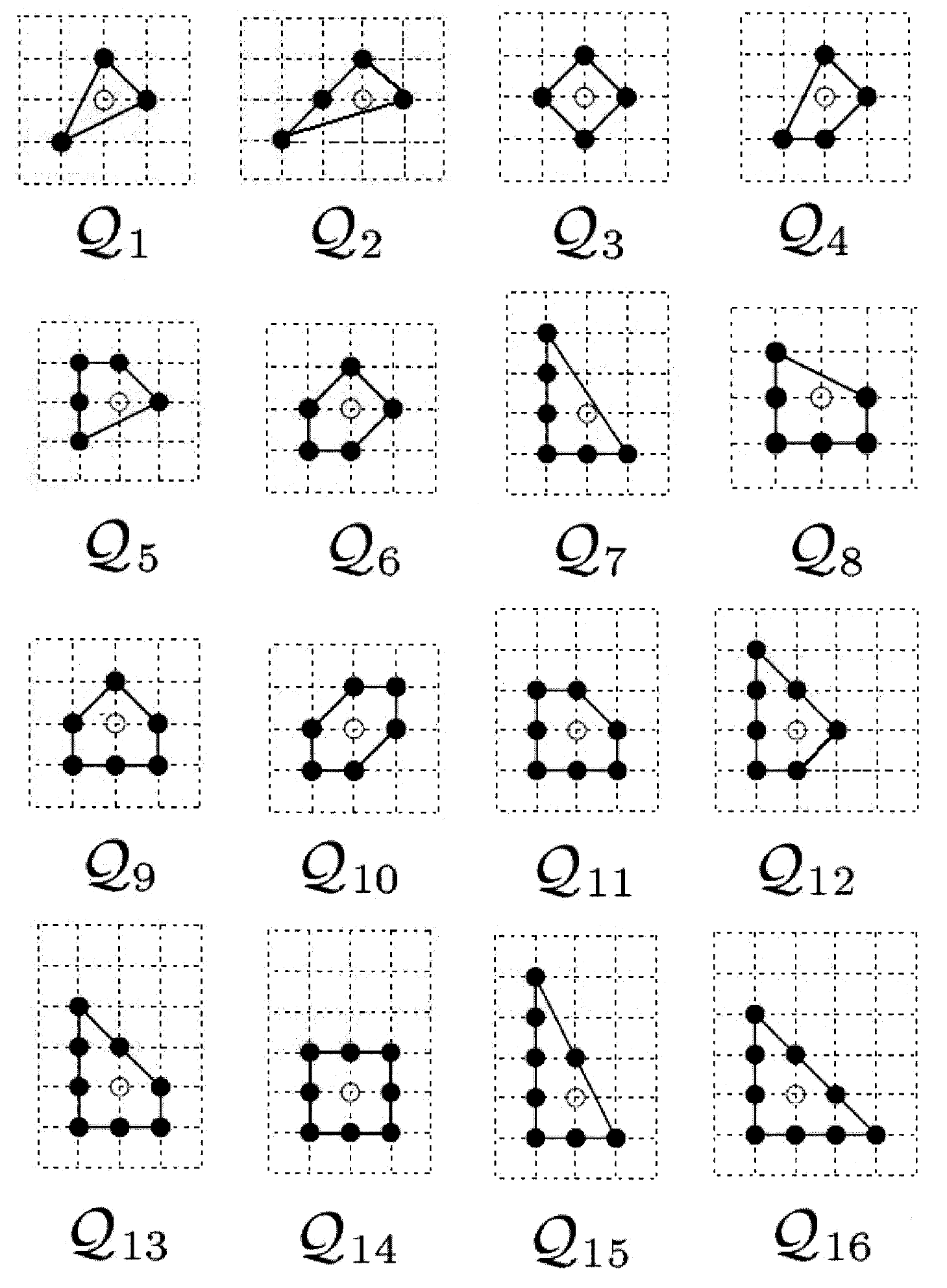}
	\caption{}\label{Fig.0}%
\end{figure}

\noindent{}\textit{Proof of Theorem \ref{12PTTHM} via Theorem \ref{CLASSIFRP}%
}. One checks directly that
\begin{equation}
\mathcal{Q}_{17-i}=\mathcal{Q}_{i}^{\circ},\text{ for }i\in \{1,2,3,4,5,6\},
\label{AUXEQ1}%
\end{equation}
and that the vertices of the polars of $\mathcal{Q}_{7},\ldots,\mathcal{Q}%
_{10}$ are the following:
\setlength\extrarowheight{3pt}
\[
\begin{tabular}
[c]{|c||c|}\hline
$i$ & vertices of\emph{ }$\mathcal{Q}_{i}^{\circ}$\\\hline\hline
$7$ & $%
\begin{array}
[c]{c}%
\tbinom{-3}{-2},\tbinom{1}{0},\tbinom{0}{1}\smallskip
\end{array}
$\\\hline
$8$ & $%
\begin{array}
[c]{c}%
\tbinom{-1}{-2},\tbinom{1}{0},\tbinom{0}{1},\tbinom{-1}{0}\smallskip
\end{array}
$\\\hline
$9$ & $%
\begin{array}
[c]{c}%
\tbinom{-1}{-1},\tbinom{1}{-1},\tbinom{1}{0},\tbinom{0}{1},\tbinom{-1}%
{0}\smallskip
\end{array}
$\\\hline
$10$ & $%
\begin{array}
[c]{c}%
\tbinom{-1}{0},\tbinom{0}{-1},\tbinom{1}{-1},\tbinom{1}{0},\tbinom{0}%
{1},\tbinom{-1}{1}\smallskip
\end{array}
$\\\hline
\end{tabular}
\ \
\]
\setlength\extrarowheight{-3pt}
Defining
\[
\mathcal{A}_{1}:=\left(
\begin{smallmatrix}
2 & 1\\
1 & 1
\end{smallmatrix}
\right)  ,\  \mathcal{A}_{2}:=\left(
\begin{smallmatrix}
0 & 1\\
1 & 1
\end{smallmatrix}
\right)  ,\  \mathcal{A}_{3}:=\left(
\begin{smallmatrix}
1 & 0\\
0 & 1
\end{smallmatrix}
\right)  ,\  \mathcal{A}_{4}:=\left(
\begin{smallmatrix}
0 & -1\\
1 & 0
\end{smallmatrix}
\right)
\]
we see that
\begin{equation}
\Phi_{\mathcal{A}_{i}}(\mathcal{Q}_{i+6})=\mathcal{Q}_{i+6}^{\circ
}\Longrightarrow \lbrack \mathcal{Q}_{i+6}]_{\mathbb{Z}^{2}}=[\mathcal{Q}%
_{i+6}^{\circ}]_{\mathbb{Z}^{2}},\text{ for }i\in \{1,2,3,4\}. \label{AUXEQ2}%
\end{equation}
\textit{ }The entries of the third columns of tables (\ref{tablesofRFP}),
combined with (\ref{AUXEQ1}) and (\ref{AUXEQ2}), give
\[
\sharp(\partial \mathcal{Q}_{i}\cap \mathbb{Z}^{2})+\sharp(\partial
\mathcal{Q}_{i}^{\circ}\cap \mathbb{Z}^{2})=12,\text{ for all }i\in
\{1,\ldots,16\},
\]
and therefore (\ref{12PTFORMULA}) is true.\hfill{}$\square$

\begin{note}\label{NOETHERHIST}
(i) Although the above proof of (\ref{12PTFORMULA}) is elementary, it is not
very enlightening because it does not explain where $12$ comes from. Moreover,
the \textit{earliest} known proof of Theorem \ref{CLASSIFRP} (due to Batyrev
\cite{Batyrev1} (reproduced in \cite[Part A, Theorem 6.1.6]{Batyrev3})) makes
use of formula (\ref{12PTFORMULA})! (This happens implicitly also in
\cite[Theorem 6.10, pp. 108-109]{Dais1}.)\smallskip \newline(ii) The first
purely combinatorial proof of Theorem \ref{CLASSIFRP} is due to Rabinowitz
\cite{Rabin} who classified lattice polygons with exactly one interior lattice
point. (In \cite[Proposition 3, p. 89]{Rabin}, the author failed to include
$[\mathcal{Q}_{13}]_{\mathbb{Z}^{2}}$ but this inaccuracy can be removed
easily by his own techniques.) For other proofs of Theorem \ref{CLASSIFRP}
(which do not use (\ref{12PTFORMULA})) the reader is referred (in
chronological order) to Koelman \cite[Theorem 4.2.4, p. 86]{Koelman}, Sato
\cite[Theorem 6.22, p. 401]{Sato}, Nill \cite[Proposition 3.4.1, pp.
55-57]{Nill}, and Kasprzyk \cite[Proposition 5.2.4, pp. 59-60]{Kasprzyk}.\smallskip \newline (iii) In fact, as it follows from the work of Batyrev
mentioned in (i), the most natural interpretation for the presence of the
number $12$ in (\ref{12PTFORMULA}) arises from the application of the
celebrated \textit{Noether's formula }for the Euler--Poincar\'{e}
characteristic of the structure sheaf of the minimally desingularized compact
toric surface which is associated with a reflexive polygon. (See also
\cite[Theorem 10.5.10, pp. 510-511]{CLS} and Note \ref{ALTPROOF12PTTHM}
below.) Max Noether \cite{Noether} discovered this remarkable formula in
$1870$'s in the framework of the theory of adjoints of algebraic surfaces. For
comments on its early history and for a modern direct proof see Gray
\cite[\S 2]{Gray}, Hulek \cite[\S 3]{Hulek} and Piene \cite{Piene}. Besides,
Noether's formula can be viewed, alternatively, as the Hirzebruch--Riemann--Roch
formula \cite[p. 154]{Hirzebruch2} in (complex) dimension $2.$ (The
coefficient of $t^{2}$ in the expansion of $\frac{t}{\exp(t)-1}$ as a Taylor
series equals $1/12$.)$\smallskip$ \newline(iv) Poonen \& Rodriguez-Villegas
\cite{P-RV} added two new proofs of Theorem \ref{12PTTHM}: (a) by stepping into
the space of reflexive polygons, and (b) by exploiting basic properties of the
universal cover of SL$_{2}(\mathbb{R})$ and of the corresponding modular
(cusp) forms of weight $12.$ (See also Castryck \cite[\S 2]{Castryck}.) Elementary geometric
proofs (using reduction to the parallelogram case and Scott's inequality
\cite{Scott}, respectively) are due to Cencelj, Repov\v{s} \& Skopenkov
\cite{CRS}, and Burns \& O'Keeffe \cite{BOK}. On the other hand, Hille \& Skarke have shown in
\cite[Theorem 1.2]{Hille-Skarke} that there is a one-to-one correspondence
between the set $\left \{  \mathcal{Q}_{1},\ldots,\mathcal{Q}_{16}\right \}  $
and certain relations of the generators $\left(
\begin{smallmatrix}
1 & 1\\
0 & 1
\end{smallmatrix}
\right)  $ and $\left(
\begin{smallmatrix}
1 & 0\\
-1 & 1
\end{smallmatrix}
\right)  $ of SL$_{2}(\mathbb{Z})$ of length $12.$ Higashitani \& Masuda \cite{Higashitani-Masuda} calculated the rotation number of \textit{unimodular sequences} (see also \textrm{\v{Z}ivaljevi\'{c}} \cite{Zivaljevic} for another approach), gave an alternative proof of the \textit{generalised} Twelve-Point Theorem (for \textit{legal loops}) of Poonen \& Rodriguez-Villegas \cite[\S 9.1]{P-RV}, and studied the Ehrhart polynomials of \textit{lattice multipolygons}. Finally, Batyrev \& Schaller \cite[Corollary 5.3]{Batyrev-Schaller} (and independently, Douai \cite[\S8]{Douai}) proved that the so-called \textit{stringy
Libgober--Wood identity} for reflexive polygons is equivalent to (\ref{12PTFORMULA}). 
\end{note}

\begin{remark}\label{NORMIERUNG} For $i\in \{1,...,16\}$ let us define the $\mathbb{Z}^{2}$-polygons
	$\overline{\mathcal{Q}}_{i}$ as follows:
	\setlength\extrarowheight{3pt}
{\scriptsize	\[
	\begin{tabular}
	[c]{|c||c|c|c|c|c|c|c|}\hline
	$i$ & vertices of\emph{ }$\overline{\mathcal{Q}}_{i}$ & $i$ & vertices
	of\emph{ }$\overline{\mathcal{Q}}_{i}$ & $i$ & vertices of\emph{ }%
	$\overline{\mathcal{Q}}_{i}$ & $i$ & vertices of\emph{ }$\overline
	{\mathcal{Q}}_{i}$\\ \hline \hline
	$1$ & $%
	\begin{array}
	[c]{c}%
	\tbinom{0}{1},\tbinom{-1}{-2},\tbinom{1}{1}\smallskip
	\end{array}
	$ & $5$ & $%
	\begin{array}
	[c]{c}%
	\tbinom{0}{1},\tbinom{-1}{0},\tbinom{-1}{-2}\smallskip,\tbinom{1}{1}\smallskip
	\end{array}
	$ & $9$ & $%
	\begin{array}
	[c]{c}%
	\tbinom{0}{1},\tbinom{-1}{0},\tbinom{0}{-1}\smallskip,\tbinom{1}{-1}%
	,\tbinom{1}{1}\smallskip
	\end{array}
	$ & $13$ & $%
	\begin{array}
	[c]{c}%
	\tbinom{0}{1},\tbinom{-2}{-1},\tbinom{1}{-1}\smallskip,\tbinom{1}{1}%
	\smallskip \! \!
	\end{array}
	$\\ \hline
	$2$ & $%
	\begin{array}
	[c]{c}%
	\tbinom{0}{1},\tbinom{-1}{-1},\tbinom{2}{1}\smallskip
	\end{array}
	$ & $6$ & $%
	\begin{array}
	[c]{c}%
	\! \tbinom{0}{1},\tbinom{-1}{0},\tbinom{-1}{-1}\smallskip,\tbinom{0}%
	{-1},\tbinom{1}{1}\! \! \smallskip
	\end{array}
	$ & $10$ & $%
	\begin{array}
	[c]{c}%
	\! \tbinom{0}{1},\tbinom{-1}{0},\tbinom{-1}{-1}\smallskip,\tbinom{0}%
	{-1},\tbinom{1}{0},\tbinom{1}{1}\! \! \! \smallskip
	\end{array}
	$ & $14$ & $%
	\begin{array}
	[c]{c}%
	\tbinom{0}{1},\tbinom{-2}{-1},\tbinom{0}{-1}\smallskip,\tbinom{2}{1}\smallskip
	\end{array}
	$\\ \hline
	$3$ & $%
	\begin{array}
	[c]{c}%
	\tbinom{0}{1},\tbinom{-1}{-1},\tbinom{0}{-1}\smallskip,\tbinom{1}{1}%
	\smallskip \! \!
	\end{array}
	$ & $7$ & $%
	\begin{array}
	[c]{c}%
	\tbinom{0}{1},\tbinom{-3}{-2},\tbinom{2}{1}\smallskip
	\end{array}
	$ & $11$ & $%
	\begin{array}
	[c]{c}%
	\tbinom{0}{1},\tbinom{-1}{0},\tbinom{-1}{-1}\smallskip,\tbinom{1}{-1}%
	,\tbinom{1}{1}\smallskip
	\end{array}
	$ & $15$ & $%
	\begin{array}
	[c]{c}%
	\tbinom{0}{1},\tbinom{-2}{-1},\tbinom{4}{1}\smallskip
	\end{array}
	$\\ \hline
	$4$ & $%
	\begin{array}
	[c]{c}%
	\tbinom{0}{1},\tbinom{-1}{0},\tbinom{0}{-1}\smallskip,\tbinom{1}{1}%
	\smallskip \! \!
	\end{array}
	$ & $8$ & $%
	\begin{array}
	[c]{c}%
	\tbinom{0}{1},\tbinom{-2}{-1},\tbinom{0}{-1}\smallskip,\tbinom{1}{1}\smallskip
	\end{array}
	$ & $12$ & $%
	\begin{array}
	[c]{c}%
	\tbinom{0}{1},\tbinom{-1}{0},\tbinom{1}{-2}\smallskip,\tbinom{1}{1}\smallskip
	\end{array}
	$ & $16$ & $%
	\begin{array}
	[c]{c}%
	\tbinom{0}{1},\tbinom{-3}{-2},\tbinom{3}{1}\smallskip
	\end{array}
	$\\ \hline
	\end{tabular}
	\]}
	\setlength\extrarowheight{-3pt}
\hspace{-0.3cm} It is straightforward to see that $[\overline{\mathcal{Q}}_{i}]_{\mathbb{Z}%
		^{2}}=[\mathcal{Q}_{i}]_{\mathbb{Z}^{2}}$ for all $i\in \{1,...,16\},$
	\[ 
	[\overline{\mathcal{Q}}_{17-j}]_{\mathbb{Z}^{2}}=[\overline{\mathcal{Q}}%
	_{j}^{\circ}]_{\mathbb{Z}^{2}}\ \ \text{for}\ \ j\in \{1,2,3,4,5,6\}, \ \text{and} \ \ [\overline
	{\mathcal{Q}}_{j+6}]_{\mathbb{Z}^{2}}=[\overline{\mathcal{Q}}_{j+6}^{\circ
	}]_{\mathbb{Z}^{2}}\ \ \text{for}\ \ j\in \{1,2,3,4\}.
\]  This \textit{particular choice}
	of representatives from each of the $16$ available equivalence classes is such that $\tbinom{0}{1}\in$ Vert$(\overline{\mathcal{Q}}_{i})$
	\textit{for all} $i\in \{1,...,16\},$ and will be convenient for the
	formulation of Theorem \ref{GLOBALTHM}.  
\end{remark}

\noindent{}$\bullet$ $\ell$-\textbf{Reflexive polygons}. Motivated by
condition \ref{CONDREFL} (iii) one generalises Definition \ref{DEFREFLPOL} as follows:

\begin{definition}
\label{DEFLREFLEXIVE}Let $Q\subset \mathbb{R}^{2}$ be an LDP-polygon of index
$\ell$ (w.r.t. $N$). $Q$ is called $\ell$-\textit{reflexive polygon} (and
$(Q,N)$ $\ell$-\textit{reflexive pair}) if $l_{F}=\ell$ \textit{for all}
$F\in$ Edg$(Q).$ (The terms \textit{reflexive} and \thinspace$1$%
-\textit{reflexive} polygon (or pair) coincide.)
\end{definition}

\begin{proposition}
\label{NOBOUNDARYPNTS}If $(Q,N)$ is an $\ell$-\textit{reflexive pair, then}%
\begin{equation}
\sharp \left(  \partial Q\cap N\right)  =\frac{2\, \text{\emph{area}}_{N}%
(Q)}{\ell}. \label{BOUNDARYAREA}%
\end{equation}

\end{proposition}

\begin{proof}
If $F\in$ Edg$(Q)$ with $\mathbf{n},\mathbf{n}^{\prime}$ as vertices, and
$T_{F}:=$ cov$(\left \{  \mathbf{0},\mathbf{n},\mathbf{n}^{\prime}\right \}  ),$
then $l_{F}=\frac{1}{\sharp \left(  F\cap N\right)  -1}\cdot \frac{\left \vert
\det(\mathbf{n},\mathbf{n}^{\prime})\right \vert }{\det(N)}=\ell$ (by
definition), and%
\[
\text{area}_{N}(Q)=\sum_{F\in \text{ Edg}(Q)}\text{area}_{N}(T_{F})=\frac{\ell
}{2}\left(  \sum_{F\in \text{ Edg}(Q)}\left(  \sharp \left(  F\cap N\right)
-1\right)  \right)  =\frac{\ell}{2}\sharp \left(  \partial Q\cap N\right)
\]
because area$_{N}(T_{F})=\frac{1}{2}\ell \left(  \sharp \left(  F\cap N\right)
-1\right)  $ for all $F\in$ Edg$(Q).$
\end{proof}

\begin{corollary}
\label{ODDNESS1}If $(Q,N)$ is an $\ell$-\textit{reflexive pair and }%
$\sharp \left(  \partial Q\cap N\right)  $ is odd, then\textit{ }$\ell$ is odd.
\end{corollary}

\begin{proof}
By (\ref{PICK}) and (\ref{BOUNDARYAREA}) we have
\[
\sharp \left(  Q\cap N\right)  -1=\text{area}_{N}(Q)+\tfrac{1}{2}\sharp \left(
\partial Q\cap N\right)  =(\tfrac{\ell+1}{2})\sharp \left(  \partial Q\cap
N\right)  ,
\]
Hence, if $\sharp \left(  \partial Q\cap N\right)  $ is odd, then\textit{
}$\ell$ has to be odd.
\end{proof}

\noindent{}Next, we introduce the notation RP$(\ell;N):=\left \{  \left.
\left[  Q\right]  _{N}\in \text{LDP}(\ell;N)\right \vert Q\text{ is an }%
\ell \text{-reflexive polygon}\right \}  ,$ and for every integer $\nu \geq3$ we
set
\[
\text{RP}_{\nu}(\ell;N):=\left \{  \left.  \left[  Q\right]  _{N}\in
\text{RP}(\ell;N)\right \vert \sharp(\text{Vert}(Q))=\nu \right \}  .
\]
As will be seen in the sequel, there are no $\ell$-reflexive polygons
having more than $6$ vertices and there are no $\ell$-reflexive polygons with
$\ell$ even (see Corollaries \ref{Vert6} and \ref{ODDNESS2}). For the time
being, taking into account the precise polygon data
from \cite{Br-Kas} we deduce the following:

\begin{corollary}
The values of the enumerating \ functions $\ell \mapsto \sharp($\emph{RP}$_{\nu
}(\ell;N))$ and $\ell \mapsto \sharp($\emph{RP}$(\ell;N))$ for $\nu
\in \{3,4,5,6\}$ and for odd $\ell \leq25$ are those given in the table\emph{:}%
\setlength\extrarowheight{2pt}
\begin{equation*}
\begin{tabular}{|c||c|c|c|c|c|c|c|c|c|c|c|c|c|}
\hline
$\ell $ & $1$ & $3$ & $5$ & $7$ & $9$ & $11$ & $13$ & $15$ & $17$ & $19$ & $%
21$ & $23$ & $25$ \\ \hline\hline
$\sharp ($\emph{RP}$_{3}(\ell ;N))$ & $5$ & $0$ & $1$ & $6$ & $0$ & $14$ & $%
20$ & $0$ & $28$ & $34$ & $0$ & $42$ & $5$ \\ \hline
$\sharp ($\emph{RP}$_{4}(\ell ;N))$ & $7$ & $0$ & $7$ & $15$ & $0$ & $33$ & $%
43$ & $0$ & $61$ & $69$ & $0$ & $87$ & $27$ \\ \hline
$\sharp ($\emph{RP}$_{5}(\ell ;N))$ & $3$ & $0$ & $3$ & $6$ & $0$ & $12$ & $%
15$ & $0$ & $21$ & $24$ & $0$ & $30$ & $15$ \\ \hline
$\sharp ($\emph{RP}$_{6}(\ell ;N))$ & $1$ & $1$ & $1$ & $2$ & $1$ & $2$ & $3$
& $1$ & $3$ & $4$ & $2$ & $4$ & $3$ \\ \hline
$\sharp ($\emph{RP}$(\ell ;N))$ & $16$ & $1$ & $12$ & $29$ & $1$ & $61$ & $81
$ & $1$ & $113$ & $131$ & $2$ & $163$ & $50$ \\ \hline
\end{tabular}%
\ \ \ 
\end{equation*}
\setlength\extrarowheight{-2pt}
\end{corollary}

\noindent{}In addition, examples \ref{EXAMPLESLREF} show that $%
{\textstyle \bigcup \limits_{\ell \text{ odd}}}
$RP$_{\nu}(\ell;N)$ is an infinite set for all $\nu \in \{3,4,5,6\}.$

\begin{definition}
\label{DEFDUAL}Let $Q\subset \mathbb{R}^{2}$ be an $\ell$-reflexive polygon
(w.r.t. the lattice $N$) \textit{and }$M:=$ Hom$_{\mathbb{Z}}(N,\mathbb{Z}).$
The $M$-polygon%
\setlength\extrarowheight{2pt}
\[
\fbox{$%
\begin{array}
[c]{ccc}
& Q^{\ast}:=\ell Q^{\circ}\subset \mathbb{R}^{2} &
\end{array}
$}%
\]
\setlength\extrarowheight{-2pt}
will be called the \textit{dual} of $Q.$ (The polar and the dual of $Q$
coincide only if $\ell=1.$)
\end{definition}

\begin{proposition}
\label{QQSTAR}If $(Q,N)$ is an $\ell$-\textit{reflexive pair and }$M$ the dual
of $N,$ then $(Q^{\ast},M)$ is again an $\ell$-reflexive pair.
\end{proposition}

\begin{proof}
Since the affine hull of every $F\in$ Edg$(Q)$ is of the form $\left \{
\left.  \mathbf{y}\in \mathbb{R}^{2}\right \vert \left \langle \boldsymbol{\eta
}_{F},\mathbf{y}\right \rangle =-\ell \right \}  $, we have%
\[
Q=\bigcap \limits_{F\in \text{Edg}(Q)}\left \{  \left.  \mathbf{y}\in
\mathbb{R}^{2}\right \vert \left \langle \boldsymbol{\eta}_{F},\mathbf{y}%
\right \rangle \geq-\ell \right \}  \Rightarrow \text{Vert}(Q^{\ast})=\left \{
\boldsymbol{\eta}_{F}\left \vert F\in \text{Edg}(Q)\right.  \right \}  \subset
M,
\]
i.e., Vert$(Q^{\ast})$ consists of primitive lattice points, and
$\mathbf{0}\in$ int$(Q^{\circ})\subseteq$ int$(Q^{\ast}).$ Since the affine
hull of any edge of $Q^{\ast}$ is of the form $\left \{  \left.  \mathbf{x}%
\in \mathbb{R}^{2}\right \vert \left \langle \mathbf{x},-\mathbf{n}\right \rangle
=\ell \right \}  $ for some $\mathbf{n}\in$ Vert$(Q),$ the integral distance
between $\mathbf{0}$ and the edge equals $\ell.$ Hence, $Q^{\ast}$ is an
LDP-polygon of index $\ell$ w.r.t. $M.$
\end{proof}

\begin{note}
In analogy to (\ref{POLARITY1}), (\ref{POLARITY2}), we establish bijections:%
\begin{equation}
\text{Vert}(Q)\ni \mathbf{n}\longmapsto \left \{  \left.  \mathbf{x}\in Q^{\ast
}\right \vert \left \langle \mathbf{x},\mathbf{n}\right \rangle =-\ell \right \}
\in \text{Edg}(Q^{\ast}),\  \label{LDUALITY1}%
\end{equation}
and%
\begin{equation}
\text{Edg}(Q)\ni F\longmapsto \left \{  \left.  \mathbf{x}\in Q^{\ast
}\right \vert \left \langle \mathbf{x},\mathbf{n}\right \rangle =\left \langle
\mathbf{x},\mathbf{n}^{\prime}\right \rangle =-\ell \right \}  =\boldsymbol{\eta
}_{F}\in \text{Vert}(Q^{\ast}), \label{LDUALITY2}%
\end{equation}
where $\mathbf{n},\mathbf{n}^{\prime}$
denote the vertices of $F.$
\end{note}

\noindent{}Based on the involution%
\[
\left \{  \ell \text{-reflexive pairs}\right \}  \ni(Q,N)\longmapsto(Q^{\ast
},M)\in \left \{  \ell \text{-reflexive pairs}\right \},
\]
on (\ref{LDUALITY1}) and (\ref{LDUALITY2}), and on Theorems \ref{CHLATT} and \ref{GLOBALTHM}, we shall give a second proof of the following:

\begin{theorem}
[Twelve-Point Theorem for $\ell$-Reflexive Pairs, \cite{KaNi}]\label{G12PTTHM}%
If $(Q,N)$ is an $\ell$-reflexive pair, then%
\setlength\extrarowheight{2pt}
\begin{equation}
\fbox{$%
\begin{array}
[c]{ccc}
& \sharp \left(  \partial Q\cap N\right)  +\sharp(\partial Q^{\ast}\cap
M)=12. &
\end{array}
$} \label{G12PTFORMULA}%
\end{equation}
\setlength\extrarowheight{-2pt}
\end{theorem}

\begin{examples}
\label{EXAMPLESLREF}Let $\ell$ be a positive odd integer. We define $\ell
$-reflexive polygons w.r.t. the standard lattice $\mathbb{Z}^{2}$ as
follows:\smallskip \  \newline(i) If $3\nmid \ell$ and $5\nmid \ell,$ then
\begin{equation}
Q:=\text{conv}\left(  \left \{  \tbinom{5}{-2\ell},\tbinom{-1}{\ell}%
,\tbinom{-4}{\ell}\right \}  \right)  \label{EXAMPLE1}%
\end{equation}
is an $\ell$-reflexive triangle having%
\begin{equation}
Q^{\ast}:=\ell Q^{\circ}=\text{conv}\left(  \left \{  \tbinom{-\ell}%
{-2},\tbinom{0}{-1},\tbinom{\ell}{3}\right \}  \right)  \label{EXSTAR1}%
\end{equation}
as its dual. Obviously,
\[
\sharp \left(  \partial Q\cap \mathbb{Z}^{2}\right)  =\text{gcd}(6,3\ell
)+\text{gcd}(3,0)+\text{gcd}(9,3\ell)=3\cdot3=9,\  \  \sharp \left(  \partial
Q^{\ast}\cap \mathbb{Z}^{2}\right)  =3.
\]
(ii) If $3\nmid \ell,$ then%
\begin{equation}
Q:=\text{conv}\left(  \left \{  \tbinom{3}{-\ell},\tbinom{-1}{\ell},\tbinom
{-3}{\ell},\tbinom{1}{-\ell}\right \}  \right)  \label{EXAMPLE2}%
\end{equation}
is an $\ell$-reflexive quadrilateral with
\begin{equation}
Q^{\ast}=\text{conv}\left(  \left \{  \tbinom{-\ell}{-2},\tbinom{0}{-1}%
,\tbinom{\ell}{2},\tbinom{0}{1}\right \}  \right)  \label{EXSTAR2}%
\end{equation}
and
\[
\sharp \left(  \partial Q\cap \mathbb{Z}^{2}\right)  =8,\  \  \sharp \left(
\partial Q^{\ast}\cap \mathbb{Z}^{2}\right)  =4.
\]
(iii) If $3\nmid \ell,$ then
\begin{equation}
Q:=\text{conv}\left(  \left \{  \tbinom{3}{-2\ell},\tbinom{1}{0},\tbinom
{-1}{\ell},\tbinom{-2}{\ell},\tbinom{-1}{0}\right \}  \right)  \label{EXAMPLE3}%
\end{equation}
is an $\ell$-reflexive pentagon with%
\begin{equation}
Q^{\ast}=\text{conv}\left(  \left \{  \tbinom{-\ell}{-1},\tbinom{-\ell}%
{-2},\tbinom{0}{-1},\tbinom{\ell}{1},\tbinom{\ell}{2}\right \}  \right)
\label{EXSTAR3}%
\end{equation}
and%
\[
\sharp \left(  \partial Q\cap \mathbb{Z}^{2}\right)  =7,\  \  \sharp \left(
\partial Q^{\ast}\cap \mathbb{Z}^{2}\right)  =5.
\]
(iv) The hexagon
\begin{equation}
Q:=\text{conv}\left(  \left \{  \tbinom{1}{0},\tbinom{-1}{\ell},\tbinom
{-2}{\ell},\tbinom{-1}{0},\tbinom{1}{-\ell},\tbinom{2}{-\ell}\right \}
\right)  \label{EXAMPLE4}%
\end{equation}
is $\ell$-reflexive having%
\begin{equation}
Q^{\ast}=\text{conv}\left(  \left \{  \tbinom{-\ell}{-2},\tbinom{0}{-1}%
,\tbinom{\ell}{1},\tbinom{\ell}{2},\tbinom{0}{1},\tbinom{-\ell}{-1}\right \}
\right)  \label{EXSTAR4}%
\end{equation}
as its dual, with $\left[  Q\right]  _{\mathbb{Z}^{2}}=\left[  Q^{\ast
}\right]  _{\mathbb{Z}^{2}}$ (because anticlockwise rotation through $\pi/2$
maps $Q$ onto $Q^{\ast}$), and%
\[
\sharp \left(  \partial Q\cap \mathbb{Z}^{2}\right)  =\sharp \left(  \partial
Q^{\ast}\cap \mathbb{Z}^{2}\right)  =6.
\]

\end{examples}

\noindent For $\ell = 3$ this is illustrated in Figure \ref{FigHex}. (Here, the $\ell = 3$ case gives an interesting example, because the associated toric log del Pezzo surface is the only log del Pezzo surface among those with Fano index $1$, anticanonical degree $\ge 2$ and singularities of type $(2,3)$ (in our notation), the regular locus of which has \textit{non-trivial} fundamental group. See Corti $\&$ Heuberger \cite[Proposition 1.8 (b), pp. 83-84]{Cor-He}.)

\begin{figure}[h]
	\includegraphics[height=7cm, width=7.5cm]{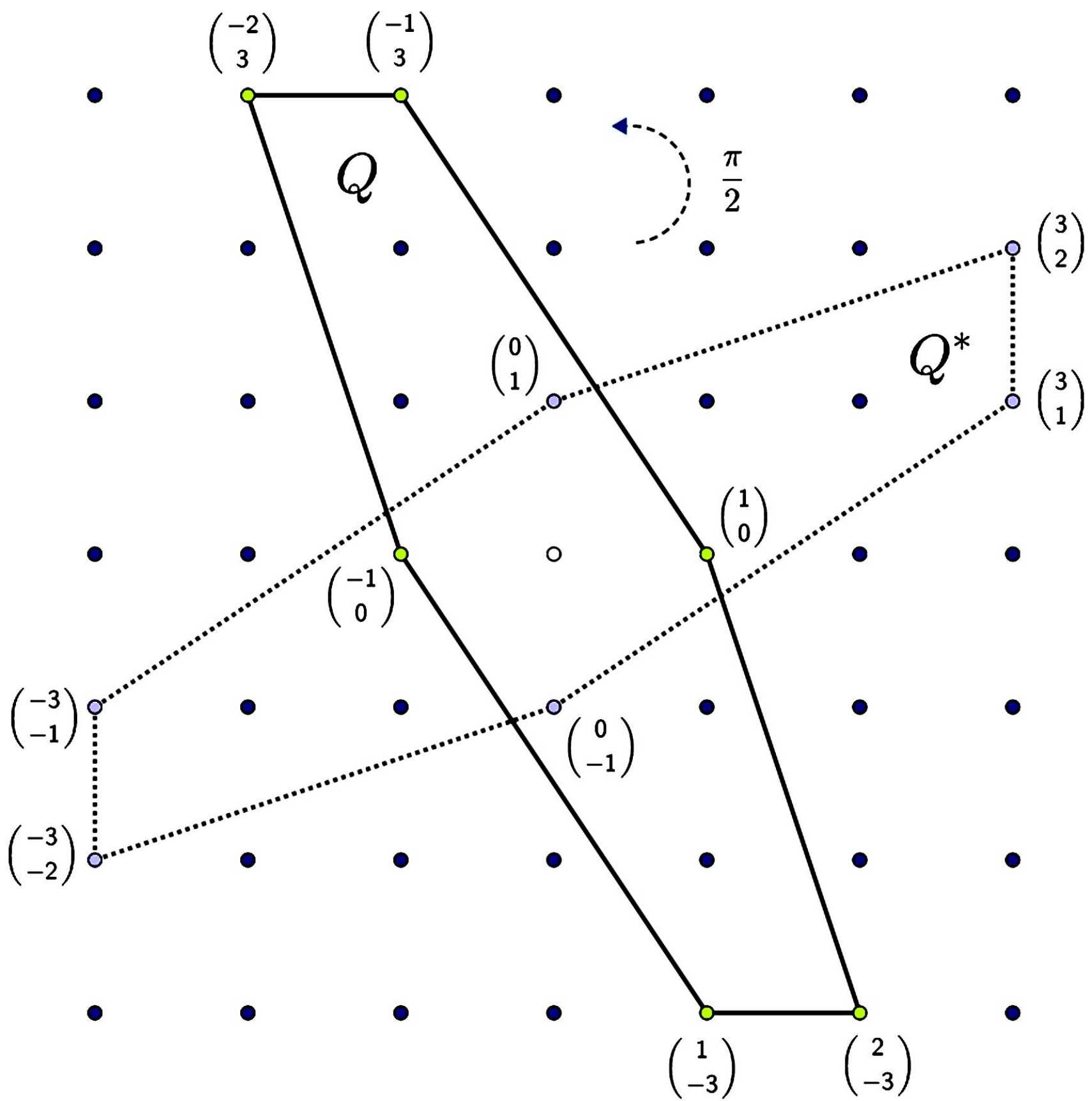}
	\caption{}\label{FigHex}%
\end{figure}

\noindent Since $\ell$ does not admit upper bound, the \textquotedblleft
exhaustion method\textquotedblright \ is apparently not the right method to
verify formula (\ref{G12PTFORMULA}). The first proof of Theorem \ref{G12PTTHM}
given by Kasprzyk and Nill in \cite{KaNi} is purely combinatorial, clear and short, and makes use of the so-called $\ell$-\textit{reflexive loops}. Nevertheless, it offers no essential information about the connection with the \textquotedblleft classical\textquotedblright approach mentioned in \ref{NOETHERHIST} (iii). In \cite[\S1.6]{KaNi} the question was raised, whether there is also another
\textit{direct} argument arising from algebraic geometry in the case of $\ell$-reflexive polygons. Here, maintaining the technique of \textit{lattice change} from \cite[\S2.2]{KaNi} in our toolbox, we shall provide such an argument and a second proof of Theorem \ref{G12PTTHM}: Its disadvantage lies in that it is by no means short (as one has to translate everything into the language of toric varieties, and this requires several steps). On the other hand, among its main advantages are included: (a) Noether's formula remains \textit{again} at least one assured reason for the appearance of $12$ (in combination with other useful formulae in the $\ell>1$ case), and (b) several other results are obtained by transferring the \textit{duality concept} from the $\ell$-reflexive polygons to the corresponding log del Pezzo surfaces.   

More precisely, the paper is organised as follows: In Section \ref{LCONES} we focus on the two non-negative, relatively prime integers $p=p_{\sigma}$ and $q=q_{\sigma}$ parametrising the $N$-cones $\sigma$ and characterising the two-dimensional toric singularities.  Moreover, we describe briefly the
minimal desingularization procedure by means of the negative-regular continued
fraction expansion of $\frac{q}{q-p}$ and by determining the exceptional prime
divisors after the Hilbert basis computation of the corresponding cone. In
section \ref{CTS} we recall the interrelation between  lattice polygons and
compact toric surfaces with a fixed ample divisor, and explain how one
computes the area and the number of lattice points lying on the boundary of a
lattice polygon via intersection numbers. (See Theorems \ref{AmplenessCORR} and \ref{SIGMAPMINDES}.) In
\S\ref{GRAPHS}-\S\ref{TORICLDPSURF} we indicate the manner in which we classify (up to isomorphism)
compact toric surfaces via the \textsc{wve}$^{2}$\textsc{c}-\textit{graphs} and, in particular, toric log del
Pezzo surfaces via LDP-\textit{polygons}.  Giving priority to those log del
Pezzo surfaces which are associated with $\ell$-\textit{reflexive polygons} we
present in \S\ref{COVTR} the geometric meaning of the \textit{lattice change} from \cite[\S2.2]{KaNi}
(which, in a sense, seems to be the standard method of reducing $\ell
$-reflexivity to $1$-reflexivity): One may patch together canonical cyclic
covers over the singularities in order to construct a finite holomorphic map
of degree $\ell$ and to represent the surfaces under consideration as
\textit{global quotients} of \textit{Gorenstein} del Pezzo surfaces by
finite cyclic groups of order $\ell.$ Results of this geometric interpretation
(e.g., Proposition \ref{KAQUADRAT}, concerning the relation between the self-intersection numbers of
the canonical divisors), combined with Noether's formula and other information derived from the desingularization, give rise to a new
proof of Theorem \ref{G12PTTHM} in \S\ref{FIRSTPROOF} and to various consequences of it
(upper bound for $\sharp($Vert$(Q)),$ a proof of \textquotedblleft
oddness\textquotedblright \ of $\ell,$ a new approach of Suyama's formula,
number-theoretic identities involving types of singularities, combinatorial
triples, Dedekind sums etc.). In section \ref{CHARDIFFSEC} we discuss
certain new phenomena which occur in the $\ell>1$ case, and give typical
examples. For instance, the \textit{characteristic differences}
{\small
\[%
\begin{array}
[c]{cc}
& \sharp \left(  \partial Q^{\ast}\cap M\right)  -K_{X(N,\widetilde{\Delta}%
	_{Q})}^{2}=e(X(N,\widetilde{\Delta}_{Q}))-\sharp \left(  \partial Q\cap
N\right) \bigskip  \\
\text{and} & \sharp \left(  \partial Q\cap N\right)  -K_{X(M,\widetilde{\Delta
	}_{Q^{\ast}})}^{2}=e(X(M,\widetilde{\Delta}_{Q^{\ast}}))-\sharp \left(
\partial Q^{\ast}\cap M\right)
\end{array}
\]}
\hspace{-0.2cm}no longer vanish (as in the $\ell=1$ case, where each \thinspace
$1$\thinspace-reflexive polygon has only the origin in its interior), but
are equal to the number of lattice points lying on the boundary of $\mathbf{I}(Q^{\ast})$ and $\mathbf{I}(Q),$ i.e., of the polygons
defined as convex hulls of the (at least $4$, non-collinear) interior lattice
points of $Q^{\ast}$ and $Q,$ respectively. Finally, in \S\ref{BATYRMS} we verify (in the lowest dimension) the existence of a large number of families of \textit{combinatorial mirror pairs} (of certain smooth curves of \textit{high} genus, owing to this new wider framework of duality) and in \S\ref{FINALE} we state some open questions. 

We use tools from discrete and classical toric geometry (adopting the standard terminology
from \cite{CLS}, \cite{Ewald}, \cite{Fulton}, and \cite{Oda}), and some basic facts and formulae from
intersection theory (see \cite[Ch. I]{Beauville},  \cite[\S4]{Dais1}, \cite[\S 2.2-2.4, \S 7.1, and \S 15.2]{FultonITH}, and \cite[\S II(b)]{Mumford}), working over $\mathbb{C}$ and within the analytic category (with complex (analytic) spaces as objects, holomorphic maps as morphisms and biholomorphic maps as isomorphisms).

\section{Two-dimensional lattice cones and toric surfaces\label{LCONES}}

{}\noindent{}$\bullet$ $N$-\textbf{cones}. A $2$-dimensional strongly convex
polyhedral cone in $\mathbb{R}^{2}$ (with the origin $\mathbf{0}\in
\mathbb{R}^{2}$ as its apex and $\mathbf{x}_{1},\mathbf{x}_{2}\in
\mathbb{R}^{2}\mathbb{r}\{ \mathbf{0}\}$ as generators) is a subset $\sigma$ of
$\mathbb{R}^{2}$ of the form
\[
\sigma=\mathbb{R}_{\geq0}\mathbf{x}_{1}+\mathbb{R}_{\geq0}\mathbf{x}%
_{2}:=\left \{  \lambda_{1}\left.  \mathbf{x}_{1}+\lambda_{1}\mathbf{x}%
_{2}\right \vert \lambda_{1},\lambda_{2}\in \mathbb{R}_{\geq0}\right \}  ,
\]
where $\mathbf{x}_{1},\mathbf{x}_{2}$ are $\mathbb{R}$-linearly independent,
and $\sigma \cap \left(  -\sigma \right)  =\{ \mathbf{0}\}.$

\begin{definition}
Let $N$ be a lattice in $\mathbb{R}^{2}$ (as defined in \ref{DEFLATTICE}). A
$2$-dimensional strongly convex polyhedral cone $\sigma=\mathbb{R}_{\geq
0}\mathbf{n}_{1}+\mathbb{R}_{\geq0}\mathbf{n}_{2}\subset \mathbb{R}^{2},$
generated by $\mathbf{n}_{1},\mathbf{n}_{2}\in N,$ will be referred to as
$N$-\textit{cone} having $\mathbf{0}\in \mathbb{R}^{2}$ as its apex
($0$-dimensional face) and $\mathbb{R}_{\geq0}\mathbf{n}_{1}:=\left \{
\lambda \left.  \mathbf{n}_{1}\right \vert \lambda \in \mathbb{R}_{\geq0}\right \}
$ and $\mathbb{R}_{\geq0}\mathbf{n}_{2}:=\left \{  \lambda \left.
\mathbf{n}_{2}\right \vert \lambda \in \mathbb{R}_{\geq0}\right \}  $ as its
\textit{rays} ($1$-dimensional faces). If for $j\in \{1,2\},$ $\mathbf{\breve
{n}}_{j}$ is the unique primitive point (w.r.t. $N$) belonging to the ray
$\mathbb{R}_{\geq0}\mathbf{n}_{j},$then we shall say that $\mathbf{\breve{n}%
}_{j}$ is the \textit{minimal generator} of $\mathbb{R}_{\geq0}\mathbf{n}_{j}$
and that $\mathbf{\breve{n}}_{1},\mathbf{\breve{n}}_{2}$ are the
\textit{minimal generators} of $\sigma.$ (Since $\sigma=\mathbb{R}_{\geq
0}\mathbf{\breve{n}}_{1}+\mathbb{R}_{\geq0}\mathbf{\breve{n}}_{2},$ one can
always replace arbitrary generators of $\sigma$ by the minimal ones). On the
set of $N$-cones we define (as we did on POL$_{\mathbf{0}}(N)$ in
\S \ref{INTRO}) the equivalence relation:%
\[
\sigma_{1}\backsim_{N}\sigma_{2}\underset{\text{def}}{\Longleftrightarrow
}\exists \Psi \in \text{Aut}_{N}(\mathbb{R}^{2}):\Psi \left(  \sigma_{1}\right)
=\sigma_{2}.
\]
If $\sigma_{1}\backsim_{N}\sigma_{2},$ we say that $\sigma_{1}$ and
$\sigma_{2}$ are \textit{equivalent up to umimodular transformation}. If
$\sigma$ is an $N$-cone, we denote by $\left[  \sigma \right]  _{N}:=\left \{
\left.  N\text{-cones }\tau \right \vert \tau \backsim_{N}\sigma \right \}  $ its
equivalence class.
\end{definition}

\begin{definition}
If $\sigma$ is an $N$-cone, then for a fixed basis matrix $\mathcal{B}\in$
GL$_{2}(\mathbb{R})$ of the lattice $N$ we have $N=\Phi_{\mathcal{B}%
}(\mathbb{Z}^{2})$ with $\Phi_{\mathcal{B}}\in$ Aut$(\mathbb{R}^{2}).$ Thus,
we may define the \textit{standard model} of $\sigma$ w.r.t. $\mathcal{B},$
i.e., the $\mathbb{Z}^{2}$-cone $\sigma^{\text{st}}:=\Phi_{\mathcal{B}^{-1}%
}(\sigma).$ By Proposition \ref{BASECHANGE}, $[\sigma^{\text{st}}%
]_{\mathbb{Z}^{2}}$ does not depend on the particular choice of $\mathcal{B}.$
\end{definition}

\noindent{}If the induced bijection $\left[  \sigma \right]
_{N}\longmapsto \lbrack \sigma^{\text{st}}]_{\mathbb{Z}^{2}}$ is taken into account, it is in many
cases convenient to work with the equivalence class of $\sigma^{\text{st}}$
instead of that of $\sigma$ (and with the standard lattice $\mathbb{Z}^{2}$
instead of $N$).

\begin{definition}
Let $\sigma=\mathbb{R}_{\geq0}\mathbf{n}_{1}+\mathbb{R}_{\geq0}\mathbf{n}_{2}$
be an $N$-cone with $\mathbf{n}_{1},\mathbf{n}_{2}$ as minimal generators. The
\textit{multiplicity} of $\sigma$ is the positive integer%
\begin{equation}
\text{mult}_{N}(\sigma):=\frac{\left \vert \det(\mathbf{n}_{1},\mathbf{n}%
_{2})\right \vert }{\det(N)}=\frac{\det(N^{\prime})}{\det(N)}=\sharp(\Pi \cap
N), \label{MULTDEF}%
\end{equation}
where $\Pi:=\left \{  \xi_{1}\mathbf{n}_{1}+\xi_{2}\mathbf{n}_{2}\left \vert
\xi_{1},\xi_{2}\in \left[  0,1\right)  \right.  \right \}  $ and $N^{\prime}$
the sublattice of $N$ having $\{ \mathbf{n}_{1},\mathbf{n}_{2}\}$ as basis (see
Proposition \ref{parallepiped}).
\end{definition}

\begin{proposition}
\noindent{}\label{PQDESCR1}Let $\sigma=\mathbb{R}_{\geq0}\mathbf{n}%
_{1}+\mathbb{R}_{\geq0}\mathbf{n}_{2}$ be an $N$-cone with $\mathbf{n}%
_{1}=\tbinom{n_{11}}{n_{21}},\mathbf{n}_{2}=\tbinom{n_{12}}{n_{22}}\in N$ as
minimal generators, and $\mathcal{B}=\left(
\begin{smallmatrix}
\mathfrak{b}_{11} & \mathfrak{b}_{12}\\
\mathfrak{b}_{21} & \mathfrak{b}_{22}%
\end{smallmatrix}
\right)  $ a basis matrix of $N.$ Denote by
\[
\overline{\mathbf{n}}_{1}=\tbinom{\overline{n}_{11}}{\overline{n}_{21}}%
:=\Phi_{\mathcal{B}^{-1}}\left(  \mathbf{n}_{1}\right)  \in \mathbb{Z}%
^{2}\  \  \text{and}\  \  \overline{\mathbf{n}}_{2}=\tbinom{\overline{n}_{12}%
}{\overline{n}_{22}}:=\Phi_{\mathcal{B}^{-1}}\left(  \mathbf{n}_{2}\right)
\in \mathbb{Z}^{2}%
\]
the minimal generators of its standard model $\sigma^{\text{\emph{st}}}%
:=\Phi_{\mathcal{B}^{-1}}(\sigma)$ w.r.t\emph{.} $\mathcal{B}$, and consider
$\kappa,\lambda \in \mathbb{Z},$ such that
\begin{equation}
\kappa \overline{n}_{11}-\lambda \overline{n}_{21}=1. \label{KAPPALAMDA1}%
\end{equation}
$\emph{(i)}$ If $q:=\left \vert \det(\overline{\mathbf{n}}_{1},\overline
{\mathbf{n}}_{2})\right \vert $ and if $p$ denotes the unique integer with%
\begin{equation}
0\leq p<q\  \  \  \text{and}\  \  \  \kappa \overline{n}_{12}-\lambda \overline
{n}_{22}\equiv p\left(  \operatorname{mod}q\right)  , \label{KAPPALAMDA2}%
\end{equation}
then \emph{gcd}$\left(  p,q\right)  =1,$ and there exists a primitive point
$\overline{\mathbf{n}}_{1}^{\prime}\in \mathbb{Z}^{2},$ such that
$\overline{\mathbf{n}}_{2}=p\overline{\mathbf{n}}_{1}+q\overline{\mathbf{n}%
}_{1}^{\prime},$ where $\left \{  \overline{\mathbf{n}}_{1},\overline
{\mathbf{n}}_{1}^{\prime}\right \}  $ is a basis of $\mathbb{Z}^{2}.\smallskip$
\newline \emph{(ii)} If $\varepsilon$ is the sign of $\det(\overline
{\mathbf{n}}_{1},\overline{\mathbf{n}}_{2})$ and $\mathcal{M}_{\sigma
}:=\left(
\begin{smallmatrix}
\frac{\varepsilon}{q}\left(  \overline{n}_{22}-\overline{n}_{21}p\right)  &
\frac{\varepsilon}{q}\left(  \overline{n}_{11}p-\overline{n}_{12}\right) \\
-\varepsilon \overline{n}_{21} & \varepsilon \overline{n}_{11}%
\end{smallmatrix}
\right)  \in$ \emph{GL}$_{2}(\mathbb{Z}),$ then%
\[
\Phi_{\mathcal{M}_{\sigma}}(\sigma^{\text{\emph{st}}})=\mathbb{R}_{\geq
0}\tbinom{1}{0}+\mathbb{R}_{\geq0}\tbinom{p}{q},\  \  \text{i.e.,}\  \  \left[
\sigma^{\text{\emph{st}}}\right]  _{\mathbb{Z}^{2}}=\left[  \mathbb{R}_{\geq
0}\tbinom{1}{0}+\mathbb{R}_{\geq0}\tbinom{p}{q}\right]  _{\mathbb{Z}^{2}}.
\]
\emph{(iii) }If $\mathbf{n}_{1}^{\prime}:=\Phi_{\mathcal{B}^{-1}}%
(\overline{\mathbf{n}}_{1}^{\prime}),$ then $\left \{  \mathbf{n}%
_{1},\mathbf{n}_{1}^{\prime}\right \}  $ is a basis of $N,$ and $\mathbf{n}%
_{2}=p\mathbf{n}_{1}+q\mathbf{n}_{1}^{\prime}$ with $q=$ \emph{mult}%
$_{N}(\sigma).$ The above integers $p=:p_{\sigma}$ and $q=:q_{\sigma}$
associated with $\sigma$ do not depend on the particular choice of
$\mathcal{B}.$
\end{proposition}

\begin{proof}
For (i) see \cite[Lemma 2.1, p. 222]{Dais2}. (ii) can be checked directly.
(Note that det$(\mathcal{M}_{\sigma})=\varepsilon.$)\smallskip \  \newline(iii)
Obviously, $\Phi_{\mathcal{B}}\left(  \left \{  \overline{\mathbf{n}}%
_{1},\overline{\mathbf{n}}_{1}^{\prime}\right \}  \right)  =\left \{
\mathbf{n}_{1},\mathbf{n}_{1}^{\prime}\right \}  $ is a basis of $N$ and
\[
\mathbf{n}_{2}=\Phi_{\mathcal{B}}(\overline{\mathbf{n}}_{2})=p\Phi
_{\mathcal{B}}(\overline{\mathbf{n}}_{1})+q\Phi_{\mathcal{B}}(\overline
{\mathbf{n}}_{1}^{\prime})=p\mathbf{n}_{1}+q\mathbf{n}_{1}^{\prime}.
\]
On the other hand, since%
\[
\overline{\mathbf{n}}_{1}=\Phi_{\mathcal{B}^{-1}}\left(  \mathbf{n}%
_{1}\right)  =\tfrac{1}{\det(\mathcal{B})}\allowbreak \tbinom{n_{11}%
\mathfrak{b}_{22}-n_{21}\mathfrak{b}_{12}}{n_{21}\mathfrak{b}_{11}%
-n_{11}\mathfrak{b}_{21}}\text{ and }\overline{\mathbf{n}}_{2}=\Phi
_{\mathcal{B}^{-1}}\left(  \mathbf{n}_{2}\right)  =\tfrac{1}{\det
(\mathcal{B})}\allowbreak \tbinom{n_{12}\mathfrak{b}_{22}-n_{22}\mathfrak{b}%
_{12}}{n_{22}\mathfrak{b}_{11}-n_{12}\mathfrak{b}_{21}},
\]
we have%
\[
q=\left \vert \det(\overline{\mathbf{n}}_{1},\overline{\mathbf{n}}%
_{2})\right \vert =\tfrac{1}{\left \vert \det(\mathcal{B})\right \vert ^{2}%
}\left \vert \det(\mathbf{n}_{1},\mathbf{n}_{2})\right \vert \left \vert
\det(\mathcal{B})\right \vert =\frac{\left \vert \det(\mathbf{n}_{1}%
,\mathbf{n}_{2})\right \vert }{\det(N)}=\text{ mult}_{N}\left(  \sigma \right)
.
\]
Therefore, $q=q_{\sigma}$ does not depend on the choice of $\mathcal{B}.$ In
addition, if $\mathcal{B}^{\prime}$ is another basis matrix of $N,$ then
$\mathcal{B}^{\prime}=\mathcal{BA}$ for some $\mathcal{A}=\left(
\begin{smallmatrix}
\mathfrak{a}_{11} & \mathfrak{a}_{12}\\
\mathfrak{a}_{21} & \mathfrak{a}_{22}%
\end{smallmatrix}
\right)  \in$ GL$_{2}(\mathbb{Z})$ (see Proposition \ref{BASECHANGE}). Let
\[
\widetilde{\mathbf{n}}_{1}:=\Phi_{(\mathcal{BA})^{-1}}\left(  \mathbf{n}%
_{1}\right)  =\Phi_{\mathcal{A}^{-1}}\left(  \overline{\mathbf{n}}_{1}\right)
=\tfrac{1}{\det(\mathcal{A})}\allowbreak \tbinom{\mathfrak{a}_{22}\overline
{n}_{11}-\mathfrak{a}_{12}\overline{n}_{21}}{\mathfrak{a}_{11}\overline
{n}_{21}-\mathfrak{a}_{21}\overline{n}_{11}}\text{ and }\widetilde{\mathbf{n}%
}_{2}:=\Phi_{\mathcal{A}^{-1}}\left(  \overline{\mathbf{n}}_{2}\right)
=\tfrac{1}{\det(\mathcal{A})}\allowbreak \tbinom{\mathfrak{a}_{22}\overline
{n}_{12}-\mathfrak{a}_{12}\overline{n}_{22}}{\mathfrak{a}_{11}\overline
{n}_{22}-\mathfrak{a}_{21}\overline{n}_{12}}%
\]
be the minimal generators of the standard model of $\sigma$ w.r.t\emph{.}
$\mathcal{B}^{\prime}.$ We consider $\widetilde{\kappa},\widetilde{\lambda}%
\in \mathbb{Z},$ such that
\[
\widetilde{\kappa}\tfrac{\left(  \mathfrak{a}_{22}\overline{n}_{11}%
-\mathfrak{a}_{12}\overline{n}_{21}\right)  }{\det(\mathcal{A})}%
-\widetilde{\lambda}\tfrac{\left(  \mathfrak{a}_{11}\overline{n}%
_{21}-\mathfrak{a}_{21}\overline{n}_{11}\right)  }{\det(\mathcal{A})}%
=\tfrac{\widetilde{\kappa}\mathfrak{a}_{22}+\widetilde{\lambda}\mathfrak{a}%
_{21}}{\det(\mathcal{A})}\overline{n}_{11}-\tfrac{\widetilde{\kappa
}\mathfrak{a}_{12}+\widetilde{\lambda}\mathfrak{a}_{11}}{\det(\mathcal{A}%
)}\overline{n}_{21}=1.
\]
The integers $\kappa:=\tfrac{\widetilde{\kappa}\mathfrak{a}_{22}%
+\widetilde{\lambda}\mathfrak{a}_{21}}{\det(\mathcal{A})}$ and $\lambda
:=\tfrac{\widetilde{\kappa}\mathfrak{a}_{12}+\widetilde{\lambda}%
\mathfrak{a}_{11}}{\det(\mathcal{A})}$ satisfy (\ref{KAPPALAMDA1}). Hence,%
\[
\kappa \overline{n}_{12}-\lambda \overline{n}_{22}=\widetilde{\kappa}%
\tfrac{\left(  \mathfrak{a}_{22}\overline{n}_{12}-\mathfrak{a}_{12}%
\overline{n}_{22}\right)  }{\det(\mathcal{A})}-\widetilde{\lambda}%
\tfrac{\left(  \mathfrak{a}_{11}\overline{n}_{22}-\mathfrak{a}_{21}%
\overline{n}_{12}\right)  }{\det(\mathcal{A})}\equiv p\left(
\operatorname{mod}q\right)  ,
\]
and $p=p_{\sigma}$ is also independent of the choice of $\mathcal{B}.$
\end{proof}

\begin{note}
\label{pcomment}It should be stressed that $p=p_{\sigma}$ \textit{does depend}
on which minimal generators of $\sigma$ is regarded as first and
which as second (because of the defining conditions (\ref{KAPPALAMDA1}) and
(\ref{KAPPALAMDA2})). For this reason, by writting $\sigma=\mathbb{R}_{\geq
0}\mathbf{n}_{1}+\mathbb{R}_{\geq0}\mathbf{n}_{2},$ with $\mathbf{n}%
_{1},\mathbf{n}_{2}$ as its minimal generators, their ordering will always be
implicit (and $p_{\sigma}$ well-defined). Proposition \ref{PQDESCR2} gives the
precise description of what happens by interchanging $\mathbf{n}_{1}$ and
$\mathbf{n}_{2}$ or, more generally, by replacing $\sigma$ with a
$\tau \backsim_{N}\sigma.$
\end{note}

\begin{definition}
Let $p,q$ two integers with $0\leq p<q$ and gcd$(p,q)=1.$ The \textit{socius
}$\widehat{p}$ of $p$ (w.r.t. $q$) is defined to be the unique integer
satisfying the conditions:
\[
0\leq \widehat{p}<q,\  \  \text{gcd}(\widehat{p},q)=1,\text{ and \ }p\widehat{p}%
\equiv1(\operatorname{mod}q).
\]

\end{definition}

\begin{proposition}
\label{PQDESCR2}If $\sigma,\tau$ are two $N$-cones, then the following
conditions are equivalent\emph{:} \smallskip \newline \emph{(i)} $\left[
\sigma \right]  _{N}=\left[  \tau \right]  _{N}$ \emph{(i.e.,} $\sigma$ and
$\tau$ are \textit{equivalent up to umimodular transformation}%
\emph{).\smallskip} \newline \emph{(ii)} For the integers $p_{\sigma},p_{\tau
},q_{\sigma},q_{\tau}$ associated with $\sigma,\tau$ \emph{(}by Proposition
\emph{\ref{PQDESCR1})} we have $q_{\tau}=q_{\sigma},$ and either $p_{\tau
}=p_{\sigma}$ or $p_{\tau}=\widehat{p}_{\sigma}.$
\end{proposition}

\begin{proof}
Let $\sigma^{\text{st}},\tau^{\text{st}}$ be the standard models of
$\sigma,\tau$ w.r.t. an arbitrary basis matrix $\mathcal{B}$ of the lattice
$N,$ and $\mathcal{M}_{\sigma},\mathcal{M}_{\tau}\in$ GL$_{2}(\mathbb{Z})$ the
corresponding matrices defined in \ref{PQDESCR1}\emph{ }(iii), so that%
\[
\Phi_{\mathcal{M}_{\sigma}}(\sigma^{\text{st}})=\mathbb{R}_{\geq0}\tbinom
{1}{0}+\mathbb{R}_{\geq0}\tbinom{p_{\sigma}}{q_{\sigma}}\  \text{ and \ }%
\Phi_{\mathcal{M}_{\tau}}(\tau^{\text{st}})=\mathbb{R}_{\geq0}\tbinom{1}%
{0}+\mathbb{R}_{\geq0}\tbinom{p_{\tau}}{q_{\tau}}.
\]
(i)$\Rightarrow$(ii) If $\left[  \sigma \right]  _{N}=\left[  \tau \right]
_{N},$ then there is a matrix $\mathcal{A}\in$ GL$_{2}(\mathbb{Z})$ such that
$\Phi_{\mathcal{BAB}^{-1}}(\sigma)=\tau \Rightarrow \Phi_{\mathcal{A}}%
(\sigma^{\text{st}})=\tau^{\text{st}}.$ Hence,%
\[
\Phi_{\mathcal{M}_{\tau}\mathcal{AM}_{\sigma}^{-1}}\left(  \mathbb{R}_{\geq
0}\tbinom{1}{0}+\mathbb{R}_{\geq0}\tbinom{p_{\sigma}}{q_{\sigma}}\right)
=\mathbb{R}_{\geq0}\tbinom{1}{0}+\mathbb{R}_{\geq0}\tbinom{p_{\tau}}{q_{\tau}%
},
\]
and either
\[
\Phi_{\mathcal{M}_{\tau}\mathcal{AM}_{\sigma}^{-1}}\left(  \tbinom{1}%
{0}\right)  =\tbinom{1}{0}\  \text{ and \ }\Phi_{\mathcal{M}_{\tau}%
\mathcal{AM}_{\sigma}^{-1}}\left(  \tbinom{p_{\sigma}}{q_{\sigma}}\right)
=\tbinom{p_{\tau}}{q_{\tau}}%
\]
or%
\[
\Phi_{\mathcal{M}_{\tau}\mathcal{AM}_{\sigma}^{-1}}\left(  \tbinom{1}%
{0}\right)  =\tbinom{p_{\tau}}{q_{\tau}}\  \text{ and \ }\Phi_{\mathcal{M}%
_{\tau}\mathcal{AM}_{\sigma}^{-1}}\left(  \tbinom{p_{\sigma}}{q_{\sigma}%
}\right)  =\tbinom{1}{0}.
\]
Thus, either
\[
\mathcal{M}_{\tau}\mathcal{AM}_{\sigma}^{-1}=\left(
\begin{smallmatrix}
1 & \frac{p_{\tau}-p_{\sigma}}{q_{\sigma}} \medskip \\
0 & \frac{q_{\tau}}{q_{\sigma}}%
\end{smallmatrix}
\right)  \  \text{ or \ }\mathcal{M}_{\tau}\mathcal{AM}_{\sigma}^{-1}=\left(
\begin{smallmatrix}
p_{\tau} & \frac{1-p_{\sigma}p_{\tau}}{q_{\sigma}} \medskip \\
q_{\tau} & -\frac{p_{\sigma}q_{\tau}}{q_{\sigma}}%
\end{smallmatrix}
\right)  .
\]
In the first case $\det(\mathcal{M}_{\tau}\mathcal{AM}_{\sigma}^{-1})$ has to
be equal to $1,$ which means that $q_{\tau}=q_{\sigma}$ and $p_{\tau
}-p_{\sigma}\equiv0(\operatorname{mod}q),$ i.e., $p_{\tau}=p_{\sigma}$
(because $0\leq p_{\sigma},p_{\tau}<q_{\sigma}=q_{\tau}$). In the second case,
$\det(\mathcal{M}_{\tau}\mathcal{AM}_{\sigma}^{-1})=-1,$ i.e., $q_{\tau
}=q_{\sigma}$ and $1-p_{\sigma}p_{\tau}\equiv0(\operatorname{mod}q)\Rightarrow
p_{\tau}=\widehat{p}_{\sigma}.\smallskip$ \newline(ii)$\Rightarrow$(i) We set
$\mathcal{A}:=\mathcal{M}_{\tau}\mathcal{DM}_{\sigma}^{-1}$ with
$\mathcal{D}\in$ GL$_{2}(\mathbb{Z})$ being defined as follows:%
\[
\mathcal{D}:=\left(
\begin{smallmatrix}
1 & 0\\
0 & 1
\end{smallmatrix}
\right)  \text{ if \ }p_{\tau}=p_{\sigma},\  \text{ and \ }\mathcal{D}:=\left(
\begin{smallmatrix}
p_{\tau} & \frac{1-p_{\sigma}p_{\tau}}{q_{\sigma}} \medskip \\
q_{\sigma} & -p_{\sigma}%
\end{smallmatrix}
\right)  \text{ if \ }p_{\tau}=\widehat{p}_{\sigma}.
\]
Obviously, $\Phi_{\mathcal{BAB}^{-1}}(\sigma)=\tau$ with $\mathcal{A}\in$
GL$_{2}(\mathbb{Z}),$ i.e., $\left[  \sigma \right]  _{N}=\left[  \tau \right]
_{N}.$
\end{proof}

\begin{definition}
Let $\sigma$ be an $N$-cone. Since the two integers $p=p_{\sigma}$ and
$\ q=q_{\sigma}$ associated with $\sigma$ (by Proposition \ref{PQDESCR1})\emph{
}parametrise uniquely the equivalence class $\left[  \sigma \right]  _{N}$ up
to replacement of\emph{\ }$p$ by its socius $\widehat{p},$ we shall henceforth
say that $\sigma$ is \textit{of type} $(p,q)$ (or simply that $\sigma$ is a
$(p,q)$-\textit{cone}).
\end{definition}

\begin{definition}
If $\sigma$ is an $N$-cone$,$ then $\sigma^{\vee}:=\left \{  \left.
\mathbf{x}\in \mathbb{R}^{2}\right \vert \left \langle \mathbf{x},\mathbf{y}%
\right \rangle \geq0,\forall \mathbf{y}\in \sigma \right \}  $ is called the
\textit{dual cone} of $\sigma.$
\end{definition}

\begin{proposition}
\label{DUALCNEIDENT}If $\sigma$ is an $N$-cone and $M:=$ \emph{Hom}%
$_{\mathbb{Z}}(N,\mathbb{Z}),$ then $\sigma^{\vee}$ is an $M$-cone, and
$(\sigma^{\vee})^{\vee}=\sigma.$ In particular, if $\sigma$ is a
$(p,q)$-cone\textit{ }with\textit{ }$q>1,$\textit{ then }$\sigma^{\vee}$ is a
$(q-p,q)$-cone.
\end{proposition}

\begin{proof}
For the first assertion see \cite[\S 1.2]{Fulton}. If $\sigma$ is a
$(p,q)$-cone\textit{ }with\textit{ }$q>1,$ then by Propostion \ref{PQDESCR1}
\[
\sigma=\mathbb{R}_{\geq0}\mathbf{n}_{1}+\mathbb{R}_{\geq0}\mathbf{n}%
_{2}=\mathbb{R}_{\geq0}\mathbf{n}_{1}+\mathbb{R}_{\geq0}\mathbf{(}%
p\mathbf{n}_{1}+q\mathbf{n}_{1}^{\prime}),
\]
where $\left \{  \mathbf{n}_{1},\mathbf{n}_{1}^{\prime}\right \}  $ is a basis
of $N.$ Let $\left \{  \mathbf{m}_{1},\mathbf{m}_{1}^{\prime}\right \}  $ be the
dual basis of $M$ (i.e., $\left \langle \mathbf{m}_{1},\mathbf{n}%
_{1}\right \rangle =\left \langle \mathbf{m}_{1}^{\prime},\mathbf{n}_{1}%
^{\prime}\right \rangle =1$ and $\left \langle \mathbf{m}_{1},\mathbf{n}%
_{1}^{\prime}\right \rangle =\left \langle \mathbf{m}_{1},\mathbf{n}_{1}%
^{\prime}\right \rangle =0$). Then
\begin{equation}
\sigma^{\vee}=\mathbb{R}_{\geq0}\mathbf{m}_{1}^{\prime}+\mathbb{R}_{\geq
0}(q\mathbf{m}_{1}-p\mathbf{m}_{1}^{\prime})=\mathbb{R}_{\geq0}\mathbf{m}%
_{1}^{\prime}+\mathbb{R}_{\geq0}((q-p)\mathbf{m}_{1}^{\prime}+q(\mathbf{m}%
_{1}-\mathbf{m}_{1}^{\prime})) \label{GENSIGMADUAL}%
\end{equation}
where $\left \{  \mathbf{m}_{1}^{\prime},\mathbf{m}_{1}-\mathbf{m}_{1}^{\prime
}\right \}  $ is a basis of $M.$ Thus the dual cone $\sigma^{\vee}$ of $\sigma$
is a $(q-p,q)$-cone (because $0<q-p<q$ and gcd$(q-p,q)=1$).
\end{proof}

\noindent{}$\bullet$ \textbf{Hibert basis}. $\sigma \cap N$ is an additive
commutative monoid for any $N$-cone $\sigma.$ It is known (by Gordan's Lemma
\cite[Proposition 1.1 (iii), p. 3]{Oda}) that $\sigma \cap N$ is finitely
generated (as a semigroup), and that among all generating systems there is a
system Hilb$_{N}\left(  \sigma \right)  $ of minimal cardinality, the so-called
\textit{Hilbert basis} of $\sigma$, which is uniquely determined (up to
reordering of its elements) by the following characterisation:
\[
\text{Hilb}_{N}\left(  \sigma \right)  =\left \{  \mathbf{n}\in \sigma \cap \left(
N\smallsetminus \left \{  \mathbf{0}\right \}  \right)  \  \left \vert \
\begin{array}
[c]{l}%
\mathbf{n}\  \text{cannot be expressed as sum of two }\\
\text{other vectors belonging\emph{\ }to }\sigma \cap \left(  N\smallsetminus
\left \{  \mathbf{0}\right \}  \right)
\end{array}
\right.  \right \}  .
\]
\noindent{}$\bullet$ \textbf{Affine toric surfaces}. Let $\sigma$ be an
$N$-cone and $M$ the dual lattice. We set $S_{\sigma}:=\sigma^{\vee}\cap M.$
Let $\mathbb{C}[S_{\sigma}]:=%
{\textstyle \bigoplus \limits_{\mathbf{m}\in S_{\sigma}}}
\mathbb{C}\mathbf{e}(\mathbf{m})$ be the $\mathbb{C}$-algebra with basis
$\left \{  \mathbf{e}(\mathbf{m})\left \vert \mathbf{m}\in S_{\sigma}\right.
\right \}  $ consisting of formal elements which fulfill the exponential law:%
\[
\mathbf{e}(\mathbf{m})\mathbf{e}(\mathbf{m}^{\prime})=\mathbf{e}%
(\mathbf{m}+\mathbf{m}^{\prime}),\  \forall \left(  \mathbf{m},\mathbf{m}%
^{\prime}\right)  \in S_{\sigma}\times S_{\sigma}\  \ [\text{with }%
\mathbf{e}(\mathbf{0})=:1_{\mathbb{C}[S_{\sigma}]}].
\]
Since $S_{\sigma}$ is finitely generated (as a semigroup), $\exists
\mathbf{m}_{1},\ldots,\mathbf{m}_{k}\in S_{\sigma}$ such that $S_{\sigma
}=\mathbb{Z}_{\geq0}\mathbf{m}_{1}+\cdots+\mathbb{Z}_{\geq0}\mathbf{m}_{k}.$
Thus $\mathbb{C}[S_{\sigma}]$ is generated by $\{ \mathbf{e}(\mathbf{m}%
_{1}),\ldots,\mathbf{e}(\mathbf{m}_{k})\}$ and %
$
\mathbb{C}[S_{\sigma}]\cong \mathbb{C}[z_{1},\ldots,z_{k}]/\mathcal{I},
$
where $\mathcal{I}$ denotes the kernel of the $\mathbb{C}$-algebra epimorphism
$\mathbb{C}[z_{1},\ldots,z_{k}]\longrightarrow \mathbb{C}[S_{\sigma}]$ which
maps $z_{j}$ onto $\mathbf{e}(\mathbf{m}_{j})$ for all $j\in \{1,\ldots,k\}.$

\begin{definition}
We denote by $U_{\sigma}$  (or, more precisely, by $U_{{\sigma},N}$, whenever it is necessary to stress that $\sigma$ is an
$N$-cone) \textit{the affine
toric surface} Spec$(\mathbb{C}[S_{\sigma}])$ which is associated with $\sigma$ and has $\mathbb{C}[S_{\sigma
}]$ as coordinate ring. ($U_{\sigma}$ is a $2$-dimensional \textit{normal}
complex analytic variety embedded in $\mathbb{C}^{k}$ as vanishing locus of finitely
many binomials which generate $\mathcal{I}$ ; see \cite[Proposition 1.2, pp.
4-5]{Oda}. To work with the embedding of $U_{\sigma}$ into an affine complex
space of minimal dimension it is enough to replace the arbitrary generating
system $\{ \mathbf{m}_{1},\ldots,\mathbf{m}_{k}\}$ of $S_{\sigma}$ by
Hilb$_{M}\left(  \sigma^{\vee}\right)  .$)
\end{definition}

\noindent{}Next, we use the identifications%
\[
\text{Hom}_{\text{semigr.}}(S_{\sigma},\mathbb{C})\overset{\text{(a)}%
}{\longleftrightarrow}\text{Hom}_{\mathbb{C}\text{-alg.}}(\mathbb{C}%
[S_{\sigma}],\mathbb{C})\overset{\text{(b)}}{\longleftrightarrow}\left \{
\text{points of }U_{\sigma}\right \}  \overset{\text{(c)}}{\longleftrightarrow
}\text{Max-Spec}(\mathbb{C}[S_{\sigma}])
\]
with (a) induced by%
\[
\text{Hom}_{\text{semigr.}}(S_{\sigma},\mathbb{C})\ni u\longmapsto \theta
_{u}\in \text{Hom}_{\mathbb{C}\text{-alg.}}(\mathbb{C}[S_{\sigma}%
],\mathbb{C}),\  \  \theta_{u}(\mathbf{e}(\mathbf{m})):=u(\mathbf{m}%
),\forall \mathbf{m}\in S_{\sigma},
\]
(b) induced by%
\[
\text{Hom}_{\mathbb{C}\text{-alg.}}(\mathbb{C}[S_{\sigma}],\mathbb{C}%
)\ni \theta_{u}\longmapsto(u(\mathbf{m}_{1}),\ldots,u(\mathbf{m}_{k}%
))\in \left \{  \text{points of }U_{\sigma}\right \}  ,
\]
and (c) by%
\[
\left \{  \text{points of }U_{\sigma}\right \}  \ni \mathfrak{p}\longmapsto
\mathfrak{m}_{\mathfrak{p}}:=\left \{  \left.  \varphi \in \mathcal{O}%
_{U_{\sigma},\mathfrak{p}}\right \vert \varphi \left(  \mathfrak{p}\right)
=0\right \}  \in \text{Max-Spec}(\mathbb{C}[S_{\sigma}]):=\left \{
\begin{array}
[c]{c}%
\text{maximal ideals}\\
\text{of }\mathbb{C}[S_{\sigma}]
\end{array}
\right \}  ,
\]
where $\mathcal{O}_{U_{\sigma},\mathfrak{p}}\cong \mathbb{C}[S_{\sigma
}]_{\mathfrak{p}}$ is the local ring of $U_{\sigma}$ at $\mathfrak{p}$ (i.e.,
the stalk of the structure sheaf $\mathcal{O}_{U_{\sigma}}$ of $U_{\sigma}$ at
$\mathfrak{p}$). The standard action of the algebraic torus%
\begin{equation}
 \mathbb{T}:=\mathbb{T}_{N}:=\text{Hom}_{\mathbb{C}\text{-alg.}}(\mathbb{C}[M],\mathbb{C}%
)=\text{Hom}_{\text{semigr.}}(M,\mathbb{C})=\text{Hom}_{\text{gr.}%
}(M,\mathbb{C}^{\times})=N\otimes_{\mathbb{Z}}(\mathbb{C}^{\times})\cong%
(\mathbb{C}^{\times})^{2} \label{ALGTORUS}%
\end{equation}
on (the set of points Hom$_{\text{semigr.}}(S_{\sigma},\mathbb{C})$ of)
$U_{\sigma}$ can be conceived as multiplication of semigroup homomorphisms:%
\begin{equation}
\mathbb{T}\times U_{\sigma}\ni(t,u)\longmapsto t\cdot u\in U_{\sigma}.
\label{TACTION}%
\end{equation}
We denote by orb$(\sigma)\in$ Hom$_{\text{semigr.}}(S_{\sigma},\mathbb{C})$ (or, more precisely, by $\text{orb}_{N}{(\sigma)}$, whenever it is necessary to stress that $\sigma$ is an
$N$-cone)
the unique point of $U_{\sigma}$ remaining fixed under (\ref{TACTION}), i.e.,
the semigroup homomorphism mapping $\mathbf{m}\in S_{\sigma}$ onto $1$
whenever $\left \langle \mathbf{m},\mathbf{y}\right \rangle =0$ for all
$\mathbf{y}\in \sigma,$ and onto $0$ otherwise. By Propositions \ref{PQDESCR2}
and \ref{ISO} the type $(p,q)$ of $\sigma$ (up to replacement of\emph{\ }$p$
by its socius $\widehat{p}$) determines the isomorphism class of the
germ\emph{\ }$\left(  U_{\sigma}\text{\emph{, }orb}\left(  \sigma \right)
\right)  .$

\begin{proposition}
\label{ISO}If $\sigma,\tau$ are two $N$-cones, then the following conditions
are equivalent\emph{:\smallskip} \newline \emph{(i)} $\left[  \sigma \right]
_{N}=\left[  \tau \right]  _{N}$\emph{.\smallskip} \newline \emph{(ii)} There is
a $\mathbb{T}_{N}$-equivariant analytic isomorphism $U_{\sigma}\overset{\cong%
}{\longrightarrow}U_{\tau}$ mapping \emph{orb}$\left(  \sigma \right)  $ onto
\emph{orb}$\left(  \tau \right)  .$
\end{proposition}

\begin{proof}
It follows from \cite[Ch. VI, Theorem 2.11, pp. 222-223]{Ewald}.
\end{proof}

\begin{proposition}
\label{BASICNESS}Let $\sigma$ be an $N$-cone of type $(p,q).$ The following
conditions are equivalent\emph{:\smallskip} \newline \emph{(i)} $q=1$
\emph{(}and consequently, $p=0$\emph{).\smallskip} \newline \emph{(ii)} The
minimal generators of $\sigma$ constitute a basis of $N.\smallskip$
\newline \emph{(iii)} $U_{\sigma}=U_{{\sigma},N}\cong \mathbb{C}^{2}.$
\end{proposition}

\begin{proof}
Since $q=$ mult$_{N}(\sigma)$ (by \ref{PQDESCR1} (iii)), $q=1$ if and only if
the triangle having the origin and the two minimal generators of $\sigma$ as
vertices does not contain any additional lattice point (see (\ref{MULTDEF})).
Hence, the equivalence of (i) and (ii) follows from Proposition
\ref{Basicness}. For the proof of the equivalence of conditions (ii) and (iii)
see \cite[Theorem 1.10, p. 15]{Oda}.
\end{proof}

\noindent{}If the conditions of Proposition \ref{BASICNESS} are satisfied,
then $\sigma$ is said to be a \textit{basic }$N$-\textit{cone}. The non-basic
$N$-cones are characterised by the following:

\begin{proposition}
\label{Sing-Prop}Let $\sigma$ be an $N$-cone of type $(p,q).$ If $q>1,$ then
$p\geq1$ and $\emph{orb}(\sigma)\in U_{\sigma}$ is a cyclic quotient
singularity. \emph{(It is often called} cyclic quotient singularity of type\footnote{In Reid's terminology \cite[p. 370]{Reid-YPG}, it is called \textit{cyclic quotient singularity of type} $\frac{1}{q}(q-p,1).$} $(p,q)$.\emph{)} In particular,
\[
U_{\sigma}=U_{{\sigma},N}\cong \mathbb{C}^{2}/G=\emph{Spec}(\mathbb{C}[z_{1},z_{2}%
]^{G}),\text{ where }G\subset \emph{GL}_{2}\left(\mathbb{C}\right)
\]
is the cyclic group of order $q$ which is generated by \emph{diag}$(\zeta
_{q}^{-p},\zeta_{q})$ \emph{(}with $\zeta_{q}:=$ \emph{exp}$(2\pi \sqrt{-1}%
/q)$\emph{)} and acts on $\mathbb{C}^{2}=$ \emph{Spec}$(\mathbb{C}[z_{1}%
,z_{2}])$ linearly and faithfully.
\end{proposition}

\begin{proof}
\noindent See Oda \cite[Proposition 1.24,
p.30]{Oda}.
\end{proof}

\begin{proposition}
\label{GORENSTPROP}Let $\sigma$ be a non-basic $N$-cone of type $(p,q).$ The
following conditions are equivalent\emph{:\smallskip}\newline \emph{(i)}
$\emph{orb}(\sigma)\in U_{\sigma}$ is a Gorenstein singularity \emph{(}i.e.,
$\mathcal{O}_{U_{\sigma},\emph{orb}(\sigma)}$ is a Gorenstein local
ring\emph{).\smallskip} \newline \emph{(ii)} $p=1.$
\end{proposition}

\begin{proof}
Since the quotient space $U_{\sigma}\cong \mathbb{C}^{2}/G,$ $G:=\left \langle
\text{diag}(\zeta_{q}^{-p},\zeta_{q})\right \rangle ,$ is Gorenstein (as
complex variety) if and only if $G\subset$ $\text{SL}_{2}\left(\mathbb{C}\right)$
(see \cite{Watanabe}), condition (i) is equivalent to $\zeta_{q}^{1-p}=1,$
i.e., to $p=1$ (because $0<p<q$).
\end{proof}

\noindent{}$\bullet$ $N$-\textbf{fans}. A set $\Delta$ consisting of finitely
many $N$-cones and their $0$- and $1$-dimensional faces (i.e., the origin and
their rays) will be referred to as an $N$-\textit{fan} if for any $N$-cones
$\sigma_{1},\sigma_{2}$ belonging to $\Delta$ with $\sigma_{1}\neq \sigma_{2},$
the intersection $\sigma_{1}\cap \sigma_{2}$ is either the singleton
$\{ \mathbf{0}\}$ or a common ray of $\sigma_{1}$ and $\sigma_{2}.$ (We shall
denote by $\Delta \left(  1\right)  $ and $\Delta \left(  2\right)  $ the set of
rays and the set of $N$-cones of $\Delta,$ respectively, and by $\left \vert
\Delta \right \vert $ the \textit{support} of $\Delta,$ i.e., the union of its
elements.) If $\Delta$ is an $N$-fan, then using the so-called Glueing Lemma
(for the affine toric surfaces $U_{\sigma},\sigma \in \Delta \left(  2\right)  $)
one defines the (normal) \textit{toric surface} $X(N,\Delta)$
\textit{associated with} $\Delta.$ The actions of the algebraic torus
(\ref{ALGTORUS}) on the affine toric surfaces $U_{\sigma},\sigma \in
\Delta \left(  2\right)  ,$ defined in (\ref{TACTION}) are compatible with the
patching isomorphisms, giving the natural action of $\mathbb{T}$ on
$X(N,\Delta)$ (which extends the multiplication in $\mathbb{T}$). All the
orbits w.r.t. it are either of the form orb$(\sigma),$ $\sigma \in \Delta \left(
2\right)  ,$ with orb$(\sigma)$ the unique $\mathbb{T}$-fixed point of
$U_{\sigma}\hookrightarrow X(N,\Delta)$ as defined before, or of the form
orb$(\varrho):=$ Hom$_{\text{gr.}}(\varrho^{\perp}\cap M,\mathbb{Cr}\{0\}),$
$\varrho \in \Delta \left(  1\right)  $ (which are $1$-dimensional subvarieties
of $X(N,\Delta)$) with $M:=$ Hom$_{\mathbb{Z}}(N,\mathbb{Z})$ and
$\varrho^{\perp}:=\left \{  \left.  \mathbf{x}\in \mathbb{R}^{2}\right \vert
\left \langle \mathbf{x},\mathbf{y}\right \rangle =0,\forall \mathbf{y}\in
\varrho \right \}  $, or, finally, orb$(\{ \mathbf{0}\})=\mathbb{T}.$ If $D$ is a
Weil divisor on $X(N,\Delta),$ then $D\sim D^{\prime}$ for some $\mathbb{T}%
$-invariant Weil divisor $D^{\prime}$ (with \textquotedblleft$\sim
$\textquotedblright \ meaning linearly equivalent). It is known that every Weil
divisor on $X(N,\Delta)$ is a $\mathbb{Q}$-Cartier divisor (see \cite[p.
65]{Fulton}). We denote by Div$_{\text{W}}^{\mathbb{T}}(X(N,\Delta))$ and
Div$_{\text{C}}^{\mathbb{T}}(X\left(  N,\Delta \right)  )$ the groups of
$\mathbb{T}$-invariant Weil and Cartier divisors on $X(N,\Delta)$,
respectively. The first of them is described as follows:%

\begin{equation}
\text{Div}_{\text{W}}^{\mathbb{T}}(X(N,\Delta))=%
{\textstyle \bigoplus \limits_{\varrho \in \Delta \left(  1\right)  }}
\mathbb{Z\,}\mathbf{V}_{\Delta}(\varrho)\  \  \ [\text{where }\mathbf{V}%
_{\Delta}(\varrho)\text{ denotes the topological closure of orb}(\varrho)].
\label{DIVW}%
\end{equation}

\begin{theorem}
[{$\mathbb{T}$-invariant Cartier divisors, \cite[p. 62]{Fulton}}%
]\label{CARTIERDIVTHM}If $D=%
{\textstyle \sum \limits_{\varrho \in \Delta \left(  1\right)  }}
\mathbb{\lambda}_{\varrho}\mathbb{\,}\mathbf{V}_{\Delta}(\varrho)\in$
\emph{Div}$_{\text{\emph{W}}}^{\mathbb{T}}(X(N,\Delta))$ $\left(
\mathbb{\lambda}_{\varrho}\in \mathbb{Z}\right)  ,$ then%
\begin{equation}
D\in \text{\emph{Div}}_{\text{\emph{C}}}^{\mathbb{T}}(X\left(  N,\Delta \right)
)\Longleftrightarrow \left \{
\begin{array}
[c]{c}%
\text{\emph{for each} }\sigma \in \Delta \left(  2\right)  \text{ \emph{there is
an} }\mathbf{l}_{\sigma}\in M \smallskip \\
\text{ \emph{such that} }\left \langle \mathbf{l}_{\sigma},\mathbf{n}_{\varrho
}\right \rangle =-\lambda_{\varrho},\  \forall \varrho \in \Delta \left(  1\right)
\cap \sigma
\end{array}
\right \}  , \label{CARTIERCRIT}%
\end{equation}
with $\mathbf{n}_{\varrho}$ denoting the minimal generator of $\varrho
\in \Delta \left(  1\right)  .$
\end{theorem}

\noindent{}The group Div$_{\text{C}}^{\mathbb{T}}(X\left(  N,\Delta \right)  )$
can be also described in terms of $\Delta$-support functions. A $\Delta
$-\textit{support function} (\textit{integral w.r.t.} $N$) is a function
$h:\left \vert \Delta \right \vert \longrightarrow \mathbb{R}$ which is linear on
each $N$-cone belonging to $\Delta,$ with $h\left(  \left \vert \Delta
\right \vert \cap N\right)  \subset \mathbb{Z}.$ Let SF$\left(  N,\Delta \right)
$ denote the additive group of all $\Delta$-support functions. Then
\begin{equation}
\text{SF}\left(  N,\Delta \right)  \ni h\longmapsto D_{h}:=-\sum_{\varrho
\in \Delta \left(  1\right)  }h(\mathbf{n}_{\varrho})\mathbf{V}_{\Delta}%
(\varrho)\in \text{Div}_{\text{C}}^{\mathbb{T}}(X\left(  N,\Delta \right)  )
\label{CORRSF1}%
\end{equation}
is a group isomorphism having%
\begin{equation}
\text{Div}_{\text{C}}^{\mathbb{T}}(X\left(  N,\Delta \right)  )\ni
D=\sum_{\varrho \in \Delta \left(  1\right)  }\lambda_{\varrho}\mathbf{V}%
_{\Delta}(\varrho)\longmapsto h_{D}\in \text{SF}\left(  N,\Delta \right)
\label{CORRSF2}%
\end{equation}
as its inverse (i.e., $D=D_{h_{D}}$ and $h=h_{D_{h}},$ cf. \cite[Theorem
4.2.12, p. 184]{CLS}), where
\begin{equation}
\left.  h_{D}\right \vert _{\sigma}(\mathbf{y}):=\left \langle \mathbf{l}%
_{\sigma},\mathbf{y}\right \rangle ,\  \forall \sigma \in \Delta \left(  2\right)
,\  \text{with\ }h_{D}(\mathbf{n}_{\varrho})=-\lambda_{\varrho},\  \forall
\varrho \in \Delta \left(  1\right)  . \label{HDDEF}%
\end{equation}

\begin{note}
[Canonical divisor and the Euler number]\label{NoteCAND}(i) The canonical (divisorial) sheaf
$\omega_{X(N,\Delta)}$ on $X(N,\Delta)$ is isomorphic to $\mathcal{O}%
_{X(N,\Delta)}(-\sum_{\varrho \in \Delta \left(  1\right)  }\mathbf{V}_{\Delta
}(\varrho))$ (see \cite[\S 4.3, pp. 85-86]{Fulton} or \cite[Theorem 8.2.3, pp.
366-367]{CLS}). So we may define the Weil divisor (\ref{KFORMULA}) as
\textit{canonical divisor} of $X(N,\Delta)$:%
\begin{equation}
K_{X(N,\Delta)}:=-\sum_{\varrho \in \Delta \left(  1\right)  }\mathbf{V}_{\Delta
}(\varrho). \label{KFORMULA}%
\end{equation}
(ii) The topological \textit{Euler number}
\[
e(X(N,\Delta)):=\sum_{j=0}^{4}\left(  -1\right)  ^{j}\dim_{\mathbb{R}}%
H^{j}(X(N,\Delta),\mathbb{R})
\]
of $X(N,\Delta)$ equals $\sharp \Delta \left(  2\right)  $ (see \cite[p.
59]{Fulton}).
\end{note}

\noindent{}Let $\Delta$ be an $N$-fan. If $\Delta^{\prime}$ is a
\textit{refinement} of $\Delta$ (i.e., if $\Delta^{\prime}$ is an $N$-fan with
$\left \vert \Delta^{\prime}\right \vert =\left \vert \Delta \right \vert $ and
each $N$-cone of $\Delta$ is a union of $N$-cones of $\Delta^{\prime}$), then
the induced $\mathbb{T}$-equivariant holomorphic map%
\begin{equation}
f:X(N,\Delta^{\prime})\longrightarrow X(N,\Delta) \label{INDUCEDMAP}%
\end{equation}
is proper and birational (see \cite[Corollary 1.17, p. 23]{Oda}). The singular
locus of $X(N,\Delta)$ is
\[
\text{Sing}(X(N,\Delta))=\left \{  \text{orb}(\sigma)\left \vert \sigma \in
\Delta \left(  2\right)  ,\  \sigma \text{ non-basic}\right.  \right \}  .
\]
In the case in which Sing$(X(N,\Delta))\neq \varnothing,$ it is always possible
to construct (by suitable successive $N$-cone subdivisions) a refinement
$\Delta^{\prime}$ of $\Delta$ such that Sing$(X(N,\Delta^{\prime
}))=\varnothing,$ i.e., such that (\ref{INDUCEDMAP}) is a resolution of the
singular points of $X(N,\Delta)$ (a \textit{desingularization} of
$X(N,\Delta)$). The so-called \textit{minimal }desingularization
$f:X(N,\widetilde{\Delta})\longrightarrow X(N,\Delta)$ of a toric surface
$X(N,\Delta)$ (which is unique, up to factorisation by an isomorphism) is that
one arising from the coarsest refinement $\widetilde{\Delta}$ of $\Delta$ with
Sing$(X(N,\widetilde{\Delta}))=\varnothing.$ \medskip

\noindent $\bullet$ \textbf{Intersection numbers.} If $X(N,\Delta)$ is smooth, then the
\textit{intersection number} $D_{1}\cdot D_{2}\in \mathbb{Z}$ of two divisors
$D_{1},D_{2}$ on $X(N,\Delta)$ with compactly supported intersection is
defined in the usual sense (see \cite[2.4.9, p. 40]{FultonITH}). If
$X(N,\Delta)$ is singular and compact, and $D_{1},D_{2}$ two $\mathbb{Q}$-Weil
divisors on $X(N,\Delta),$ we set%
\begin{equation}
D_{1}\cdot D_{2}:=f^{\star}(D_{1})\cdot f^{\star}(D_{2})\in \mathbb{Q},
\label{INTNUMBSING}%
\end{equation}
where $f:X(N,\Delta^{\prime})\longrightarrow X(N,\Delta)$ is an arbitrary
$\mathbb{T}$-equivariant desingularization of $X(N,\Delta)$ and $f^{\star
}(D_{j})$ the pullback of $D_{j},$ $j\in \{1,2\},$ via $f$ in the sense of
Mumford (see \cite[pp. 17-18]{Mumford} and \cite[\S 1]{Dais1}). It is easy to
prove that (\ref{INTNUMBSING}) does not depend on the choice of the
desingularization (cf. Fulton \cite[7.1.16, p. 125]{FultonITH}).\medskip

\noindent{}$\bullet$ \textbf{Continued fractions and minimal desingularization
of} $U_{\sigma}.$ Let $\sigma=\mathbb{R}_{\geq0}\mathbf{n}_{1}+\mathbb{R}%
_{\geq0}\mathbf{n}_{2}$ be an $N$-cone with $\mathbf{n}_{1},\mathbf{n}_{2}$ as
minimal generators. The affine toric surface $U_{\sigma}$ can be viewed as
$X(N,\Delta_{\sigma})$ with $\Delta_{\sigma}:=\left \{  \left \{  \mathbf{0}%
\right \}  ,\mathbb{R}_{\geq0}\mathbf{n}_{1},\mathbb{R}_{\geq0}\mathbf{n}%
_{2},\sigma \right \}  .$ If $\sigma$ is \textit{non-basic }of type $(p,q)$ (as
in Proposition \ref{Sing-Prop}), we consider the negative-regular continued
fraction expansion
\begin{equation}
\dfrac{q}{q-p}=\left[  \! \left[  b_{1},b_{2},\ldots,b_{s}\right]  \! \right]
:=b_{1}-\frac{1}{b_{2}-\cfrac{1}{%
\begin{array}
[c]{cc}%
\ddots & \\
& b_{s-1}-\dfrac{1}{b_{s}}%
\end{array}
}}\ , 
\label{HJCONTFRAC}
\end{equation}
of $\tfrac{q}{q-p}$ (with $b_{s}\geq2$) and define recursively $\mathbf{u}%
_{0},\mathbf{u}_{1},\ldots,\mathbf{u}_{s},\mathbf{u}_{s+1}\in N$ by setting
\begin{equation}
\mathbf{u}_{0}:=\mathbf{n}_{1},\mathbf{u}_{1}:=\frac{1}{q}((q-p)\mathbf{n}%
_{1}+\mathbf{n}_{2}),\  \text{and \ }\mathbf{u}_{j+1}:=b_{j}\mathbf{u}%
_{j}-\mathbf{u}_{j-1},\  \  \forall j\in \{1,\ldots,s\}. \label{LATPOINTSSIGMA1}%
\end{equation}
It is easy to calculate $b_{1},b_{2},\ldots,b_{s}$ (see, e.g., \cite[Lemma 3.4 (i)]{DHH} and \cite{Karmazyn}) and to verify that%
\begin{equation}
\mathbf{u}_{s}=\frac{1}{q}(\mathbf{n}_{1}+(q-\widehat{p})\mathbf{n}%
_{2}),\text{ }\mathbf{u}_{s+1}=\mathbf{n}_{2},\text{ and }b_{j}\geq2\ \ \text{for
all }j\in \{1,\ldots,s\}. \label{LATPOINTSSIGMA2}%
\end{equation}

\begin{note}
(i) $p,\widehat{p},q$ and the sum $b_{1}+\cdots+b_{s}\,$\ are related to each
other via the formula%
\begin{equation}
12\, \text{DS}\left(  p,q\right)  =\sum_{j=1}^{s}\left(  3-b_{j}\right)
+\frac{1}{q}\left(  p+\widehat{p}\, \right)  -2, \label{DSFORMULA}%
\end{equation}
where
\[
\text{DS}\left(  p,q\right)  :=\sum_{\iota=1}^{q-1}\left(  \! \! \! \left(
\frac{\iota}{q}\right)  \! \! \! \right)  \left(  \! \! \! \left(  \frac{p\iota}%
{q}\right)  \! \! \! \right)
\]
is the \textit{Dedekind sum} of the pair $(p,q).$ (For an $x\in \mathbb{Q},$
$\left(  \! \left(  x\right)  \! \right)  :=$ $x-\left \lfloor x\right \rfloor
-\frac{1}{2}$ if $x\notin \mathbb{Z}$ and is $0$ if $x\in \mathbb{Z}%
.$)\smallskip \  \newline(ii) Since $\sigma^{\vee}=\mathbb{R}_{\geq0}%
\mathbf{m}_{1}^{\prime}+\mathbb{R}_{\geq0}(q\mathbf{m}_{1}-p\mathbf{m}%
_{1}^{\prime})$ is a $(q-p,q)$-cone (see (\ref{GENSIGMADUAL})), considering
the negative-regular continued fraction expansion
\[
\dfrac{q}{p}=\dfrac{q}{q-(q-p)}=\left[  \! \left[  b_{1}^{\vee},b_{2}^{\vee
},\ldots,b_{t}^{\vee}\right]  \! \right]
\]
of $\tfrac{q}{p}$ we may define analogously $\mathbf{u}_{0}^{\vee}%
,\mathbf{u}_{1}^{\vee},\ldots,\mathbf{u}_{t}^{\vee},\mathbf{u}_{t+1}^{\vee}\in
M:=$ Hom$_{\mathbb{Z}}(N,\mathbb{Z})$ by setting
\begin{equation}
\mathbf{u}_{0}^{\vee}:=\mathbf{m}_{1}^{\prime},\text{ }\mathbf{u}_{1}^{\vee
}:=\mathbf{m}_{1},\  \text{and \ }\mathbf{u}_{j+1}^{\vee}:=b_{j}^{\vee
}\mathbf{u}_{j}^{\vee}-\mathbf{u}_{j-1}^{\vee},\  \  \forall j\in \{1,\ldots,t\},
\label{LATPOINTSSIGMADUAL1}%
\end{equation}
with%
\begin{equation}
\mathbf{u}_{t}^{\vee}=\frac{1}{q}(\mathbf{m}_{1}^{\prime}+\widehat
{p}(q\mathbf{m}_{1}-p\mathbf{m}_{1}^{\prime})),\text{ }\mathbf{u}_{t+1}^{\vee
}=q\mathbf{m}_{1}-p\mathbf{m}_{1}^{\prime},\text{ and }b_{j}^{\vee}\geq2\text{
for all }j\in \{1,\ldots,t\}. \label{LATPOINTSSIGMADUAL2}%
\end{equation}
It is known (cf. \cite[p. 29]{Oda}) that%
\begin{equation}
\left(  b_{1}+\cdots+b_{s}\right)  -s=\left(  b_{1}^{\vee}+\cdots+b_{t}^{\vee
}\right)  -t=s+t-1. \label{BBDUALREL}%
\end{equation}

\end{note}

\begin{proposition}
\label{HILBB2DIM}If we define
\begin{equation}
\Theta_{\sigma}:=\text{ \emph{conv}}\left(  \sigma \cap \left(  N\mathbb{r}%
\left \{  \mathbf{0}\right \}  \right)  \right)  ,\  \text{\emph{resp.,}%
}\  \  \Theta_{\sigma^{\vee}}:=\text{ \emph{conv}}\left(  \sigma^{\vee}%
\cap \left(  M\mathbb{r}\left \{  \mathbf{0}\right \}  \right)  \right)  ,
\label{BIGTHETAS}%
\end{equation}
and denote by $\partial \Theta_{\sigma}^{\mathbf{cp}}$ \emph{(}resp., by
$\partial \Theta_{\sigma^{\vee}}^{\mathbf{cp}}$\emph{)} the part of the
boundary $\partial \Theta_{\sigma}$ \emph{(}resp., of $\partial \Theta
_{\sigma^{\vee}}$\emph{)} containing only its compact edges, then%
\[
\emph{Hilb}_{N}\left(  \sigma \right)  =\partial \Theta_{\sigma}^{\mathbf{cp}%
}\cap N=\left \{  \mathbf{u}_{j}\  \left \vert \ 0\leq j\leq s\right.
+1\right \}  ,\smallskip \text{ }\emph{Hilb}_{M}\left(  \sigma^{\vee}\right)
=\partial \Theta_{\sigma^{\vee}}^{\mathbf{cp}}\cap M=\left \{  \mathbf{u}%
_{j}^{\vee}\  \left \vert \ 0\leq j\leq t\right.  +1\right \}  .
\]
\emph{(See Figure \ref{Fig.2}.) }
\end{proposition}

\begin{proof}
\noindent This follows from \cite[pp. 26-29]{Oda} and \cite[Theorem 3.16, pp.
226-228]{DHH}.
\end{proof}

\begin{figure}[h]
	\includegraphics[height=5cm, width=10cm]{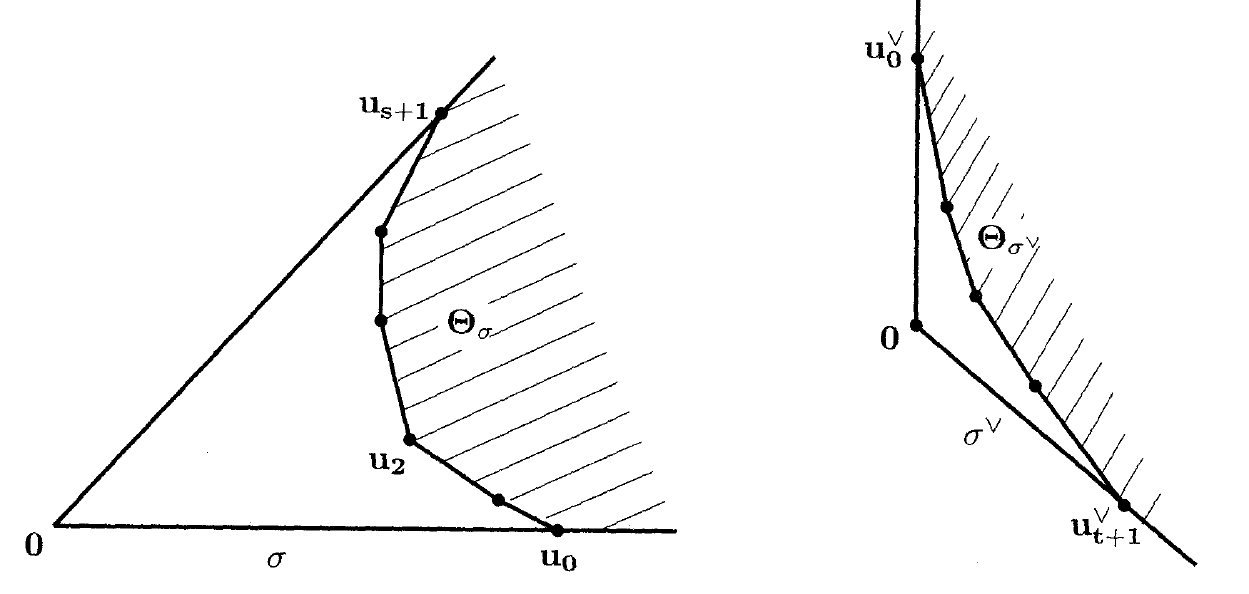}
	\caption{}\label{Fig.2}%
\end{figure}

\begin{theorem}
[Toric version of Hirzebruch's desingularization]\label{HIRZVER}The
refinement
\[
\widetilde{\Delta}_{\sigma}:=\left \{  \left \{  \mathbb{R}_{\geq0}%
\, \mathbf{u}_{j}+\mathbb{R}_{\geq0}\, \mathbf{u}_{j+1}\  \left \vert \ 1\leq
j\leq s\right.  +1\right \}  \  \  \text{\emph{together with their faces}%
}\right \}
\]
of the $N$-fan $\Delta_{\sigma}$ (having the Hilbert basis elements of
$\sigma$ as minimal generators of its rays) contains only basic $N$-cones, and
constitutes the coarsest refinement of $\Delta_{\sigma}$ with this property.
Therefore, it gives rise to the construction of the \emph{minimal}
$\mathbb{T}$-\emph{equivariant resolution}
\[
f:X(N,\widetilde{\Delta}_{\sigma})\longrightarrow X(N,\Delta_{\sigma
})=U_{\sigma}%
\]
of the singular point \emph{orb}$(\sigma)\in U_{\sigma}.$ The exceptional
prime divisors w.r.t. $f$ are $E_{j}:=\mathbf{V}_{\widetilde{\Delta}_{\sigma}%
}\left(  \mathbb{R}_{\geq0}\, \mathbf{u}_{j}\right)  \cong \mathbb{P}%
_{\mathbb{C}}^{1},$ $j\in \left \{  1,\ldots,s\right \}  ,$ and have
self-intersection number $E_{j}^{2}:=E_{j}\cdot E_{j}=-b_{j}.$
\end{theorem}

\begin{proof}
\noindent See Hirzebruch \cite[pp. 15-20]{Hirzebruch1} who constructs
$X(N,\widetilde{\Delta}_{\sigma})$ by resolving the unique singularity lying
over the origin of $\mathbb{C}^{3}$ in the normalisation of the hypersurface
$\left \{  \left.  \left(  z_{1},z_{2},z_{3}\right)  \in \mathbb{C}%
^{3}\right \vert \,z_{1}^{q}-z_{2}z_{3}^{q-p}=0\right \}  ,$ and Oda \cite[pp.
24-30]{Oda} for a proof which uses only tools from toric geometry.
\end{proof}

\section{Compact toric surfaces and lattice polygons\label{CTS}}

\noindent{}Let $N$ be a lattice in $\mathbb{R}^{2}$ and $M:=$ Hom$_{\mathbb{Z}%
}(N,\mathbb{Z}).$ An $N$-fan $\Delta$ is said to be \textit{complete} whenever
$\left \vert \Delta \right \vert =\mathbb{R}^{2}.$

\begin{theorem}
[{cf. \cite[Theorem 1.11, p. 16]{Oda}}]If $\Delta$ is an $N$-fan, then the
toric surface $X\left(  N,\Delta \right)  $ is compact if and only if $\Delta$
is complete.
\end{theorem}

\noindent{}$\bullet$ \textbf{Nef and ample Cartier divisors on compact toric
surfaces}. From now on we shall work with a fixed complete $N$-fan $\Delta.$
For a given $D\in$ Div$_{\text{C}}^{\mathbb{T}}(X\left(  N,\Delta \right)  )$,
we set
\begin{equation}
\fbox{$%
\begin{array}
[c]{ccc}
&
\begin{array}
[c]{r}%
P_{D}:=P_{D,\Delta}:=\left \{  \left.  \mathbf{x}\in \mathbb{R}^{2}\right \vert
\left \langle \mathbf{x},\mathbf{y}\right \rangle \geq h_{D}(\mathbf{y)}%
,\forall \mathbf{y}\in \mathbb{R}^{2}\right \}  \bigskip \\
=\left \{  \left.  \mathbf{x}\in \mathbb{R}^{2}\right \vert \left \langle
\mathbf{x},\mathbf{n}_{\varrho}\right \rangle \geq h_{D}(\mathbf{n}_{\varrho
}\mathbf{)},\forall \varrho \in \Delta \left(  1\right)  \right \}
\end{array}
&
\end{array}
$} \label{PEDE}%
\end{equation}
with $h_{D}\in$ SF$\left(  N,\Delta \right)  $ as defined in (\ref{HDDEF}). (We
write $P_{D,\Delta}$ instead of $P_{D}$ if we wish to emphasise which our
reference fan is.) $P_{D}$ is bounded and its affine hull has dimension
$\leq2.$ Moreover, there is a \textit{unique} set $\left \{  \left.
\mathbf{l}_{\sigma}\right \vert \sigma \in \Delta \left(  2\right)  \right \}
\subset M$ of lattice points satisfying (\ref{CARTIERCRIT}), and%
\[
H^{0}(X\left(  N,\Delta \right)  ,\mathcal{O}_{X\left(  N,\Delta \right)
}(D))=\bigoplus \limits_{\mathbf{m}\in P_{D}\cap M}\mathbb{C}\mathbf{e}%
(\mathbf{m}).
\]
We denote by%
\[
\text{UCSF}\left(  N,\Delta \right)  :=\left \{  h\in \text{SF}\left(
N,\Delta \right)  \left \vert
\begin{array}
[c]{c}%
h(t\mathbf{y}+\left(  1-t\right)  \mathbf{y}^{\prime})\geq th(\mathbf{y}%
)+(1-t)h(\mathbf{y}^{\prime}),\medskip \\
\text{for all }\mathbf{y},\mathbf{y}^{\prime}\in \mathbb{R}^{2}\text{ and }%
t\in \lbrack0,1]
\end{array}
\right.  \right \}
\]
and
\[
\text{SUCSF}\left(  N,\Delta \right)  :=\left \{  h\in \text{UCSF}\left(
N,\Delta \right)  \left \vert
\begin{array}
[c]{c}%
\text{for all }\sigma_{1},\sigma_{2}\in \Delta \left(  2\right)  \text{ with
}\sigma_{1}\neq \sigma_{2}, \medskip\\
\left.  h\right \vert _{\sigma_{1}},\left.  h\right \vert _{\sigma_{2}}\text{
are different linear functions}%
\end{array}
\right.  \right \}
\]
the sets of upper convex and strictly upper convex $\Delta$-support functions, respectively.

\begin{theorem}
[Neffity Criterion]\label{NEFFITY}If $D=%
{\textstyle \sum \limits_{\varrho \in \Delta \left(  1\right)  }}
\mathbb{\lambda}_{\varrho}\mathbb{\,}\mathbf{V}_{\Delta}(\varrho)\in$
\emph{Div}$_{\text{\emph{C}}}^{\mathbb{T}}(X(N,\Delta))$ $\left(
\mathbb{\lambda}_{\varrho}\in \mathbb{Z}\right)  ,$ then the following
conditions are equivalent\emph{:\smallskip} \newline \emph{(i)} $D$ is
basepoint free, i.e., $\mathcal{O}_{X\left(  N,\Delta \right)  }(D)$ is
generated by its global sections.\smallskip \  \newline \emph{(ii)} $h_{D}\in$
\emph{UCSF}$\left(  N,\Delta \right)  .\smallskip$ \newline \emph{(iii)}
$\mathbf{l}_{\sigma}\in P_{D},\  \forall \sigma \in \Delta \left(  2\right)
.\smallskip$ \newline \emph{(iv)} $P_{D}=$ \emph{conv}$\left(  \left \{  \left.
\mathbf{l}_{\sigma}\right \vert \sigma \in \Delta \left(  2\right)  \right \}
\right)  .\smallskip$ \newline \emph{(v)} $\left \{  \left.  \mathbf{l}_{\sigma
}\right \vert \sigma \in \Delta \left(  2\right)  \right \}  $ is the vertex set of
$P_{D}$ (possibly with repetitions)$.\smallskip$ \newline \emph{(vi) }%
$h_{D}(\mathbf{y})=\min \left \{  \left \langle \mathbf{m},\mathbf{y}%
\right \rangle \left \vert \mathbf{m}\in P_{D}\cap M\right.  \right \}
=\min \left \{  \left \langle \mathbf{l}_{\sigma},\mathbf{y}\right \rangle
\left \vert \sigma \in \Delta \left(  2\right)  \right.  \right \}  ,\  \forall
\mathbf{y}\in \mathbb{R}^{2}.\smallskip$ \newline \emph{(vii)} $D\cdot
\mathbf{V}_{\Delta}(\varrho)\geq0,\  \forall \varrho \in \Delta \left(  1\right)
.\smallskip$ \newline \emph{(viii)} $D$ is \textquotedblleft
nef\textquotedblright \ (numerically effective), i.e., the intersection number
of $D$ with any (irreducible compact) curve on $X\left(  N,\Delta \right)  $ is non-negative.
\end{theorem}

\begin{proof}
See \cite[Theorem 2.7, pp. 76-77]{Oda} and \cite[Theorems 6.1.7, pp. 266-267,
and 6.3.12, p. 291]{CLS}.
\end{proof}

\begin{theorem}
[Ampleness Criterion]\label{AMPLENESS}If $D=%
{\textstyle \sum \limits_{\varrho \in \Delta \left(  1\right)  }}
\mathbb{\lambda}_{\varrho}\mathbb{\,}\mathbf{V}_{\Delta}(\varrho)\in$
\emph{Div}$_{\text{\emph{C}}}^{\mathbb{T}}(X(N,\Delta))$ $\left(
\mathbb{\lambda}_{\varrho}\in \mathbb{Z}\right)  ,$ then the following
conditions are equivalent\emph{:\smallskip} \newline \emph{(i)} $D$ is
ample.\smallskip \  \newline \emph{(ii)} $h_{D}\in$ \emph{SUCSF}$\left(
N,\Delta \right)  .\smallskip$ \newline \emph{(iii)} $\mathbf{l}_{\sigma}\in
P_{D},\  \forall \sigma \in \Delta \left(  2\right)  ,$ and $\mathbf{l}_{\sigma
_{1}}\neq \mathbf{l}_{\sigma_{2}}$ for all $\sigma_{1},\sigma_{2}\in
\Delta \left(  2\right)  $ with $\sigma_{1}\neq \sigma_{2}.\smallskip$
\newline \emph{(iv)} $P_{D}$ is an $M$-polygon with \emph{Vert}$(P_{D}%
)=\left \{  \left.  \mathbf{l}_{\sigma}\right \vert \sigma \in \Delta \left(
2\right)  \right \}  $ (without repetitions)$.\smallskip$ \newline \emph{(v)}
$D\cdot \mathbf{V}_{\Delta}(\varrho)>0,\  \forall \varrho \in \Delta \left(
1\right)  .$
\end{theorem}

\begin{proof}
The equivalence of the conditions (i), (ii), (iii) and (iv) follows from
\cite[Lemma 2.12, p. 82]{Oda}. (i)$\Leftrightarrow$(v) is the so-called Toric
Nakai Criterion (see \cite[Theorem 2.18, p. 86, and Remark on p. 87]{Oda}).
\end{proof}

\begin{note}\label{PROJECTIVE}
(i) If $\ D\in$ Div$_{\text{C}}^{\mathbb{T}}(X\left(  N,\Delta \right)  )$ and
$M\cap P_{D}=\{ \mathbf{m}_{0},\mathbf{m}_{1},\ldots,\mathbf{m}_{k}\},$ then
(according to a result of Ewald \& Wessels \cite{EWALD-WESSELS}) $D$ is ample
if and only if $D$ is very ample, i.e., if and only if $D$ is nef and the
map\
\[
X\left(  N,\Delta \right)  \ni \mathfrak{p}\longmapsto \left[  \mathbf{e}%
(\mathbf{m}_{0})(\mathfrak{p}):\mathbf{e}(\mathbf{m}_{1})(\mathfrak{p}%
):\ldots:\mathbf{e}(\mathbf{m}_{k})(\mathfrak{p})\right]  \in \mathbb{P}%
_{\mathbb{C}}^{k}%
\]
is a closed embedding.\smallskip \  \newline (ii) Compact toric surfaces are
\textit{projective} because there exist always ample $\mathbb{T}$-invariant
Cartier divisors on them (see \cite[Proposition 6.3.25, p. 297]{CLS}).
\end{note}

\begin{theorem}
The self-intersection number of a nef divisor $D\in$ \emph{Div}%
$_{\text{\emph{C}}}^{\mathbb{T}}(X(N,\Delta))$ is%
\begin{equation}
D^{2}=2\, \text{\emph{area}}_{M}(P_{D}). \label{DSQUAREPD}%
\end{equation}

\end{theorem}

\begin{proof}
This follows directly from the highest power term in the (generalised)
Riemann-Roch formula:
\[
\int \nolimits_{X(N,\Delta)}\left[  D\right]  ^{2}=2\, \text{area}_{M}(P_{D}).
\]
For details see \cite[Theorems 13.4.1 (b), pp. 655-656, and 13.4.3, p.
657]{CLS}.
\end{proof}

\noindent{}$\bullet$ \textbf{Lattice polygons and normal fans}. For given
$D\in$ Div$_{\text{C}}^{\mathbb{T}}(X\left(  N,\Delta \right)  )$ we have
defined $P_{D}=P_{D,\Delta}$ in (\ref{PEDE}) which is an $M$-polygon whenever
$D$ is ample. Starting with an $M$-polygon$\ P$ one can, conversely, construct
a compact toric surface $X\left(  N,\Sigma_{P}\right)  $ and a distinguished
ample divisor $D_{P}$ on it.

\begin{definition}
Let $P$ be an $M$-polygon. For every $\mathbf{m}\in$ Vert$(P)$ we define the
$M$-cone
\begin{equation}
\varpi_{\mathbf{m}}:=\left \{  \lambda(\mathbf{x}-\mathbf{m})\left \vert
\  \lambda \in \mathbb{R}_{\geq0},\  \mathbf{x}\in P\right.  \right \}  .
\label{PMCONES}%
\end{equation}
It is easy to see that
\setlength\extrarowheight{2pt}
\[
\fbox{$%
\begin{array}
[c]{ccc}
& \Sigma_{P}:=\left \{  \text{ the }N\text{-cones }\left \{  \varpi_{\mathbf{m}%
}^{\vee}\left \vert \  \mathbf{m}\in \text{Vert}(P)\right.  \right \}  \text{
together with their faces}\right \}  &
\end{array}
$}%
\]
\setlength\extrarowheight{-2pt}
is a complete $N$-fan. $\Sigma_{P}$ is called the \textit{normal fan} of $P.$
Denoting by $\boldsymbol{\eta}_{F}\in N$ the (primitive) inward-pointing
normal of an $F\in$ Edg$(P)$ (cf. \ref{DEFLOCIND} (ii)) we observe that
\[
\varpi_{\mathbf{m}}=\left \{  \mathbf{x}\in \mathbb{R}^{2}\left \vert
\  \left \langle x,\boldsymbol{\eta}_{F}\right \rangle \geq0\text{ and
}\left \langle x,\boldsymbol{\eta}_{F^{\prime}}\right \rangle \geq0\right.
\right \}  \  \text{ and \ }\varpi_{\mathbf{m}}^{\vee}=\mathbb{R}_{\geq
0}\boldsymbol{\eta}_{F}+\mathbb{R}_{\geq0}\boldsymbol{\eta}_{F^{\prime}},
\]
where $F,F^{\prime}$ are the edges of $P$ having $\mathbf{m}$ as their common
vertex. Now writing $P$ in the form%
\[
P=\bigcap \limits_{F\in \text{Edg}(P)}\left \{  \left.  \mathbf{x}\in
\mathbb{R}^{2}\right \vert \left \langle \mathbf{x},\boldsymbol{\eta}%
_{F}\right \rangle \geq-\lambda_{F}\right \}  \text{ \ (with }\lambda_{F}%
\in \mathbb{Z},\forall F\in \text{Edg}(P)\text{)}%
\]
we set%
\setlength\extrarowheight{2pt}
\[
\fbox{$%
\begin{array}
[c]{ccc}
& D_{P}:=%
{\displaystyle \sum \limits_{F\in \text{Edg}(P)}}
\lambda_{F}\mathbf{V}_{\Sigma_{P}}(\mathbb{R}_{\geq0}\boldsymbol{\eta}_{F}%
)\in \text{Div}_{\text{C}}^{\mathbb{T}}(X\left(  N,\Sigma_{P}\right)  ). &
\end{array}
$}%
\]
\setlength\extrarowheight{-2pt}
\end{definition}

\begin{proposition}
\label{SUPPFCTPOL}$D_{P}$ is ample and its support function $h_{P}:=h_{D_{P}%
}:\mathbb{R}^{2}\rightarrow \mathbb{R}$ \emph{(}often called \emph{the}
\emph{support function of }$P$\emph{)} is defined as follows:%
\begin{equation}
h_{P}(\mathbf{y}):=\min \left \{  \left \langle \mathbf{x},\mathbf{y}%
\right \rangle \left \vert \mathbf{x}\in P\right.  \right \}  ,\  \forall
\mathbf{y\in}\mathbb{R}^{2}, \label{SUPPFCTP}%
\end{equation}
with $h_{P}(\boldsymbol{\eta}_{F})=-\lambda_{F},\forall F\in$ \emph{Edg}$(P).$
Moreover, $P=P_{D_{P}},$ and%
\begin{equation}
H^{0}(X\left(  N,\Sigma_{P}\right)  ,\mathcal{O}_{X\left(  N,\Sigma
_{P}\right)  }(D_{P}))=\bigoplus \limits_{\mathbf{m}\in P\cap M}\mathbb{C}%
\mathbf{e}(\mathbf{m}). \label{H0BASIS}%
\end{equation}

\end{proposition}

\begin{proof}
See \cite[Propositions 4.2.14, pp. 186-187, and 6.1.10 (a), pp. 269-270]{CLS}.
\end{proof}

\noindent{}Next, we consider the set of pairs%
\[
\mathfrak{Y}:=\left \{  \left(  X\left(  N,\Delta \right)  ,D\right)  \left \vert
\Delta \text{ a complete }N\text{-fan and }D\in \text{Div}_{\text{C}%
}^{\mathbb{T}}(X\left(  N,\Delta \right)  )\text{ ample}\right.  \right \}  .
\]

\begin{theorem}
\label{AmplenessCORR}If $\left(  X\left(  N,\Delta \right)  ,D\right)
\in \mathfrak{Y},$ then $\Delta=\Sigma_{P_{D}}$ and $D=D_{P_{D}}.$ Thus,
\[
\fbox{\ $%
	\begin{array}
	[c]{c}%
	\begin{array}
	[c]{c}%
	\\
	\emph{POL}(M)\ni P\longmapsto \left(  X\left(  N,\Sigma_{P}\right)
	,D_{P}\right)  \in \mathfrak{Y\ }\ \emph{and}\ \  \mathfrak{Y}\ni \left(  X\left(
	N,\Delta \right)  ,D\right)  \longmapsto P_{D}\in \emph{POL}(M)\medskip\\
	\end{array}
	\end{array}
	\ $}%
\]

are bijections which are inverses of each other.
\end{theorem}

\begin{proof}
It follows from Proposition \ref{SUPPFCTPOL}, Theorem \ref{AMPLENESS} and
\cite[Theorem 6.2.1, p. 277]{CLS}.
\end{proof}

\begin{theorem}
\label{SIGMAPMINDES}Let $P$ be an $M$-polygon. If $f:X(N,\widetilde{\Sigma
_{P}})\longrightarrow X(N,\Sigma_{P})$ is the minimal desingularization of
$X(N,\Sigma_{P}),$ then the pullback $f^{\star}(D_{P})$ of the ample divisor
$D_{P}$ is the unique nef divisor on $X(N,\widetilde{\Sigma_{P}})$ for which
$P=P_{D_{P}}=P_{f^{\star}(D_{P})}$ \emph{(}or, more precisely, $P=P_{D_{P}%
,\Sigma_{P}}=P_{f^{\star}(D_{P}),\widetilde{\Sigma_{P}}}$\emph{), }and for
which%
\begin{equation}
\chi(\mathcal{O}_{X(N,\widetilde{\Sigma_{P}})}(kf^{\star}(D_{P}%
)))=\text{\emph{Ehr}}_{M}(P;k),\  \text{ \emph{for all} }k\in \mathbb{Z}_{\geq
0}. \label{hiEhrh}%
\end{equation}
\emph{(\ref{hiEhrh})} gives
\begin{equation}
D_{P}^{2}=f^{\star}(D_{P})^{2}=2\, \text{\emph{area}}_{M}(P) \label{DPSQUARE}%
\end{equation}
and%
\begin{equation}
-f^{\star}(D_{P})\cdot K_{X(N,\widetilde{\Sigma_{P}})}=\sharp \left(  \partial
P\cap M\right)  . \label{DPTHETAP}%
\end{equation}

\end{theorem}

\begin{proof}
For the first assertion see \cite[Proposition 6.2.7, p. 281]{CLS}. ($D_{P}$ has strictly upper convex support function and therefore $f^{\star}(D_{P})$ has upper convex support function, and $P=P_{D_{P}%
	,\Sigma_{P}}=P_{f^{\star}(D_{P}),\widetilde{\Sigma_{P}}}$ because by Theorem \ref{NEFFITY} the $M$-polygon associated with a nef divisor is determined by its support function.) Now let $k$
be an arbitrary non-negative integer. By the Demazure Vanishing Theorem (cf.
\cite[Theorem 2.7 (d), pp. 76-77]{Oda} or \cite[Theorem 9.2.3, p. 410]{CLS})
we obtain
\begin{equation}
\dim_{\mathbb{C}}H^{j}(X(N,\widetilde{\Sigma_{P}}),\mathcal{O}_{X(N,\widetilde
{\Sigma_{P}})}(kf^{\star}(D_{P}))))=0,\text{ for }j=1,2. \label{DEMVANISHING}%
\end{equation}
Thus, the Euler-Poincar\'{e} characteristic
\[
\chi(X(N,\widetilde{\Sigma_{P}}),\mathcal{O}_{X(N,\widetilde{\Sigma_{P}}%
)}(kf^{\star}(D_{P})))=\sum_{j=0}^{2}(-1)^{j}\dim_{\mathbb{C}}H^{j}%
(X(N,\widetilde{\Sigma_{P}}),\mathcal{O}_{X(N,\widetilde{\Sigma_{P}}%
)}(kf^{\star}(D_{P}))))
\]
of the sheaf $\mathcal{O}_{X(N,\widetilde{\Sigma_{P}})}(kf^{\star}(D_{P}))$ is
computed via (\ref{DEMVANISHING}) and (\ref{H0BASIS}) as follows:%
\begin{align*}
\chi(\mathcal{O}_{X(N,\widetilde{\Sigma_{P}})}(kf^{\star}(D_{P})))  &
=\dim_{\mathbb{C}}H^{0}(X(N,\widetilde{\Sigma_{P}}),\mathcal{O}%
_{X(N,\widetilde{\Sigma_{P}})}(kf^{\star}(D_{P}))))\\
&  =\dim_{\mathbb{C}}H^{0}(X(N,\Sigma_{P}),\mathcal{O}_{X(N,\Sigma_{P}%
)}(kD_{P}))\\
&  =\dim_{\mathbb{C}}H^{0}(X(N,\Sigma_{P}),\mathcal{O}_{X(N,\Sigma_{P}%
)}(D_{kP}))\\
&  =\sharp \left(  kP\cap M\right)  ,
\end{align*}
and (\ref{hiEhrh}) is therefore true. On the other hand, Riemann-Roch Theorem
for the projective smooth toric surface $X(N,\widetilde{\Sigma_{P}})$ gives
\begin{equation}
\chi(\mathcal{O}_{X(N,\widetilde{\Sigma_{P}})}(kf^{\star}(D_{P})))=\tfrac
{1}{2}f^{\star}(D_{P})^{2}k^{2}-\tfrac{1}{2}(f^{\star}(D_{P})\cdot
K_{X(N,\widetilde{\Sigma_{P}})})k+1 \label{RR2}%
\end{equation}
(cf. \cite[Ex. 15.2.2, p. 289]{FultonITH}). (\ref{DPSQUARE}) and
(\ref{DPTHETAP}) follow from (\ref{hiEhrh}) and (\ref{Ehrhartpol}) after
coefficient comparison. (To prove (\ref{DPSQUARE}) one may alternatively use
(\ref{DSQUAREPD}) for the divisor $f^{\star}(D_{P}).$)
\end{proof}

\section{\textsc{wve}$^{2}$\textsc{c}-graphs and classification up to
isomorphism\label{GRAPHS}}

\noindent{}Given two complete $N$-fans $\Delta,$ $\Delta^{\prime},$ under
which conditions are the corresponding compact toric surfaces $X\left(
N,\Delta \right)  $ and $X\left(  N,\Delta^{\prime}\right)  $ biholomorphically equivalent, i.e.,
\textit{isomorphic} in the analytic category? The answer to this question requires the use of the so-called
\textquotedblleft \textsc{wve}$^{2}$\textsc{c}-graphs\textquotedblright, the
weights of which are the types of the $N$-cones of $\Delta$ and $\Delta
^{\prime},$ and some additional characteristic integers determined by the
minimal desingularizations\ of $X\left(  N,\Delta \right)  $ and $X\left(
N,\Delta^{\prime}\right)  $ (see below Theorem \ref{CLASSIFTHM}). Let $\Delta$
be a complete $N$-fan, and let $\sigma_{i}=\mathbb{R}_{\geq0}\mathbf{n}%
_{i}+\mathbb{R}_{\geq0}\mathbf{n}_{i+1},$ $i\in \{1,\ldots,\nu \},$ be its
$N$-cones (with $\nu \geq3$ and $\mathbf{n}_{i}$ primitive for all
$i\in \{1,\ldots,\nu \}$), enumerated in such a way that $\mathbf{n}_{1}%
,\ldots,\mathbf{n}_{\nu}$ go \textit{anticlockwise} around the origin exactly
once in this order. (\textit{Convention}. We set $\mathbf{n}_{\nu
+1}:=\mathbf{n}_{1}$ and $\mathbf{n}_{0}:=\mathbf{n}_{\nu}.$ In general, in
definitions and formulae involving enumerated sets of numbers or vectors in
which the index set $\{1,\ldots,\nu \}$ is meant as a cycle, we shall read the
indices $i$ \textquotedblleft$\operatorname{mod}\nu$\textquotedblright, even
if it is not mentioned explicitly.) By (\ref{DIVW}) we have
\begin{equation}
\text{Div}_{\text{W}}^{\mathbb{T}}(X(N,\Delta))=%
{\textstyle \bigoplus \limits_{i=1}^{\nu}}
\mathbb{Z\,}C_{i},\  \text{where }C_{i}:=\mathbf{V}_{\Delta}(\mathbb{R}_{\geq
0}\, \mathbf{n}_{i}),\  \forall i\in \{1,\ldots,\nu \}. \label{DEFCIS}%
\end{equation}
Suppose that $\sigma_{i}$ is a $(p_{i},q_{i})$-cone for all $i\in
\{1,\ldots,\nu \}$ and introduce the notation%
\begin{equation}
I_{\Delta}:=\left \{  \left.  i\in \{1,\ldots,\nu \} \  \right \vert \ q_{i}%
>1\right \}  ,\  \ J_{\Delta}:=\left \{  \left.  i\in \{1,\ldots,\nu
\} \  \right \vert \ q_{i}=1\right \}  , \label{IJNOT}%
\end{equation}
to separate the indices corresponding to non-basic from those corresponding to
basic $N$-cones. Obviously, Sing$(X\left(  N,\Delta \right)  )=\left \{  \left.
\text{orb}(\sigma_{i})\  \right \vert \ i\in I_{\Delta}\right \}  .$ For all
$i\in I_{\Delta}$ consider the negative-regular continued fraction expansions
\begin{equation}
\frac{q_{i}}{q_{i}-p_{i}}=\left[  \! \! \left[  b_{1}^{(i)},b_{2}^{(i)}%
,\ldots,b_{s_{i}}^{(i)}\right]  \! \! \right]  , \label{EXPPQCF}%
\end{equation}
define recursively, in accordance with what is already mentioned in
(\ref{LATPOINTSSIGMA1}) and (\ref{LATPOINTSSIGMA2}) for a single non-basic
$N$-cone, lattice points $\mathbf{u}_{0}^{\left(  i\right)  },\mathbf{u}%
_{1}^{\left(  i\right)  },\ldots,\mathbf{u}_{s_{i}}^{\left(  i\right)
},\mathbf{u}_{s_{i}+1}^{\left(  i\right)  }\in N$ by
\begin{equation}
\mathbf{u}_{0}^{\left(  i\right)  }:=\mathbf{n}_{i},\mathbf{u}_{1}^{\left(
i\right)  }:=\frac{1}{q_{i}}((q_{i}-p_{i})\mathbf{n}_{i}+\mathbf{n}%
_{i+1}),\  \text{and \ }\mathbf{u}_{j+1}^{(i)}:=b_{j}^{(i)}\mathbf{u}_{j}%
^{(i)}-\mathbf{u}_{j-1}^{(i)},\  \  \forall j\in \{1,\ldots,s_{i}\}, \label{UIJS}%
\end{equation}
with%
\begin{equation}
\mathbf{u}_{s_{i}}^{(i)}=\frac{1}{q_{i}}(\mathbf{n}_{i}+(q_{i}-\widehat{p}%
_{i})\mathbf{n}_{i+1}),\text{ }\mathbf{u}_{s_{i}+1}^{\left(  i\right)
}=\mathbf{n}_{i+1},\text{ and }b_{j}^{(i)}\geq2,\  \  \forall j\in
\{1,\ldots,s_{i}\}, \label{BJBIGGER1}%
\end{equation}
and, finally, define the complete $N$-fan%
\[
\widetilde{\Delta}:=\left \{
\begin{array}
[c]{c}%
\text{ the }N\text{-cones }\left \{  \left.  \sigma_{i}\  \right \vert \ i\in
J_{\Delta}\right \}  \text{ and }\\
\left \{  \left.  \mathbb{R}_{\geq0}\, \mathbf{u}_{j}^{(i)}+\mathbb{R}_{\geq
0}\, \mathbf{u}_{j+1}^{(i)}\  \right \vert \ i\in I_{\Delta},\ j\in
\{0,1,\ldots,s_{i}\} \right \}  ,\\
\text{together with their faces}%
\end{array}
\right \}  .
\]
By construction, the induced $\mathbb{T}$-equivariant proper birational map
$f:X(N,\widetilde{\Delta})\longrightarrow X(N,\Delta)$ is the\textit{ }minimal
desingularization of $X(N,\Delta)$ (as we just patch together the minimal
desingularizations of $U_{\sigma_{i}}$'s, $i\in I_{\Delta},$ established in
Theorem \ref{HIRZVER})$.$ Setting
\begin{equation}
\left \{
\begin{array}
[c]{ll}%
E_{j}^{(i)}:=\mathbf{V}_{\widetilde{\Delta}}(\mathbb{R}_{\geq0}\, \mathbf{u}%
_{j}^{(i)}), & \forall i\in I_{\Delta}\text{ and }\forall j\in \{1,\ldots
,s_{i}\},\\
\  & \\
\overline{C}_{i}:=\mathbf{V}_{\widetilde{\Delta}}(\mathbb{R}_{\geq
0}\, \mathbf{n}_{i}), & \forall i\in \{1,\ldots,\nu \},
\end{array}
\right.  \label{DEFEIJCS}%
\end{equation}
we observe that $\overline{C}_{i}$ is the strict transform of $C_{i}$ w.r.t.
$f,$ $E^{(i)}:=\sum_{j=1}^{s_{i}}E_{j}^{(i)}$ the exceptional divisor
replacing the singular point orb$(\sigma_{i})$ via $f,$ and
\[
\text{Div}_{\text{C}}^{\mathbb{T}}(X(N,\widetilde{\Delta}))\otimes
_{\mathbb{Z}}\mathbb{Q=}\left(
{\textstyle \bigoplus \limits_{i=1}^{\nu}}
\mathbb{Q\,}\overline{C}_{i}\right)  \oplus \left(
{\textstyle \bigoplus \limits_{i\in I_{\Delta}}}
{\textstyle \bigoplus \limits_{j=1}^{s_{i}}}
\mathbb{Q\,}E_{j}^{(i)}\right)  .
\]
The \textit{discrepancy divisor}\ w.r.t. $f$ is%
\begin{equation}
K_{X(N,\widetilde{\Delta})}-f^{\star}(K_{X(N,\Delta)})\sim \sum_{i\in
I_{\Delta}}K(E^{(i)}), \label{TORDISCREP}%
\end{equation}
where each of the $K(E^{(i)})$'s is a $\mathbb{Q}$-Cartier divisor (the
\textit{local canonical divisor} of $X(N,\widetilde{\Delta})$ at
orb$(\sigma_{i})$ in the sense of \cite[\S 1]{Dais1}) supported in the union
${\bigcup \limits_{j=1}^{s_{i}}}E_{j}^{(i)}$. If $K_{X(N,\widetilde{\Delta}%
)}\sim f^{\star}(K_{X(N,\Delta)}),$ then $f$ is said to be \textit{crepant}$.$

\begin{proposition}
\label{CREPANTPROP}$f$ is crepant if and only if $X(N,\Delta)$ has at worst
Gorenstein singularities.
\end{proposition}

\begin{proof}
By Proposition \ref{GORENSTPROP}, $X(N,\Delta)$ has at worst Gorenstein
singularities if and only if $p_{i}=1$ for all $i\in I_{\Delta}.$ This can be
shown to be equivalent to $K(E^{(i)})\sim0,$ for all $i\in I_{\Delta},$ by
using the explicit description of $K(E^{(i)})$'s given in \cite[Proposition
4.4, pp. 94-95]{Dais1}.
\end{proof}

\begin{definition}
[The additional characteristic integers $r_{i}$]\label{DEFRIS}For every index
$i\in \{1,\ldots,\nu \}$ we introduce integers $r_{i}$ uniquely determined by
the conditions:%
\begin{equation}
r_{i}\mathbf{n}_{i}=\left \{
\begin{array}
[c]{ll}%
\mathbf{u}_{s_{i-1}}^{(i-1)}+\mathbf{u}_{1}^{(i)}, & \text{if }i\in I_{\Delta
}^{\prime}, \medskip\\
\mathbf{n}_{i-1}+\mathbf{u}_{1}^{(i)}, & \text{if\emph{\ }}i\in I_{\Delta
}^{\prime \prime}, \medskip \\
\mathbf{u}_{s_{i-1}}^{(i-1)}+\mathbf{n}_{i+1}, & \text{if }i\in J_{\Delta
}^{\prime},\medskip \\
\mathbf{n}_{i-1}+\mathbf{n}_{i+1}, & \text{if }i\in J_{\Delta}^{\prime \prime},
\end{array}
\right.  \label{CONDri}%
\end{equation}
where%
\[
I_{\Delta}^{\prime}:=\left \{  \left.  i\in I_{\Delta}\  \right \vert
\ q_{i-1}>1\right \}  ,\  \ I_{\Delta}^{\prime \prime}:=\left \{  \left.  i\in
I_{\Delta}\  \right \vert \ q_{i-1}=1\right \}  ,
\]
and%
\[
J_{\Delta}^{\prime}:=\left \{  \left.  i\in J_{\Delta}\  \right \vert
\ q_{i-1}>1\right \}  ,\  \ J_{\Delta}^{\prime \prime}:=\left \{  \left.  i\in
J_{\Delta}\  \right \vert \ q_{i-1}=1\right \}  ,
\]
with $I_{\Delta},J_{\Delta}$ as in (\ref{IJNOT}). The triples $(p_{i}%
,q_{i},r_{i}),$ $i\in \{1,\ldots,\nu \},$ will be referred to as the
\textit{combinatorial triples} of $\Delta.$
\end{definition}

\begin{note}
\noindent{}(i) The self-intersection number $(E_{j}^{(i)})^{2}\,$of
$E_{j}^{(i)}$ equals $-b_{j}^{(i)}$ for all $i\in I_{\Delta}$ and all
$j\in \{1,\ldots,s_{i}\}$ (cf. Theorem \ref{HIRZVER}). On the other hand, the
opposite $-r_{i}$ of the integer $r_{i}$ defined by (\ref{CONDri}) is nothing
but the self-intersection number $\overline{C}_{i}^{2}$ of $\overline{C}_{i}$
for all $i\in \{1,\ldots,\nu \}.$ For the proof of this fact, as well as for the
computation of the intersection numbers of the rest of pairs of divisors
(\ref{DEFEIJCS}) which constitute the given $\mathbb{Q}$-basis of
Div$_{\text{C}}^{\mathbb{T}}(X(N,\widetilde{\Delta}))\otimes_{\mathbb{Z}%
}\mathbb{Q},$ we refer to \cite[Lemma 4.4, pp. 93-94]{Dais1}.\smallskip
\  \newline(ii) The (fractional) intersection numbers $C_{i}\cdot C_{i^{\prime
}}\in \mathbb{Q}$ of any pair $C_{i},C_{i^{\prime}}$ of generators of
Div$_{\text{W}}^{\mathbb{T}}(X(N,\Delta))$ (with $i,i^{\prime}\in
\{1,\ldots,\nu \}$) can be expressed in terms of the coordinates of the
combinatorial triples of $\Delta$ and the socii of their first coordinates
(see \cite[Lemma 4.7, pp. 97-98]{Dais1}).
\end{note}

\begin{definition}
A \textit{circular graph }is a plane graph whose vertices are points on a
circle and whose edges are the corresponding arcs (on this circle, each of
which connects two consecutive vertices). We say that a circular graph
$\mathfrak{G}$ is $\mathbb{Z}$-\textit{weighted at its vertices} and
\textit{double} $\mathbb{Z}$-\textit{weighted} \textit{at its edges} (and call
it \textsc{wve}$^{2}$\textsc{c}-\textit{graph}, for short) if it is
accompanied by two maps
\[
\left \{  \text{Vertices of }\mathfrak{G}\right \}  \longmapsto \mathbb{Z}%
,\  \left \{  \text{Edges of }\mathfrak{G}\right \}  \longmapsto \mathbb{Z}^{2},
\]
assigning to each vertex an integer and to each edge a pair of integers,
respectively. To the complete $N$-fan $\Delta$ (as described above) we
associate an anticlockwise directed \textsc{wve}$^{2}$\textsc{c}-graph
$\mathfrak{G}_{\Delta}$ with
\[
\left \{  \text{Vertices of }\mathfrak{G}_{\Delta}\right \}  =\{ \mathbf{v}%
_{1},\ldots,\mathbf{v}_{\nu}\} \text{\ }\  \text{and\  \ }\left \{  \text{Edges of
}\mathfrak{G}_{\Delta}\right \}  =\{ \overline{\mathbf{v}_{1}\mathbf{v}_{2}%
},\ldots,\overline{\mathbf{v}_{\nu}\mathbf{v}_{1}}\},
\]
$(\mathbf{v}_{\nu+1}:=\mathbf{v}_{1}),$ by defining its \textquotedblleft
weights\textquotedblright \ as follows:
\[
\mathbf{v}_{i}\longmapsto-r_{i},\text{ \ }\  \  \overline{\mathbf{v}%
_{i}\mathbf{v}_{i+1}}\longmapsto \left(  p_{i},q_{i}\right)  ,\  \forall
i\in \{1,\ldots,\nu \}.
\]
The \textit{reverse graph} $\mathfrak{G}_{\Delta}^{\text{rev}}$ of
$\mathfrak{G}_{\Delta}$ is the directed \textsc{wve}$^{2}$\textsc{c}-graph
which is obtained by changing the double weight $\left(  p_{i},q_{i}\right)  $
of the edge $\overline{\mathbf{v}_{i}\mathbf{v}_{i+1}}$ into $(\widehat{p}%
_{i},q_{i})$ and reversing the initial anticlockwise direction of
$\mathfrak{G}_{\Delta}$ into clockwise direction (see Figure \ref{Fig.1}).
\end{definition}

\begin{figure}[h]
	\includegraphics[height=5cm, width=10cm]{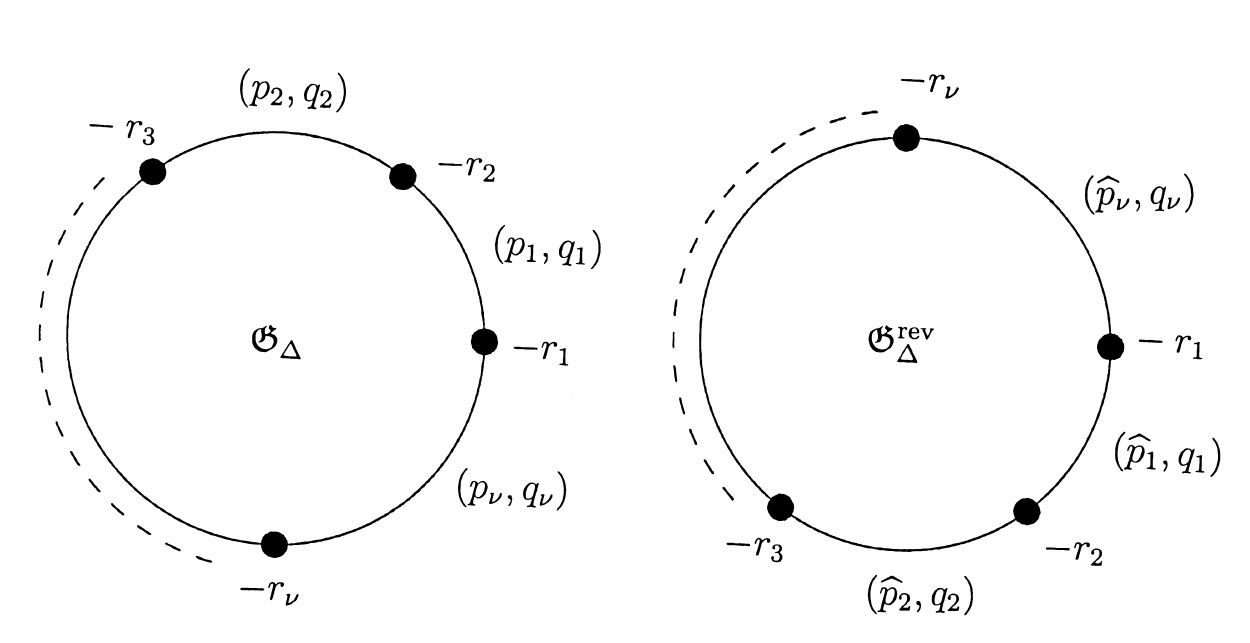}
	\caption{}\label{Fig.1}%
\end{figure}

\begin{theorem}
[Classification up to isomorphism]\label{CLASSIFTHM}Let $\Delta,$
$\Delta^{\prime}$ be two complete $N$-fans. Then the following conditions are
equivalent\emph{:\smallskip \ }\newline \emph{(i)} The compact toric surfaces
$X\left(  N,\Delta \right)  $ and $X\left(  N,\Delta^{\prime}\right)  $ are
isomorphic.\smallskip \  \newline \emph{(ii)} Either $\mathfrak{G}_{\Delta
^{\prime}}\underset{\text{\emph{gr.}}}{\cong}\mathfrak{G}_{\Delta}$ or
$\mathfrak{G}_{\Delta^{\prime}}\underset{\text{\emph{gr}.}}{\cong}%
\mathfrak{G}_{\Delta}^{\text{\emph{rev}}}.$
\end{theorem}

\noindent Here \textquotedblleft$\, \overset{\text{gr.}}{\cong}\,$%
\textquotedblright \ indicates graph-theoretic isomorphism (i.e., a bijection
between the sets of vertices which preserves the corresponding weights). For
further details and for the proof of Theorem \ref{CLASSIFTHM} see
\cite[\S 5]{Dais1}. (\textit{Conventions for the drawings}. When we draw
concrete \textsc{wve}$^{2}$\textsc{c}-graphs in the plane we attach, for
simplification's sake, only the weight $-r_{i}$ at $\mathbf{v}_{i}$ without
mentioning $\mathbf{v}_{i}$ itself, for $i\in \{1,\ldots,\nu \},$ and the double
weight $\left(  p_{i},q_{i}\right)  $ at the edge $\overline{\mathbf{v}%
_{i}\mathbf{v}_{i+1}},$ for $i\in I_{\Delta},$ while we leave edges
$\overline{\mathbf{v}_{i}\mathbf{v}_{i+1}},$ $i\in J_{\Delta},$ without any
decoration in order to switch to the notation for the $\mathbb{Z}$-weighted
circular graphs introduced by Oda in \cite[pp. 42-46]{Oda} which are used for
the study of \textit{smooth} compact toric surfaces.)

\section{Toric log del Pezzo surfaces\label{TORICLDPSURF}}

\noindent{}A compact complex surface is called \textit{log del Pezzo surface}
if \ (a) it has at worst log-terminal singularities, i.e., quotient
singularities, and (b) there is a positive integer multiple of its
anticanonical divisor which is a Cartier ample divisor. The \textit{index} of
a log del Pezzo surface is defined to be the least positive integer having
property (b). Every smooth compact toric surface possesses a unique
\textit{anticanonical model} (in the sense of Sakai \cite{Sakai}) which has to
be a toric log del Pezzo surface; and conversely, every toric log del Pezzo
surface is the anticanonical model of its minimal desingularization (see
\cite[Theorem 6.5, p. 106]{Dais1}).

\begin{definition}
Let $Q\in$ POL$_{\mathbf{0}}(N)$ be an LDP-polygon $($see \ref{DEFLOCIND}
(i)). For each $F\in$ Edg$(Q)$ we define the $N$-cone $\sigma_{F}:=\left \{
\lambda \mathbf{x}\left \vert \mathbf{x}\in F\text{ and }\lambda \in
\mathbb{R}_{\geq0}\right.  \right \}  $ supporting $F,$ and the complete
$N$-fan%
\setlength\extrarowheight{2pt}
\[
\fbox{$%
\begin{array}
[c]{ccc}
& \Delta_{Q}:=\left \{  \left.  \text{the }N\text{-cones }\sigma_{F}\text{
together with their faces}\right \vert F\in \text{Edg}(Q)\right \}  . &
\end{array}
$}%
\]
\setlength\extrarowheight{-2pt}
$X\left(  N,\Delta_{Q}\right)  $ is said to be the compact toric surface
\textit{associated with} $Q,$ and $\mathfrak{G}_{\Delta_{Q}}$ the \textsc{wve}%
$^{2}$\textsc{c}-\textit{graph of} $Q.$
\end{definition}

\begin{proposition}
\label{LDPPROP1}Let $\Delta$ be a complete $N$-fan. Then the following
conditions are equivalent\emph{:\smallskip} \newline \emph{(i)} $X\left(
N,\Delta \right)  $ is a log del Pezzo surface of index $\ell$.\smallskip
\  \newline \emph{(ii)} There exists an LDP-polygon $Q$ of index $\ell$ w.r.t.
$N$ \emph{(}see \emph{\ref{DEFLOCIND} (ii))} such that $\Delta=\Delta_{Q}.$
\end{proposition}

\begin{proof}
Suppose that $\ell:=\min \left \{  k\in \mathbb{Z}_{>0}\left \vert -kK_{X\left(
N,\Delta \right)  }\in \text{Div}_{\text{C}}^{\mathbb{T}}(X(N,\Delta))\text{ and
is ample}\right.  \right \}  .$ By Theorem \ref{CARTIERDIVTHM} and
(\ref{KFORMULA}),
\[
-\ell K_{X\left(  N,\Delta \right)  }=\ell(\sum_{\varrho \in \Delta \left(
1\right)  }\mathbf{V}_{\Delta}(\varrho))\in \text{Div}_{\text{C}}^{\mathbb{T}%
}(X(N,\Delta))
\]
means that there is a unique set $\left \{  \left.  \mathbf{l}_{\sigma
}\right \vert \sigma \in \Delta \left(  2\right)  \right \}  \subset M:=$
Hom$_{\mathbb{Z}}\left(  N,\mathbb{Z}\right)  $ such that $\left \langle
\mathbf{l}_{\sigma},\mathbf{n}_{\varrho}\right \rangle =-\ell \ $for $\varrho
\in \Delta \left(  1\right)  \cap \sigma.$ From the implication (i)$\Rightarrow
$(iv) in Theorem \ref{AMPLENESS} (applied for the divisor $D=-\ell
K_{X(N,\Delta_{Q})}$) we deduce that $P_{-\ell K_{X\left(  N,\Delta \right)  }%
}$ is an $M$-polygon with vertex set Vert$(P_{-\ell K_{X\left(  N,\Delta
\right)  }})=\left \{  \left.  \mathbf{l}_{\sigma}\right \vert \sigma \in
\Delta \left(  2\right)  \right \}  $ (without repetitions). We observe that the
polygon $\frac{1}{\ell}P_{-\ell K_{X\left(  N,\Delta \right)  }}:=$
conv$\left(  \left \{  \frac{1}{\ell}\left.  \mathbf{l}_{\sigma}\right \vert
\sigma \in \Delta \left(  2\right)  \right \}  \right)  $ contains $\mathbf{0}$ in
its interior. Since $\left \langle \frac{1}{\ell}\mathbf{l}_{\sigma}%
,\mathbf{n}_{\varrho}\right \rangle =-1\ $for $\varrho \in \Delta \left(
1\right)  \cap \sigma,$ its polar polygon is
\[
(\frac{1}{\ell}P_{-\ell K_{X\left(  N,\Delta \right)  }})^{\circ}%
=\text{conv}(\left \{  \left.  \mathbf{n}_{\varrho}\right \vert \varrho \in
\Delta \left(  1\right)  \right \}  )\in \text{POL}_{\mathbf{0}}(N)
\]
(by (\ref{POLARITY1}) and (\ref{POLARITY2})). Setting $Q:=(\frac{1}{\ell
}P_{-\ell K_{X\left(  N,\Delta \right)  }})^{\circ}$ we see that $Q$ is an
LDP-polygon because $\mathbf{n}_{\varrho}$ is primitive for all $\varrho
\in \Delta \left(  1\right)  .$ Moreover, by our hypothesis, $\ell=\min \left \{
k\in \mathbb{Z}_{>0}\left \vert \text{Vert}(kQ^{\circ})\subset M\right.
\right \}  .$ Thus, the index of $Q$ equals $\ell,$ $\Delta=\Delta_{Q},$ and
(i)$\Rightarrow$(ii) is true. The proof of the reverse implication
(ii)$\Rightarrow$(i) is similar.
\end{proof}

\begin{proposition}
\label{LDPPROP2} Let $Q,Q^{\prime}\in$ \emph{POL}$_{\mathbf{0}}(N)$ be two
LDP-polygons. Then the following are equivalent\emph{:\smallskip}
\newline \emph{(i) }$X\left(  N,\Delta_{Q}\right)  $ and $X(N,\Delta
_{Q^{\prime}})$ are isomorphic.\smallskip \  \newline \emph{(ii)} $\left[
Q\right]  _{N}=[Q^{\prime}]_{N}.$
\end{proposition}

\begin{proof}
We have $\left[  Q\right]  _{N}=[Q^{\prime}]_{N}$ if and only if there exist a
basis matrix $\mathcal{B}$ of $N$ and a matrix $\mathcal{A}\in$ GL$\left(
2,\mathbb{Z}\right)  $ such that%
\[
\Phi_{\mathcal{BAB}^{-1}}(Q)=Q^{\prime}\Rightarrow \Phi_{\mathcal{AB}^{-1}%
}(Q)=\Phi_{\mathcal{B}^{-1}}(Q^{\prime})\Rightarrow \Phi_{\mathcal{A}%
}(Q^{\text{st}})=Q^{\prime\!\text{ st}}\Rightarrow \lbrack Q^{\text{st}%
}]_{\mathbb{Z}^{2}}=[Q^{\prime}]_{\mathbb{Z}^{2}},
\]
where $Q^{\text{st}},Q^{\prime\!\text{ st}}$ are the standard models of $Q,$ and
$Q^{\prime},$ respectively, w.r.t. $\mathcal{B}$. It is a easy to verify that
this is equivalent to
\[
\mathfrak{G}_{\Delta_{Q^{\prime \text{ st}}}}\underset{\text{gr.}}{\cong%
}\left \{
\begin{array}
[c]{ll}%
\mathfrak{G}_{\Delta_{Q^{\text{st}}}},\smallskip & \text{if }\det
(\mathcal{A})=1,\\
\mathfrak{G}_{\Delta_{Q^{\text{st}}}}^{\text{rev}}, & \text{if }%
\det(\mathcal{A})=-1,
\end{array}
\right.  \Longleftrightarrow \mathfrak{G}_{\Delta_{Q^{\prime}}}\underset
{\text{gr.}}{\cong}\left \{
\begin{array}
[c]{ll}%
\mathfrak{G}_{\Delta_{Q}},\smallskip & \text{if }\det(\mathcal{A})=1,\\
\mathfrak{G}_{\Delta_{Q}}^{\text{rev}}, & \text{if }\det(\mathcal{A})=-1.
\end{array}
\right.
\]
Thus, (ii)$\Leftrightarrow$(i) can be seen to be true by making use of Theorem
\ref{CLASSIFTHM}.
\end{proof}

\begin{note}
\label{NOTELF}(i) By Propositions \ref{LDPPROP1} and \ref{LDPPROP2} the
following map is a bijection:%
\[
\left \{
\begin{array}
[c]{c}%
\text{ equivalence classes }\medskip \\
\text{of LDP-polygons}\medskip \\
\text{of index }\ell \text{ w.r.t. \textquotedblleft}\backsim_{N}%
\text{\textquotedblright}%
\end{array}
\right \}  \ni \left[  Q\right]  _{N}\longmapsto \left[  X\left(  N,\Delta
_{Q}\right)  \right]  \in \left \{
\begin{array}
[c]{c}%
\text{ isomorphism classes of toric } \medskip\\
\text{log Del Pezzo surfaces}\medskip \\
\text{of index }\ell \text{ }%
\end{array}
\right \}  .
\]
Thus, the classification of toric log del Pezzo surfaces of index $\ell$
\textit{up to isomorphism} is equivalent to the classification of LDP-polygons
of index $\ell$ \textit{up to unimodular transformation}.\smallskip
\  \newline(ii) Let $Q\in$ POL$_{\mathbf{0}}(N)$ be an LDP-polygon. Enumerating
the edges, say, $F_{1},\ldots,F_{\nu}$ (and the vertices $\mathbf{n}%
_{1},\ldots,\mathbf{n}_{\nu}$) of $\ Q$ anticlockwise (as in \S \ref{GRAPHS}),
with $F_{i}:=$ conv$(\left \{  \mathbf{n}_{i},\mathbf{n}_{i+1}\right \}  )$ and
$T_{F_{i}}:=$ conv$(\left \{  \mathbf{0},\mathbf{n}_{i},\mathbf{n}%
_{i+1}\right \}  ),$ $i\in \left \{  1,\ldots,\nu \right \}  ,$ and assuming that
the $N$-cone $\sigma_{i}:=\sigma_{F_{i}}=\mathbb{R}_{\geq0}\mathbf{n}%
_{i}+\mathbb{R}_{\geq0}\mathbf{n}_{i+1}$ supporting $F_{i}$ is of type
$(p_{i},q_{i}),$ we obtain
\begin{equation}
q_{i}=\text{mult}_{N}(\sigma_{i})=\frac{\det(\mathbf{n}_{i},\mathbf{n}_{i+1}%
)}{\det(N)}=2\, \text{area}_{N}(T_{F_{i}}),\text{ }\forall i\in \{1,\ldots
,\nu \}. \label{QIFORMULA}%
\end{equation}
By Proposition \ref{PQDESCR1} there exist a basis matrix $\mathcal{B}$ of $N$
and a matrix $\mathcal{M}_{\sigma_{i}}\in$ GL$\left(  2,\mathbb{Z}\right)  $
such that
\[
\Phi_{\mathcal{M}_{\sigma_{i}}\mathcal{B}^{-1}}(\sigma_{i})=\Phi
_{\mathcal{M}_{\sigma_{i}}}(\sigma_{i}^{\text{st}})=\mathbb{R}_{\geq0}%
\tbinom{1}{0}+\mathbb{R}_{\geq0}\tbinom{p_{i}}{q_{i}},
\]
where $\sigma_{i}^{\text{st}}$ is the standard model of $\sigma_{i}$ w.r.t.
$\mathcal{B}.$ Taking into account that
\[
\text{mult}_{N}(\sigma_{i})=\text{mult}_{\mathbb{Z}^{2}}(\sigma_{i}%
^{\text{st}})=\text{mult}_{\mathbb{Z}^{2}}(\Phi_{\mathcal{M}_{\sigma_{i}}%
}(\sigma_{i}^{\text{st}}))=q_{i},
\]
the local index $l_{F_{i}}$ of $F_{i}$ w.r.t. $Q$ (as defined in
\ref{DEFLOCIND} (ii)) is given by the formula
\begin{equation}
l_{F_{i}}=\frac{\text{mult}_{N}(\sigma_{i})}{\sharp \left(  F_{i}\cap N\right)
-1}=\frac{\text{mult}_{\mathbb{Z}^{2}}(\Phi_{\mathcal{M}_{\sigma_{i}}}%
(\sigma_{i}^{\text{st}}))}{\sharp \left(  \text{conv}\left(  \left \{
\tbinom{1}{0},\tbinom{p_{i}}{q_{i}}\right \}  \right)  \cap \mathbb{Z}%
^{2}\right)  -1}=\frac{q_{i}}{\text{gcd}(q_{i},p_{i}-1)}.
\label{LOCALINDEXFORMULA}%
\end{equation}

\end{note}

\noindent{}Since we are mainly interested in the geometric properties of the
toric log del Pezzo surfaces $X\left(  N,\Delta_{Q}\right)  $ and $X\left(
M,\Delta_{Q^{\ast}}\right)  $ which are associated with $\ell$%
-\textit{reflexive} polygons $Q$ and their duals $Q^{\ast}:=\ell Q^{\circ},$
respectively, being defined in \ref{DEFLREFLEXIVE} and \ref{DEFDUAL}, let us
first determine the \textsc{wve}$^{2}$\textsc{c}-graphs $\mathfrak{G}%
_{\Delta_{Q}}$ and $\mathfrak{G}_{\Delta_{Q^{\ast}}}$ for the examples
mentioned in \ref{EXAMPLESLREF}. (For  the \textsc{wve}$^{2}$\textsc{c}-graphs of \textit{all} $1$-reflexive polygons cf. \cite[Figures 8, 9 and 10, pp. 108-109]{Dais1}.)

\begin{examples}
	\label{EXGRAPHS}(i) The \textsc{wve}$^{2}$\textsc{c}-graph $\mathfrak{G}%
	_{\Delta_{Q}}$ of the $\ell$-reflexive triangle (\ref{EXAMPLE1}) is shown in
	Figure \ref{Fig.3}. For $\ell\geq7$ we set%
	\[
	\ell^{\prime}:=\left\{
	\begin{array}
	[c]{ll}%
	\frac{1}{5}(6\ell-1),\smallskip & \text{if }\ell\equiv1(\operatorname{mod}5)\\
	\frac{1}{5}(3\ell-1),\smallskip & \text{if }\ell\equiv2(\operatorname{mod}5)\\
	\frac{1}{5}(8\ell+1),\smallskip & \text{if }\ell\equiv3(\operatorname{mod}5)\\
	\frac{1}{5}(9\ell-1), & \text{if }\ell\equiv4(\operatorname{mod}5)
	\end{array}
	\right.  \text{ \ and \ }\ell^{\prime\prime}:=\left\{
	\begin{array}
	[c]{ll}%
	\frac{1}{4}(9\ell-5),\smallskip & \text{if }\ell\equiv1(\operatorname{mod}4)\\
	\frac{1}{4}(3\ell-5), & \text{if }\ell\equiv3(\operatorname{mod}4)
	\end{array}
	\right.
	\]%
\end{examples}

\begin{figure}[h]
	\includegraphics[height=4.5cm, width=11.5cm]{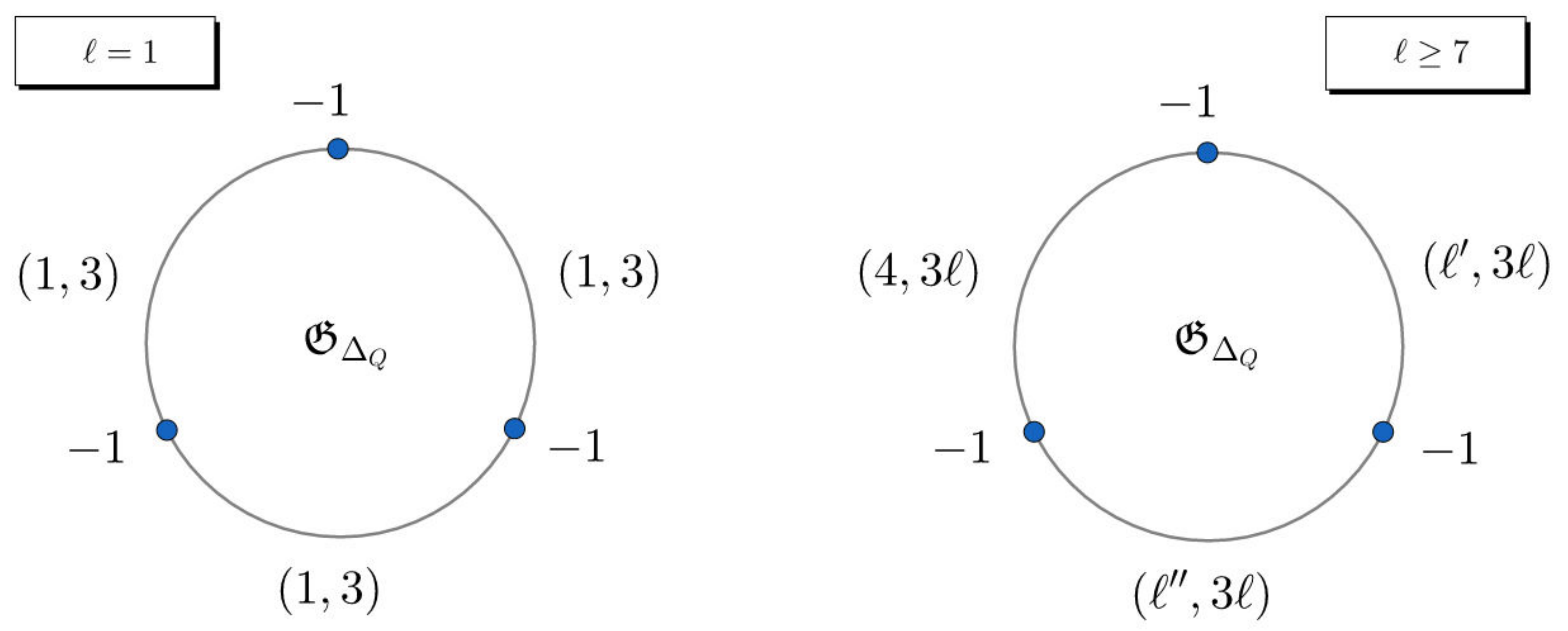}
	\caption{}\label{Fig.3}%
\end{figure}

\noindent The \textsc{wve}$^{2}$\textsc{c}-graph $\mathfrak{G}_{\Delta_{Q^{\ast}}}$ of
its dual (\ref{EXSTAR1}) is shown in Figure \ref{Fig.4}, where for $\ell\geq7$ we set
$\ell^{\prime\prime\prime}:=\frac{2}{3}(\ell-1)$ if $\ell\equiv
1(\operatorname{mod}3)$ and $\ell^{\prime\prime\prime}:=\frac{1}{3}(\ell-2)$
if $\ell\equiv2(\operatorname{mod}3).$\smallskip

\begin{figure}[h]
	\includegraphics[height=4.5cm, width=11.5cm]{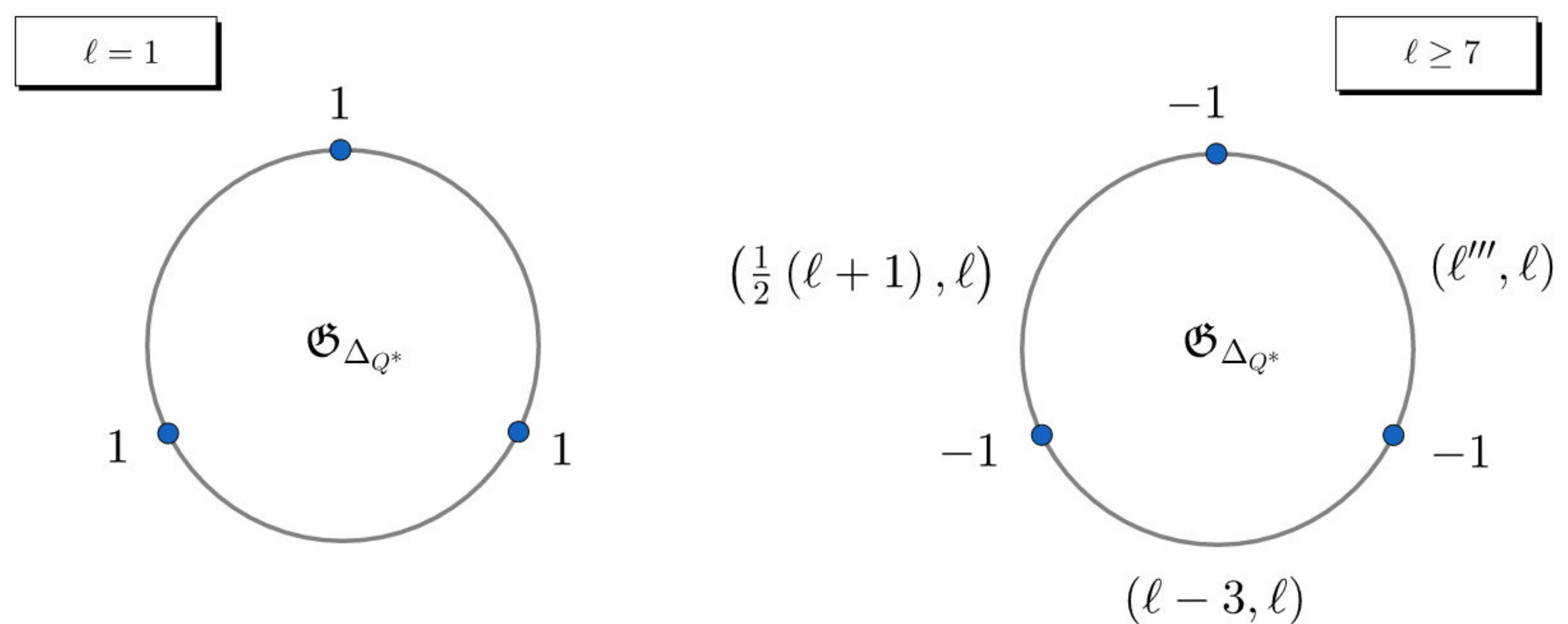}
	\caption{}\label{Fig.4}%
\end{figure}

\noindent (ii) The \textsc{wve}$^{2}$\textsc{c}-graph $\mathfrak{G}_{\Delta_{Q}}$ of the
$\ell$-reflexive quadrilateral (\ref{EXAMPLE2}) is illustrated in Figure \ref{Fig.5}, where
for $\ell\geq5$ we set $\ell^{\prime}:=\frac{4 \ell -1}{3}$ whenever $\ell\equiv1(\operatorname{mod}3)$
and  $\ell^{\prime}:=\frac{2 \ell -1}{3}$ whenever $\ell\equiv2(\operatorname{mod}3).$

\begin{figure}[h]
	\includegraphics[height=4.2cm, width=11.5cm]{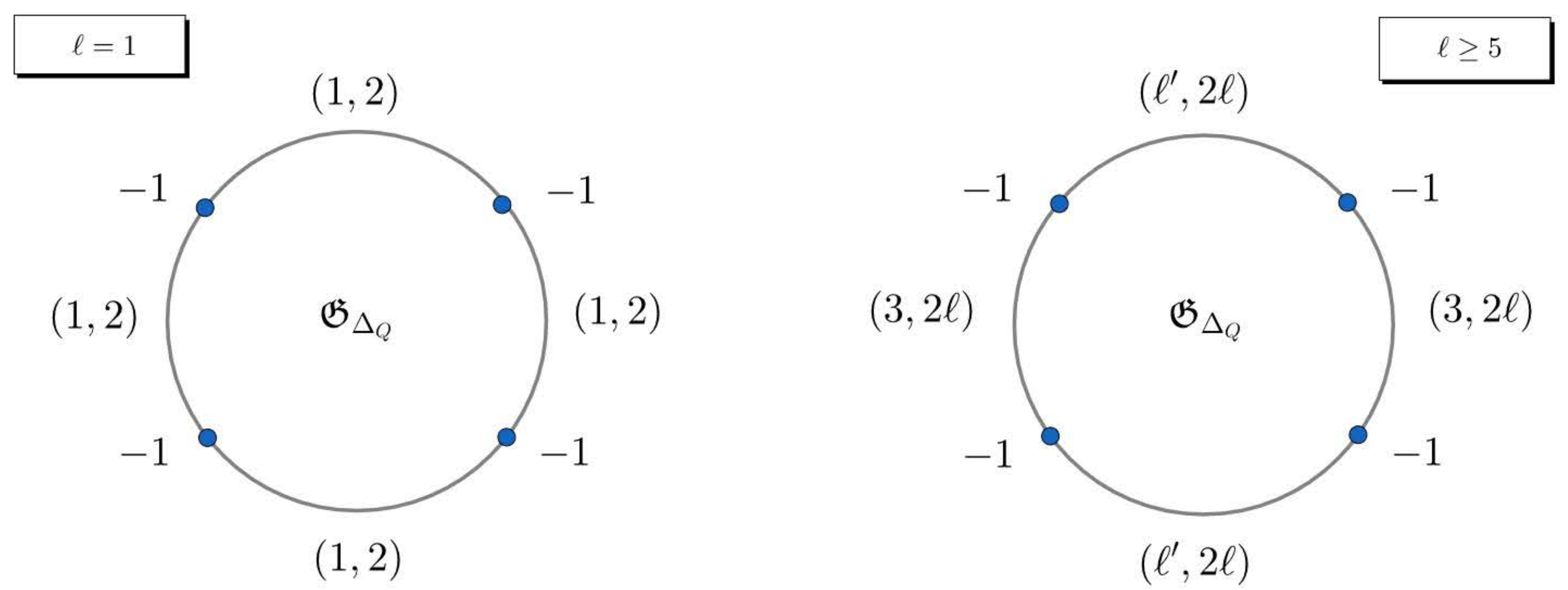}
	\caption{}\label{Fig.5}%
\end{figure}

For the \textsc{wve}$^{2}$\textsc{c}-graph $\mathfrak{G}_{\Delta_{Q^{\ast}}}$
of its dual (\ref{EXSTAR2}) see Figure \ref{Fig.6}.

\begin{figure}[h]
	\includegraphics[height=4.3cm, width=11.5cm]{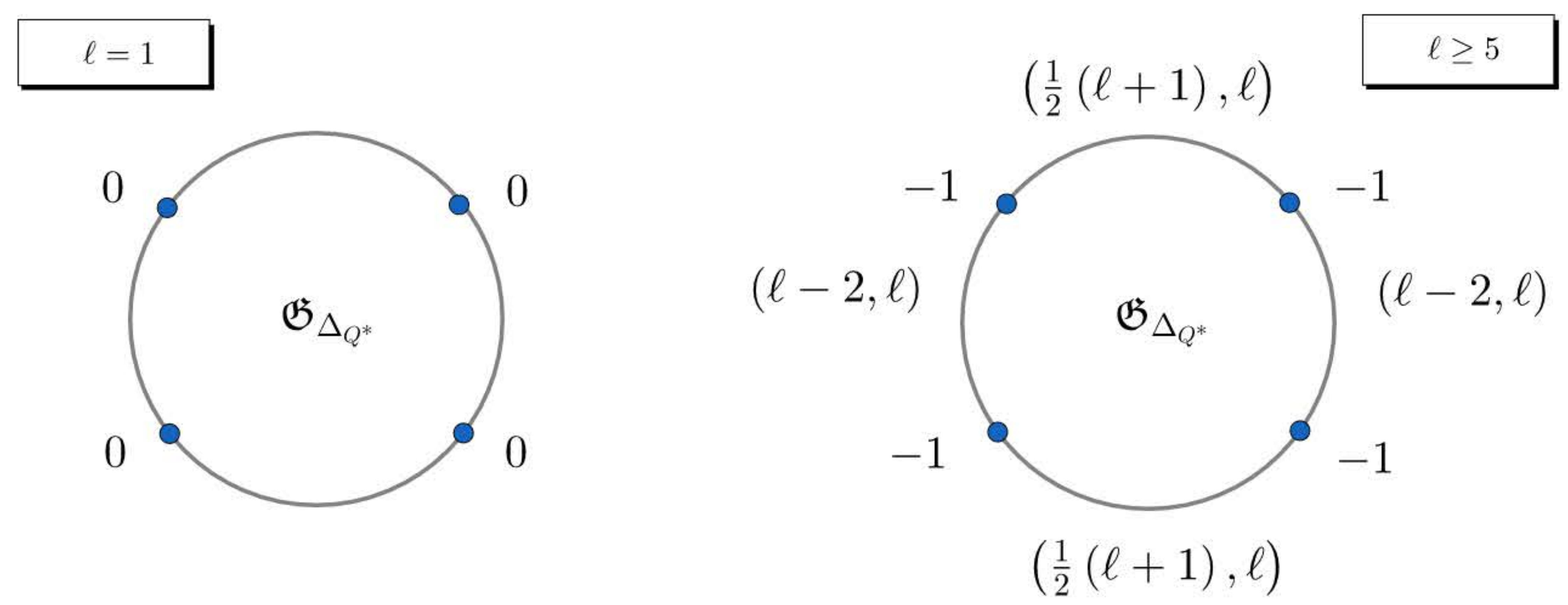}
	\caption{}\label{Fig.6}%
\end{figure}
\newpage

\noindent (iii) The \textsc{wve}$^{2}$\textsc{c}-graph $\mathfrak{G}_{\Delta_{Q}}$ of
the $\ell$-reflexive pentagon (\ref{EXAMPLE3}) is illustrated in Figure \ref{Fig.7}.
For $\ell\geq5$ we set $\ell^{\prime}:=\frac{1}{3}(2\ell+1)$ if $\ell
\equiv1(\operatorname{mod}3)$ and $\ell^{\prime}:=\frac{1}{3}(4\ell+1)$ if
$\ell\equiv2(\operatorname{mod}3).$

\begin{figure}[h]
	\includegraphics[height=4.5cm, width=12.5cm]{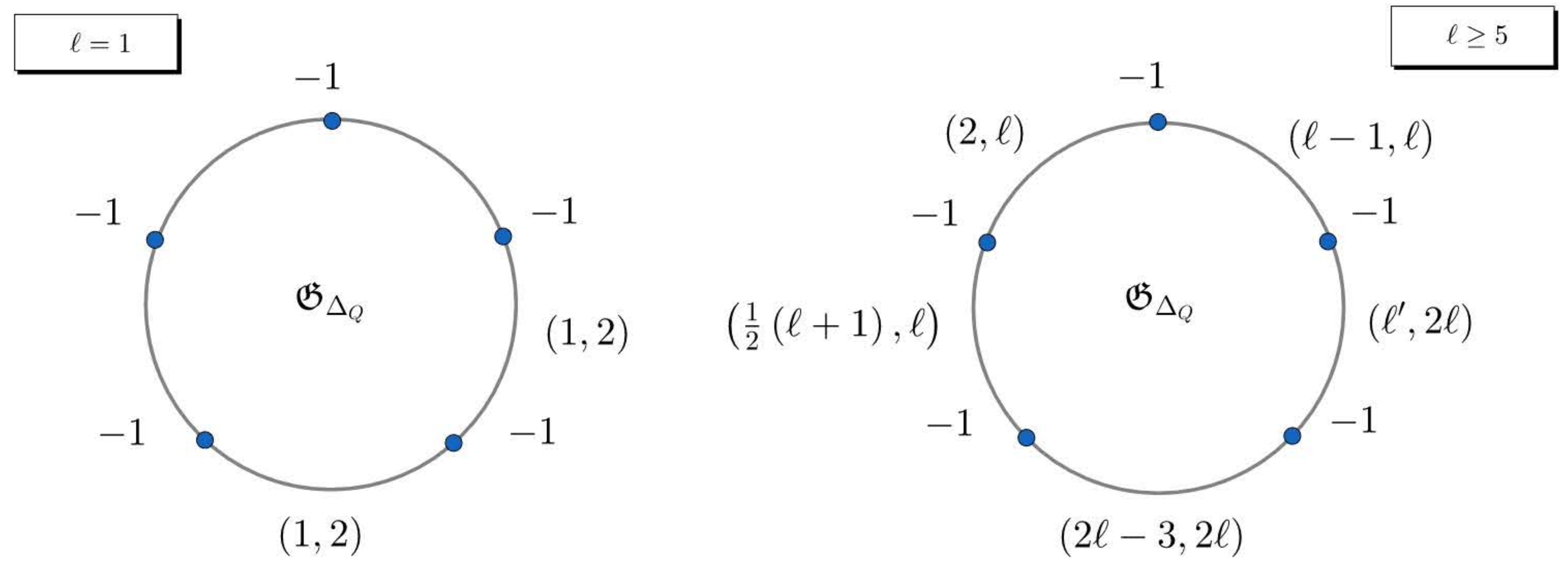}
	\caption{}\label{Fig.7}%
\end{figure}

\noindent The \textsc{wve}$^{2}$\textsc{c}-graph $\mathfrak{G}_{\Delta_{Q^{\ast}}}$ of
its dual (\ref{EXSTAR3}) is given in Figure \ref{Fig.8}.%

\begin{figure}[h]
	\includegraphics[height=4.5cm, width=12.7cm]{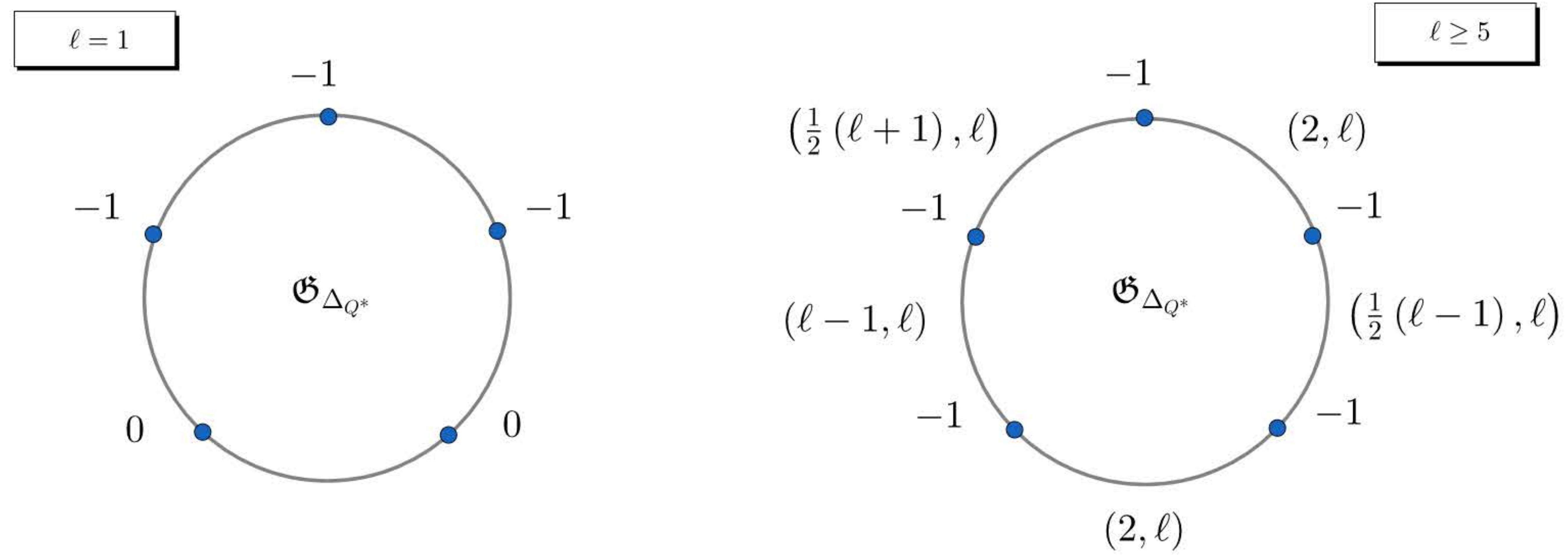}
	\caption{}\label{Fig.8}%
\end{figure}

\noindent (iv) The \textsc{wve}$^{2}$\textsc{c}-graph $\mathfrak{G}_{\Delta_{Q}}$ of the
$\ell$-reflexive hexagon (\ref{EXAMPLE4}) is shown in Figure \ref{Fig.9}. Note that for
its dual (\ref{EXSTAR4}) we have $\mathfrak{G}_{\Delta_{Q^{\ast}}}%
\underset{\text{gr.}}{\cong}\mathfrak{G}_{\Delta_{Q}}^{\text{rev}}$ (If
$\ell\geq3,$ the socii of $2,\frac{1}{2}(\ell+1),\ell-1$ w.r.t. $\ell$ are
$\frac{1}{2}(\ell+1),2,$ and $\ell-1,$ respectively.)

\begin{figure}[h]
	\includegraphics[height=4.5cm, width=12.7cm]{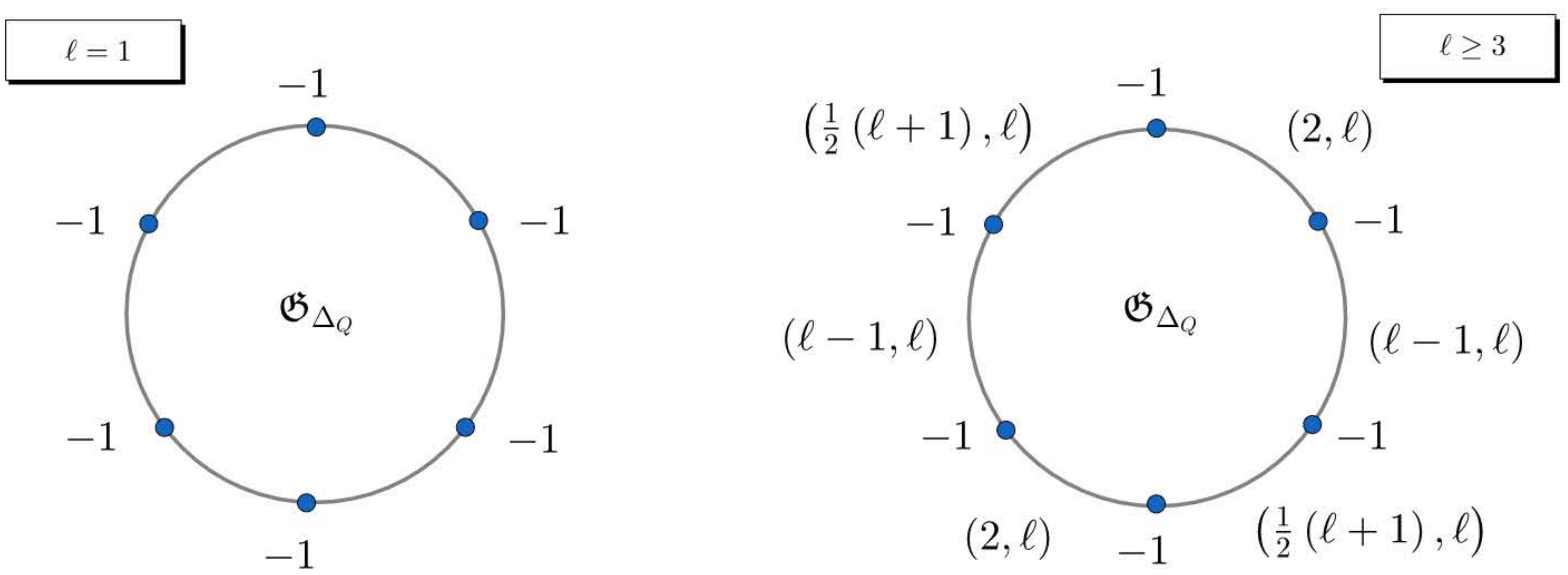}
	\caption{}\label{Fig.9}%
\end{figure}

\noindent{}If $Q$ is an $\ell$-reflexive polygon, examples \ref{EXGRAPHS}
suggest that there should be a particular connection betweeen the
combinatorial triples of $\Delta_{Q}$ and $\Delta_{Q^{\ast}}$ (and,
consequently, between the \textsc{wve}$^{2}$\textsc{c}-graphs $\mathfrak{G}%
_{\Delta_{Q}}$ and $\mathfrak{G}_{\Delta_{Q^{\ast}}}$) due to bijections
(\ref{LDUALITY1}) and (\ref{LDUALITY2}). This will be clarified below in
Propositions \ref{EXPRESSIONQSTAR}, \ref{QISTAREXPRESSION} and \ref{PISTAREXPRESSION}.

\section{Lattice change and cyclic covering trick whenever $\ell>1$}\label{COVTR}
\noindent $\bullet$ \textbf{Degree.} Let $f:X\longrightarrow Y$ be a proper holomorphic
map between two complex (analytic) spaces. $f$ is called \textit{finite} if it
is closed (as map) and for every $y\in Y$ the fibre $f^{-1}(\{y\})$ consists
of finitely many points.  $f$ is called \textit{generically finite} if there
is a non-empty open subset $V\subset Y$ such that $\left.  f\right \vert
_{f^{-1}(V)}:f^{-1}(V)\longrightarrow V$ is finite. If $X$ and $Y$ are complex
\textit{varieties}, $f$ generically finite and $f(X)$ dense in $Y,$ then the
field extension defined by $f^{\star}:\mathbb{C}(Y)\hookrightarrow
\mathbb{C}(X)$ is finite and
$
\text{deg}(f):=[\mathbb{C}(X):f^{\star}(\mathbb{C}(Y))]
$
is said to be \textit{the degree of} $f.$ (Note that, in this case, the set
$\left \{  y\in V\left \vert \sharp(f^{-1}(\{y\}))=\text{ deg}(f)\right.
\right \}$ is dense in $V.$) 
\medskip
 
\noindent $\bullet$ \textbf{\'{E}tale holomorphic maps}. For any complex space $X$ let
us denote by $\mathcal{O}_{X}$ its \textit{structure sheaf} and by $\Omega
_{X}^{1}$ \textit{the sheaf of germs of holomorhic} $1$-\textit{forms} on $X$
(or \textit{the cotangent sheaf on} $X,$ cf. \cite[\S 2.9 and \S 2.21]{Fischer}). If $X$
and $Y$ are two complex spaces, $\ f:X\longrightarrow Y$ a holomorphic map and
$Df:f^{\star}\Omega_{Y}^{1}\longrightarrow \Omega_{X}^{1}$ the associated
homomorphism (which is determined by means of the Jacobian), then one defines
\textit{the sheaf} $\Omega_{X|Y}^{1}:=$ Coker$(Df)$ \textit{of germs of
	relative} $1$-\textit{forms w.r.t. }$f.$ The holomorphic map $f$ is called
\textit{flat at} $x\in X$ if the stalk $\mathcal{O}_{X,x}$ is a flat
$\mathcal{O}_{Y,f(x)}$-module. ($\mathcal{O}_{X,x}$ becomes an $\mathcal{O}%
_{Y,f(x)}$-module via the natural homomorphism $\mathcal{O}_{Y,f(x)}%
\longrightarrow \mathcal{O}_{X,x}.$)  $f$ is called \textit{flat} if it is flat
at all points of $X.$ $f$ is said to be  \textit{\'{e}tale} \textit{at a
	point} $x\in X$ if it is flat at $x$ and simultaneously \textit{unramified at} $x$, i.e.,  $\mathfrak{m}_{Y,f(x)}\mathcal{O}%
_{X,x}=$ $\mathfrak{m}_{X,x}$ (where $\mathfrak{m}_{X,x}$ and $\mathfrak{m}%
_{Y,f(x)}$ denote the maximal ideals of the local rings $\mathcal{O}_{X,x}$
and $\mathcal{O}_{Y,f(x)},$ respectively). $f$ is called \textit{\'{e}tale} if
it is \'{e}tale at all points of $X.$ $f\, \ $is, in particular, \'{e}tale if
and only if it is flat and $\Omega_{X|Y}^{1}=0.$
\medskip 

\noindent $\bullet$ \textbf{Analytic spectrum.} Let $X$ be a complex space and $\mathcal{G}$ be an arbitrary sheaf of
$\mathcal{O}_{X}$-modules (\textit{an} $\mathcal{O}_{X}$-\textit{module}, for
short). $\mathcal{G}$ is said to be \textit{of finite
	type} at $x\in X$ if there is an open neighborhood $\mathcal{U}_{x}$ of $x$
and a $\left.  \mathcal{G}\right \vert _{\mathcal{U}_{x}}$-epimorhism
$\mathcal{O}_{X}^{\kappa_{x}}\rightarrow \left.  \mathcal{G}\right \vert
_{\mathcal{U}_{x}}$ for a positive integer $\kappa_{x}.$ $\mathcal{G}$ is
called \textit{of finite type }on $X$ if it is of finite type at all points
$x\in X.$ $\mathcal{G}$ is \textit{coherent} if $\mathcal{G}$ is of finite
type on $X$ and, in addition, for every $x\in X$ and every finite system
$\mathfrak{s}_{1},...,\mathfrak{s}_{\kappa}\in \mathcal{G}(\mathfrak{U}_{x})$
of sections over an open neighborhood $\mathfrak{U}_{x}$ of $x$ the sheaf of
relations Rel$_{x}(\mathfrak{s}_{1},...,\mathfrak{s}_{\kappa})$ (which is the
kernel of the $\left.  \mathcal{G}\right \vert _{\mathfrak{U}_{x}}$-homomorhism
$\mathcal{O}_{\mathfrak{U}_{x}}^{\kappa}\rightarrow \left.  \mathcal{G}%
\right \vert _{\mathfrak{U}_{x}}$ determined by $\mathfrak{s}_{1}%
,...,\mathfrak{s}_{\kappa}$) is of finite type at $x.$ If $\mathcal{G}$
happens to be a sheaf of $\mathcal{O}_{X}$-algebras (\textit{an}
$\mathcal{O}_{X}$-\textit{algebra}, for short), i.e., if $\mathcal{G}_{x}$ is
an $\mathcal{O}_{X,x}$-module and at the same time endowed with a \textit{ring
	structure} for all $x\in X,$ then the following is of particular importance.

\begin{theorem}\label{SPECANCON}
	Let $X$ be a complex space and $\mathcal{G}$ be a coherent $\mathcal{O}_{X}%
	$-algebra. Then there exists a unique (up to analytic isomorphism) complex
	space \emph{Specan}$(\mathcal{G}),$ the so-called \emph{analytic spectrum }of
	$\mathcal{G},$ as well as a finite holomorphic map $\pi:\emph{Specan}%
	(\mathcal{G})\longrightarrow X,$ such that\smallskip \  \newline \emph{(i)} there
	is an isomorphism $\pi_{\ast}(\mathcal{O}_{\emph{Specan}(\mathcal{G})}%
	)\cong \mathcal{G},$ and\smallskip \  \newline \emph{(ii) }there is a bijection
	$\pi^{-1}(x)\leftrightarrow \emph{Max}$\emph{-}$\emph{Spec}(\mathcal{G}_{x})$
	between the set of points of the fibre of $\pi$ over $x$ and the set of
	maximal ideals of the stalk of $\mathcal{G}$ at $x,$ for all $x\in X$.
\end{theorem}

\noindent For a rough local description of this \textquotedblleft
spectrum\textquotedblright \ in the analytic category we refer to \cite[pp.
59-62]{Fischer} and \cite[45.B.1, p. 172]{Kaup-Kaup}, and for more details on
the construction and the main properties of $\pi$ to Houzel \cite{Houzel}.\medskip 

\noindent {}$\bullet $ \textbf{Normal complex varieties which are $\mathbb{Q}$-Gorenstein.} If $X$
is a normal complex variety, then its Weil divisors can be
described by means of ``divisorial'' sheaves.

\begin{lemma}\label{LEMHART}
	\emph{(\cite[1.6]{Hartshorne}). }For a coherent $\mathcal{O%
	}_{X}$-module $\mathcal{F}$ the following conditions are equivalent\emph{:}\newline
	\emph{(i)} $\mathcal{F}$ is \emph{reflexive} \emph{(}i.e., $\mathcal{F}\cong 
	\mathcal{F}^{\vee \vee },$ with $\mathcal{F}^{\vee }:=Hom_{\mathcal{O}%
		_{X}}\left( \mathcal{F},\mathcal{O}_{X}\right) $ denoting the dual of $%
	\mathcal{F}$\emph{) }and has rank one.\newline
	\emph{(ii)} If $X^{0}$ is a non-singular open subvariety of $X$ with \emph{%
		codim}$_{X}\left( X\mathbb{r}X^{0}\right) \geq 2,$ then $\mathcal{F}\left|
	_{X^{0}}\right. $ is invertible and 
	\begin{equation*}
	\mathcal{F}\cong \iota _{\star }\left( \mathcal{F}\left| _{X^{0}}\right.
	\right) \cong \iota _{\star }\iota ^{\star }\left( \mathcal{F}\right) ,
	\end{equation*}
	where $\iota :X^{0}\hookrightarrow X$ denotes the inclusion map.
\end{lemma}

\noindent {}The \textit{divisorial sheaves} are exactly those satisfying the above conditions. Since a divisorial sheaf is torsion free, there is
a non-zero section $\mathfrak{s} \in H^{0}\left( X,\mathcal{M}_{X}\otimes _{\mathcal{O}%
	_{X}}\mathcal{F}\right) $, with 
$
H^{0}\left( X,\mathcal{M}_{X}\otimes _{\mathcal{O}_{X}}\mathcal{F}\right) \cong 
\mathbb{C}\left( X\right) \cdot \mathfrak{s} ,
$
and $\mathcal{F}$ can be considered as a subsheaf of the constant sheaf $%
\mathcal{M}_{X}$ of meromorphic functions of $X,$ i.e., as a special \textit{fractional
	ideal sheaf.} Let $\mathcal{M}^{\ast}_{X}$ and  $\mathcal{O}^{\ast}_{X}$ be the sheaves of germs of not identically vanishing meromorphic functions and of nowhere vanishing holomorphic functions on $X$, respectively.

\begin{proposition}
	\emph{(\cite[Appendix of \S 1]{Reid}) }The correspondence 
	\begin{equation*}
	\left\{ 
	\begin{array}{c}
	\text{classes of Weil divisors on}\ X \\ 
	\text{(w.r.t. rational equivalence)}%
	\end{array}
	\right\}
	 \ni \{D\}\overset{\mathfrak{d}}{\longmapsto }\mathcal{O}%
	_{X}\left( D\right)\in \left\{ 
	\begin{array}{c}
	\text{divisorial coherent } \\ 
	\text{subsheaves of }\mathcal{M}_{X}%
	\end{array}
	\right\} \ /\ H^{0}\left( X,\mathcal{O}_{X}^{\ast }\right)
	\end{equation*}
	with $\mathcal{O}_{X}\left( D\right) $ defined by sending every non-empty
	open set $\mathcal{U}$ of $X$ onto 
	\begin{equation*}
	\Gamma(\mathcal{U},\mathcal{O}_{X}(D)):=\mathcal{O}_{X}\left( D\right) \left(\mathcal{U}\right) :=\left\{
	\varphi \in \mathcal{M}^{\ast}_{X}(\mathcal{U})\ \left| \ \left( \text{\emph{%
			div}}\left( \varphi \right) +D\right) \left| _{\mathcal{U}}\right. \geq 0\right.
	\right\} \cup \{\mathbf{0}\},
	\end{equation*}
	is a bijection, and induces a $\mathbb{Z}$-module isomorphism. In fact, to
	avoid torsion, one defines this $\mathbb{Z}$-module structure by setting 
	\begin{equation*}
	\mathfrak{d} \left(\{ D_{1}+D_{2}\}\right) :=(\mathcal{O}_{X}\left( D_{1}\right)
	\otimes \mathcal{O}_{X}\left( D_{2}\right) )^{\vee \vee }\ \text{\emph{and} }\ %
	\mathfrak{d} \left( \{j D\}\right):=\mathcal{O}_{X}\left( j D\right) ^{\vee \vee },
	\end{equation*}
	for any Weil divisors $D,D_{1},D_{2}$ and $j \in \mathbb{Z}.$
	
\end{proposition}

\noindent Let now $\Omega^{1}_{\text{Reg}(X)}$ be the cotangent sheaf on Reg$\left( X\right): =X\mathbb{r}$Sing$\left( X\right) 
\overset{\iota}{\hookrightarrow}$ $X,$ and for $%
j\geq2$ let us set 
$
\Omega_{\text{Reg}(X)}^{j}:=\bigwedge\limits^{j}\ \Omega^{1}_{\text{%
		Reg}(X)}\ .
$
The canonical divisor $K_{X}$ of $X$ is that one, the class of which is mapped by $\mathfrak{d}$ onto the \textit{canonical divisorial sheaf} 
$
\omega_{X}:=\iota_{\ast}\left( \Omega_{\text{Reg}(X)}^{\dim_{%
		\mathbb{C}}\left( X\right) }\right) .
$ Note that $\omega _{X}=\omega _{X}^{\left[
	1 \right] }:=\mathcal{O}_{X}\left( K_{X}\right)$ and that 
$\omega _{X}^{\left[j \right] }:=\mathcal{O}_{X}\left( jK_{X}\right)=(\omega _{X}^{\otimes j})^{\vee \vee }= \iota_{\ast}((\Omega_{\text{Reg}%
	(X)}^{\text{dim}_{\mathbb{C}}(X)})^{\otimes j})$
 for all $j \in \mathbb{Z}.$

\begin{definition}
	$X$ is called $\mathbb{Q}$-\textit{Gorenstein} if its canonical divisorial sheaf $\omega _{X}=\mathcal{O}_{X}\left( K_{X}\right)$ is such that $K_{X}$ is $\mathbb{Q}$-Cartier divisor. If $X$ is $\mathbb{Q}$-Gorenstein, then we set index$(X)$:= min$\left \{  j\in \mathbb{Z}_{\geq1}\left \vert jK_{X}\text{ is
		Cartier}\right.  \right \}$.
\end{definition}

\noindent $\bullet$ \textbf{Canonical cyclic coverings}. Given a point $x_{0}$ of a normal $\mathbb{Q}$-Gorenstein complex variety $X,$ we
consider an affine neighborhood $U$ of $x_{0}$ representing the set germ at $x_{0},$
and a nowhere vanishing section $\mathfrak{s}\in H^{0}%
(U,\mathcal{O}_{U}\left(-\text{index}(U)K_{U}\right)$ such that%
\[
H^{0}%
(U,\mathcal{O}_{U}\left(-\text{index}(U)K_{U}\right)=H^{0}(U,\omega_{U}^{[\text{-index}(U)]})\cong \mathcal{O}_{U}\cdot
\mathfrak{s}\cong \mathcal{O}_{U}.
\]
If $\alpha \in \omega_{U}^{[i]},\beta \in \omega_{U}^{[j]}$ and $\mathfrak{v}_{U}:\omega_{U}^{[i]}\otimes \omega_{U}^{[j]}\longrightarrow \omega_{U}^{[i+j]}$ is the natural map, then the coherent $\mathcal{O}_{U}$-module
\[\mathcal{R}_{U}:=\mathcal{O}_{U}\oplus \omega_{U}\oplus \omega_{U}^{[2]}%
\oplus \cdots \oplus \omega_{U}^{[\text{index}(U)-1]},\] equipped with the
multiplication \textquotedblleft$\odot$\textquotedblright \ being induced by
setting
\[
\alpha \odot \beta:=\left \{
\begin{array}
[c]{ll}%
\mathfrak{v}_{U}(\alpha \otimes \beta)\in \omega_{U}^{[i+j]}, & \text{if }%
i+j\leq \text{index}(U)-1\medskip,\\
\mathfrak{v}_{U}(\alpha \otimes \beta)\cdot \mathfrak{s}\in \omega
_{U}^{[i+j-\text{index}(U)]}, & \text{if }i+j\ge\text{index}(U),
\end{array}
\right.
\]
becomes an $\mathcal{O}_{U}$-algebra. Let $\pi_{U}:$ Specan$(\mathcal{R}%
_{U})\longrightarrow U$ be the corresponding finite holomorphic map
constructed by Theorem \ref{SPECANCON}. Wahl (in the algebraic category, cf.
\cite[Appendix, pp. 260-262]{Wahl}) and Reid \cite[Appendix of \S 1, pp.
281-285]{Reid} were the first who initiated the use of $\pi_{U}$ in order to replace $U$ by $U^{\text{can}}:= \text{Specan}(\mathcal{R}%
_{U})$ \textit{of index} $1$ in the case in which $x_{0}$ is \textit{singular}. 

\begin{theorem}\label{WAHLREID}	\emph{(\cite[1.9]{Reid}, \cite[\S 3.6]{Reid-YPG} and \cite[4-5-1 \& 4-5-2, pp. 183-186]{Matsuki})}
The pair $(U^{\emph{can}},\pi_{U})$ has (and is up to an analytic isomorphism uniquely determined by)
	the following properties\emph{:\smallskip} \newline \emph{(i)} $U^{\emph{can}}$ is
	a normal complex variety and the fiber $\pi_{U}^{-1}(\{x_{0}\})$ over $x_{0}$ is a singleton
	\emph{(}say $\{y_{0}\}$\emph{)}.\smallskip \  \newline \emph{(ii)} The field
	extension $\mathbb{C}(U^{\emph{can}})$ of $\mathbb{C}(U)$ is Galois with Galois
	group  $G_{U}\cong \mathbb{Z}/(\emph{index}%
	(U))\mathbb{Z}$ and with a
	generator $\mathfrak{g}$ of $G_{U}$ acting on $\mathcal{R}_{U}$ as follows:
	\[
	(\mathfrak{g},\alpha_{0}+\alpha_{1}+\alpha_{2}+\cdots+\alpha_{\emph{index}(U)-1}%
	)\longmapsto \alpha_{0}+\alpha_{1}\zeta_{\emph{index}(U)}+\alpha_{2}%
	\zeta_{\emph{index}(U)}^{2}\cdots+\alpha_{\emph{index}(U)-1}\zeta
	_{\emph{index}(U)}^{\emph{index}(U)-1}%
	\]
	\emph{(}with $\zeta_{\emph{index}(U)}:=\exp(2\pi \sqrt{-1}/\emph{index}(U)),\alpha_{0}\in \mathcal{O}_{U}\ \text{and}\ \alpha_{i}\in  \omega_{U}^{[i]}\ \text{for all}\ i\in\{1,...,\ \emph{index}(U)-1\}$\emph{)}.
	\smallskip \newline \emph{(iii)} $\pi_{U}$ is
	\'{e}tale in codimension\footnote{This means \'{e}tale outside a subvariety of codimension $\ge 2$.\smallskip} $1$. \smallskip \newline \emph{(iv)} $ \mathcal{O}_{U^{\emph{can}}}(K_{U^{\emph{can}}%
	})\cong \mathcal{O}_{U^{\emph{can}}},$ i.e., $K_{U^{\emph{can}}}$ is a Cartier divisor,
	and $U^{\emph{can}}$ is a $\mathbb{Q}$-Gorenstein affine complex variety of index
	$1.\smallskip$ \newline \emph{(v)} There is a non-vanishing section
	$\mathfrak{s}^{\prime}\in H^{0}(U^{\emph{can}},\mathcal{O}_{U^{\emph{can}}%
	}(K_{U^{\emph{can}}}))$ around the point $y_{0}$ which is semi-invariant w.r.t. the
	action of $G_{U}$ and on which $G_{U}$ acts faithfully. \smallskip \newline \emph{($\pi_{U}: U^{\text{can}}\longrightarrow U$ is said to be the} canonical cyclic cover of $U$ \emph{of degree deg}$(\pi_{U})=\emph{index}(U)$.\emph{)}
\end{theorem}
\begin{remark} (i) In particular, $\pi_{U}:U^{\text{can}}\longrightarrow U$ is surjective\footnote{Since  $\pi_{U}$ is finite and surjective, we have, in particular, $\text{dim}_{\mathbb{C}}(U^{\text{can}})= \text{dim}_{\mathbb{C}}(U)$.} and can be viewed as the quotient map by an appropriate
	identification $U\cong U^{\text{can}}/G_{U}.$ \smallskip \newline (ii) If $\varphi \in \mathbb{C}(U)$ is such that div$(\varphi)+K_{U}=0,$ then the
	polynomial $\mathsf{T}^{\text{index}(U)}-\varphi$ \ is irreducible in
	$\mathbb{C}(U)[\mathsf{T}],$ the Galois extension
	\[
	\mathbb{C}(U)[\sqrt[\uproot{3}\text{index}(U)]{\varphi}]=\mathbb{C}%
	(U)[\mathsf{T}]/(\mathsf{T}^{\text{index}(U)}-\varphi)
	\]
of $\mathbb{C}(U)$	has $G_{U}$ as Galois group, and 
	\[U^{\text{can}}\cong \text{Spec}(%
	\bigoplus_{j=0}^{\text{index}(U)-1}
	 \Gamma(U,\omega_{U}^{\left[  j\right]  })
	\cdot(\sqrt[\uproot{3}\text{index}(U)]{\varphi})^{j}),\  \  \forall j\in \{0,1,...,\text{index}%
	(U)-1\}.
	\] 
\end{remark}

\noindent $\bullet$ \textbf{Back to our specific toric log del Pezzo surfaces.} 
Let $\ell$ be an integer $> 1$, and $(Q,N)$, $(Q^{\ast},M)$ two $\ell
$-reflexive pairs, where $M:=$
Hom$_{\mathbb{Z}}(N,\mathbb{Z}),$ with  $X(N,\Delta_{Q})$ and $X(M,\Delta_{Q^{\ast}})$ the corresponding toric log del
Pezzo surfaces. Assume that
$\mathbf{n}_{1},\ldots,\mathbf{n}_{\nu}$ are the vertices of $Q$ ordered anticlockwise,
and for $i\in \left \{  1,\ldots,\nu \right \}$ define $F_{i}:=$ conv$(\left \{  \mathbf{n}_{i},\mathbf{n}_{i+1}\right \}  )$
to be the $i$-th edge of $Q$ (as in \ref{NOTELF} (ii)) and $\sigma
_{i}:=\sigma_{F_{i}}=\mathbb{R}_{\geq0}\mathbf{n}_{i}+\mathbb{R}_{\geq
	0}\mathbf{n}_{i+1}$ the $N$-cone of type $(p_{i},q_{i})$ supporting $F_{i}.$  
It is easy to verify that \[U_{\sigma_{i},N}\cong \text{Spec}(\mathbb{C}[z_{1},z_{2}]^{G_{i}}),\ \text{where}\ G_{i}:=\left \langle\text{diag}(\zeta_{q_{i}}^{-p_{i}},\zeta_{q_{i}})\right \rangle \subset \text{GL}_{2}\left(\mathbb{C}\right)\ \text{(as in Proposition \ref{Sing-Prop}),}\] is
$\mathbb{Q}$-Gorenstein and that index$(U_{\sigma_{i},N})=l_{F_{i}},$ where
$l_{F_{i}}=q_{i}/{\text{gcd}(q_{i},p_{i}-1)}$ is the \textit{local index}
of $F_{i}$ (w.r.t. $Q$) as defined in \ref{DEFLOCIND} (ii). (See \cite[Notes 3.19, p. 89, and 4.5 (b), p.96]{Dais1}.) 

\begin{lemma}\label{TRICKYLEMMA}
	For every $i\in \left \{  1,...,\nu \right \}$  let $\Lambda_{F_{i}}\subseteq N$ be the sublattice generated by the lattice points of $F_{i}.$ The canonical cyclic cover
	\begin{equation}
	\pi^{ }_{U_{\sigma_{i},N}}:\emph{Specan}(\mathcal{R}_{U_{\sigma_{i},N}%
	})\longrightarrow U_{\sigma_{i},N}\cong \mathbb{C}^{2}/G_{i}=\emph{Spec}(\mathbb{C}[z_{1},z_{2}]^{G_{i}%
	}) \label{CANCYCLC}
	\end{equation}
	has degree $l_{F_{i}}$, with
	\[
	\emph{Specan}%
	(\mathcal{R}_{U_{\sigma_{i},N}})\cong U_{\sigma_{i},\Lambda_{F_{i}}}\cong \mathbb{C}^{2}/G_{i}^{\prime}=\emph{Spec}(\mathbb{C}[z_{1},z_{2}]^{G_{i}^{\prime
	}}),\  \text{where} \ G_{i}^{\prime}:=G_{i}\cap \text{ \emph{SL}}_{2}(\mathbb{C}),
	\]
	and can be
	viewed as the quotient map by the identification 
\[
U_{\sigma_{i},N}\cong U_{\sigma_{i},\Lambda_{F_{i}}}/(N/\Lambda_{F_{i}})\ \ \text{
	with}\ \ N/\Lambda_{F_{i}}\cong G_{i}/G^{\prime}_{i}\cong \mathbb{Z}/l_{F_{i}}\mathbb{Z}.
\]%
Moreover, $\pi_{U_{\sigma_{i},N}}^{-1}(\{$\emph{orb}$_{N}(\sigma_{i})\})=\{$%
\emph{orb}$_{\Lambda_{F_{i}}}(\sigma_{i})\},$ where the point \emph{orb}$_{\Lambda_{F_{i}}%
}(\sigma_{i})\in  U_{\sigma_{i},\Lambda_{F_{i}}} $ is either nonsingular or a Gorenstein cyclic quotient
singularity of type $(1,\tfrac{q_{i}}{l_{F_{i}}}).$

\end{lemma}
\begin{proof} Consider an arbitrary $i\in \{1,...,\nu \}.$ Firstly, index$(U_{\sigma_{i}%
		,N})=l_{F_{i}}=\left \vert N:\Lambda_{F_{i}}\right \vert .$ Secondly,
	\[
	\text{diag}(\zeta_{q_{i}}^{-p_{i}\ell},\zeta_{q_{i}}^{\ell})=\text{diag}%
	(\zeta_{\frac{q_{i}}{l_{F_{i}}}}^{-p_{i}},\zeta_{\frac{q_{i}}{l_{F_{i}}}%
	})=\text{diag}(\zeta_{\frac{q_{i}}{l_{F_{i}}}}^{-1},\zeta_{\frac{q_{i}%
		}{l_{F_{i}}}}),
	\]
	and therefore
	\[
	G_{i}^{\prime}:=G_{i}\cap \text{SL}_{2}(\mathbb{Z})=\text{Ker}(G_{i}%
	\overset{\text{det}}{\longrightarrow}\mathbb{C}^{\times})=\left \langle
	\text{diag}(\zeta_{\frac{q_{i}}{l_{F_{i}}}}^{-1},\zeta_{\frac{q_{i}}{l_{F_{i}%
	}}})\right \rangle ,\text{ \ with \ }\left \vert G_{i}^{\prime}\right \vert
	=\frac{q_{i}}{l_{F_{i}}}.
	\]
	Since $G_{i}\cong \mathbb{Z}/q_{i}\mathbb{Z},$ $G_{i}^{\prime}\cong%
	\mathbb{Z}/\frac{q_{i}}{l_{F_{i}}}\mathbb{Z}$ and $G_{i}/G_{i}^{\prime}%
	\cong \mathbb{Z}/l_{F_{i}}\mathbb{Z},$ the diagram 
	\[
	\xymatrixrowsep{0.2in}
	\xymatrixcolsep{0.2in}
	\xymatrix{
		&  &  & \{0\} \ar[d] & \\
		\{0\} \ar[r] & \hspace{0.2cm} \mathbb{Z}\mathbf{n}_{i}\oplus \mathbb{Z}\mathbf{n}_{i+1} \ar@{=}[d] \hspace{0.15cm} \ar@{^{(}->}[r] & \begin{array}
		[c]{c}%
		\Lambda_{F_{i}}\smallskip%
		\end{array} \vspace{0.3cm} \ar@{^{(}->}[d]  \ar@{->>}[r]
		&  \begin{array}
		[c]{c}%
		\mathbb{Z}/\frac{q_{i}}{l_{F_{i}}}\mathbb{Z}
		\end{array} \vspace{0.5cm} \ar@{^{(}->}[d] \ar[r] & \{0\}\\
		\{0\} \ar[r] & \begin{array}
		[c]{c}%
		\mathbb{Z}\mathbf{n}_{i}\oplus \mathbb{Z}\mathbf{n}_{i+1}\smallskip%
		\end{array} \ar@{^{(}->}[d] \ar@{^{(}->}[r] & N \ar@{=}[d] \ar@{->>}[r] & \mathbb{Z}/q_{i}\mathbb{Z} \ar@{->>}[d] \ar[r] & \{0\}\\
		\{0\} \ar[r] & \Lambda_{F_{i}}  \hspace{0.1cm} \ar@{^{(}->}[r] & N \ar@{->>}[r] & \mathbb{Z}/l_{F_{i}}\mathbb{Z}  \ar[d] \ar[r] &
		\{0\}\\
		&  &  & \{0\} &}
	\]
	(the three rows and the last column of which are short exact sequences of
	additive groups) combined with \cite[Proposition 1.13.18 and Ex. 1.3.20, pp. 44-46]{CLS}
	gives $U_{\sigma_{i},\Lambda_{F_{i}}}\cong \mathbb{C}^{2}/G_{i}^{\prime},$
	$U_{\sigma_{i},N}\cong \mathbb{C}^{2}/G_{i},$ and%
	\[
	U_{\sigma_{i},N}\cong U_{\sigma_{i},\Lambda_{F_{i}}}/(G_{i}/G_{i}^{\prime
	})\cong \mathbb{C}^{2}/G_{i}^{\prime}/(G_{i}/G_{i}^{\prime}).
	\]
	Now we apply Theorem \ref{WAHLREID} for $U_{\sigma_{i},N}.$ For every $j\in
	\{0,1,...,l_{F_{i}}-1\}$ the divisor $-jK_{U_{\sigma_{i},N}}$ is
	$\mathbb{T}_{N}$-invariant and $\Gamma(U_{\sigma_{i},N},\omega_{U_{\sigma
			_{i},N}}^{\left[  j\right]  })$ is a reflexive $\mathbb{C}[\sigma_{i}^{\vee
	}\cap M]$-module of rank $1.$ Therefore $\mathcal{R}_{U_{\sigma_{i},N}}:=%
	{\textstyle \bigoplus \nolimits_{j=0}^{l_{F_{i}}-1}}
	\omega_{U_{\sigma_{i},N}}^{\left[  j\right]  }\, \ $is a $\mathbb{T}_{N}%
	$-invariant $\mathcal{O}_{U_{\sigma_{i},N}}$-algebra, Specan$(\mathcal{R}%
	_{U_{\sigma_{i},N}})$ is an affine \textit{toric} surface which is
	$\mathbb{Q}$-Gorenstein and of index $1$ (which means that it is a
	two-dimensional Gorenstein variety\footnote{If a $\mathbb{Q}$-Gorenstein
		variety of index $1$ is Cohen-Macauley, then it is a \textit{Gorenstein
			variety}, i.e., the local ring at each of its points is a Gorenstein ring.}),
	and the canonical cover map (\ref{CANCYCLC}) is equivariant. Setting $\varphi_{i}%
	:=-\frac{d\mathbf{e}(\mathbf{m}_{i})}{\mathbf{e}(\mathbf{m}_{i})}\wedge
	\frac{d\mathbf{e}(\mathbf{m}_{i}^{\prime})}{\mathbf{e}(\mathbf{m}_{i}^{\prime
		})},$ with $\{ \mathbf{m}_{i},\mathbf{m}_{i}^{\prime}\}$ a basis of $M,$ we
	have $\operatorname{div}(\varphi_{i})=-K_{U_{\sigma_{i},N}}$ (cf. Oda \cite[p. 71]{Oda}),
	$\  \mathsf{T}^{l_{F_{i}}}-\varphi_{i}$ is irreducible in $\mathbb{C}%
	$($U_{\sigma_{i},N}$)$[\mathsf{T}]$ and the Galois extension $\mathbb{C}%
	$($U_{\sigma_{i},N}$)$[\sqrt[\uproot{5}{\scriptstyle l_{F_{i}}}]{\varphi_{i}}]$ of $\mathbb{C}%
	$($U_{\sigma_{i},N}$) has a cyclic Galois group, say $G_{i}^{\prime \prime
	}\cong \mathbb{Z}/l_{F_{i}}\mathbb{Z}$ (because deg$(\pi_{U_{\sigma_{i},N}%
	})=l_{F_{i}}$). Since
	\[
	\text{Specan}(\mathcal{R}_{U_{\sigma_{i},N}})\cong \text{Spec}\left(
	{\textstyle \bigoplus \limits_{j=0}^{l_{F_{i}}-1}}
	\Gamma(U_{\sigma_{i},N},\omega_{U_{\sigma_{i},N}}^{\left[  j\right]  }%
	)\cdot(\sqrt[\uproot{5}{\scriptstyle l_{F_{i}}}]{\varphi_{i}})^{j}\right)
	\]
	is a \textit{Gorenstein} toric affine surface, it suffices for our purposes to recall
	that it has to appear as the quotient of $\mathbb{C}^{2}$ by a finite cyclic
	subgroup $H_{i}$ of SL$_{2}(\mathbb{C})$ acting diagonally. W.l.o.g. we may
	assume that Specan$(\mathcal{R}_{U_{\sigma_{i},N}})\cong \mathbb{C}^{2}%
	/H_{i}\cong U_{\sigma_{i},L_{i}}$ (i.e., the toric affine surface associated
	with the \textit{same} cone $\sigma_{i}$ but with respect to \textit{another}
	lattice $L_{i}\subset$ $\mathbb{R}^{2},$ such that $\left \vert L_{i}%
	:\mathbb{Z}\mathbf{n}_{i}\oplus \mathbb{Z}\mathbf{n}_{i+1}\right \vert =\left \vert H_{i}\right \vert%
	$)$,$ that $H_{i}\subseteq G_{i}^{\prime}$ and that $\pi_{U_{\sigma_{i},N}%
	}^{-1}(\{$orb$_{N}(\sigma_{i})\})=\{$orb$_{L_{i}}(\sigma_{i})\}.$ Using the
	equivariant holomorphic map determined by the dotted arrow in the diagram:%
\[
\xymatrix{
	\mathbb{C}^{2} \ar@{<->}[r]^{\cong \hspace{0.8cm}} \ar@{->>}[d] & U_{\sigma_{i},\mathbb{Z}\mathbf{n}_{i}\oplus \mathbb{Z}%
		\mathbf{n}_{i+1}} \ar@{->>}[d] \ar@{=}[rr]  &  & U_{\sigma_{i},\mathbb{Z}\mathbf{n}_{i}\oplus
		\mathbb{Z}\mathbf{n}_{i+1}} \ar@{->>}[d] \ar@{<->}[r]^{\hspace{0.8cm}\cong} & \mathbb{C}^{2} \ar@{->>}[d]\\
	\mathbb{C}^{2}/G_{i}^{\prime} \ar@{->>}[d] \ar@{<->}[r]^{\cong} & U_{\sigma_{i},\Lambda_{F_{i}}} \ar@{->>}[d] &  &
	U_{\sigma_{i},L_{i}} \ar@{->>}[d] \ar@{->>}[dll]  | {\pi_{U_{\sigma_{i},N}}} \ar@{-->}[ll]  \ar@{<->}[r]^{\cong}  & \mathbb{C}^{2}/H_{i} \\
	\mathbb{C}^{2}/G_{i} \ar@{<->}[r]^{\cong} & U_{\sigma_{i},N}  \ar@{<->}[rr]^{\cong}  &  & U_{\sigma_{i},L_{i}}%
	/G_{i}^{\prime \prime}
}
\]
	we verify that $\Lambda_{F_{i}}/L_{i}\cong G_{i}^{\prime}/H_{i}.$ On the other
	hand, the restriction
	\[
	\xi_{i}:=\left.  \pi_{U_{\sigma_{i},N}}\right \vert _{U_{\sigma_{i},L_{i}%
		}\mathbb{r\{}\text{orb}_{L_{i}}(\sigma_{i})\}}:U_{\sigma_{i},L_{i}%
	}\mathbb{r\{}\text{orb}_{L_{i}}(\sigma_{i})\} \twoheadrightarrow U_{\sigma
		_{i},N}\mathbb{r}\{ \text{orb}_{N}(\sigma_{i})\}
	\]
	is an \'{e}tale holomorphic map (and, in particular, a topological, i.e., an
	unramified covering map), and
	\[
	G_{i}^{\prime \prime}\cong \pi_{1}(U_{\sigma_{i},N}\mathbb{r}\{ \text{orb}%
	_{N}(\sigma_{i})\})/\xi_{i\, \ast}(\pi_{1}(U_{\sigma_{i},L_{i}}\mathbb{r\{}%
	\text{orb}_{L_{i}}(\sigma_{i})\}))
	\]
	(where $\pi_{1}(...)$ denotes the fundamental group of these pathwise connected spaces, cf. \cite[Theorem 2.8, p. 18]{Morishita}).  Furthermore, the composite of the \'{e}tale holomorphic
	maps%
\[
\xymatrix{
	\mathbb{C}^{2}\mathbb{r}\{ \mathbf{0}\} \ar@{->>}[r] & U_{\sigma_{i},L_{i}}\mathbb{r\{}%
	\text{orb}_{L_{i}}(\sigma_{i})\} \ar@{->>}[r] & U_{\sigma_{i},N}\mathbb{r}\{ \text{orb}%
	_{N}(\sigma_{i})\}
}
\]
	(where $\mathbb{C}^{2}\mathbb{r}\{ \mathbf{0}\}$ is the universal cover of
	$U_{\sigma_{i},L_{i}}\mathbb{r\{}$orb$_{L_{i}}(\sigma_{i})\}$ which is simply
	connected) gives the following short exact sequence of fundamental groups:%
	\[
	{\small
		\xymatrixrowsep{0.2in}
		\xymatrixcolsep{0.2in}
		\xymatrix{
			\pi_{1}(\mathbb{C}^{2}\mathbb{r}\{ \mathbf{0}\}) \ar[r] \ar@{=}[d] & \pi_{1}(U_{\sigma_{i},L_{i}%
			}\mathbb{r\{}\text{orb}_{L_{i}}(\sigma_{i})\}) \hspace{0.1cm}\ar@{^{(}->}[r]^{\hspace{0.1cm}\xi_{i\ast}} \ar@{=}[d] & \pi_{1}(U_{\sigma_{i}%
				,N}\mathbb{r}\{ \text{orb}_{N}(\sigma_{i})\}) \ar@{->>}[r]  \ar@{=}[d] & G_{i}^{\prime \prime} \ar[r] \ar@{=}[d] & \{1\}  \ar@{=}[d] \\
			\{1\} \ar[r] & H_{i} \ar@{^{(}->}[r] & G_{i} \ar@{->>}[r]  & G_{i}^{\prime \prime} \ar[r] & \{1\}
	}}
	\]
	Since $H_{i},G_{i}$ and $G_{i}^{\prime \prime}$ are \textit{finite} groups, we
	have $l_{F_{i}}=\left \vert G_{i}^{\prime \prime}\right \vert =\left \vert
	G_{i}\right \vert /\left \vert H_{i}\right \vert =\frac{q_{i}}{\left \vert
		H_{i}\right \vert }\Rightarrow \left \vert H_{i}\right \vert =\frac{q_{i}%
	}{l_{F_{i}}}=\left \vert G_{i}^{\prime}\right \vert ,$ and we conclude that
	$H_{i}=G_{i}^{\prime}$ and $L_{i}=\Lambda_{F_{i}}.$ Finally, it is by construction obvious that the orbit   orb$_{\Lambda_{F_{i}}}(\sigma_{i})\in U_{\sigma_{i}%
		,\Lambda_{F_{i}}}\cong \mathbb{C}^{2}/G_{i}^{\prime}$ \ is either a smooth
	point (whenever $G_{i}^{\prime}$ is trivial) or a cyclic quotient singularity of type
	 $(1,\tfrac{q_{i}}{l_{F_{i}}})$ (whenever \ $\left \vert G_{i}^{\prime}\right \vert
	>1$). 
	\end{proof}
\noindent Now let $\Lambda_{Q}\subseteq N$ be the sublattice generated by the boundary
lattice points of $Q$ and $\Lambda_{Q^{\ast}}\subseteq M$ be the sublattice
generated by the boundary lattice points of $Q^{\ast}.$

\begin{theorem} \emph{(Kasprzyk \& Nill} \cite[\S 2]{KaNi}\emph{)}\label{CHLATT}
	We have \emph{Hom}$_{\mathbb{Z}}(\Lambda_{Q},\mathbb{Z})=\frac{1}{\ell}%
	\Lambda_{Q^{\ast}}$ and
	\[
	\left \vert N:\Lambda_{Q}\right \vert =\ell=\left \vert M:\Lambda_{Q^{\ast}%
	}\right \vert .
	\]
	In addition, $(Q,\Lambda_{Q})$ and $(Q^{\ast},\Lambda_{Q^{\ast}})$
	\emph{(}with $Q^{\ast}=\ell Q^{\circ}$\emph{)} are $1$-reflexive pairs, where
	$(Q^{\ast},\Lambda_{Q^{\ast}})$ is to be identified with $(Q^{\circ}%
	,\emph{Hom}_{\mathbb{Z}}(\Lambda_{Q},\mathbb{Z}))$.
\end{theorem}
The \textquotedblleft beauty\textquotedblright \ of being $\ell$-reflexive is
mainly embodied in the following property: All \textit{local indices} of the
edges $F_{i}$ of $Q$ coincide with the \textit{index} $\ell$ of the toric log
del Pezzo surface $X(N,\Delta_{Q}),$ and this allows us to patch together the
canonical cyclic covers over the affine neighborhoods of its singularities in
order to create a single \textit{global} finite holomorphic map $\pi_{Q}$ of
\textit{degree} $\ell$ and represent $X(N,\Delta_{Q})$ as a \textit{global}
quotient space. 
\begin{theorem}\label{GLOBALTHM}
	There is an equivariant 	\emph{(}w.r.t. the actions of
	the algebraic tori $\mathbb{T}_{\Lambda_{Q}}$ and $\mathbb{T}_{N}$\emph{)} finite holomorphic map
	\begin{equation}
	\pi_{Q}:X(\Lambda_{Q},\Delta_{Q})\longrightarrow X(N,\Delta_{Q}) \label{BIGCOVMAP}
	\end{equation}
 which has degree $\ell$ and coincides with the
	quotient map by the identification%
	\[
	X(N,\Delta_{Q})\cong X(\Lambda_{Q},\Delta_{Q})/(N/\Lambda_{Q})\  \text{
	}with\  \text{ \emph{Ker}}[\mathbb{T}_{\Lambda_{Q}}\rightarrow \mathbb{T}%
	_{N}]\cong N/\Lambda_{Q}\cong \mathbb{Z}/\ell \,\mathbb{Z}.
	\]
	Moreover, there exist bases $\mathcal{B}$ and $\mathcal{B}^{\diamondsuit}$ of the
	lattices $\Lambda_{Q}$ and $N,$ respectively, as well as a $k\in
	\{1,...,\ell -1\}$ with \emph{gcd}$(k,\ell)=1$ and exactly one $j\in \{1,...,16\},$ such
	that $\Phi_{\mathcal{A}_{\ell,k}}(\overline{\mathcal{Q}}_{j})=Q^{\diamondsuit},$ where
	$\overline{\mathcal{Q}}_{1},...,\overline{\mathcal{Q}}_{16}$ are the
	representatives of the $16$ equivalent classes of the $1$-reflexive
	$\mathbb{Z}^{2}$-polygons given in the table of \emph{\ref{NORMIERUNG},} 
	\[
	\Phi_{\mathcal{A}_{\ell,k}}:\mathbb{R}^{2}\longrightarrow \mathbb{R}^{2}%
	,\  \tbinom{x_{1}}{x_{2}}\longmapsto \Phi_{\mathcal{A}_{\ell,k}}\left(  \tbinom{x_{1}%
	}{x_{2}}\right)  :=\mathcal{A}_{\ell,k}\tbinom{x_{1}}{x_{2}},\ \text{
		with}\footnote{Note that det$(\mathcal{A}_{\ell, k})=\ell.$}\ \ \mathcal{A}%
	_{\ell,k}:=\left(
	\begin{smallmatrix}
	\ell & 0\\
	k & 1
	\end{smallmatrix}
	\right)  ,
	\]
	$Q=\Phi_{\mathcal{B}}(\overline{\mathcal{Q}}_{j}),$ and $Q^{\diamondsuit}%
	:=\Phi_{\mathcal{B}^{\diamondsuit \,-1}}(Q).$ Hence, the dotted arrow (which denotes
	the $\mathbb{T}_{\mathbb{Z}^{2}}$-equivariant holomorphic map induced by
	$\Phi_{\mathcal{A}_{\ell,k}}$) in the following diagram%
	\[
	\xymatrix{
		X(\mathbb{Z}^{2},\Delta_{\overline{\mathcal{Q}}_{j}}) \ar@{-->}[rr] &  & X(\mathbb{Z}%
		^{2},\Delta_{\mathcal{Q}^{\diamondsuit}}\!)\\
		X(\Lambda_{Q},\Delta_{Q}) \ar@{<->}[u]^{\cong} \ar[rr]^{\pi_{Q}} &  & \ar@{<->}[u]_{\cong} X(N,\Delta_{Q})
	}
	\]
	can be viewed again as a quotient map.%
\end{theorem}

\begin{proof}Since $Q$ is $\ell$-reflexive, we have $l_{F_{i}}=\ell$ and $U_{\sigma_{i}%
		,\Lambda_{F_{i}}}=U_{\sigma_{i},\Lambda_{Q}},$ and for the canonical cyclic
	covers $\pi_{U_{\sigma_{i},N}}$ which are constructed by Lemma \ref{TRICKYLEMMA} we obtain
	\[
	\left.  \text{ }\pi_{U_{\sigma_{i},N}}\right \vert _{U_{\sigma_{i},\Lambda_{Q}%
		}\cap U_{\sigma_{i+1},\Lambda_{Q}}}=\left.  \pi_{U_{\sigma_{i+1},N}%
	}\right \vert _{U_{\sigma_{i},\Lambda_{Q}}\cap U_{\sigma_{i+1},\Lambda_{Q}}},
	\]
	for all $i\in \{1,...,\nu \}.$ Since $\left \{  \left.  U_{\sigma_{i},\Lambda
		_{Q}}\right \vert i\in \{1,...,\nu \} \right \}  $ is an open covering of
	$X(\Lambda_{Q},\Delta_{Q}),$ we may patch them together by setting%
	\[
	\pi_{Q}(\mathfrak{x}):=\pi_{U_{\sigma_{i},N}}(\mathfrak{x}),\  \  \forall
	\mathfrak{x}\in U_{\sigma_{i},\Lambda_{Q}}.
	\]
	$\pi_{Q}$ is by definition a finite holomorphic map of degree $\ell=$
	$\left \vert N:\Lambda_{Q}\right \vert ,$ with%
	\[
	\text{Ker}[\mathbb{T}_{\Lambda_{Q}}\longrightarrow \mathbb{T}_{N}%
	]=\text{Hom}_{\mathbb{Z}}(\text{Hom}_{\mathbb{Z}}(\Lambda_{Q},\mathbb{Z}%
	)/M,\mathbb{C}^{\times})\cong N/\Lambda_{Q}\cong \mathbb{Z}/\ell \, \mathbb{Z},
	\]
	and it suffices to apply 
	\cite[Corollary 1.16, pp. 22-23]{Oda} or \cite[Proposition 3.3.7, pp. 127-128]{CLS}.  On the other hand, $(Q,\Lambda_{Q})$ is an $1$-reflexive pair, and utilising suitable bases $\mathcal{B}$ and $\mathcal{B}^{\diamondsuit}$ of the
	lattices $\Lambda_{Q}$ and $N,$ respectively, we may transfer $Q$ to $\mathbb{Z}^{2}$-polygons. To define carefully the matrix $\mathcal{A}_{\ell,k}$, so that $\Phi_{\mathcal{A}_{\ell,k}}$ maps $\overline{\mathcal{Q}}_{j}$ onto  $Q^{\diamondsuit}$, one has to make use of the Hermite normal form. (For details see \cite[Corollary 13]{KaNi}.) 
\end{proof}
\begin{example}	The $\mathbb{Z}^{2}$-triangle $Q:=$ conv$\left(  \left \{  \binom{0}{1}%
	,\binom{14}{3},\binom{-21}{-5}\right \}  \right)  $ is $7$-reflexive, and via
	$\mathcal{A}_{7,1}$ we get $\Phi_{\mathcal{A}_{7,1}}(\overline{\mathcal{Q}%
	}_{7})=Q.$ The toric del Pezzo surface $X(\mathbb{Z}^{2},\Delta_{Q})$ has
	three cyclic quotient singularities: One of type $(5,14)$, one of type
	$(16,21),$ and one of type $(5,7)$. \ $X(\mathbb{Z}^{2},\Delta_{\overline
		{\mathcal{Q}}_{7}})$ inherits a Gorenstein cyclic quotient singularity of type
	$(1,2)$ over the first, a Gorenstein cyclic quotient singularity of type
	$(1,3)$ over the second, and a smooth point over the third.	
\end{example}

\begin{remark} Clearly, Theorem \ref{GLOBALTHM} gives $\sharp($RP$(\ell;N))\leq16\, \phi(\ell),$ where
	$\phi$ is Euler's totient function, but this is only a rough upper bound.
	In fact, $\sharp($RP$(\ell;N))$ depends essentially on number-theoretic
	restrictions on the weights of the possible \textsc{wve}$^{2}$\textsc{c}-graphs. In practice, for the
	classification of $\ell$-reflexive polygons and for the construction of precise
	tables like those in \cite{Br-Kas}, one has to perform ad-hoc tests to distinguish lattice-inequivalent polygons. (Cf. Grinis \& Kasprzyk \cite{GR_KAS} for a more
	general discussion on the normal forms of lattice polytopes.) 
	\end{remark}

\begin{lemma}
	Let $Y$ and $Z$ be two normal projective surfaces and $\pi:Y\longrightarrow Z$
	be a generically finite and surjective holomorphic map of degree $d.$ If $D_{1},D_{2}$ are two $\mathbb{Q}$-Weil
	divisors on $Z,$ then%
	\begin{equation}
	D_{1}\cdot D_{2}=\frac{1}{d}(\pi^{\star}(D_{1})\cdot \pi^{\star}(D_{2}%
	)),\label{INTNUMBCOV}%
	\end{equation}
	where $\pi^{\star}(D_{j})$ is the pullback of $D_{j},$ $j\in \{1,2\},$ via $\pi$
	\emph{(}in the sense of \cite[p. 32]{FultonITH}\emph{)}.
\end{lemma}

\begin{proof}
	Denoting by $\rho:\widetilde{Z}\longrightarrow Z$ the minimal
	desingularization of $Z,$ by $\delta:Y^{\prime}\longrightarrow \widetilde
	{Z}\times_{Z}Y$ the normalisation of the fiber product $\widetilde{Z}%
	\times_{Z}Y,$ and by $\gamma:\widetilde{Y}\longrightarrow Y^{\prime}$ the
	minimal desingularization of $Y^{\prime},$ we obtain a commutative diagram of
	the form:
	\[
	\xymatrix{
		\widetilde{Y} \ar[ddrrrr]_{{\psi:=\varepsilon_{1}\circ\delta\circ\gamma}\hspace{0.5cm}} \ar[rr]^{\gamma}&  & Y^{\prime} \ar[rr]^{{\delta}\hspace{0.5cm}} &  & \widetilde{Z}\times_{Z}Y \ar@{}[ddrr] |{\circlearrowright} \ar[dd]_{{\varepsilon}_{1}} \ar[rr]^{{\varepsilon}_{2}} &  & Y \ar[dd]^{\pi}\\
		&  &  &  &  &  & \\
		&  &  &  & \widetilde{Z} \ar[rr]_{\rho} &  & Z}
	\]
	Since both $\widetilde{Z}$ and $\widetilde{Y}$ are \textit{smooth}, and
	$\psi:\widetilde{Y}\longrightarrow \widetilde{Z}$ is generically finite and
	surjective (of degree $d$), we have%
	\[%
	\begin{array}
	[c]{ll}%
	D_{1}\cdot D_{2}:=\rho^{\star}(D_{1})\cdot \rho^{\star}(D_{2})\medskip &
	\text{(by \cite[pp. 17-18]{Mumford})}\\
	=\frac{1}{d}({\psi}^{\star}(\rho^{\star}(D_{1}))\cdot {\psi}^{\star}(\rho^{\star}(D_{2}))) &
	\text{(by \cite[Proposition I.8 (ii), pp. 4-5]{Beauville}).}%
	\end{array}
	\]
	On the other hand,%
	\[
	\begin{array}
	[c]{ll}%
	{\psi}^{\star}(\rho^{\star}(D_{1}))\cdot {\psi}^{\star}(\rho^{\star}(D_{2}))=(\rho\circ
	\psi)^{\star}(D_{1})\cdot(\rho\circ \psi)^{\star}(D_{2})\medskip \\
	=(\pi \circ \varepsilon_{2} \circ \delta \circ \gamma)^{\star}(D_{1})\cdot
	(\pi \circ \varepsilon_{2} \circ \delta \circ \gamma)^{\star}(D_{2})\medskip \\
	=( \varepsilon_{2} \circ \delta \circ \gamma)^{\star}(\pi^{\star}(D_{1}))\cdot
	( \varepsilon_{2} \circ \delta \circ \gamma)^{\star}(\pi^{\star}(D_{2}))\medskip \\
	=\pi^{\star}(D_{1})\cdot \pi^{\star}(D_{2}) \ \  \text{(by \cite[pp. 17-18]{Mumford} and \cite[7.1.16, p. 125]{FultonITH}) }%
	\end{array}
	\]
	and therefore (\ref{INTNUMBCOV}) is true. \end{proof}
\begin{proposition}\label{KAQUADRAT}
	The self-intersection number of the canonical divisor of $X(N,\Delta_{Q})$ is
	\begin{equation}
	K_{X(N,\Delta_{Q})}^{2}=\frac{1}{\ell}K_{X(\Lambda_{Q},\Delta_{Q})}%
	^{2}.\label{KSQUARE1}%
	\end{equation}
	Correspondingly, the self-intersection number of the canonical divisor of
	$X(M,\Delta_{Q^{\ast}})$ is%
	\begin{equation}
	K_{X(M,\Delta_{Q^{\ast}})}^{2}=\frac{1}{\ell}K_{X(\Lambda_{Q^{\ast}}%
		,\Delta_{Q^{\ast}})}^{2}.\label{KSQUARE2}%
	\end{equation}
\end{proposition}

\begin{proof} Let $\iota:$ Reg$(X(\Lambda_{Q},\Delta_{Q}))\hookrightarrow X(\Lambda
	_{Q},\Delta_{Q})$ and $\iota^{\prime}:$ Reg$(X(N,\Delta_{Q}))\hookrightarrow
	X(N,\Delta_{Q})$ be the natural inclusions of the regular loci of
	$X(\Lambda_{Q},\Delta_{Q})$ and $X(N,\Delta_{Q})$ into themselves. Obviously,
	\[
	\pi_{Q}^{-1}(\text{Reg}(X(N,\Delta_{Q})))\subseteq \text{Reg}(X(\Lambda
	_{Q},\Delta_{Q}))\text{ and codim}_{X(\Lambda_{Q},\Delta_{Q})}(X(\Lambda
	_{Q},\Delta_{Q})\mathbb{r}\pi_{Q}^{-1}(\text{Reg}(X(N,\Delta_{Q}))))=2.
	\]
	Since
	\[
	\left.  \pi_{Q}\right \vert _{\pi_{Q}^{-1}(\text{Reg}(X(N,\Delta_{Q})))}%
	:\pi_{Q}^{-1}(\text{Reg}(X(N,\Delta_{Q})))\longrightarrow \text{Reg}%
	(X(N,\Delta_{Q}))
	\]
	is an \'{e}tale holomorphic map, we have
	\[
	\Omega_{\pi_{Q}^{-1}(\text{Reg}(X(N,\Delta_{Q})))|\text{Reg}(X(N,\Delta_{Q}%
		))}^{1}=0\Longrightarrow \Omega_{\pi_{Q}^{-1}(\text{Reg}(X(N,\Delta_{Q})))}%
	^{1}\cong \pi_{Q}^{\star}(\Omega_{\text{Reg}(X(N,\Delta_{Q}))}^{1}).
	\]
	Passing to ${\textstyle \bigwedge \nolimits^{\! \!2}}\ldots$ and taking into
	account \ref{LEMHART} (ii) this implies
	\[%
	\begin{array}
	[c]{l}%
	\mathcal{O}_{X(\Lambda_{Q},\Delta_{Q})}(K_{X(\Lambda_{Q},\Delta_{Q})}%
	)=\omega_{X(\Lambda_{Q},\Delta_{Q})}=\iota_{\star}(\Omega_{\text{Reg}%
		(X(\Lambda_{Q},\Delta_{Q}))}^{2})=\iota_{\star}(\Omega_{\pi_{Q}^{-1}%
		(\text{Reg}(X(N,\Delta_{Q})))}^{2})\bigskip \\
	\iota_{\star}(\pi_{Q}^{\star}(\Omega_{\text{Reg}(X(N,\Delta_{Q}))}^{2}))=\pi
	_{Q}^{\star}(\iota_{\star}^{\prime}(\Omega_{\text{Reg}(X(N,\Delta_{Q}))}%
	^{2}))=\pi_{Q}^{\star}(\omega_{X(N,\Delta_{Q})})=\mathcal{O}_{X(\Lambda
		_{Q},\Delta_{Q})}(\pi_{Q}^{\star}(K_{X(N,\Delta_{Q})}),
	\end{array}
	\]
	i.e., $K_{X(\Lambda_{Q},\Delta_{Q})}\sim \pi_{Q}^{\star}(K_{X(N,\Delta_{Q})}).$ 
Furthermore, both $X(N,\Delta_{Q})$ and $X(\Lambda_{Q},\Delta_{Q})$ are projective. (See \ref{PROJECTIVE} (ii).) Thus (\ref{INTNUMBCOV}) can be applied for the finite holomorphic map (\ref{BIGCOVMAP}) of degree $\ell$ and for the $\mathbb{Q}$-Weil
divisor $D_{1}=D_{2}=K_{X(N,\Delta_{Q})}$ giving 	
	\[
	K_{X(N,\Delta_{Q})}^{2}=\frac{1}{\ell}\pi_{Q}^{\star}(K_{X(N,\Delta_{Q})}%
	)^{2}=\frac{1}{\ell}K_{X(\Lambda_{Q},\Delta_{Q})}^{2},
	\] 
	i.e., (\ref{KSQUARE1}). The proof of the equality (\ref{KSQUARE2}) is similar. \end{proof}

\section{Second proof and consequences of Theorem \ref{G12PTTHM}%
\label{FIRSTPROOF}}

\noindent{}$\bullet$ \textbf{Notation and basic facts}. Let $\ell$ be a
positive integer. Throughout this section we shall work with fixed $\ell
$-reflexive pairs\textit{ }$(Q,N)$ and $(Q^{\ast},M),$ where $M:=$
Hom$_{\mathbb{Z}}(N,\mathbb{Z}),$ and with the corresponding toric log del
Pezzo surfaces $X(N,\Delta_{Q})$ and $X(M,\Delta_{Q^{\ast}}).$ Let
$\mathbf{n}_{1}=\tbinom{n_{1,1}}{n_{2,1}},\ldots,\mathbf{n}_{\nu}%
=\tbinom{n_{1,\nu}}{n_{2,\nu}}$ be the vertices of $Q$ ordered anticlockwise,
and $F_{i}:=$ conv$(\left \{  \mathbf{n}_{i},\mathbf{n}_{i+1}\right \}  ),$
$i\in \left \{  1,\ldots,\nu \right \}  ,$ be the edges of $Q$ (as in
\S \ref{GRAPHS} and in \ref{NOTELF} (ii)). In these terms, the bijections
(\ref{LDUALITY1}) and (\ref{LDUALITY2}) become
\[
\text{Vert}(Q)\ni \tbinom{n_{1,i}}{n_{2,i}}=\mathbf{n}_{i}\longmapsto
F_{i}^{\ast}:=\text{ conv}(\left \{  \mathbf{m}_{i-1},\mathbf{m}_{i}\right \}
)\in \text{ Edg}(Q^{\ast}),
\]
and%
\[
\text{Edg}(Q)\ni F_{i}\longmapsto \mathbf{m}_{i}:=\boldsymbol{\eta}_{F_{i}%
}=\left(
\begin{smallmatrix}
n_{1,i} & n_{2,i}\\
n_{1,i+1} & n_{2,i+1}%
\end{smallmatrix}
\right)  ^{-1}\tbinom{-\ell}{-\ell}=\tfrac{\ell}{\det \left(  \mathbf{n}%
_{i},\mathbf{n}_{i+1}\right)  }\tbinom{n_{2,i}-n_{2,i+1}}{n_{1,i+1}-n_{1,i}%
}\in \text{ Vert}(Q^{\ast}),
\]
respectively. (By definition, $F_{i},F_{i}^{\ast}$ preserve the involution,
i.e., $\boldsymbol{\eta}_{F_{i}^{\ast}}=\mathbf{n}_{i},$ for all $i\in \left \{
1,\ldots,\nu \right \}  .$ Note that the vertices $\mathbf{m}_{1},\ldots
,\mathbf{m}_{\nu}$ of $Q^{\ast}$ are \ also equipped with anticlockwise
order.) Next, for $i\in \left \{  1,\ldots,\nu \right \}  $ denote by $\sigma
_{i}:=\sigma_{F_{i}}=\mathbb{R}_{\geq0}\mathbf{n}_{i}+\mathbb{R}_{\geq
0}\mathbf{n}_{i+1}$ the $N$-cone supporting $F_{i},$ by $\sigma_{i}^{\ast
}:=\sigma_{F_{i}^{\ast}}=\mathbb{R}_{\geq0}\mathbf{m}_{i-1}+\mathbb{R}_{\geq
0}\mathbf{m}_{i}$ the $M$-cone supporting $F_{i}^{\ast},$ and assume that
$\sigma_{i}$ is a $(p_{i},q_{i})$-cone with $q_{i}=2\,$area$_{N}(T_{F_{i}})$
(see (\ref{QIFORMULA})), and that $\sigma_{i}^{\ast}$ is a $(p_{i}^{\ast
},q_{i}^{\ast})$-cone with $q_{i}^{\ast}=2\,$area$_{M}(T_{F_{i}^{\ast}}).$

\begin{definition}
[Auxiliary cones]\label{AUXCONES}For $i\in \left \{  1,\ldots,\nu \right \}  $ the $N$-cone%
\[
\tau_{i}:=\mathbb{R}_{\geq0}(\frac{\ell}{q_{i-1}}(\mathbf{n}_{i-1}%
-\mathbf{n}_{i}))+\mathbb{R}_{\geq0}(\frac{\ell}{q_{i}}(\mathbf{n}%
_{i+1}-\mathbf{n}_{i}))
\]
will be called \textit{the auxiliary cone} associated with the vertex
$\mathbf{n}_{i}$ of $Q.$ Analogously, the $M$-cone
\[
\tau_{i}^{\ast}:=\mathbb{R}_{\geq0}(\frac{\ell}{q_{i}^{\ast}}(\mathbf{m}%
_{i-1}-\mathbf{m}_{i}))+\mathbb{R}_{\geq0}(\frac{\ell}{q_{i+1}^{\ast}%
}(\mathbf{m}_{i+1}-\mathbf{m}_{i}))
\]
will be the auxiliary cone associated with $\mathbf{m}_{i}\in$ Vert$(Q^{\ast}).$
(Their generators given here are the minimal ones.)$\allowbreak$
\end{definition}

\begin{lemma}
\label{SIGMATAU}$\sigma_{i}^{\ast}=\tau_{i}^{\vee}$ and $\sigma_{i}=(\tau
_{i}^{\ast})^{\vee}$ for all $i\in \left \{  1,\ldots,\nu \right \}  .$
\end{lemma}

\begin{proof}
For each $i\in \left \{  1,\ldots,\nu \right \}  $ the minimal generators of
$\tau_{i}$ are $\tfrac{\ell}{q_{i-1}}\tbinom{n_{1,i-1}-n_{1,i}}{n_{2,i-1}%
	-n_{2,i}}$ and $\tfrac{\ell}{q_{i}}\tbinom{n_{1,i+1}-n_{1,i}}{n_{2,i+1}%
	-n_{2,i}}$  . Since $\sigma_{i}^{\ast}$ is $(p_{i}^{\ast},q_{i}^{\ast})$-cone,
we have $\mathbf{m}_{i}=p_{i}^{\ast}\mathbf{m}_{i-1}+q_{i}^{\ast}%
\mathbf{m}_{i-1}^{\prime},$ where $\left \{  \mathbf{m}_{i-1},\mathbf{m}%
_{i-1}^{\prime}\right \}  $ is a basis of $M.$ The corresponding basis matrix
is
\[
\mathcal{B}:=\tfrac{\ell}{\det \left(  \mathbf{n}_{i-1},\mathbf{n}_{i}\right)
}\left(
\begin{smallmatrix}
n_{2,i-1}-n_{2,i} & \frac{1}{q_{i}^{\ast}}\left(  \tfrac{\det \left(
\mathbf{n}_{i-1},\mathbf{n}_{i}\right)  }{\det \left(  \mathbf{n}%
_{i},\mathbf{n}_{i+1}\right)  }(n_{2,i}-n_{2,i+1})-p_{i}^{\ast}(n_{2,i-1}%
-n_{2,i})\right) \\
n_{1,i}-n_{1,i-1} & \frac{1}{q_{i}^{\ast}}\left(  \tfrac{\det \left(
\mathbf{n}_{i-1},\mathbf{n}_{i}\right)  }{\det \left(  \mathbf{n}%
_{i},\mathbf{n}_{i+1}\right)  }(n_{1,i+1}-n_{1,i})-p_{i}^{\ast}(n_{1,i}%
-n_{1,i-1})\right)
\end{smallmatrix}
\right)  .
\]
Thus, the members of the dual basis of $\left \{  \mathbf{m}_{i},\mathbf{m}%
_{i}^{\prime}\right \}  $ are%
\begin{align*}
(\mathcal{B}^{\intercal})^{-1}\tbinom{1}{0}  &  =\tfrac{\det \left(
\mathbf{n}_{i},\mathbf{n}_{i+1}\right)  }{\ell \left(  \det \left(
\mathbf{n}_{i-1},\mathbf{n}_{i}\right)  +\det \left(  \mathbf{n}_{i}%
,\mathbf{n}_{i+1}\right)  -\det \left(  \mathbf{n}_{i-1},\mathbf{n}%
_{i+1}\right)  \right)  }\left(
\begin{array}
[c]{c}%
\tfrac{\det \left(  \mathbf{n}_{i-1},\mathbf{n}_{i}\right)  \left(
n_{1,i+1}-n_{1,i}\right)  }{\det \left(  \mathbf{n}_{i},\mathbf{n}%
_{i+1}\right)  }-p_{i}^{\ast}\left(  n_{1,i}-n_{1,i-1}\right)  \smallskip \\
-\tfrac{\det \left(  \mathbf{n}_{i-1},\mathbf{n}_{i}\right)  \left(
n_{2,i}-n_{2,i+1}\right)  }{\det \left(  \mathbf{n}_{i},\mathbf{n}%
_{i+1}\right)  }+p_{i}^{\ast}\left(  n_{2,i-1}-n_{2,i}\right)  )
\end{array}
\right) \\
&  =\tfrac{\ell}{q_{i}^{\ast}q_{i-1}}\tbinom{\tfrac{q_{i-1}}{q_{i}}\left(
n_{1,i+1}-n_{1,i}\right)  -p_{i}^{\ast}\left(  n_{1,i}-n_{1,i-1}\right)
}{-\tfrac{q_{i-1}}{q_{i}}\left(  n_{2,i}-n_{2,i+1}\right)  +p_{i}^{\ast
}\left(  n_{2,i-1}-n_{2,i}\right)  )}%
\end{align*}
and%
\[
(\mathcal{B}^{\intercal})^{-1}\tbinom{0}{1}=\tfrac{q_{i}^{\ast}\det \left(
\mathbf{n}_{i},\mathbf{n}_{i+1}\right)  }{\ell \left(  \det \left(
\mathbf{n}_{i-1},\mathbf{n}_{i}\right)  +\det \left(  \mathbf{n}_{i}%
,\mathbf{n}_{i+1}\right)  -\det \left(  \mathbf{n}_{i-1},\mathbf{n}%
_{i+1}\right)  \right)  }\tbinom{-\left(  n_{1,i}-n_{1,i-1}\right)
}{n_{2,i-1}-n_{2,i}}=\tfrac{\ell}{q_{i-1}}\tbinom{n_{1,i-1}-n_{1,i}}%
{n_{2,i-1}-n_{2,i}}.
\]
The minimal generators of the $N$-cone $(\sigma_{i}^{\ast})^{\vee}$ are
$\tfrac{\ell}{q_{i-1}}\tbinom{n_{1,i-1}-n_{1,i}}%
{n_{2,i-1}-n_{2,i}}$ and
\[
q_{i}^{\ast}\mathbf{m}_{i-1}-p_{i}^{\ast}\mathbf{m}_{i-1}^{\prime}=\tfrac
{\ell}{q_{i-1}}\tbinom{\tfrac{q_{i-1}}{q_{i}}\left(  n_{1,i+1}-n_{1,i}\right)
-p_{i}^{\ast}\left(  n_{1,i}-n_{1,i-1}\right)  }{-\tfrac{q_{i-1}}{q_{i}%
}\left(  n_{2,i}-n_{2,i+1}\right)  +p_{i}^{\ast}\left(  n_{2,i-1}%
-n_{2,i}\right)  )}-\tfrac{\ell p_{i}^{\ast}}{q_{i-1}}\tbinom{n_{1,i-1}%
-n_{1,i}}{n_{2,i-1}-n_{2,i}}=\tfrac{\ell}{q_{i}}\tbinom{n_{1,i+1}-n_{1,i}%
}{n_{2,i+1}-n_{2,i}},
\]
i.e., $(\sigma_{i}^{\ast})^{\vee}=\tau_{i}\Rightarrow \sigma_{i}^{\ast}%
=\tau_{i}^{\vee}.$ The proof of the equality $\sigma_{i}=(\tau_{i}^{\ast
})^{\vee}$ is similar.
\end{proof}

\begin{proposition}
\label{DELTASIGMAS}$\Delta_{Q}=\Sigma_{Q^{\ast}}$ and $\Delta_{Q^{\ast}%
}=\Sigma_{Q}.$
\end{proposition}

\begin{proof}
Since $\tau_{i}^{\ast}=\varpi_{\mathbf{m}_{i}}$ (see (\ref{PMCONES})), Lemma
\ref{SIGMATAU} implies that $\varpi_{\mathbf{m}_{i}}^{\vee}=(\tau_{i}^{\ast
})^{\vee}=\sigma_{i}$ for all $i\in \left \{  1,\ldots,\nu \right \}  .$ Hence,
$\Sigma_{Q^{\ast}}=\Delta_{Q}.$ (Alternatively, one may apply Theorem
\ref{AmplenessCORR} for $\Delta=\Delta_{Q}$ and $D=-\ell K_{X(N,\Delta_{Q})},$
because $P_{-\ell K_{X(N,\Delta_{Q})}}=Q^{\ast}.$) Interchanging the roles of
$Q$ and $Q^{\ast}$ we find $\Sigma_{Q}=\Delta_{Q^{\ast}}$ by the same arguments.
\end{proof}

\begin{proposition}
	\label{K2PROP}The self-intersection number of the canonical divisor of
	$X(N,\Delta_{Q})$ is%
	\begin{equation}
	K_{X(N,\Delta_{Q})}^{2}=\frac{1}{\ell}\sharp \left(  \partial Q^{\ast}\cap
	M\right)  . \label{K2FORMULA}%
	\end{equation}
	Correspondingly, the self-intersection number of the canonical divisor of $X(M,\Delta_{Q^{\ast}})$ is
	\begin{equation}
	K_{X(M,\Delta_{Q^{\ast}})}^{2}=\frac{1}{\ell}\sharp \left(  \partial Q\cap
	N\right)  . \label{K2FORMULA2}%
	\end{equation}	
\end{proposition}

\begin{proof}
	Applying Proposition \ref{NOBOUNDARYPNTS} for the $\ell$-reflexive pair
	$(Q^{\ast},M)$ and formula (\ref{DPSQUARE}) (for $P=Q^{\ast}$) we get
	\[
	\ell \, \sharp \left(  \partial Q^{\ast}\cap M\right)  =2\, \text{area}%
	_{M}(Q^{\ast})=(-\ell K_{X(N,\Sigma_{Q^{\ast}})})^{2}=(-\ell K_{X(N,\Delta
		_{Q})})^{2}=\ell^{2}K_{X(N,\Delta_{Q})}^{2}%
	\]
	which gives (\ref{K2FORMULA}). The proof of (\ref{K2FORMULA2}) is similar.
\end{proof}

\noindent $\bullet$ \textbf{Passing to the minimal desingularizations.}
$\ $Let $f:X(N,\widetilde{\Delta}_{Q})\longrightarrow X(N,\Delta_{Q})$ be the
minimal desingularization of $X(N,\Delta_{Q}).$ Consider $\left \{
C_{i}\left \vert i\in \{1,\ldots,\nu \} \right.  \right \}  ,$ the regular and the
negative-regular continued fraction expansions
\begin{equation}
\frac{q_{i}}{q_{i}-p_{i}}=\left[  \! \! \left[  b_{1}^{(i)},b_{2}^{(i)},\ldots,b_{s_{i}%
}^{(i)}\right]  \! \! \right]  ,\  \forall i\in I_{\Delta_{Q}},
\label{KETTENBRUCH1}%
\end{equation}
and
\begin{equation}
\frac{q_{i}}{p_{i}}=\left[  \! \! \left[  b_{1}^{\ast \text{\thinspace}(i)}%
,b_{2}^{\ast \text{\thinspace}(i)},\ldots,b_{t_{i}}^{\ast \text{\thinspace}%
(i)}\right]  \! \! \right] ,\ \text{with}\ \sum_{j=1}^{s_{i}}
(b_{j}^{(i)}-1)=%
\sum_{k=1}^{t_{i}}
(b_{k}^{\ast \,(i)}-1)=s_{i}+t_{i}-1,
\  \forall i\in I_{\Delta_{Q}},
\label{KETTENBRUCH1b}%
\end{equation}
and Hilb$_{N}(\sigma_{i})=\left \{  \left.  \mathbf{u}_{j}^{(i)}\right \vert
j\in \{0,1,\ldots,s_{i}+1\} \right \}  $\ for all $i\in \{1,\ldots,\nu \},$
$\left \{  \overline{C}_{i}\left \vert i\in \{1,\ldots,\nu \} \right.  \right \}
,$
\[
\left \{  \left.  E_{j}^{(i)}\right \vert i\in I_{\Delta_{Q}},\ j\in
\{1,\ldots,s_{i}\} \right \}  ,\  \left \{  \left.  K(E^{(i)})\right \vert i\in
I_{\Delta_{Q}}\right \}  ,\text{ }\left \{  r_{i}\left \vert i\in \{1,\ldots
,\nu \} \right.  \right \}
\]
as in (\ref{DEFCIS}), (\ref{EXPPQCF}), (\ref{UIJS}), (\ref{DEFEIJCS}),
(\ref{TORDISCREP}), and (\ref{DEFRIS}), respectively (where now $\Delta
=\Delta_{Q}$). In the dual sense, consider the regular and the
negative-regular continued fraction expansions
\[
\frac{q_{i}^{\ast}}{q_{i}^{\ast}-p_{i}^{\ast}}=\left[  \! \! \left[
c_{1}^{\ast \text{\thinspace}(i)},c_{2}^{\ast \text{\thinspace}(i)}%
,\ldots,c_{s_{i}^{\ast}}^{\ast \text{\thinspace}(i)}\right]  \! \! \right]
,\  \forall i\in I_{\Delta_{Q^{\ast}}},%
\]
and%
\[
\frac{q_{i}^{\ast}}{p_{i}^{\ast}}=\left[  \! \! \left[  c_{1}^{(i)},c_{2}%
^{(i)},\ldots,c_{t_{i}^{\ast}}^{(i)}\right]  \! \! \right],\ \text{with}\ \sum_{j=1}^{s_{i}^{\ast}}
(c_{j}^{\ast \,(i)}-1)=
\sum_{k=1}^{t_{i}^{\ast}}
(c_{k}^{(i)}-1)=s_{i}^{\ast}+t_{i}^{\ast}-1,
\  \forall i\in
I_{\Delta_{Q^{\ast}}},%
\]
attached to the minimal desingularization, say $\varphi:X(M,\widetilde{\Delta
}_{Q^{\ast}})\longrightarrow X(M,\Delta_{Q^{\ast}}),$ of $X(M,\Delta_{Q^{\ast
}}),$ as well as the other data $\left \{  C_{i}^{\ast}\left \vert
i\in \{1,\ldots,\nu \} \right.  \right \}  ,$ Hilb$_{M}(\sigma_{i}^{\ast
})=\left \{  \left.  \mathbf{u}_{j}^{\ast \text{\thinspace}(i)}\right \vert
j\in \{0,1,\ldots,s_{i}^{\ast}+1\} \right \}  $\ for all $i\in \{1,\ldots,\nu \},$
$\left \{  \overline{C}_{i}^{\ast}\left \vert i\in \{1,\ldots,\nu \} \right.
\right \},$ $\left \{  \left.  E_{j}^{\ast \text{\thinspace}(i)}\right \vert i\in
I_{\Delta_{Q^{\ast}}},\ j\in \{1,\ldots,s_{i}^{\ast}\} \right \}  ,\  \left \{
\left.  K(E^{\ast \text{\thinspace}(i)})\right \vert i\in I_{\Delta_{Q^{\ast}}%
}\right \}  ,\text{ }\left \{  r_{i}^{\ast}\left \vert i\in \{1,\ldots
,\nu \} \right.  \right \}$
\newline which are defined analogously for $\Delta=\Delta_{Q^{\ast}}.$ All the above
accompanying data of $f$ and $\varphi$ will play a crucial role in what
follows.\medskip \ 

\noindent{}$\bullet$ \textbf{Noether's formula}. Since $H^{j}(X(N,\widetilde
{\Delta}_{Q}),\mathcal{O}_{X(N,\widetilde{\Delta}_{Q})})$ is trivial for
$j=1,2,$ the Euler-Poincar\'{e} characteristic
\[
\chi(X(N,\widetilde{\Delta}_{Q}),\mathcal{O}_{X(N,\widetilde{\Delta}_{Q}%
)}):=\sum_{j=0}^{2}\left(  -1\right)  ^{j}\text{dim}_{\mathbb{C}}%
H^{j}(X(N,\widetilde{\Delta}_{Q}),\mathcal{O}_{X(N,\widetilde{\Delta}_{Q})})
\]
of the structure sheaf $\mathcal{O}_{X(N,\widetilde{\Delta}_{Q})}$ equals $1.$
Thus, Noether's formula \cite[p. 154]{Hirzebruch2}:
\[
\chi(X(N,\widetilde{\Delta}_{Q}),\mathcal{O}_{X(N,\widetilde{\Delta}_{Q}%
)})=\frac{1}{12}(K_{X(N,\widetilde{\Delta}_{Q})}^{2}+e(X(N,\widetilde{\Delta
}_{Q})))
\]
can be written as follows:%
\setlength\extrarowheight{3pt}
\begin{equation}
\fbox{$%
\begin{array}
[c]{ccc}
& K_{X(N,\widetilde{\Delta}_{Q})}^{2}+e(X(N,\widetilde{\Delta}_{Q}))=12. &
\end{array}
$} \label{NOETHERSFORMULA}%
\end{equation}
\setlength\extrarowheight{-3pt}
$\bullet$ \textbf{Case 1.} $\fbox{$\ell=1.$}$ In this case, $Q^{\ast}%
=Q^{\circ},$ $l_{F_{i}}=1$ for all $i\in \left \{  1,\ldots,\nu \right \}  $ (see
Proposition \ref{CONDREFL}), and by (\ref{LOCALINDEXFORMULA}) we infer that
\[
q_{i}=\text{gcd}(q_{i},p_{i}-1),\  \forall i\in I_{\Delta_{Q}}\Rightarrow
p_{i}=1,\ s_{i}=q_{i}-1,\  \forall i\in I_{\Delta_{Q}}.
\]
Therefore $X(N,\Delta_{Q})$ is either smooth (whenever $I_{\Delta_{Q}%
}=\varnothing$) or has only Gorenstein singularities (whenever $I_{\Delta_{Q}%
}\neq \varnothing$); cf. Proposition \ref{GORENSTPROP}. Moreover, by
Proposition \ref{CREPANTPROP} $f$ is crepant.

\begin{note}
[Alternative proof of Theorem \ref{12PTTHM}.]\label{ALTPROOF12PTTHM}Combining
the fact that $f$ is crepant with (\ref{K2FORMULA}) the self-intersection
number of the canonical divisor of $X(N,\widetilde{\Delta}_{Q})$ equals%
\begin{equation}
K_{X(N,\widetilde{\Delta}_{Q})}^{2}=K_{X(N,\Delta_{Q})}^{2}=\sharp \left(
\partial Q^{\circ}\cap M\right)  . \label{K2QSTAR}%
\end{equation}
On the other hand, one computes the topological Euler characteristic of
$X(N,\widetilde{\Delta}_{Q})$ by \ref{NoteCAND} (ii) and (\ref{BOUNDARYAREA}%
):
\begin{equation}
e(X(N,\widetilde{\Delta}_{Q}))=\nu+\sum_{i\in I_{\Delta_{Q}}}s_{i}=\nu
+\sum_{i\in I_{\Delta_{Q}}}\left(  q_{i}-1\right)  =\sum_{i=1}^{\nu}%
q_{i}=2\, \text{area}_{N}(Q)=\sharp \left(  \partial Q\cap N\right)  .
\label{EBOUNDARY}%
\end{equation}
Formula (\ref{12PTFORMULA}) follows from (\ref{K2QSTAR}), (\ref{EBOUNDARY})
and (\ref{NOETHERSFORMULA}). Obviously,%

\begin{equation}
\sharp(\partial Q^{\circ}\cap M)-K_{X(N,\widetilde{\Delta}_{Q})}%
^{2}=0=e(X(N,\widetilde{\Delta}_{Q}))-\sharp(\partial Q\cap
N),\label{CHARDIFF1}%
\end{equation}
and, analogously,%
\begin{equation}
\sharp(\partial Q\cap N)-K_{X(M,\widetilde{\Delta}_{Q^{\circ}})}%
^{2}=0=e(X(M,\widetilde{\Delta}_{Q^{\circ}}))-\sharp(\partial Q^{\circ}\cap
M).\label{CHARDIFF2}%
\end{equation}
We shall hereafter call these two couples of differences occuring in
(\ref{CHARDIFF1}) and (\ref{CHARDIFF2}) \textit{characteristic differences}
w.r.t. $Q$ (and w.r.t. $Q^{\ast}%
=Q^{\circ},$ respectively). As we shall verify below in
\S  \ref{CHARDIFFSEC}, these do not vanish whenever $\ell>1,$ and they have an interesting
geometric interpretation. (See (\ref{THETAIQ1F}) and (\ref{THETAIQ1F2}).)  

\end{note}

\noindent{}$\bullet$ \textbf{Case 2.} $\fbox{$\ell>1.$}$ In this case,
$I_{\Delta_{Q}}=\left \{  1,\ldots,\nu \right \}  ,$ and $X(N,\Delta_{Q})$ has
exactly $\nu$ singularities, all of which are non-Gorenstein singularities
(see Proposition \ref{GORENSTPROP}) because by hypothesis and by
(\ref{LOCALINDEXFORMULA}) we conclude
\begin{equation}
l_{F_{i}}=\frac{q_{i}}{\text{gcd}(q_{i},p_{i}-1)}=\ell \geq2\Rightarrow
p_{i}\geq2,\  \forall i\in \left \{  1,\ldots,\nu \right \}  . \label{PIGE2}%
\end{equation}
Analogously, $I_{\Delta_{Q^{\ast}}}=\left \{  1,\ldots,\nu \right \}  ,$
$p_{i}^{\ast}\geq2$ for all $i\in \left \{  1,\ldots,\nu \right \}  ,$ and
$X(M,\Delta_{Q^{\ast}})$ has exactly $\nu$ (non-Gorenstein) singularities.  \smallskip \newline
\noindent $\blacktriangleright$ \textit{Proof of Theorem \ref{G12PTTHM} for}
$\ell>1$. Passing from lattice $N$ to lattice $\Lambda_{Q}$ (and, respectively,  from
$M$ to $\Lambda_{Q^{\ast}}$) we denote by $\widehat{f}:X(\Lambda_{Q}%
,\widehat{\Delta}_{Q})\longrightarrow X(\Lambda_{Q},\Delta_{Q})$ (resp., by
$\widehat{\varphi}:X(\Lambda_{Q^{\ast}},\widehat{\Delta}_{Q^{\ast}%
})\longrightarrow X(\Lambda_{Q^{\ast}},\Delta_{Q^{\ast}})$) the minimal
desingularization of the Gorenstein toric log del Pezzo surface $X(\Lambda
_{Q},\Delta_{Q})$ (resp., of $X(\Lambda_{Q^{\ast}},\Delta_{Q^{\ast}})$). 
 Since orb$_{\Lambda_{Q}}(\sigma_{i})$ is either a nonsingular point (whenever
$\sigma_{i}$ is a basic $\Lambda_{Q}$-cone and $q_{i}=\ell$) or a Gorenstein
cyclic quotient singularity (whenever $\sigma_{i}$ is a non-basic $\Lambda_{Q}$-cone, necessarily of type $(1,\frac{q_{i}}{\ell})$ with\footnote{The number of the lattice points lying in the interior of an edge of an $1$-reflexive polygon is $\le3$. (See Figure \ref{Fig.0}.)} $\frac{q_{i}}{\ell}\in\{2,3,4\}$), formula (\ref{EBOUNDARY}) applied for the lattice $\Lambda_{Q}$ and
the refinement $\widehat{\Delta}_{Q}$ of the  $\Lambda_{Q}$-fan $\Delta_{Q}$ gives%
\begin{equation}
e(X(\Lambda_{Q},\widehat{\Delta}_{Q}))=\sum_{i=1}^{\nu}\frac{q_{i}}{\ell}\underset{\text{(\ref{QIFORMULA})}}{=}\sum_{i=1}^{\nu}\frac{2\, \text{area}_{N}(T_{F_{i}})}{\ell
}=\frac{2\, \text{area}_{N}(Q)}{\ell
}\underset{\text{(\ref{BOUNDARYAREA})}}{=}\sharp(\partial Q\cap N).\label{EDENTAQDACH}%
\end{equation}%
On the other hand, since $\widehat{f}$ is crepant,
\begin{equation}
K_{X(\Lambda_{Q},\widehat{\Delta}%
	_{Q})}^{2}=K_{X(\Lambda_{Q},\Delta_{Q})}^{2}. \label{KSQTWOTIMES}
\end{equation}
Hence,
\[%
\begin{array}
[c]{ll}%
12\underset{\text{(\ref{12PTFORMULA})}}{=}\sharp(\partial Q\cap \Lambda_{Q})+\sharp(\partial
Q^{\circ}\cap \text{Hom}_{\mathbb{Z}}(\Lambda_{Q},\mathbb{Z}))=\sharp(\partial
Q\cap \Lambda_{Q})+\sharp(\partial Q^{\ast}\cap \Lambda_{Q^{\ast}}) \medskip & \\
=e(X(\Lambda_{Q},\widehat{\Delta}_{Q}))+K_{X(\Lambda_{Q},\widehat{\Delta}%
	_{Q})}^{2}=\sharp(\partial Q\cap N)+K_{X(\Lambda_{Q},\Delta_{Q})}^{2}
& \text{(by (\ref{EDENTAQDACH}) and (\ref{KSQTWOTIMES}))}\medskip\\
=\sharp(\partial Q\cap N)+\ell K_{X(N,\Delta_{Q})}^{2}=\sharp(\partial Q\cap
N)+\sharp(\partial Q^{\ast}\cap M) & \text{(by (\ref{KSQUARE1}) and (\ref{K2FORMULA})).}%
\end{array}
\]
One could, of course, use (\ref{12PTFORMULA}), $e(X(\Lambda_{Q^{\ast}},\widehat{\Delta}_{Q\ast
}))=\sharp(\partial Q^{\ast}\cap M)$, the fact that $\widehat{\varphi}$ is crepant (leading to $K_{X(\Lambda_{Q^{\ast}},\widehat{\Delta}%
_{Q^{\ast}})}^{2}=K_{X(\Lambda_{Q^{\ast}},\Delta_{Q^{\ast}})}^{2}$), (\ref{KSQUARE2}) and (\ref{K2FORMULA2}), instead. Thus (\ref{G12PTFORMULA})
is true. \hfill $\square$ \medskip

\noindent{}$\bullet$ \textbf{Consequences of Theorem \ref{G12PTTHM}}. Let
$\ell$ be an integer $\geq1$ and let $(Q,N)$ be an $\ell$-reflexive
pair\textit{.} Maintaining the notation introduced above, formula
(\ref{G12PTFORMULA}) gives significant information about $Q,Q^{\ast},$
$\sharp \left(  \partial Q\cap N\right)  ,$ $\sharp(\partial Q^{\ast}\cap
M),\ell,$ and the combinatorial triples of the corresponding fans $\Delta
_{Q},\Delta_{Q^{\ast}}.$

\begin{corollary}
[Upper bound for the number of vertices]\label{Vert6}$\sharp(\emph{Vert}%
(Q))=\sharp(\emph{Vert}(Q^{\ast}))=\nu \leq6.$
\end{corollary}

\noindent\textit{First proof.} Since the number of the vertices of $Q$ (resp., of $Q^{\ast}$) does not change by passing from lattice $N$ to lattice $\Lambda_{Q}$ (resp., from $M$ to $\Lambda_{Q^{\ast}}$), the claim is correct by Theorem \ref{CLASSIFRP}.  \newline \medskip
\noindent \textit{Second proof.} (\ref{G12PTFORMULA}) directly implies%
\begin{align*}
12  &  =\sharp \left(  \partial Q\cap N\right)  +\sharp(\partial Q^{\ast}\cap
M)=\frac{1}{\ell}\left(  2\, \text{area}_{N}(Q)+2\, \text{area}_{M}(Q^{\ast
})\right) \\
&  =\frac{2}{\ell}\left(  \sum_{i=1}^{\nu}\text{area}_{N}(T_{F_{i}}%
)+\sum_{i=1}^{\nu}\text{area}_{M}(T_{F_{i}^{\ast}})\right)
=\sum_{i=1}^{\nu}\left(  \sharp \left(  F_{i}\cap N\right)  -1\right)
+\sum_{i=1}^{\nu}\left(  \sharp \left(  F_{i}^{\ast}\cap M\right)  -1\right)
\geq2\nu,
\end{align*}
i.e., $\sharp($Vert$(Q))=\sharp($Vert$(Q^{\ast}))=\nu \leq6.$ \hfill $\square$

\begin{corollary}
[All possible values of $\sharp \left(  \partial Q\cap N\right)  $ and
$\sharp(\partial Q^{\ast}\cap M)$]\label{VALUESBOUNDARY}We have%
\[
(\sharp \left(  \partial Q\cap N\right)  ,\sharp(\partial Q^{\ast}\cap
M))\in \left \{  \left(  3,9\right)
,(4,8),(5,7),(6,6),(7,5),(8,4),(9,3)\right \}  .
\]

\end{corollary}

\begin{proof}
Since $\sharp \left(  \partial Q\cap N\right)  \geq3$ and $\sharp(\partial
Q^{\ast}\cap M)\geq3,$ this follows directly from (\ref{G12PTFORMULA}).
\end{proof}

\begin{corollary}
[\textquotedblleft Oddness\textquotedblright of $\ell$]\label{ODDNESS2}The index $\ell$ of $Q$ is always odd.
\end{corollary}

\begin{proof}
Suppose that the index $\ell$ of $Q$ is \textit{even}. By Corollary
\ref{ODDNESS1} $\sharp \left(  \partial Q\cap N\right)  $ has to be even.
Therefore, by Corollary \ref{VALUESBOUNDARY}, $\sharp \left(  \partial Q\cap
N\right)  \in \{4,6,8\}.$ Taking into account that%
\begin{equation}
\frac{2}{\ell}\, \text{area}_{N}(Q)=\dfrac{1}{\ell}\left(
{\displaystyle \sum \limits_{i=1}^{\nu}}
q_{i}\right)  =%
{\displaystyle \sum \limits_{i=1}^{\nu}}
\text{gcd}(q_{i},p_{i}-1)=\sharp \left(  \partial Q\cap N\right)  ,
\label{GCDBOUNDARY}%
\end{equation}
we examine the three cases separately:\smallskip \  \newline(i) If
$\sharp \left(  \partial Q\cap N\right)  =4,$ then $\nu \in \{3,4\}$ and by
(\ref{GCDBOUNDARY}) $\exists i_{\bullet}\in \{1,\ldots,\nu \}:$
gcd$(q_{i_{\bullet}},p_{i_{\bullet}}-1)=1.$ Since $\ell=q_{i_{\bullet}}$ is
even, $p_{i_{\bullet}}$ is even $\geq2.$ This is impossible because
gcd$(p_{i_{\bullet}},q_{i_{\bullet}})=1.\smallskip$ \newline(ii) If
$\sharp \left(  \partial Q\cap N\right)  =8,$ then $\sharp(\partial Q^{\ast
}\cap M)=4,$ which is again impossible (by using the same argument as in case
(i) but this time with $Q^{\ast}$ in the place of $Q$).\smallskip
\  \newline(iii) If $\sharp \left(  \partial Q\cap N\right)  =6,$ then $\nu
\in \{3,4,5,6\}.$ For $\nu \in \{4,5,6\}$ equality (\ref{GCDBOUNDARY}) informs us
that there is an $i_{\bullet}\in \{1,\ldots,\nu \}:$ gcd$(q_{i_{\bullet}%
},p_{i_{\bullet}}-1)=1,$ leading to contradiction (as in case (i)). It remains
to see what happens for $\nu=3$ under the assumption that $\nexists
i_{\bullet}\in \{1,2,3\}:$ gcd$(q_{i_{\bullet}},p_{i_{\bullet}}-1)=1.$ In this
case we have necessarily%
\[
\text{gcd}(q_{1},p_{1}-1)=\text{gcd}(q_{2},p_{2}-1)=\text{gcd}(q_{3}%
,p_{3}-1)=2\text{ and }q_{1}=q_{2}=q_{3}=2\ell,
\]
and consequently $p_{1},p_{2}$ and $p_{3}$ are odd $\geq3.$ By \cite[Lemma
6.2, pp. 232-233]{Dais2} we obtain%
\[
q_{1}q_{2}\mid \widehat{p}_{1}q_{2}+p_{2}q_{1}+q_{3}\Longrightarrow2\ell
\mid \widehat{p}_{1}+p_{2}+1.
\]
Since the socius $\widehat{p}_{1}$ of $p_{1}$ is odd too, the last
divisibility condition is impossible (because $\widehat{p}_{1}+p_{2}+1$ is an
odd integer). By (i), (ii) and (iii) we conclude that $\ell$ is always odd.
\end{proof}

\begin{note} (i) All possible values for the numbers of boundary lattice points described in Corollary \ref{VALUESBOUNDARY} can be taken, as it is shown by examples \ref{EXAMPLESLREF}. \smallskip \newline (ii) If  $\ell=3\lambda$, where  $\lambda$ is a positive odd integer, then it can be proven that  $Q$ has to be a lattice \textit{hexagon}. (See \cite[\S 2.5]{KaNi}.)  
\end{note}

\begin{proposition}
	\label{EXPRESSIONQSTAR}For each $i\in \left \{  1,\ldots,\nu \right \}  $ we have
	\begin{equation}
	q_{i}^{\ast}=\ell^{2}\left(  \frac{1}{q_{i-1}}+\frac{1}{q_{i}}-\frac
	{1}{q_{i-1}q_{i}}\frac{\det \left(  \mathbf{n}_{i-1},\mathbf{n}_{i+1}\right)
	}{\det(N)}\right)  \label{QISTAR1}%
	\end{equation}
	and%
	\begin{equation}
	q_{i}=\ell^{2}\left(  \frac{1}{q_{i}^{\ast}}+\frac{1}{q_{i+1}^{\ast}}-\frac
	{1}{q_{i}^{\ast}q_{i+1}^{\ast}}\frac{\det \left(  \mathbf{m}_{i},\mathbf{m}%
		_{i+2}\right)  }{\det(M)}\right)  . \label{QI1}%
	\end{equation}
	
\end{proposition}

\begin{proof}
	Since $\mathbf{m}_{1},\ldots,\mathbf{m}_{\nu}$ and $\mathbf{n}_{1}%
	,\ldots,\mathbf{n}_{\nu}$ are ordered anticlockwise, we have
	\[
	\det(\mathbf{m}_{i-1},\mathbf{m}_{i})>0,\  \  \det \left(  \mathbf{n}%
	_{i-1},\mathbf{n}_{i}\right)  >0,\  \  \det \left(  \mathbf{n}_{i},\mathbf{n}%
	_{i+1}\right)  >0,
	\]
	and (\ref{QIFORMULA}) (applied for the $M$-cone $\sigma_{i}^{\ast}$) gives
	\begin{align*}
	q_{i}^{\ast}  &  =\frac{\det(\mathbf{m}_{i-1},\mathbf{m}_{i})}{\det(M)}%
	=\tfrac{\ell^{2}\det(N)}{\det \left(  \mathbf{n}_{i-1},\mathbf{n}_{i}\right)
		\det \left(  \mathbf{n}_{i},\mathbf{n}_{i+1}\right)  }\det \left(
	\begin{smallmatrix}
	n_{2,i-1}-n_{2,i} & n_{2,i}-n_{2,i+1}\\
	n_{1,i}-n_{1,i-1} & n_{1,i+1}-n_{1,i}%
	\end{smallmatrix}
	\right) \\
	&  =\tfrac{\ell^{2}\det(N)}{\det \left(  \mathbf{n}_{i-1},\mathbf{n}%
		_{i}\right)  \det \left(  \mathbf{n}_{i},\mathbf{n}_{i+1}\right)  }(\det \left(
	\mathbf{n}_{i},\mathbf{n}_{i+1}\right)  +\det \left(  \mathbf{n}_{i-1}%
	,\mathbf{n}_{i}\right)  -\det \left(  \mathbf{n}_{i-1},\mathbf{n}_{i+1}\right)
	)\\
	&  =\ell^{2}\left(  \frac{1}{\frac{\det \left(  \mathbf{n}_{i-1},\mathbf{n}%
			_{i}\right)  }{\det(N)}}+\frac{1}{\frac{\det \left(  \mathbf{n}_{i}%
			,\mathbf{n}_{i+1}\right)  }{\det(N)}}-\frac{\frac{\det \left(  \mathbf{n}%
			_{i-1},\mathbf{n}_{i+1}\right)  }{\det(N)}}{\frac{\det \left(  \mathbf{n}%
			_{i-1},\mathbf{n}_{i}\right)  }{\det(N)}\frac{\det \left(  \mathbf{n}%
			_{i},\mathbf{n}_{i+1}\right)  }{\det(N)}}\right) \\
	&  =\ell^{2}\left(  \frac{1}{q_{i-1}}+\frac{1}{q_{i}}-\frac{1}{q_{i-1}q_{i}%
	}\frac{\det \left(  \mathbf{n}_{i-1},\mathbf{n}_{i+1}\right)  }{\det
		(N)}\right)  ,
	\end{align*}
	for all $i\in \left \{  1,\ldots,\nu \right \}  .$ The proof of equality
	(\ref{QI1}) is similar.
\end{proof}

\begin{corollary}
[Determinantal identities]The multiplicities $q_{1},\ldots,q_{\nu}$ of the
$N$-cones of $\Delta_{Q}$ satisfy the following identity\emph{:}%
\begin{equation}%
{\displaystyle \sum \limits_{i=1}^{\nu}}
\left(  \dfrac{q_{i}}{\ell}+\dfrac{2\ell}{q_{i}}\right)  =%
{\displaystyle \sum \limits_{i=1}^{\nu}}
\frac{\ell}{q_{i}q_{i+1}}\frac{\det \left(  \mathbf{n}_{i},\mathbf{n}%
_{i+2}\right)  }{\det(N)}+12 \label{DETIDENTITY1}%
\end{equation}
In dual terms, the multiplicities $q_{1}^{\ast},\ldots,q_{\nu}^{\ast}$ of
the $M$-cones of $\Delta_{Q^{\ast}}$ satisfy the identity\emph{:}
\begin{equation}%
{\displaystyle \sum \limits_{i=1}^{\nu}}
\left(  \dfrac{q_{i}^{\ast}}{\ell}+\dfrac{2\ell}{q_{i}^{\ast}}\right)  =%
{\displaystyle \sum \limits_{i=1}^{\nu}}
\frac{\ell}{q_{i}^{\ast}q_{i+1}^{\ast}}\frac{\det \left(  \mathbf{m}%
_{i},\mathbf{m}_{i+2}\right)  }{\det(M)}+12 \label{DETIDENTITY2}%
\end{equation}

\end{corollary}

\begin{proof}
Formula (\ref{G12PTFORMULA}) can be rewritten via Proposition \ref{EXPRESSIONQSTAR}
in the form%
\begin{align*}
12  &  =\sharp \left(  \partial Q\cap N\right)  +\sharp(\partial Q^{\ast}\cap
M)=\frac{1}{\ell}\left(  2\, \text{area}_{N}(Q)+2\, \text{area}_{M}(Q^{\ast
})\right) \\
&  =\frac{1}{\ell}\sum_{i=1}^{\nu}q_{i}+\frac{1}{\ell}\sum_{i=1}^{\nu}%
q_{i}^{\ast} =\frac{1}{\ell}\sum_{i=1}^{\nu}q_{i}+\sum_{i=1}^{\nu}\ell \left(  \frac
{2}{q_{i}}-\frac{1}{q_{i}q_{i+1}}\frac{\det \left(  \mathbf{n}_{i}%
,\mathbf{n}_{i+2}\right)  }{\det(N)}\right)  .
\end{align*}
Hence, (\ref{DETIDENTITY1}) is true. The proof of (\ref{DETIDENTITY2}) is similar.
\end{proof}

\begin{corollary}
[Dedekind sum identities]If $\ell>1,$ then the Dedekind sums of the pairs $(p_{1},q_{1}%
),\ldots,(p_{\nu},q_{\nu})$ satisfy the identity\emph{:}%
\begin{equation}
12\left(
{\displaystyle \sum \limits_{i=1}^{\nu}}
\text{\emph{DS}}(p_{i},q_{i})\right)  =12-3\nu+%
{\displaystyle \sum \limits_{i=1}^{\nu}}
\frac{1}{q_{i}q_{i+1}}\frac{\det \left(  \mathbf{n}_{i},\mathbf{n}%
_{i+2}\right)  }{\det(N)}. \label{DSFORMULA1}%
\end{equation}
In dual terms, the Dedekind sums of the pairs $(p_{1}^{\ast},q_{1}^{\ast
}),\ldots,(p_{\nu}^{\ast},q_{\nu}^{\ast})$ satisfy the identity\emph{:}
\begin{equation}
12\left(
{\displaystyle \sum \limits_{i=1}^{\nu}}
\text{\emph{DS}}(p_{i}^{\ast},q_{i}^{\ast})\right)  =12-3\nu+%
{\displaystyle \sum \limits_{i=1}^{\nu}}
\frac{1}{q_{i}^{\ast}q_{i+1}^{\ast}}\frac{\det \left(  \mathbf{m}%
_{i},\mathbf{m}_{i+2}\right)  }{\det(M)}. \label{DSFORMULA2}%
\end{equation}

\end{corollary}

\begin{proof}
By (\ref{K2FORMULA}) and (\ref{QISTAR1}) we obtain%
\begin{equation}
K_{X(N,\Delta_{Q})}^{2}=\frac{1}{\ell}\sharp \left(  \partial Q^{\ast}\cap
M\right)  =\frac{1}{\ell^{2}}\sum_{i=1}^{\nu}q_{i}^{\ast}=\sum_{i=1}^{\nu
}\left(  \frac{2}{q_{i}}-\frac{1}{q_{i}q_{i+1}}\frac{\det \left(
\mathbf{n}_{i},\mathbf{n}_{i+2}\right)  }{\det(N)}\right)  . \label{K2FIRST}%
\end{equation}
Formula (\ref{DSFORMULA}) leads to another version of Noether's formula (see
\cite[Corollary 4.10, p. 99]{Dais1}):
\begin{equation}
K_{X(N,\Delta_{Q})}^{2}=12-\nu+\sum_{i=1}^{\nu}\left(  \frac{2}{q_{i}%
}-12\text{DS}(p_{i},q_{i})-2\right)  . \label{K2SECOND}%
\end{equation}
(\ref{DSFORMULA1}) follows from (\ref{K2FIRST}) and (\ref{K2SECOND}). The
proof of (\ref{DSFORMULA2}) is similar.
\end{proof}

\noindent $\bullet$ \textbf{Suyama's formula}. Let $N\subset \mathbb{R}^{2}$ be a
lattice. If $\mathbf{v}_{1,\ldots,}\mathbf{v}_{\nu}\in N$
is a sequence of \textit{primitive} lattice points with $\mathbf{v}%
_{0}:=\mathbf{v}_{\nu}$ and $\mathbf{v}_{\nu+1}:=\mathbf{v}_{1}$, one denotes by
\[
\text{Rot}\left(  \mathbf{v}_{1,\ldots,}\mathbf{v}_{\nu}\right)  :=\frac
{1}{2\pi}\sum_{i=1}^{\nu}\int_{\text{conv}(\{ \mathbf{v}_{i},\mathbf{v}%
	_{i+1}\})}\frac{-\mathfrak{y}\ d\mathfrak{x}+\mathfrak{x}\ d\mathfrak{y}%
}{\mathfrak{x}^{2}+\mathfrak{y}^{2}}%
\]
the \textit{rotation} (or \textit{winding}) \textit{number} of $\mathbf{v}_{1,\ldots,}\mathbf{v}_{\nu
}$ around $\mathbf{0}.$ Suyama gave a nice formula in \cite[Theorem 6, p. 854]{Suyama}, by means
of which one computes Rot$\left(  \mathbf{v}_{1,\ldots,}\mathbf{v}_{\nu
}\right)  .$ Applying his formula for the (very special) sequence
$\mathbf{n}_{1,\ldots,}\mathbf{n}_{\nu}$ of the vertices of $Q$ for $\ell >1$ (and taking
into account the continued fraction expansion (\ref{KETTENBRUCH1}) of $\frac{q_{i}}%
{q_{i}-p_{i}}$ for all $i\in \{1,...,\nu \}\left(=I_{\Delta}\right)$) we obtain
\begin{equation}
\text{Rot}\left(  \mathbf{n}_{1,\ldots,}\mathbf{n}_{\nu}\right)  =\tfrac
{1}{12}%
{\textstyle \sum \limits_{i=1}^{\nu}}
\left(  3(s_{i}+1)-%
{\textstyle \sum \limits_{j=1}^{s_{i}}}
b_{j}^{(i)}-\tfrac{1}{q_{i-1}q_{i}}\tfrac{\text{det}(\mathbf{n}_{i-1}%
	,\mathbf{n}_{i+1})}{\text{det}(N)}-\tfrac{(q_{i}-p_{i})+(q_{i}-\widehat{p}%
	_{i})}{q_{i}}\right)  .\label{SUYAMFORM}%
\end{equation}
Obviously, by construction,
\begin{equation}
\text{Rot}\left(  \mathbf{n}_{1,\ldots,}\mathbf{n}_{\nu}\right)
=1.\label{ROTGL1}%
\end{equation}

\begin{proposition}
	If $\ell>1,$ then \emph{(\ref{SUYAMFORM})} and \emph{(\ref{ROTGL1})} are equivalent to the known
	formula%
	\begin{equation}%
	{\textstyle \sum \limits_{i=1}^{\nu}}
	r_{i}=3\nu-12-%
	{\textstyle \sum \limits_{i=1}^{\nu}}
	{\textstyle \sum \limits_{j=1}^{s_{i}}}
	(b_{j}^{(i)}-3),\label{DAISFORMULA}%
	\end{equation}
	which follows from a generalised version of Noether's formula. \emph{(See \cite[pp. 99-100]{Dais1}.)}
\end{proposition}
\begin{proof} By (\ref{QISTAR1}), Proposition \ref{NOBOUNDARYPNTS} and (\ref{QIFORMULA}) (applied for $Q^{\ast}$ and $q_{i}^{\ast}$, respectively), and (\ref{K2FORMULA}) we have
	\[
	\tfrac{1}{q_{i-1}q_{i}}\tfrac{\text{det}(\mathbf{n}_{i-1},\mathbf{n}_{i+1}%
		)}{\text{det}(N)}=-\tfrac{q_{i}^{\ast}}{\ell^{2}}+\tfrac{1}{q_{i-1}}+\tfrac
	{1}{q_{i}}\Rightarrow \sum_{i=1}^{\nu}\tfrac{1}{q_{i-1}q_{i}}\tfrac
	{\text{det}(\mathbf{n}_{i-1},\mathbf{n}_{i+1})}{\text{det}(N)}=-K_{X(N,\Delta
		_{Q})}^{2}+\sum_{i=1}^{\nu}\tfrac{2}{q_{i}}.%
	\]
	By \cite[Proposition 4.8, p. 98]{Dais1} we know that
	\[%
	\begin{array}
	[c]{l}%
	{\textstyle \sum \limits_{i=1}^{\nu}}
	\tfrac{2}{q_{i}}=K_{X(N,\Delta_{Q})}^{2}+%
	{\textstyle \sum \limits_{i=1}^{\nu}}
	r_{i}-%
	{\textstyle \sum \limits_{i=1}^{\nu}}

	\tfrac{(q_{i}-p_{i})+(q_{i}-\widehat{p}_{i})}{q_{i}}\medskip \\
	\Rightarrow%
	{\textstyle \sum \limits_{i=1}^{\nu}}
	\left(  -\tfrac{1}{q_{i-1}q_{i}}\tfrac{\text{det}(\mathbf{n}_{i-1}%
		,\mathbf{n}_{i+1})}{\text{det}(N)}-\tfrac{(q_{i}-p_{i})+(q_{i}-\widehat{p}%
		_{i})}{q_{i}}\right)  =%
	{\textstyle \sum \limits_{i=1}^{\nu}}
	r_{i}.%
	\end{array}
	\]
	Hence, (\ref{SUYAMFORM}) and (\ref{ROTGL1}) give
	\[
	12=%
	{\textstyle \sum \limits_{i=1}^{\nu}}\left(  3(s_{i}+1)-%
	{\textstyle \sum \limits_{j=1}^{s_{i}}}
	b_{j}^{(i)}\right)  +%
	{\textstyle \sum \limits_{i=1}^{\nu}}
	r_{i}\Rightarrow3\nu-12-
	{\textstyle \sum \limits_{i=1}^{\nu}}
	{\textstyle \sum \limits_{j=1}^{s_{i}}}
	(b_{j}^{(i)}-3)=%
	{\textstyle \sum \limits_{i=1}^{\nu}}
	r_{i},
	\] i.e., (\ref{DAISFORMULA}). 
\end{proof}

\noindent $\bullet$ \textbf{Further interrelations of the data of both sides}. The duality established by the bijections
(\ref{LDUALITY1}) and (\ref{LDUALITY2}) implies certain additional number-theoretic identities which involve the combinatorial triples of both sides.   

\begin{proposition}
	\label{QISTAREXPRESSION}If $\ell>1,$ then for each $i\in \left \{  1,\ldots,\nu \right \}  $ we
	have
	\begin{equation}
	q_{i}^{\ast}=\ell^{2}\left(  \frac{q_{i-1}-\widehat{p}_{i-1}+1}{q_{i-1}}%
	+\frac{q_{i}-p_{i}+1}{q_{i}}-r_{i}\right)  \label{QISTAR2}%
	\end{equation}
	and%
	\begin{equation}
	q_{i}=\ell^{2}\left(  \frac{q_{i}^{\ast}-\widehat{p_{i}^{\ast}}+1}{q_{i}%
		^{\ast}}+\frac{q_{i+1}^{\ast}-p_{i+1}^{\ast}+1}{q_{i+1}^{\ast}}-r_{i}^{\ast
	}\right)  . \label{QI2}%
	\end{equation}
	
\end{proposition}

\begin{proof}
	Since $\sigma_{i}=\mathbb{R}_{\geq0}\mathbf{n}_{i}+\mathbb{R}_{\geq
		0}\mathbf{n}_{i+1}$ is a $(p_{i},q_{i})$-cone, there exist a basis matrix
	$\mathcal{B}$ of $N$ and a matrix $\mathcal{M}_{\sigma_{i}}\in$ $\text{GL}_{2}\left(\mathbb{Z}\right)  $ such that
	\[
	\Phi_{\mathcal{M}_{\sigma_{i}}\mathcal{B}^{-1}}(\sigma_{i})=\Phi
	_{\mathcal{M}_{\sigma_{i}}}(\sigma_{i}^{\text{st}})=\mathbb{R}_{\geq0}%
	\tbinom{1}{0}+\mathbb{R}_{\geq0}\tbinom{p_{i}}{q_{i}},
	\]
	where $\sigma_{i}^{\text{st}}$ is the standard model of $\sigma_{i}$ w.r.t.
	$\mathcal{B}$ (see Proposition \ref{PQDESCR1} and Figure \ref{ANG}). $\Phi_{\mathcal{M}_{\sigma
			_{i}}\mathcal{B}^{-1}}$ maps $\mathbf{n}_{i}$ onto $\tbinom{1}{0}$ and
	$\mathbf{n}_{i-1}$ onto a point $\tbinom{\mathfrak{n}_{1,i-1}}{\mathfrak{n}%
		_{2,i-1}}\in \mathbb{Z}^{2},$ i.e., $\sigma_{i-1}$ onto the $\mathbb{Z}^{2}%
	$-cone $\mathbb{R}_{\geq0}\tbinom{\mathfrak{n}_{1,i-1}}{\mathfrak{n}_{2,i-1}%
	}+\mathbb{R}_{\geq0}\tbinom{1}{0}$ with%
	\begin{equation}
	-\mathfrak{n}_{2,i-1}=\det \left(
	\begin{smallmatrix}
	\mathfrak{n}_{1,i-1} & 1\\
	\mathfrak{n}_{2,i-1} & 0
	\end{smallmatrix}
	\right)  =\text{mult}_{N}(\sigma_{i-1})=q_{i-1}\Rightarrow \mathfrak{n}%
	_{2,i-1}=-q_{i-1}. \label{NFR2}%
	\end{equation}
	We observe that the point of $\partial \Theta_{\Phi_{\mathcal{M}_{\sigma_{i}%
			}\mathcal{B}^{-1}}(\sigma_{i})}^{\mathbf{cp}}\cap \mathbb{Z}^{2}$ (resp., of
	$\partial \Theta_{\Phi_{\mathcal{M}_{\sigma_{i}}\mathcal{B}^{-1}}(\sigma
		_{i-1})}^{\mathbf{cp}}\cap \mathbb{Z}^{2}$) closest to $\tbinom{1}{0}$ is
	$\tbinom{1}{1}$ (resp., $\frac{1}{q_{i-1}}\left(  \tbinom{\mathfrak{n}%
		_{1,i-1}}{\mathfrak{n}_{2,i-1}}+\left(  q_{i-1}-\widehat{p}_{i-1}\right)
	\tbinom{1}{0}\right)  .$ (Use (\ref{LATPOINTSSIGMA1}) and
	(\ref{LATPOINTSSIGMA2}) for the $\mathbb{Z}^{2}$-cones $\Phi_{\mathcal{M}%
		_{\sigma_{i}}\mathcal{B}^{-1}}(\sigma_{i})$ and $\Phi_{\mathcal{M}_{\sigma
			_{i}}\mathcal{B}^{-1}}(\sigma_{i-1}),$ respectively.) By the linearity of
	$\Phi_{\mathcal{M}_{\sigma_{i}}\mathcal{B}^{-1}}$ we infer that
	\begin{equation}
	\tbinom{\frac{1}{q_{i-1}}(\mathfrak{n}_{1,i-1}+q_{i-1}-\widehat{p}_{i-1})}%
	{-1}+\tbinom{1}{1}=r_{i}\tbinom{1}{0}\Rightarrow \mathfrak{n}_{1,i-1}=\left(
	r_{i}-2\right)  q_{i-1}+\widehat{p}_{i-1}. \label{NFR1}%
	\end{equation}

\begin{figure}[h]
	\includegraphics[height=10cm, width=10.5cm]{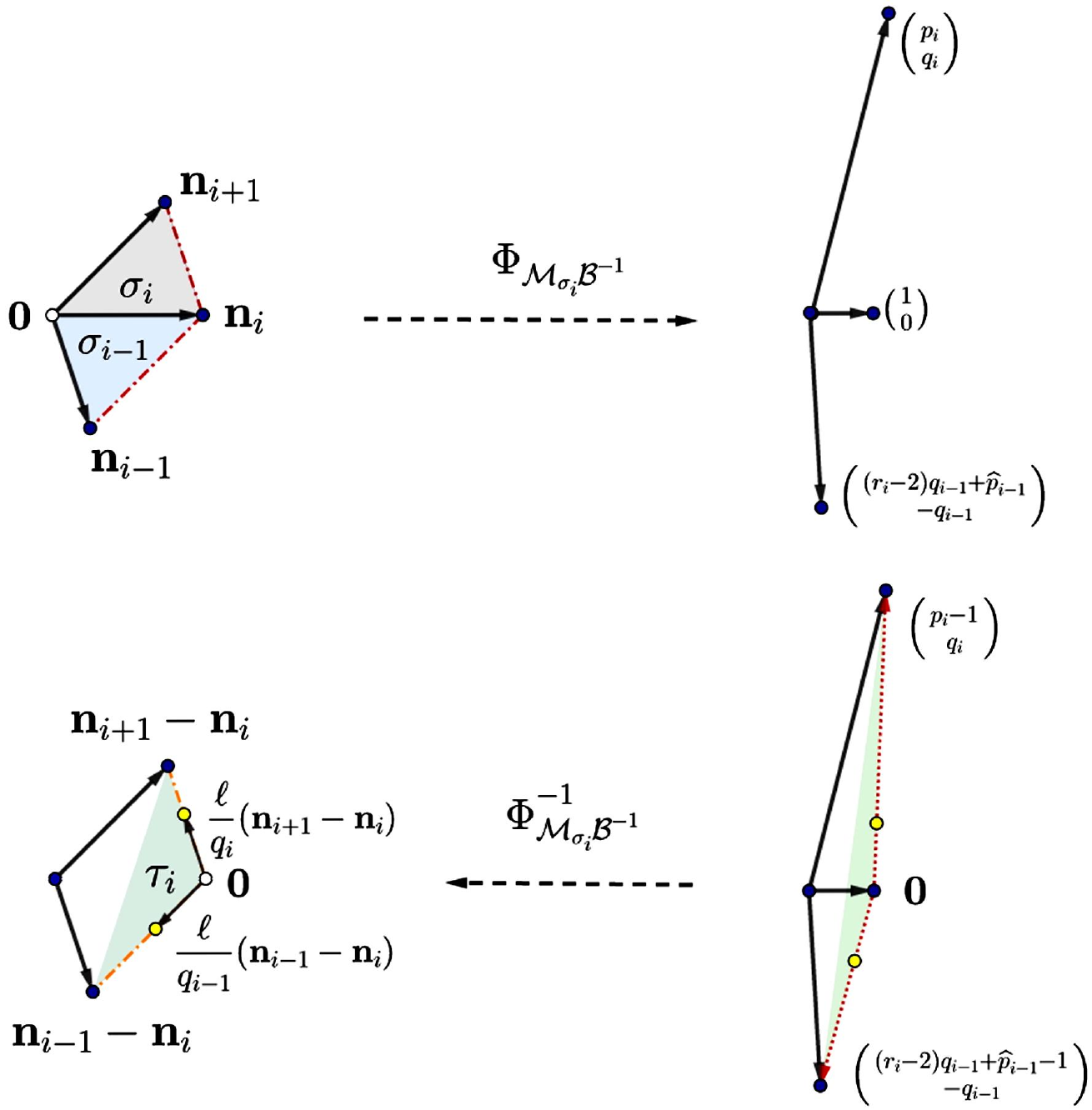}
	\caption{}\label{ANG}%
\end{figure}

	\noindent Using (\ref{NFR2}) and (\ref{NFR1}) we compute the multiplicity of $\tau
	_{i}=(\sigma_{i}^{\ast})^{\vee}$:
	\begin{align*}
	q_{i}^{\ast}  &  =\text{mult}_{M}(\sigma_{i}^{\ast})=\text{mult}_{N}(\tau
	_{i})=\det \left(  \tfrac{\ell}{q_{i}}\left(  \tbinom{p_{i}}{q_{i}}-\tbinom
	{1}{0}\right)  ,\tfrac{\ell}{q_{i-1}}\left(  \tbinom{\mathfrak{n}_{1,i-1}%
	}{\mathfrak{n}_{2,i-1}}-\tbinom{1}{0}\right)  \right)  \smallskip \\
	&  =\det \left(
	\begin{smallmatrix}
	\frac{\ell}{q_{i}}\left(  p_{i}-1\right)  & \frac{\ell}{q_{i-1}}\left(
	\left(  r_{i}-2\right)  q_{i-1}+\widehat{p}_{i-1}-1\right) \\
	\ell & -\ell
	\end{smallmatrix}
	\right)  \smallskip \\
	&  =-\ell^{2}\left(  \frac{\widehat{p}_{i-1}-1}{q_{i-1}}+\frac{p_{i}-1}{q_{i}%
	}+(r_{i}-2)\right)
	\end{align*}
	and obtain (\ref{QISTAR2}). The proof of (\ref{QI2}) is similar.
\end{proof}

\begin{corollary}
[Identities with combinatorial triples]If $\ell>1,$ then the combinatorial
triples
\[
(p_{1},q_{1},r_{1}),\ldots,(p_{\nu},q_{\nu},r_{\nu})
\]
of $\Delta_{Q}$ satisfy the identity\emph{:}
\begin{equation}%
{\displaystyle \sum \limits_{i=1}^{\nu}}
\left(  \dfrac{q_{i}}{\ell}+\dfrac{2\ell}{q_{i}}\right)  =12-2\ell \nu+%
{\displaystyle \sum \limits_{i=1}^{\nu}}
\ell \left(  \frac{p_{i}+\widehat{p}_{i}}{q_{i}}+r_{i}\right)  ,
\label{CTIDENTITY1}%
\end{equation}
and, in the dual sense, the combinatorial triples $(p_{1}^{\ast},q_{1}^{\ast
},r_{1}^{\ast}),\ldots,(p_{\nu}^{\ast},q_{\nu}^{\ast},r_{\nu}^{\ast})$ of
$\Delta_{Q^{\ast}}$ satisfy the identity\emph{:}%
\begin{equation}%
{\displaystyle \sum \limits_{i=1}^{\nu}}
\left(  \dfrac{q_{i}^{\ast}}{\ell}+\dfrac{2\ell}{q_{i}^{\ast}}\right)
=12-2\ell \nu+%
{\displaystyle \sum \limits_{i=1}^{\nu}}
\ell \left(  \frac{p_{i}^{\ast}+\widehat{p_{i}^{\ast}}}{q_{i}^{\ast}}%
+r_{i}^{\ast}\right)  . \label{CTIDENTITY2}%
\end{equation}

\end{corollary}

\begin{proof}
Formula (\ref{G12PTFORMULA}) can be rewritten via (\ref{QISTAR2}) in the form
{\small
\begin{align*}
12  &  =\sharp \left(  \partial Q\cap N\right)  +\sharp(\partial Q^{\ast}\cap
M)=\frac{1}{\ell}\sum_{i=1}^{\nu}q_{i}+\frac{1}{\ell}\sum_{i=1}^{\nu}%
q_{i}^{\ast}\\
&  =\frac{1}{\ell}\sum_{i=1}^{\nu}q_{i}+\sum_{i=1}^{\nu}\ell \left(
\frac{q_{i-1}-\widehat{p}_{i-1}+1}{q_{i-1}}+\frac{q_{i}-p_{i}+1}{q_{i}}%
-r_{i}\right) =\frac{1}{\ell}\sum_{i=1}^{\nu}q_{i}+\sum_{i=1}^{\nu}\ell \left(
\frac{1-\widehat{p}_{i}}{q_{i}}+\frac{1-p_{i}}{q_{i}}-(r_{i}-2)\right)  .
\end{align*}}
Hence, (\ref{CTIDENTITY1}) is true. The proof of (\ref{CTIDENTITY2}) is similar.
\end{proof}

\noindent Finally, it remains to give the explicit number-theoretic description of
the link between $p_{i}^{\ast},	\widehat{p_{i}}^{\ast}$ and the multiplicity
$q_{i}^{\ast}$, provided that $p_{i}$ and $q_{i}$ are assumed to be known, and, respectively, of the link between $p_{i},\widehat{p}_{i}$ and the multiplicity $q_{i}$, provided that $p_{i}^{\ast}$ and $q_{i}^{\ast}$ are assumed to be known.
\begin{proposition}
	\label{PISTAREXPRESSION}Let $\ell$ be again $>1$. For each $i\in \left \{  1,\ldots,\nu \right \}  $
	consider the regular continued fraction expansion%
	\[
	\sdfrac{\ell}{\frac{\ell}{q_{i}}(p_{i}-1)}=\frac{q_{i}}{p_{i}-1}=\left[
	d_{1}^{\left(  i\right)  },	d_{2}^{\left(  i\right)  }\ldots,d_{\rho}^{\left(
		i\right)  }\right]:=	d_{1}^{\left(  i\right)  }+\frac{1}%
	{	d_{2}^{\left(  i\right)  }+\cfrac{1}{%
		\begin{array}
		[c]{cc}%
		\ddots & \\
		& d_{\rho -1}^{\left(  i\right) }+\dfrac{1}{d_{\rho}^{\left(  i\right) }}%
		\end{array}
}}%
	\]
	of $\cfrac{\ell}{\frac{\ell}{q_{i}}(p_{i}-1)}$ and set {\small
	\[
	\kappa_{i}:=\left \{
	\begin{array}
	[c]{ll}%
	\tfrac{-\varepsilon \ell}{\left[  d_{\rho}^{\left(  i\right)
		},d_{\rho-1}^{\left(  i\right)  },\ldots,d_{2}^{\left(
			i\right)  },d_{1}^{\left(  i\right)  }\right]  },\medskip &
	\text{\emph{if} }d_{1}^{\left(  i\right)  }\geq2,\\
	\tfrac{-\varepsilon \ell}{\left[  d_{\rho}^{\left(  i\right)
		},d_{\rho-1}^{\left(  i\right)  },\ldots,d_{3}^{\left(
			i\right)  },d_{2}^{\left(  i\right)  }+1\right]  }, &
	\text{\emph{if} }d_{1}^{\left(  i\right)  }=1,
	\end{array}
	\right.  \text{ \  \emph{and} \ }\lambda_{i}:=\left \{
	\begin{array}
	[c]{ll}%
	\tfrac{-\frac{\varepsilon \ell}{q_{i}}(p_{i}-1)}{\left[  d_{\rho
		}^{\left(  i\right)  },d_{\rho-1}^{\left(  i\right)  }%
		,\ldots,d_{3}^{\left(  i\right)  },d_{2}^{\left(
			i\right)  }\right]  },\medskip & \text{\emph{if} }d_{2}^{\left(
		i\right)  }\geq2,\\
	\tfrac{-\frac{\varepsilon \ell}{q_{i}}(p_{i}-1)}{\left[  d_{\rho
		}^{\left(  i\right)  },d_{\rho-1}^{\left(  i\right)  }%
		,\ldots,d_{4}^{\left(  i\right)  },d_{3}^{\left(
			i\right)  }+1\right]  }, & \text{\emph{if} }d_{2}^{\left(
		i\right)  }=1,
	\end{array}
	\right.
	\]}
	with $\varepsilon=1$ for $\rho$ even and $\varepsilon=-1$ for $\rho$ odd. Then
	$\kappa_{i},\lambda_{i}\in \mathbb{Z}$ and
	\begin{equation}
	\kappa_{i}\frac{\ell}{q_{i}}(p_{i}-1)-\lambda_{i}\ell=1. \label{KAPPA+LAMDA1}%
	\end{equation}
	Denoting by $\mathfrak{z}_{i}$ the unique positive integer which is smaller
	than $q_{i}^{\ast}$ and satisfies%
	\begin{equation}
	\kappa_{i}\frac{\ell}{q_{i-1}}\left(  \left(  r_{i}-2\right)  q_{i-1}%
	+\widehat{p}_{i-1}-1\right)  +\lambda_{i}\ell \equiv \mathfrak{z}_{i}\left(
	\operatorname{mod}q_{i}^{\ast}\right)  , \label{ZETAI}%
	\end{equation}
	we obtain
	\begin{equation}
	\mathfrak{z}_{i}=\left \{
	\begin{array}
	[c]{ll}%
	q_{i}^{\ast}\left(  1-\tfrac{1}{\left[  d_{\rho}^{\left(  i\right)
		},d_{\rho-1}^{\left(  i\right)  },\ldots,d_{2}^{\left(
			i\right)  },d_{1}^{\left(  i\right)  }\right]  }\right)
	-1,\smallskip & \text{\emph{if} }\rho \text{ \emph{is odd and} }d%
	_{1}^{\left(  i\right)  }\geq2, \medskip \\
	q_{i}^{\ast}\left(  1-\frac{1}{\left[  d_{\rho}^{\left(  i\right)
		},d_{\rho-1}^{\left(  i\right)  },\ldots,d_{3}^{\left(
			i\right)  },d_{2}^{\left(  i\right)  }+1\right]  }\right)  -1, &
	\text{\emph{if} }\rho \text{ \emph{is odd and} }d_{1}^{\left(
		i\right)  }=1, \medskip \\
	\frac{q_{i}^{\ast}}{\left[  d_{\rho}^{\left(  i\right)
		},d_{\rho-1}^{\left(  i\right)  },\ldots,d_{2}^{\left(
			i\right)  },d_{1}^{\left(  i\right)  }\right]  }-1, &
	\text{\emph{if} }\rho \text{ \emph{is even and} }d_{1}^{\left(
		i\right)  }\geq2,\medskip \\
	\frac{q_{i}^{\ast}}{\left[  d_{\rho}^{\left(  i\right)
		},d_{\rho-1}^{\left(  i\right)  },\ldots,d_{3}^{\left(
			i\right)  },d_{2}^{\left(  i\right)  }+1\right]  }-1, &
	\text{\emph{if} }\rho \text{ \emph{is even and} }d_{1}^{\left(
		i\right)  }=1,
	\end{array}
	\right.  \label{ZETAICOMP}%
	\end{equation}
	and
	\[
	\widehat{p_{i}^{\ast}}=q_{i}^{\ast}-\mathfrak{z}_{i},\  \  \ p_{i}^{\ast}%
	=q_{i}^{\ast}-\widehat{\mathfrak{z}_{i}},
	\]
	where $\widehat{\mathfrak{z}_{i}}$ is the socius of $\mathfrak{z}_{i}$ w.r.t.
	$q_{i}^{\ast}.$
\end{proposition}
\begin{proof}
	$\tau_{i}$ is mapped by $\Phi_{\mathcal{M}_{\sigma_{i}}\mathcal{B}^{-1}}$
	(with $\Phi_{\mathcal{M}_{\sigma_{i}}\mathcal{B}^{-1}}$ as in the proof of
	Proposition \ref{QISTAREXPRESSION}) onto the $\mathbb{Z}^{2}$-cone
	\begin{equation}
	\Phi_{\mathcal{M}_{\sigma_{i}}\mathcal{B}^{-1}}(\tau_{i})=\Phi_{\left(
		\mathcal{BM}_{\sigma_{i}}^{-1}\right)  ^{-1}}(\tau_{i})=\mathbb{R}_{\geq
		0}\tbinom{\frac{\ell}{q_{i-1}}\left(  \left(  r_{i}-2\right)  q_{i-1}%
		+\widehat{p}_{i-1}-1\right)  }{-\ell}+\mathbb{R}_{\geq0}\tbinom{\frac{\ell
		}{q_{i}}\left(  p_{i}-1\right)  }{\ell} \label{PHIMB}%
	\end{equation}
	which is the standard model of $\tau_{i}$ w.r.t. $\mathcal{BM}_{\sigma_{i}%
	}^{-1}$ with mult$_{\mathbb{Z}^{2}}(\Phi_{\mathcal{M}_{\sigma_{i}}%
		\mathcal{B}^{-1}}(\tau_{i}))=$ mult$_{N}(\tau_{i})=q_{i}^{\ast}.$
	(\ref{KAPPA+LAMDA1}) is valid by the definition of $\kappa_{i},\lambda_{i}$
	(see \cite[Remark 3.2, p. 217]{DHH}). Assume that \[\mathbb{R}_{\geq0}%
	\tbinom{\frac{\ell}{q_{i}}\left(  p_{i}-1\right)  }{\ell}+\mathbb{R}_{\geq
		0}\tbinom{\frac{\ell}{q_{i-1}}\left(  \left(  r_{i}-2\right)  q_{i-1}%
		+\widehat{p}_{i-1}-1\right)  }{-\ell}\] (defined by interchanging the ordering
	of the minimal generators of (\ref{PHIMB})) is a $(\mathfrak{z}_{i}%
	,q_{i}^{\ast})$-cone. By Proposition \ref{PQDESCR1} $\mathfrak{z}_{i}$ has to
	be the unique positive integer which is smaller than $q_{i}^{\ast}$ and
	satisfies (\ref{ZETAI}). Using (\ref{KAPPA+LAMDA1}) and (\ref{QISTAR2}) we can
	write the left-hand side of (\ref{ZETAI}) as follows:
	\begin{align*}
	&  \kappa_{i}\frac{\ell}{q_{i-1}}\left(  \left(  r_{i}-2\right)
	q_{i-1}+\widehat{p}_{i-1}-1\right)  +\kappa_{i}\frac{\ell}{q_{i}}(p_{i}-1)-1\\
	&  =\kappa_{i}\ell \left(  \frac{\widehat{p}_{i-1}-1}{q_{i-1}}+\frac{p_{i}%
		-1}{q_{i}}+(r_{i}-2)\right)  -1=-\frac{\kappa_{i}q_{i}^{\ast}}{\ell}-1.
	\end{align*}
	Thus, (\ref{ZETAICOMP}) is true and (\ref{PHIMB}) is a $(\widehat
	{\mathfrak{z}_{i}},q_{i}^{\ast})$-cone (cf. Note \ref{pcomment} and the proof
	of Proposition \ref{PQDESCR2}). Since both $\tau_{i}=(\sigma_{i}^{\ast}%
	)^{\vee}$ and $\Phi_{\mathcal{M}_{\sigma_{i}}\mathcal{B}^{-1}}(\tau_{i})$ are
	$(q_{i}^{\ast}-p_{i}^{\ast},q_{i}^{\ast})$-cones (cf. Proposition
	\ref{DUALCNEIDENT}), we have $q_{i}^{\ast}-p_{i}^{\ast}=\widehat
	{\mathfrak{z}_{i}}.$
\end{proof}

\begin{note}
	Similarly, one shows that $p_{i}=q_{i}-\widehat{\mathfrak{z}_{i}^{\ast}}$ for
	all $i\in \left \{  1,\ldots,\nu \right \}  ,$ where $\mathfrak{z}_{i}^{\ast}$ is
	determined by the dual procedure, and $\widehat{\mathfrak{z}_{i}^{\ast}}$ is
	its socius w.r.t. $q_{i}$.
\end{note}

\section{Geometric interpretation of the characteristic differences whenever $\ell>1$} \label{CHARDIFFSEC}

\noindent Throughout this section we assume that $\ell >1.$ By (\ref{PIGE2}) we have int$(T_{F_{i}})\cap N\neq \varnothing,$
$\forall i\in \left \{  1,\ldots,\nu \right \}  ,$ and int$(Q)\cap N$ consists of
at least $\nu+1\geq4$ non-collinear lattice points. This means that
\setlength\extrarowheight{2pt}
\[
\fbox{$%
	\begin{array}
	[c]{ccc}
	& \mathbf{I}(Q):=\text{conv}\left(  \text{int}(Q)\cap N\right)  &
	\end{array}
	$}%
\]
\setlength\extrarowheight{-2pt}
is an $N$-polygon. Analogously,
\setlength\extrarowheight{2pt}
\[
\fbox{$%
	\begin{array}
	[c]{ccc}
	& \mathbf{I}(Q^{\ast}):=\text{conv}\left(  \text{int}(Q^{\ast})\cap M\right)
	&
	\end{array}
	$}%
\]
\setlength\extrarowheight{-2pt}
is an $M$-polygon. We wish to relate $\sharp \left(  \partial(\mathbf{I}%
(Q^{\ast}))\cap M\right)  $ with $\sharp \left(  \partial Q^{\ast}\cap
M\right)  $ and $\sharp \left(  \partial Q\cap N\right)  .$

\begin{lemma}
	\label{FSTARNEF}The divisor $f^{\star}(-\ell K_{X(N,\Delta_{Q})}%
	)+K_{X(N,\widetilde{\Delta}_{Q})}$ on $X(N,\widetilde{\Delta}_{Q})$ is nef.
	Moreover, using the notation introduced in \emph{(\ref{PEDE}),}%
	\begin{equation}
	P_{f^{\star}(-\ell K_{X(N,\Delta_{Q})})+K_{X(N,\widetilde{\Delta}_{Q})}%
	}=\mathbf{I}(Q^{\ast}). \label{PFSTARIQSTAR}%
	\end{equation}
	
\end{lemma}

\begin{proof}
	Since $\Sigma_{Q^{\ast}}=\Delta_{Q}$ and $-\ell K_{X(N,\Delta_{Q})}%
	=D_{Q^{\ast}},$ Theorem \ref{SIGMAPMINDES} (applied for the lattice
	$M$-polygon $Q^{\ast}$) implies that the pullback $f^{\star}(-\ell
	K_{X(N,\Delta_{Q})})=f^{\star}(D_{Q^{\ast}})$ of $D_{Q^{\ast}}$ via $f$ is the
	unique nef divisor on $X(N,\widetilde{\Delta}_{Q})$ for which $Q^{\ast
	}=P_{D_{Q^{\ast}}}=P_{f^{\star}(D_{Q^{\ast}})}.$ Hence,
	\[
	Q^{\ast}=\left \{  \left.  \mathbf{x}\in \mathbb{R}^{2}\right \vert \left \langle
	\mathbf{x},\mathbf{n}_{\varrho}\right \rangle \geq h_{f^{\star}(D_{Q^{\ast}}%
		)}(\mathbf{n}_{\varrho}\mathbf{)},\forall \varrho \in \widetilde{\Delta}%
	_{Q}\left(  1\right)  \right \}
	\]
	and
	\[
	\mathbf{I}(Q^{\ast})=\left \{  \left.  \mathbf{x}\in \mathbb{R}^{2}\right \vert
	\left \langle \mathbf{x},\mathbf{n}_{\varrho}\right \rangle >h_{f^{\star
		}(D_{Q^{\ast}})}(\mathbf{n}_{\varrho}\mathbf{)},\forall \varrho \in
	\widetilde{\Delta}_{Q}\left(  1\right)  \right \}  .
	\]
	We define the function $h^{\prime}:\mathbb{R}^{2}\longrightarrow \mathbb{R}$ by
	setting
	\[
	h^{\prime}(\mathbf{y}):=\min \left \{  \left \langle \mathbf{x},\mathbf{y}%
	\right \rangle \left \vert \left \langle \mathbf{x},\mathbf{n}_{\varrho
	}\right \rangle \geq h_{f^{\star}(D_{Q^{\ast}})}(\mathbf{n}_{\varrho}%
	\mathbf{)}+1,\forall \varrho \in \widetilde{\Delta}_{Q}\left(  1\right)  \right.
	\right \}  ,\  \forall \mathbf{y}\in \mathbb{R}^{2},
	\]
	(with $h^{\prime}(\mathbf{n}_{\varrho})=h_{f^{\star}(D_{Q^{\ast}})}%
	(\mathbf{n}_{\varrho}\mathbf{)}+1,\forall \varrho \in \widetilde{\Delta}%
	_{Q}\left(  1\right)  $.) This function is an upper convex $\widetilde{\Delta
	}_{Q}$-support function because $h_{f^{\star}(D_{Q^{\ast}})}$ is upper convex
	$\widetilde{\Delta}_{Q}$-support function (by the implication
	(viii)$\Rightarrow$(ii) in Theorem \ref{NEFFITY} for the divisor $f^{\star
	}(D_{Q^{\ast}})$) and $\widetilde{\Delta}_{Q}$ contains \textit{only basic}
	$N$-cones. Thus, by (\ref{CORRSF1}) and (\ref{CORRSF2}) (and by the
	implication (ii)$\Rightarrow$(viii) in Theorem \ref{NEFFITY} for $h^{\prime}%
	$), $h^{\prime}$ determines a unique nef divisor $D_{h^{\prime}}\in$
	Div$_{\text{C}}^{\mathbb{T}}(X(N,\widetilde{\Delta}_{Q})),$ namely%
	\begin{align*}
	D_{h^{\prime}}  &  =-\sum_{\varrho \in \widetilde{\Delta}_{Q}\left(  1\right)
	}h^{\prime}(\mathbf{n}_{\varrho})\mathbf{V}_{\widetilde{\Delta}_{Q}}%
	(\varrho)=-\sum_{\varrho \in \widetilde{\Delta}_{Q}\left(  1\right)
	}h_{f^{\star}(D_{Q^{\ast}})}(\mathbf{n}_{\varrho})\mathbf{V}_{\widetilde
		{\Delta}_{Q}}(\varrho)-\sum_{\varrho \in \widetilde{\Delta}_{Q}\left(  1\right)
	}\mathbf{V}_{\widetilde{\Delta}_{Q}}(\varrho)\bigskip \\
	&  =f^{\star}(D_{Q^{\ast}})+K_{X(N,\widetilde{\Delta}_{Q})}=f^{\star}(-\ell
	K_{X(N,\Delta_{Q})})+K_{X(N,\widetilde{\Delta}_{Q})}%
	\end{align*}
	(according to (\ref{KFORMULA}) for the $N$-fan $\widetilde{\Delta}_{Q}$).
	Since
	$
	\mathbf{I}(Q^{\ast})=\left \{  \left.  \mathbf{x}\in \mathbb{R}^{2}\right \vert
	\left \langle \mathbf{x},\mathbf{n}_{\varrho}\right \rangle \geq h^{\prime
	}(\mathbf{n}_{\varrho}\mathbf{)},\forall \varrho \in \widetilde{\Delta}%
	_{Q}\left(  1\right)  \right \}  ,
	$
	(\ref{PFSTARIQSTAR}) is true.
\end{proof}

\begin{note}
	An alternative proof of the neffity of $f^{\star}(-\ell K_{X(N,\Delta_{Q}%
		)})+K_{X(N,\widetilde{\Delta}_{Q})}$ (from the point of view of intersection
	theory) comes from the fact that 
	$
	f^{\star}(-\ell K_{X(N,\Delta_{Q})})\sim-\ell K_{X(N,\widetilde{\Delta}_{Q}%
		)}+\ell \sum_{i=1}^{\nu}K(E^{(i)})
	$
	(cf. (\ref{TORDISCREP})), which gives
	\[
	f^{\star}(-\ell K_{X(N,\Delta_{Q})})+K_{X(N,\widetilde{\Delta}_{Q})}\sim
	(\ell-1)(-K_{X(N,\widetilde{\Delta}_{Q})})+\ell \sum_{i=1}^{\nu}K(E^{(i)}).
	\]
	Since $-\ell K_{X(N,\Delta_{Q})}\in$ Div$_{\text{C}}^{\mathbb{T}}(X\left(
	N,\Delta \right)  )$ is ample, the implication (i)$\Rightarrow$(v) in Theorem
	\ref{AMPLENESS} (applied for the $N$-fan $\Delta_{Q}$ and the divisor $-\ell
	K_{X(N,\Delta_{Q})}$) and \cite[Lemma 4.7, pp. 97-98]{Dais1} inform us that
	for all $i\in \left \{  1,\ldots,\nu \right \}  ,$%
	\begin{align*}
	(-\ell K_{X(N,\Delta_{Q})})\cdot C_{i}  &  =\ell \left(  C_{i-1}\cdot
	C_{i}+C_{i}^{2}+C_{i}\cdot C_{i+1}\right) \\
	&  =\ell \left(  -r_{i}+\tfrac{q_{i}-p_{i}+1}{q_{i}}+\tfrac{q_{i-1}-\widehat
		{p}_{i-1}+1}{q_{i-1}}\right)  =\ell \left(  -r_{i}+2-(\tfrac{p_{i}+1}{q_{i}%
	}+\tfrac{\widehat{p}_{i-1}+1}{q_{i-1}})\right)  >0\Rightarrow r_{i}\leq1
	\end{align*}
	as it is $\tfrac{p_{i}+1}{q_{i}},\tfrac{\widehat{p}_{i-1}+1}{q_{i-1}}%
	\in(0,1]\cap \mathbb{Q}.$ (Alternatively, by (\ref{QISTAR2}), $(-\ell
	K_{X(N,\Delta_{Q})})\cdot C_{i}=\frac{q_{i}^{\ast}}{\ell}=$ gcd$(q_{i}^{\ast
	},p_{i}^{\ast}-1)\geq1.$) Using \cite[Lemma 4.3, pp.
	93-94]{Dais1} we infer that {\small
		\[
		((\ell-1)(-K_{X(N,\widetilde{\Delta}_{Q})})+\ell \sum_{i=1}^{\nu}(K(E^{(i)}))\cdot \overline{C}_{i} =(\ell-1)\cdot 2+(\ell-1)\left(  -r_{i}\right)
		+\ell \cdot 2 =4 \ell -2 + (\ell-1)\left(  -r_{i}\right)>0 
		\]}
	\hspace{-0.3cm} because $\ell \geq2$ and $-r_i\ge -1$ for all $i\in \left \{  1,\ldots,\nu \right \}.$
	Furthermore, since each of $E_{j}^{(i)}$'s is isomorphic to $\mathbb{P}_{\mathbb{C}}^{1}$,
	adjunction formula and \cite[Lemma 4.3, pp.
	93-94]{Dais1}\ give%
	\[
	K_{X(N,\widetilde{\Delta}_{Q})} \cdot E_{j}^{(i)}= K(E^{(i)})\cdot E_{j}^{(i)}=-2-(E_{j}^{(i)})^{2}=b_{j}^{\left(  i\right)
	}-2,
	\]
	i.e.,
	\[
	\left(  (\ell-1)(-K_{X(N,\widetilde{\Delta}_{Q})})+\ell \sum_{i=1}^{\nu
	}(K(E^{(i)})\right)  \cdot E_{j}^{(i)}=\left(  1-\ell \right)  (b_{j}^{\left(
		i\right)  }-2)+\ell(b_{j}^{\left(  i\right)  }-2)=b_{j}^{\left(  i\right)
	}-2\geq0
	\]
	for all $i\in \left \{  1,\ldots,\nu \right \}  $ and all $j\in \{1,\ldots,s_{i}\}$
	(see (\ref{BJBIGGER1})). From the implication (vii)$\Rightarrow$(viii) in
	Theorem \ref{NEFFITY} (applied for the $N$-fan $\widetilde{\Delta}_{Q}$ and
	the Cartier divisor $f^{\star}(-\ell K_{X(N,\Delta_{Q})})+K_{X(N,\widetilde
		{\Delta}_{Q})}$) we conclude that $f^{\star}(-\ell K_{X(N,\Delta_{Q}%
		)})+K_{X(N,\widetilde{\Delta}_{Q})}$ is indeed nef.
\end{note}

\begin{lemma}
	The area of the $M$-polygon $\mathbf{I}(Q^{\ast})$ is given by the formula
	\begin{equation}
	\text{\emph{area}}_{M}(\mathbf{I}(Q^{\ast}))=\frac{1}{2}\left(  (\ell
	-2)\sharp \left(  \partial Q^{\ast}\cap M\right)  +K_{X(N,\widetilde{\Delta
		}_{Q})}^{2}\right)  . \label{AREAIQSTAR}%
	\end{equation}
	
\end{lemma}

\begin{proof}
	Using formula (\ref{DSQUAREPD}) for the nef divisor $f^{\star}(-\ell
	K_{X(N,\Delta_{Q})})+K_{X(N,\widetilde{\Delta}_{Q})}$ we deduce from
	(\ref{K2FORMULA}):%
	\begin{align*}
	2\, \text{area}_{M}(P_{f^{\star}(-\ell K_{X(N,\Delta_{Q})})+K_{X(N,\widetilde
			{\Delta}_{Q})}})  &  =\left(  f^{\star}(-\ell K_{X(N,\Delta_{Q})}%
	)+K_{X(N,\widetilde{\Delta}_{Q})}\right)  ^{2}\\
	&  =f^{\star}(-\ell K_{X(N,\Delta_{Q})})^{2}+2f^{\star}(-\ell K_{X(N,\Delta
		_{Q})})\cdot K_{X(N,\widetilde{\Delta}_{Q})}+K_{X(N,\widetilde{\Delta}_{Q}%
		)}^{2}\\
	&  =\ell^{2}K_{X(N,\Delta_{Q})}^{2}+2f^{\star}(-\ell K_{X(N,\Delta_{Q})})\cdot
	K_{X(N,\widetilde{\Delta}_{Q})}+K_{X(N,\widetilde{\Delta}_{Q})}^{2}\\
	&  =\ell \, \sharp \left(  \partial Q^{\ast}\cap M\right)  +2f^{\star}(-\ell
	K_{X(N,\Delta_{Q})})\cdot K_{X(N,\widetilde{\Delta}_{Q})}+K_{X(N,\widetilde
		{\Delta}_{Q})}^{2}.
	\end{align*}
	Since $\Sigma_{Q^{\ast}}=\Delta_{Q}$ and $-\ell K_{X(N,\Delta_{Q})}%
	=D_{Q^{\ast}},$ applying (\ref{DPTHETAP}) for the lattice $M$-polygon
	$Q^{\ast}$ we get
	\[
	f^{\star}(-\ell K_{X(N,\Delta_{Q})})\cdot K_{X(N,\widetilde{\Delta}_{Q}%
		)}=-\sharp \left(  \partial Q^{\ast}\cap M\right)  .
	\]
	Hence,
	\begin{equation}
	2\, \text{area}_{M}(P_{f^{\star}(-\ell K_{X(N,\Delta_{Q})})+K_{X(N,\widetilde
			{\Delta}_{Q})}})=(\ell-2)\sharp \left(  \partial Q^{\ast}\cap M\right)
	+K_{X(N,\widetilde{\Delta}_{Q})}^{2}. \label{AREAPFSTAR}%
	\end{equation}
	(\ref{AREAIQSTAR}) follows from (\ref{AREAPFSTAR}) and (\ref{PFSTARIQSTAR}).
\end{proof}

\begin{theorem}
	\label{THETAIQ2}The number of lattice points lying on the boundary of
	$\mathbf{I}(Q^{\ast})$ is given by the formulae
	\setlength\extrarowheight{2pt}
	\begin{equation}
	\fbox{$%
		\begin{array}
		[c]{ccc}
		& \sharp \left(
		\partial Q^{\ast}\cap M\right)  -K_{X(N,\widetilde{\Delta}_{Q})}^{2}=\sharp \left(  \partial(\mathbf{I}(Q^{\ast}))\cap M\right)=e(X(N,\widetilde{\Delta
		}_{Q}))-\sharp \left(  \partial Q\cap N\right). &
		\end{array}
		$} \label{THETAIQ1F}%
	\end{equation}
	\setlength\extrarowheight{-2pt}
\end{theorem}

\begin{proof}
	At first we apply Pick's formula (\ref{PICK}) for the $M$-polygon
	$\mathbf{I}(Q^{\ast})$:
	\begin{equation}
	\sharp \left(  \mathbf{I}(Q^{\ast})\cap M\right)  =\text{area}_{M}%
	(\mathbf{I}(Q^{\ast}))+\tfrac{1}{2}\sharp \left(  \partial(\mathbf{I}(Q^{\ast
	}))\cap M\right)  +1. \label{PICKIQSTAR}%
	\end{equation}
	By (\ref{Ehrhartint}) and (\ref{BOUNDARYAREA}) we obtain
	\begin{equation}
	\sharp \left(  \text{int}(Q^{\ast})\cap M\right)  =\text{area}_{M}(Q^{\ast
	})-\tfrac{1}{2}\sharp \left(  \partial Q^{\ast}\cap M\right)  +1=\tfrac{1}%
	{2}\left(  \ell-1\right)  \sharp \left(  \partial Q^{\ast}\cap M\right)  +1.
	\label{PICKINTQSTAR}%
	\end{equation}
	Obviously,%
	\begin{equation}
	\sharp \left(  \mathbf{I}(Q^{\ast})\cap M\right)  =\sharp \left(  \text{int}%
	(Q^{\ast})\cap M\right)  . \label{EQUALITYIQSTARINT}%
	\end{equation}
	The first of the equalities (\ref{THETAIQ1F}) follows from (\ref{PICKIQSTAR}), (\ref{PICKINTQSTAR}),
	(\ref{EQUALITYIQSTARINT}) and (\ref{AREAIQSTAR}). The second one follows directly from (\ref{G12PTFORMULA}) and (\ref{NOETHERSFORMULA}).
\end{proof}

\begin{note}
	(i) The second term in the left-hand side of (\ref{THETAIQ1F}) can be written (by
	(\ref{K2FORMULA}) and \cite[Corollary 4.6, p. 96]{Dais1}) as
	\[
	-K_{X(N,\widetilde{\Delta}_{Q})}^{2}=-K_{X(N,\Delta_{Q})}^{2}-\sum_{i=1}^{\nu
	}K(E^{(i)})^{2}=-\frac{1}{\ell}\sharp \left(  \partial Q^{\ast}\cap M\right)  -\sum_{i=1}%
	^{\nu}\left(  \frac{p_{i}+\widehat{p}_{i}-2}{q_{i}}\right)  +\sum_{i=1}^{\nu
	}\sum_{j=1}^{s_{i}}(b_{j}^{\left(  i\right)  }-2).
	\]
	(ii) \textit{Dual formulae}. Interchanging the roles of the $\ell$-reflexive pairs\textit{
	}$(Q^{\ast},M)$\textit{ }and\textit{ }$(Q,N),$ and using the minimal
	desingularization $\varphi:X(M,\widetilde{\Delta}_{Q^{\ast}})\longrightarrow
	X(M,\Delta_{Q^{\ast}})$ of $X(M,\Delta_{Q^{\ast}}),$ we obtain
	\setlength\extrarowheight{2pt}
	\begin{equation}
	\fbox{$%
		\begin{array}
		[c]{ccc}
		& \sharp \left(
		\partial Q\cap N\right)  -K_{X(M,\widetilde{\Delta}_{Q^{\ast}})}^{2}=	\sharp \left(\partial(\mathbf{I}(Q))\cap N\right)  =e(X(M,\widetilde{\Delta
		}_{Q^{\ast}}))-\sharp \left(  \partial Q^{\ast}\cap M\right) &
		\end{array}
		$} \label{THETAIQ1F2}%
	\end{equation}
	\setlength\extrarowheight{-2pt}
	
	with
	\[
	-K_{X(M,\widetilde{\Delta}_{Q^{\ast}})}^{2}=-\frac{1}{\ell}\sharp \left(
	\partial Q\cap N\right)  -\sum_{i=1}^{\nu}\left(  \frac{p_{i}^{\ast}%
		+\widehat{p_{i}^{\ast}}-2}{q_{i}^{\ast}}\right)  +\sum_{i=1}^{\nu}\sum
	_{j=1}^{s_{i}^{\ast}}(c_{j}^{\ast \, \left(  i\right)  }-2).
	\]
(The numbers of lattice points counted in (\ref{THETAIQ1F}) and (\ref{THETAIQ1F2}) are not necessarily equal. See example \ref{EXVARTHNOT}.) \newline \smallskip
\noindent (iii) If for each $i\in \left \{  1,\ldots,\nu \right \}  $ we denote by
\[
\mathbf{I}(\Theta_{\tau_{i}}):=\text{conv}(\text{int}\left(  \tau_{i}\right)
\cap N)\subset \Theta_{\tau_{i}}\text{ \ (resp., }\mathbf{I}(\Theta_{\tau
	_{i}^{\ast}}):=\text{conv}(\text{int}\left(  \tau_{i}^{\ast}\right)  \cap
M)\subset \Theta_{\tau_{i}^{\ast}}\text{)}%
\]
the convex hull of the lattice points lying in the interior of the auxiliary
cone $\tau_{i}$ (resp., of $\tau_{i}^{\ast}$) and by $\partial \mathbf{I}%
(\Theta_{\tau_{i}})^{\mathbf{cp}}\subset \partial \Theta_{\tau_{i}}%
^{\mathbf{cp}}$ (resp., by $\partial \mathbf{I}(\Theta_{\tau_{i}^{\ast}%
})^{\mathbf{cp}}\subset \partial \Theta_{\tau_{i}^{\ast}}^{\mathbf{cp}}$) the
part of the boundary of $\mathbf{I}(\Theta_{\tau_{i}})$ (resp., of
$\mathbf{I}(\Theta_{\tau_{i}^{\ast}})$) containing only its compact edges,
then%
\[
\text{Vert}\left(  \mathbf{I}(Q)\right)  =%
{\displaystyle \bigcup \limits_{i=1}^{\nu}}
\text{Vert}(\partial \mathbf{I}(\Theta_{\tau_{i}})^{\mathbf{cp}}+\mathbf{n}%
_{i})\  \  \  \text{and \  \ Vert}\left(  \mathbf{I}(Q^{\ast})\right)  =%
{\displaystyle \bigcup \limits_{i=1}^{\nu}}
\text{Vert}(\partial \mathbf{I}(\Theta_{\tau_{i}^{\ast}})^{\mathbf{cp}%
}+\mathbf{m}_{i}).
\]
Moreover, setting%
\[%
\begin{array}
[c]{c}%
\mathfrak{K}_{Q}^{(i)}:=\left \{  \mathbf{u}+\mathbf{n}_{i}\left \vert
\mathbf{u}\in \text{Hilb}_{N}(\tau_{i})\mathbb{r}\left \{
\begin{array}
[c]{c}%
\text{the two minimal}\\
\text{generators of }\tau_{i}%
\end{array}
\vspace{0.2cm}
\right \}  \right.  \right \}  ,\medskip \\
\mathfrak{K}_{Q^{\ast}}^{(i)}:=\left \{  \mathbf{w}+\mathbf{m}_{i}\left \vert
\mathbf{w}\in \text{Hilb}_{M}(\tau_{i}^{\ast})\mathbb{r}\left \{
\begin{array}
[c]{c}%
\text{the two minimal}\\
\text{generators of }\tau_{i}^{\ast}%
\end{array}
\right \}  \right.  \right \}  ,
\end{array}
\]
and denoting by $\mathbf{u}_{\text{last}}^{(i)}$ (resp., by $\mathbf{w}%
_{\text{last}}^{(i)}$) the last lattice point of $\mathfrak{K}_{Q}^{(i)}$
(resp., of $\mathfrak{K}_{Q^{\ast}}^{(i)}$) and by $\mathbf{u}_{\text{first}%
}^{(i+1)}$ (resp., by $\mathbf{w}_{\text{first}}^{(i+1)}$) the first lattice
point of $\mathfrak{K}_{Q}^{(i+1)}$ (resp., of $\mathfrak{K}_{Q^{\ast}%
}^{(i+1)}$) w.r.t. the anticlockwise direction, then
\[
\partial(\mathbf{I}(Q))\cap N=%
{\textstyle \bigcup \limits_{i=1}^{\nu}}
(\mathfrak{K}_{Q}^{(i)}\cup \mathfrak{L}_{Q}^{(i)})\  \  \  \text{and
	\  \ }\partial(\mathbf{I}(Q^{\ast}))\cap M=%
{\textstyle \bigcup \limits_{i=1}^{\nu}}
(\mathfrak{K}_{Q^{\ast}}^{(i)}\cup \mathfrak{L}_{Q^{\ast}}^{(i)}),
\]
where%
\[
\mathfrak{L}_{Q}^{(i)}:=\left \{
\begin{array}
[c]{ll}%
\text{int}(\text{conv}(\{ \mathbf{u}_{\text{last}}^{(i)},\mathbf{u}%
_{\text{first}}^{(i+1)}\}))\cap N, & \text{if }\mathbf{u}_{\text{last}}%
^{(i)}\neq \mathbf{u}_{\text{first}}^{(i+1)}, \medskip\\
\varnothing, & \text{if }\mathbf{u}_{\text{last}}^{(i)}=\mathbf{u}%
_{\text{first}}^{(i+1)},
\end{array}
\right.
\]
and
\[
\mathfrak{L}_{Q^{\ast}}^{(i)}:=\left \{
\begin{array}
[c]{ll}%
\text{int}(\text{conv}(\{ \mathbf{w}_{\text{last}}^{(i)},\mathbf{w}%
_{\text{first}}^{(i+1)}\}))\cap M, & \text{if }\mathbf{w}_{\text{last}}%
^{(i)}\neq \mathbf{w}_{\text{first}}^{(i+1)},\medskip \\
\varnothing, & \text{if }\mathbf{w}_{\text{last}}^{(i)}=\mathbf{w}%
_{\text{first}}^{(i+1)}.
\end{array}
\right.
\]
	 
\end{note}

\begin{example}
\bigskip Let $Q$ be the $5$-reflexive $\mathbb{Z}^{2}$-pentagon of Figure \ref{Fig.10} with vertices%
\[
\mathbf{n}_{1}=\tbinom{3}{-10},\  \mathbf{n}_{2}=\tbinom{1}{0},\  \mathbf{n}%
_{3}=\tbinom{-1}{5},\  \mathbf{n}_{4}=\tbinom{-2}{5},\  \mathbf{n}_{5}%
=\tbinom{-1}{0},
\]
(i.e., (\ref{EXAMPLE3}) for $\ell=5$). Its dual $Q^{\ast}$ has the vertices%
\[
\mathbf{m}_{1}=\tbinom{-5}{-1},\  \mathbf{m}_{2}=\tbinom{-5}{-2},\  \mathbf{m}%
_{3}=\tbinom{0}{-1},\  \mathbf{m}_{4}=\tbinom{5}{1},\  \mathbf{m}_{5}%
=\tbinom{5}{2}.
\]
For the first cone $\sigma_{1}=\mathbb{R}_{\geq0}\mathbf{n}_{1}+\mathbb{R}%
_{\geq0}\mathbf{n}_{2}$ we have $\left \vert \text{det}(\mathbf{n}%
_{1},\mathbf{n}_{2})\right \vert =10,$ and since $\left(  -3\right)
\cdot3-1\cdot(-10)=1$ and
\[
-3=\left(  -3\right)  \cdot1-1\cdot0=7(\operatorname{mod}\text{ 10}),
\]
Proposition \ref{PQDESCR1} implies that $\sigma_{1}$ is of type $(7,10).$ Working
similarly with all the other cones in $\Delta_{Q}$, we conclude with the table:%
	\setlength\extrarowheight{3pt}
{\small \[
\begin{tabular}
[c]{|c||c|c|c|c|c|}\hline
$i$ & $2$-dim. cones $\sigma_{i}$ in $\Delta_{Q}$ & of type $(p_{i},q_{i})$ &
socius of $p_{i}$ & ($-$)-continued fraction expansion & length\\ \hline \hline
$1$ & $%
\begin{array}
[c]{c}%
\sigma_{1}=\mathbb{R}_{\geq0}\mathbf{n}_{1}+\mathbb{R}_{\geq0}\mathbf{n}_{2}  \smallskip
\end{array}
$ & $(7,10)$ & $\widehat{p}_{1}=3$ & $%
\begin{array}
[c]{c}%
\frac{q_{1}}{q_{1}-p_{1}}=\frac{10}{10-7}=\frac{10}{3}=\left[ \! \left[
4,2,2\right]  \! \right]  \smallskip
\end{array}
$ & $s_{1}=3$\\ \hline
$2$ & $%
\begin{array}
[c]{c}%
\sigma_{2}=\mathbb{R}_{\geq0}\mathbf{n}_{2}+\mathbb{R}_{\geq0}\mathbf{n}_{3}  \smallskip
\end{array}
$ & $(4,5)$ & $\widehat{p}_{2}=4$ & $\frac{q_{2}}{q_{2}-p_{2}}=\frac{5}%
{5-4}=\frac{5}{1}=\left[  \! \left[  5\right]  \!  \right]  \!$ &
$s_{2}=1$\\ \hline
$3$ & $%
\begin{array}
[c]{c}%
 \! \!%
\begin{array}
[c]{c}%
\sigma_{3}=\mathbb{R}_{\geq0}\mathbf{n}_{3}+\mathbb{R}_{\geq0}\mathbf{n}_{4}  \smallskip
\end{array}
\end{array}
$ & $(2,5)$ & $\widehat{p}_{3}=3$ & $%
\begin{array}
[c]{c}%
\frac{q_{3}}{q_{3}-p_{3}}=\frac{5}{5-2}=\frac{5}{3}=\left[  \! \left[
2,3\right]  \! \right]  \smallskip
\end{array}
$ & $s_{3}=2$\\ \hline
$4$ & $ \! \!%
\begin{array}
[c]{c}%
\sigma_{4}=\mathbb{R}_{\geq0}\mathbf{n}_{4}+\mathbb{R}_{\geq0}\mathbf{n}_{5} \smallskip
\end{array}
$ & $(3,5)$ & $\widehat{p}_{4}=2$ & $%
\begin{array}
[c]{c}%
\frac{q_{4}}{q_{4}-p_{4}}=\frac{5}{5-3}=\frac{5}{2}=\left[  \! \left[3,2\right]  \!  \right] \smallskip
\end{array}
$ & $s_{4}=2$\\ \hline
$5$ & $%
\begin{array}
[c]{c}%
\sigma_{5}=\mathbb{R}_{\geq0}\mathbf{n}_{5}+\mathbb{R}_{\geq0}\mathbf{n}_{1} \smallskip
\end{array}
$ & $(7,10)$ & $\widehat{p}_{5}=3$ & $\frac{q_{5}}{q_{5}-p_{5}}=\frac
{10}{10-7}=\frac{10}{3}=\left[  \!  \left[  4,2,2\right]   \! \right]  $ &
$s_{5}=3$\\ \hline
\end{tabular}
\]}
	\setlength\extrarowheight{-3pt}

\begin{figure}[h]
	\includegraphics[height=9.6cm, width=14cm]{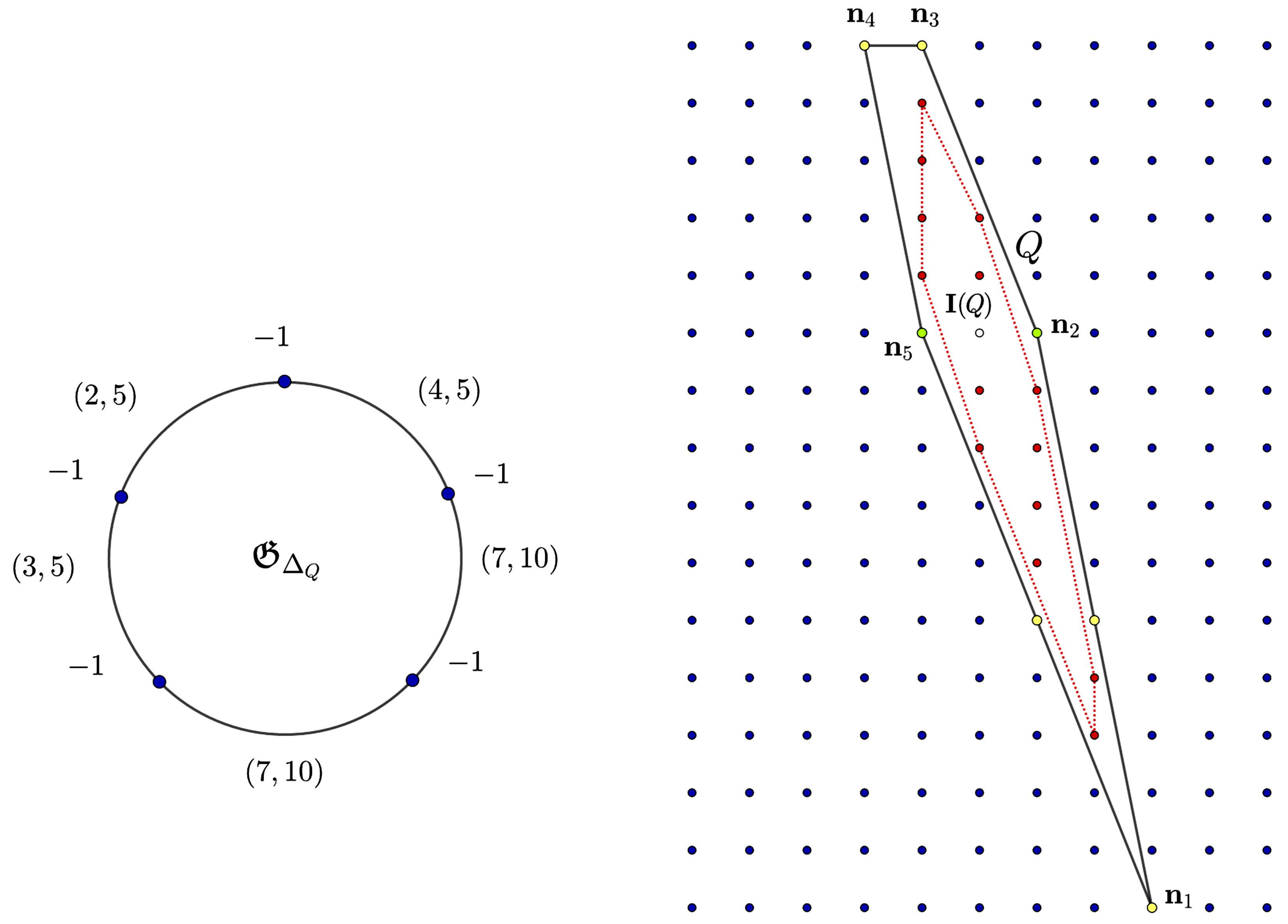}
	\caption{}\label{Fig.10}%
\end{figure}

\noindent Since $e(X(\mathbb{Z}^{2},\widetilde{\Delta}_{Q}))=\sum_{i=1}^{5}%
(s_{i}+1)=2\cdot4+2+2\cdot3=16,$ and
\begin{align*}
-K_{X(\mathbb{Z}^{2},\widetilde
	{\Delta}_{Q})}^{2} &  =-\frac{1}{5}\sharp \left(  \partial
Q^{\ast}\cap \mathbb{Z}^{2}\right)  -\sum_{i=1}^{5}\left(  \tfrac
{p_{i}+\widehat{p}_{i}-2}{q_{i}}\right)  +\sum_{i=1}^{5}\sum_{j=1}^{s_{i}%
}(b_{j}^{\left(  i\right)  }-2)\\
&  =-1-2\left(  \tfrac{7+3-2}{10}\right)  -\left(  \tfrac{4+4-2}{5}\right)
-2\left(  \tfrac{2+3-2}{5}\right)  +2+3+1+1+2=4,
\end{align*}
we have
\[
\sharp(\partial Q^{\ast}\cap \mathbb{Z}^{2})-K_{X(\mathbb{Z}^{2},\widetilde
	{\Delta}_{Q})}^{2}=5+4=9=\sharp(\partial(\mathbf{I}(Q^{\ast}))\cap
\mathbb{Z}^{2})=16-7=e(X(\mathbb{Z}^{2},\widetilde{\Delta}_{Q}))-\sharp
(\partial Q\cap \mathbb{Z}^{2}).
\]
In particular, $\mathbf{I}(Q^{\ast})=$ conv$\left \{  \tbinom{2}{1},\tbinom
{-4}{-1},\tbinom{-1}{-1},\tbinom{2}{0},\tbinom{4}{1}\right \}  ,$ and
\[
\partial(\mathbf{I}(Q^{\ast}))\cap \mathbb{Z}^{2}=\left \{  \tbinom{2}%
{1},\tbinom{-1}{0},\tbinom{-4}{-1},\tbinom{-3}{-1},\tbinom{-2}{-1},\tbinom
{-1}{-1},\tbinom{2}{0},\tbinom{4}{1},\tbinom{3}{1}\right \}  .
\]
(See Figure \ref{Fig.11}.) Analogously, one constructs the following table:
	\setlength\extrarowheight{3pt}
{\small\[
\begin{tabular}
[c]{|c||c|c|c|c|c|}\hline
$i$ & $2$-dim. cones $\sigma_{i}^{\ast}$ in $\Delta_{Q^{\ast}}$ & of type
$(p_{i}^{\ast},q_{i}^{\ast})$ & socius of $p_{i}^{\ast}$ & ($-$)-continued
fraction expansion & length\\ \hline \hline
$1$ & $%
\begin{array}
[c]{c}%
\sigma_{1}^{\ast}=\mathbb{R}_{\geq0}\mathbf{m}_{5}+\mathbb{R}_{\geq
	0}\mathbf{m}_{1} \smallskip
\end{array}
$ & $(2,5)$ & $\widehat{p}_{1}^{\, \ast}=3$ & $%
\begin{array}
[c]{c}%
\frac{q_{1}^{\ast}}{q_{1}^{\ast}-p_{1}^{\ast}}=\frac{5}{5-2}=\frac{5}%
{3}=\left[  \!  \left[  2,3\right]  \!  \right]  \smallskip
\end{array}
$ & $s_{1}^{\ast}=2$\\ \hline
$2$ & $%
\begin{array}
[c]{c}%
\sigma_{2}^{\ast}=\mathbb{R}_{\geq0}\mathbf{m}_{1}+\mathbb{R}_{\geq
	0}\mathbf{m}_{2} \smallskip
\end{array}
$ & $(2,5)$ & $\widehat{p}_{2}^{\, \ast}=3$ & $\frac{q_{2}^{\ast}}{q_{2}^{\ast
	}-p_{2}^{\ast}}=\frac{5}{5-2}=\frac{5}{3}=\left[  \! \left[  2,3\right]
\!  \right]  \!  $ & $s_{2}^{\ast}=2$  \\ \hline
$3$ & $%
\begin{array}
[c]{c}%
\sigma_{3}^{\ast}=\mathbb{R}_{\geq0}\mathbf{m}_{2}+\mathbb{R}_{\geq
	0}\mathbf{m}_{3} \smallskip
\end{array}
$ & $(3,5)$ & $\widehat{p}_{3}^{\, \ast}=2$ & $%
\begin{array}
[c]{c}%
\frac{q_{3}^{\ast}}{q_{3}^{\ast}-p_{3}^{\ast}}=\frac{5}{5-3}=\frac{5}%
{2}=\left[  \! \left[  3,2\right]  \! \right] \smallskip
\end{array}
$ & $s_{3}^{\ast}=2$\\ \hline
$4$ & $ \! \!%
\begin{array}
[c]{c}%
\sigma_{4}^{\ast}=\mathbb{R}_{\geq0}\mathbf{m}_{3}+\mathbb{R}_{\geq
	0}\mathbf{m}_{4} \smallskip
\end{array}
$ & $(4,5)$ & $\widehat{p}_{4}^{\, \ast}=4$ & $%
\begin{array}
[c]{c}%
\frac{q_{4}^{\ast}}{q_{4}^{\ast}-p_{4}^{\ast}}=\frac{5}{5-4}=\frac{5}%
{1}=\left[  \! \left[  5\right]  \! \right]  \smallskip
\end{array}
$ & $s_{4}^{\ast}=1$\\ \hline
$5$ & $%
\begin{array}
[c]{c}%
\sigma_{5}^{\ast}=\mathbb{R}_{\geq0}\mathbf{m}_{4}+\mathbb{R}_{\geq
	0}\mathbf{m}_{5} \smallskip
\end{array}
$ & $(2,5)$ & $\widehat{p}_{5}^{\, \ast}=3$ & $\frac{q_{5}^{\ast}}{q_{5}^{\ast
	}-p_{5}^{\ast}}=\frac{5}{5-2}=\frac{5}{3}=\left[  \! \left[  2,3\right] \smallskip
\! \right]  $ & $s_{5}^{\ast}=2$\\ \hline
\end{tabular}
\]}
	\setlength\extrarowheight{-3pt}

\begin{figure}[h]
	\includegraphics[height=5cm, width=14.8cm]{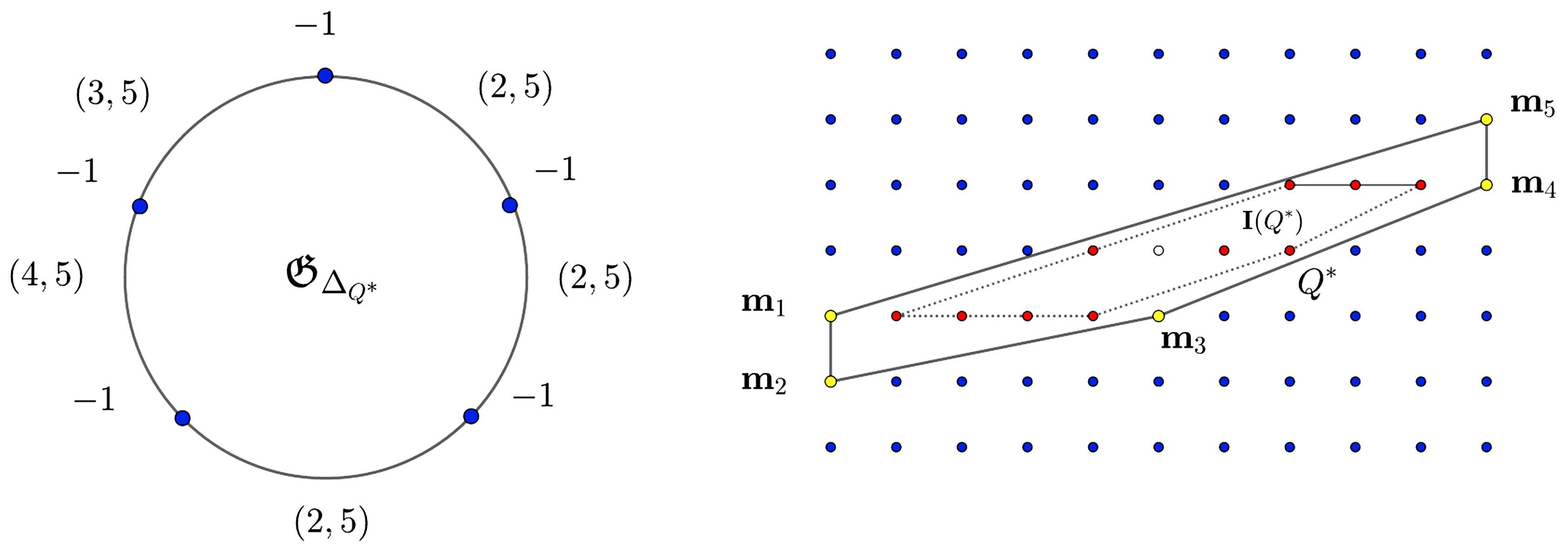}
	\caption{}\label{Fig.11}%
\end{figure}
	
Since $e(X(\mathbb{Z}^{2},\widetilde{\Delta}_{Q^{\ast}}))=\sum_{i=1}^{5}%
(s_{i}^{\ast}+1)=4\cdot3+2=14,$ and
\begin{align*}
-K_{X(\mathbb{Z}^{2},\widetilde{\Delta}_{Q^{\ast}})}^{2} &  =-\frac{1}%
{5}\sharp \left(  \partial Q\cap \mathbb{Z}^{2}\right)  -\sum_{i=1}^{5}\left(
\tfrac{p_{i}^{\ast}+\widehat{p}_{i}^{\ast}-2}{q_{i}^{\ast}}\right)
+\sum_{i=1}^{5}\sum_{j=1}^{s_{i}^{\ast}}(c_{j}^{\ast \, \left(  i\right)  }-2)\\
&  =-\tfrac{7}{5}-4\left(  \tfrac{2+3-2}{5}\right)  -\left(  \tfrac{4+4-2}%
{5}\right)  +4\cdot1+3=2,
\end{align*}
we have
\[
\sharp(\partial Q\cap \mathbb{Z}^{2})-K_{X(\mathbb{Z}^{2},\widetilde{\Delta
	}_{Q^{\ast}})}^{2}=7+2=9=\sharp(\partial(\mathbf{I}(Q))\cap \mathbb{Z}%
^{2})=14-5=e(X(\mathbb{Z}^{2},\widetilde{\Delta}_{Q^{\ast}}))-\sharp(\partial
Q^{\ast}\cap \mathbb{Z}^{2}).
\]
In particular, $\mathbf{I}(Q)=$ conv$\left \{  \tbinom{-1}{4},\tbinom{-1}%
{1},\tbinom{0}{-2},\tbinom{2}{-7},\tbinom{2}{-6},\tbinom{1}{-1},\tbinom{0}%
{2}\right \}  ,$ and
\[
\partial(\mathbf{I}(Q))\cap \mathbb{Z}^{2}=\left \{  \tbinom{-1}{4},\tbinom
{-1}{3},\tbinom{-1}{2},\tbinom{-1}{1},\tbinom{0}{-2},\tbinom{2}{-7},\tbinom
{2}{-6},\tbinom{1}{-1},\tbinom{0}{2}\right \}  .
\]
It is worth mentioning that, though each of $\mathbf{I}(Q^{\ast})$ and
$\mathbf{I}(Q)$ has $9$ lattice points on its boundary, $\mathbf{I}(Q^{\ast})$
is a $\mathbb{Z}^{2}$-pentagon while $\mathbf{I}(Q)$ is a $\mathbb{Z}^{2}%
$-heptagon, and  
\[
\tfrac{19}{2}=\text{area}_{\mathbb{Z}^{2}}(\mathbf{I}%
(Q))\neq \text{area}_{\mathbb{Z}^{2}%
}(\mathbf{I}(Q^{\ast}))=\tfrac{11}{2},\text{ }6=\sharp(\text{int}(\mathbf{I}%
(Q))\cap \mathbb{Z}^{2})\neq \sharp(\text{int}(\mathbf{I}(Q^{\ast}%
))\cap \mathbb{Z}^{2})=2.
\]
In general, the area, as well as the precise location of the vertices and of the
lattices points lying on the boundary or in the interior of $\mathbf{I}(Q)$ and $\mathbf{I}%
(Q^{\ast})$ depend on the types of auxiliary cones. In our example, \ref{AUXCONES} gives%
\begin{align*}
\tau_{1} &  =\mathbb{R}_{\geq0}\tfrac{5}{q_{5}}\left(  \mathbf{n}%
_{5}-\mathbf{n}_{1}\right)  +\mathbb{R}_{\geq0}\tfrac{5}{q_{1}}\left(
\mathbf{n}_{2}-\mathbf{n}_{1}\right)  \\
&  =\mathbb{R}_{\geq0}\tfrac{5}{10}\left(  \tbinom{-1}{0}-\tbinom{3}%
{-10}\right)  +\mathbb{R}_{\geq0}\tfrac{5}{10}\left(  \tbinom{1}{0}-\tbinom
{3}{-10}\right)  =\mathbb{R}_{\geq0}\tbinom{-2}{5}+\mathbb{R}_{\geq0}%
\tbinom{-1}{5}%
\end{align*}
(and similarly for the other four auxiliary cones). From the table
	\setlength\extrarowheight{3pt}
{\small\[
\begin{tabular}
[c]{|c||c|c|c|c|}\hline
$i$ & auxiliary cones $\tau_{i}=(\sigma_{i}^{\ast})^{\vee}$ & of type
$(q_{i}^{\ast}-p_{i}^{\ast},q_{i}^{\ast})$ & ($-$)-continued fraction
expansion & length\\ \hline \hline
$1$ & $%
\begin{array}
[c]{c}%
\tau_{1}=\mathbb{R}_{\geq0}\tbinom{-2}{5}+\mathbb{R}_{\geq0}\tbinom{-1}{5} \smallskip
\end{array}
$ & $(3,5)$ & $%
\begin{array}
[c]{c}%
\frac{q_{1}^{\ast}}{p_{1}^{\ast}}=\frac{5}{2}=\left[  \! \left[  3,2\right]
\! \right]  \! \smallskip
\end{array}
$ & $t_{1}^{\ast}=2$\\ \hline
$2$ & $%
\begin{array}
[c]{c}%
\tau_{2}=\mathbb{R}_{\geq0}\tbinom{1}{5}+\mathbb{R}_{\geq0}\tbinom{-2}{5} \smallskip
\end{array}
$ & $(3,5)$ & $\frac{q_{2}^{\ast}}{p_{2}^{\ast}}=\frac{5}{2}=\left[
\! \left[  3,2\right]  \! \right]  \! $ & $t_{2}^{\ast} \smallskip
=2$\\ \hline
$3$ & $%
\begin{array}
[c]{c}%
\tau_{3}=\mathbb{R}_{\geq0}\tbinom{2}{5}+\mathbb{R}_{\geq0}\tbinom{-1}{0} \smallskip
\end{array}
$ & $(2,5)$ & $%
\begin{array}
[c]{c}%
\frac{q_{3}^{\ast}}{p_{3}^{\ast}}=\frac{5}{3}=\left[  \! \left[  2,3 \right]  \!  
 \right] \smallskip
\end{array}
$ & $t_{3}^{\ast}=2$\\ \hline
$4$ & $ \! \!%
\begin{array}
[c]{c}%
\tau_{4}=\mathbb{R}_{\geq0}\tbinom{1}{0}+\mathbb{R}_{\geq0}\tbinom{1}{-5} \smallskip
\end{array}
$ & $(1,5)$ & $%
\begin{array}
[c]{c}%
\frac{q_{4}^{\ast}}{p_{4}^{\ast}}=\frac{5}{4}=\left[  \! \left[
2,2,2,2\right]  \! \right]  \smallskip
\end{array}
$ & $t_{4}^{\ast}=4$\\ \hline
$5$ & $%
\begin{array}
[c]{c}%
\tau_{5}=\mathbb{R}_{\geq0}\tbinom{-1}{5}+\mathbb{R}_{\geq0}\tbinom{2}{-5} \smallskip
\end{array}
$ & $(3,5)$ & $\frac{q_{5}^{\ast}}{p_{5}^{\ast}}=\frac{5}{2}=\left[
\! \left[  3,2\right]  \! \right]  $ & $t_{5}^{\ast}=2$\\ \hline
\end{tabular}
\]}
\setlength\extrarowheight{-3pt}
\hspace{-0.2cm} we obtain again
\[
\sharp(\partial(\mathbf{I}(Q))\cap \mathbb{Z}^{2})=t_{1}^{\ast}+(t_{2}^{\ast}+t_{3}^{\ast}-1)+(t_{4}^{\ast}-1)+(t_{5}^{\ast}-1)=9
\]
(with the $3$ aces subtracted in order to avoid counting lattice points
twice). Correspondingly, from 
\setlength\extrarowheight{3pt}
{\small\[
\begin{tabular}
[c]{|c||c|c|c|c|}\hline
$i$ & auxiliary cones $\tau_{i}^{\ast}=\sigma_{i}^{\vee}$ & of type
$(q_{i}-p_{i},q_{i})$ & ($-$)-continued fraction expansion &
length\\ \hline \hline
$1$ & $%
\begin{array}
[c]{c}%
\tau_{1}^{\ast}=\mathbb{R}_{\geq0}\tbinom{10}{3}+\mathbb{R}_{\geq0}\tbinom
{0}{-1} \smallskip
\end{array}
$ & $(3,10)$ & $%
\begin{array}
[c]{c}%
\frac{q_{1}}{p_{1}}=\frac{10}{7}=\left[  \!  \left[  2,2,4\right]
\!  \right]  \! \! \smallskip
\end{array}
$ & $t_{1}=3$\\ \hline
$2$ & $%
\begin{array}
[c]{c}%
\tau_{2}^{\ast}=\mathbb{R}_{\geq0}\tbinom{0}{1}+\mathbb{R}_{\geq0}\tbinom
{5}{1} \smallskip
\end{array}
$ & $(1,5)$ & $\frac{q_{2}}{p_{2}}=\frac{5}{4}=\left[  \! \left[
2,2,2,2\right]  \!  \right]  \!  $ & $t_{2}=4$\\ \hline
$3$ & $%
\begin{array}
[c]{c}%
\tau_{3}^{\ast}=\mathbb{R}_{\geq0}\tbinom{-5}{-1}+\mathbb{R}_{\geq0}\tbinom
{5}{2} \smallskip
\end{array}
$ & $(3,5)$ & $%
\begin{array}
[c]{c}%
\frac{q_{3}}{p_{3}}=\frac{5}{2}=\left[  \! \left[  3,2\right]  \! \right]\smallskip
\end{array}
$ & $t_{3}=2$\\ \hline
$4$ & $ \! \!%
\begin{array}
[c]{c}%
\tau_{4}^{\ast}=\mathbb{R}_{\geq0}\tbinom{-5}{-2}+\mathbb{R}_{\geq0}\tbinom
{0}{1}%
\end{array}
$ & $(2,5)$ & $%
\begin{array}
[c]{c}%
\frac{q_{4}}{p_{4}}=\frac{5}{3}=\left[  \! \left[  2,3\right]  \! \right]
\smallskip
\end{array}
$ & $t_{4}=2$\\ \hline
$5$ & $%
\begin{array}
[c]{c}%
\tau_{5}^{\ast}=\mathbb{R}_{\geq0}\tbinom{0}{-1}+\mathbb{R}_{\geq0}%
\tbinom{-10}{-3} \smallskip
\end{array}
$ & $(3,10)$ & $\frac{q_{5}}{p_{5}}=\frac{10}{7}=\left[  \! \left[
2,2,4\right]  \! \right]  $ & $t_{5}=3$\\ \hline
\end{tabular}
\]}
\setlength\extrarowheight{-3pt}
\hspace{-0.2cm}we obtain $\mathfrak{L}_{Q^{\ast}}^{(i)}=\varnothing$ for all $i\in
\{1,2,3,4,5\}$ and
\[
\sharp(\partial(\mathbf{I}(Q^{\ast}))\cap \mathbb{Z}^{2})=\sum_{i=1}^{5}%
(t_{i}-1)\underset{\text{(\ref{KETTENBRUCH1b})}}{=}\sum_{i=1}^{5}\sum_{j=1}^{s_{i}}%
(b_{j}^{\left(  i\right)  }-2)=9.
\]
\end{example}

\begin{example}\label{EXVARTHNOT}
	Taking the $13$-reflexive $\mathbb{Z}^{2}$-quadrilateral
	\[
	Q=\text{ conv}\left \{  \tbinom{3}{-13},\tbinom{-1}{13},\tbinom{-3}{13}%
	,\tbinom{1}{-13}\right \}  \  \text{ which has \ }Q^{\ast}=\text{ conv}\left \{
	\tbinom{-13}{-2},\tbinom{0}{-1},\tbinom{13}{2},\tbinom{0}{1}\right \}
	\]
	as its dual (i.e., (\ref{EXAMPLE2}) and (\ref{EXSTAR2}) for $\ell=13$), we see easily that $X(\mathbb{Z}%
	^{2},\Delta_{Q})$ has two cyclic quotient singularities of type $(3,26)$ and
	two cyclic quotient singularities of type $(17,26),$ and that $X(\mathbb{Z}%
	^{2},\Delta_{Q^{\ast}})$ has two cyclic quotient singularities of type
	$(11,13)$ and two cyclic quotient singularities of type $(7,13).$ (Cf. the \textsc{wve}$^{2}$\textsc{c}-graphs
	in Figures \ref{Fig.5} and \ref{Fig.6}.) Since%
	\[%
	\begin{array}
	[c]{ccc}%
	\tfrac{26}{26-3}=\tfrac{26}{23}=\left[  \! \left[  2,2,2,2,2,2,2,3,2 \right]  \! \right] , &  & \tfrac{26}{26-17}=\tfrac{26}{9}=\left[  \! \left[
	3,9\right]  \! \right]  ,\\
	&  & \\
	\tfrac{13}{13-11}=\tfrac{13}{2}=\left[  \! \left[ 7,2\right]  \! \right]
	, &  & \tfrac{13}{13-7}=\tfrac{13}{6}=\left[  \! \left[   3,2,2,2,2,2\right]  \! \right] ,
	\end{array}
	\]
	formulae  (\ref{THETAIQ1F}) and (\ref{THETAIQ1F2}) give%
	\[%
	\begin{array}
	[c]{c}%
	\sharp(\partial(\mathbf{I}(Q^{\ast}))\cap \mathbb{Z}^{2})=e(X(\mathbb{Z}%
	^{2},\widetilde{\Delta}_{Q}))-\sharp(\partial Q\cap \mathbb{Z}^{2}%
	)=2\cdot10+2\cdot3-8=18,\text{ and\medskip}\\
 \\
	\sharp(\partial(\mathbf{I}(Q))\cap \mathbb{Z}^{2})=e(X(\mathbb{Z}%
	^{2},\widetilde{\Delta}_{Q^{\ast}}))-\sharp(\partial Q^{\ast}\cap
	\mathbb{Z}^{2})=2\cdot3+2\cdot7-4=16,
	\end{array}
	\]
	respectively, i.e., $\sharp(\partial(\mathbf{I}(Q))\cap \mathbb{Z}^{2}%
	)\neq \sharp(\partial(\mathbf{I}(Q^{\ast}))\cap \mathbb{Z}^{2}).$ 
\end{example}

\begin{remark}
	Another method to compute $\sharp \left(  \partial(\mathbf{I}(Q^{\ast}))\cap M\right)$ is to apply Theorem \ref{SIGMAPMINDES} for the
	normal fan
	\[
	\Sigma_{\mathbf{I}(Q^{\ast})}:=\left \{  \text{ the }N\text{-cones }\left \{
	\varpi_{\mathfrak{v}}^{\vee}\left \vert \  \mathfrak{v}\in \text{Vert}%
	(\mathbf{I}(Q^{\ast}))\right.  \right \}  \text{ together with their
		faces}\right \}
	\]
	of $\mathbf{I}(Q^{\ast}),$ where $\varpi_{\mathfrak{v}}:=\left \{
	\lambda(\mathbf{x}-\mathfrak{v})\left \vert \  \lambda \in \mathbb{R}_{\geq
		0},\  \mathbf{x}\in \mathbf{I}(Q^{\ast})\right.  \right \}  $ for all
	$\mathfrak{v}\in$ Vert$(\mathbf{I}(Q^{\ast})),$ and to work with the minimal
	desingularization, say%
	\begin{equation}
	\vartheta:X(N,\widetilde{\Sigma_{\mathbf{I}(Q^{\ast})}})\longrightarrow
	X(N,\Sigma_{\mathbf{I}(Q^{\ast})}) \label{MINDESTHETAVAR}%
	\end{equation}
	of $X(N,\Sigma_{\mathbf{I}(Q^{\ast})}).$ If $F\in$ Edg$(\mathbf{I}(Q^{\ast}))$ and
	$\boldsymbol{\eta}_{F}\in N\mathbb{r}\{ \mathbf{0}\}$ is the (primitive)
	inward-pointing normal of $F,$ then it is easy to see that 
	$
	h_{\mathbf{I}(Q^{\ast})}(\boldsymbol{\eta}_{F})=h_{Q^{\ast}}(\boldsymbol{\eta
	}_{F})+1,
	$
	where $h_{Q^{\ast}}$ and $h_{\mathbf{I}(Q^{\ast})}$ are the support functions
	of $Q^{\ast}$ and $\mathbf{I}(Q^{\ast}),$ respectively (cf.
		(\ref{SUPPFCTP})). Moreoever, $X(N,\Sigma_{\mathbf{I}(Q^{\ast})})$ has at worst
	Gorenstein singularities, and (\ref{MINDESTHETAVAR}) is crepant (with $\widetilde{\Delta}_{Q}=\widetilde{\Sigma_{Q^{\ast}}}$ being a refinement
	of $\widetilde{\Sigma_{\mathbf{I}(Q^{\ast})}}$,  $\{\boldsymbol{\eta}_{F}|{F\in \text{Edg}(\mathbf{I}(Q^{\ast}%
		))}\}\subset{\displaystyle \bigcup \limits_{i=1}^{\nu}}\text{Vert}(\partial \Theta_{\sigma_{i}}^{\mathbf{cp}})$ and $\text{Vert}(\mathbf{I}(Q^{\ast}))$ and $\text{Edg}(\mathbf{I}(Q^{\ast}))$ exactly computable via \cite[\S 3]{DHH} applied for ${\tau
	_{i}}^{\ast}=\sigma_{i}^{\vee}$ for all $i\in \left \{  1,\ldots,\nu \right \}).$ Now writing $\mathbf{I}(Q^{\ast})$ in the form
	\[
	\mathbf{I}(Q^{\ast})=\bigcap \limits_{F\in \text{Edg}(\mathbf{I}(Q^{\ast}))}\left \{  \left.  \mathbf{x}\in \mathbb{R}^{2}\right \vert \left \langle
	\mathbf{x},\boldsymbol{\eta}_{F}\right \rangle \geq h_{\mathbf{I}(Q^{\ast}%
		)}(\boldsymbol{\eta}_{F})\right \}
	\]
	and denoting by 
	$
	D_{\mathbf{I}(Q^{\ast})}:=-\sum \limits_{F\in \text{Edg}(\mathbf{I}(Q^{\ast}%
		))}h_{\mathbf{I}(Q^{\ast})}(\boldsymbol{\eta}_{F})\mathbf{V}_{\Sigma
		_{\mathbf{I}(Q^{\ast})}}(\mathbb{R}_{\geq0}\boldsymbol{\eta}_{F})\in
	\text{Div}_{\text{C}}^{\mathbb{T}}(X(N,\Sigma_{\mathbf{I}(Q^{\ast})}))
	$
	the distinguished ample divisor on $X(N,\Sigma_{\mathbf{I}(Q^{\ast})}),$ we
	obtain the following:
\end{remark}
\begin{proposition}
	The number of lattice points lying on the boundary of $\mathbf{I}(Q^{\ast})$
	is given by the formulae%
	\[
	\sharp \left(  \partial(\mathbf{I}(Q^{\ast}))\cap M\right)  =-D_{\mathbf{I}%
		(Q^{\ast})}\cdot K_{X(N,\Sigma_{\mathbf{I}(Q^{\ast})})}=\sum \limits_{F\in
		\text{\emph{Edg}}(\mathbf{I}(Q^{\ast}))}\left(  D_{\mathbf{I}(Q^{\ast})}%
	\cdot \mathbf{V}_{\Sigma_{\mathbf{I}(Q^{\ast})}}(\mathbb{R}_{\geq
		0}\boldsymbol{\eta}_{F})\right).\]
	\end{proposition}
\begin{proof}
	Applying formula (\ref{DPTHETAP}) of Theorem \ref{SIGMAPMINDES} (for
	$P=\mathbf{I}(Q^{\ast})$ and (\ref{MINDESTHETAVAR})) we deduce that
	\[
	\sharp \left(  \partial(\mathbf{I}(Q^{\ast}))\cap M\right)  =-\vartheta^{\star
	}(D_{\mathbf{I}(Q^{\ast})})\cdot K_{X(N,\widetilde{\Sigma_{\mathbf{I}(Q^{\ast
				})}})}=-D_{\mathbf{I}(Q^{\ast})}\cdot K_{X(N,\Sigma_{\mathbf{I}(Q^{\ast})})},
	\]
	because $K_{X(N,\widetilde{\Sigma_{\mathbf{I}(Q^{\ast})}})}\sim \vartheta
	^{\star}(K_{X(N,\Sigma_{\mathbf{I}(Q^{\ast})})}).$ (See (\ref{INTNUMBSING}).) The second formula follows
	from \ref{NoteCAND} (i).
\end{proof}

\section{Families of combinatorial mirror pairs in the lowest dimension}\label{BATYRMS}

\noindent Batyrev's combinatorial mirror symmetry construction \cite{{Batyrev2}} is completely
efficient whenever the \textquotedblleft ambient spaces\textquotedblright \ are
toric Fano varieties with at worst Gorenstein singularities of (complex)
dimension $\geq4$ or at least of dimension $3.$ In the latter case, the
general members of the linear system defined by their anticanonical divisors
are \textit{K3-surfaces}. In the \textit{lowest} dimension $2$ (i.e., when the
\textquotedblleft ambient spaces\textquotedblright \ are Gorenstein log del
Pezzo surfaces), the corresponding general members are \textit{elliptic
	curves}. The generalisation (in dimension $2$) which takes place by passing
from Gorenstein log del Pezzo surfaces (defined by $1$-reflexive polygons) to
log del Pezzo surfaces defined by $\ell$-reflexive polygons leaves little room
for the determination of \textquotedblleft combinatorial mirrors\textquotedblright, and as yet only up to \textit{ homeomorphism}: The corresponding
general members are smooth projective curves with Hodge diamond having (as
unique non-trivial number) their \textit{genus} (also called \textit{sectional
	genus}) at the left and at the right corner. This genus is  $>1$ whenever
$\ell>1.$

\begin{definition}
	Let $(Q,N)$ be an $\ell$-reflexive pair and $M:=$ Hom$_{\mathbb{Z}%
	}(N,\mathbb{Z}).$ \noindent Since the Cartier divisor $-\ell K_{X(N,\Delta
		_{Q})}$ is very ample  on $X(N,\Delta_{Q})$ (with $\Delta_{Q}=\Sigma_{Q^{\ast}}$) the complete linear system
	$\left \vert -\ell K_{X(N,\Delta_{Q})}\right \vert $ induces the closed embedding%
\[
\xymatrix{\mathbb{T}_{N} \hspace{0.2cm} \ar@{^{(}->}[rr] \ar@/_1.7pc/[rrrr] & &X(N,\Delta_{Q}) \hspace{0.2cm} \ar@{^{(}->}[rr] & & \mathbb{P}_{\mathbb{C}}^{\sharp(Q^{\ast}\cap M)-1}}
\]
	with (the composition mapping)
	\[
	\mathbb{T}_{N}\ni t\longmapsto \lbrack...:z_{\mathbf{m}}:...]_{\mathbf{m}\in
		Q^{\ast}\cap M}\in \mathbb{P}_{\mathbb{C}}^{\sharp(Q^{\ast}\cap M)-1}%
	,\ z_{\mathbf{m}}:=\mathbf{e}(\mathbf{m})(t),
	\]
	where $\mathbf{e}(\mathbf{m}):\mathbb{T}_{N}\rightarrow \mathbb{C}^{\times}$ is the
	character associated with the lattice point $\mathbf{m}$, for all $\mathbf{m}\in Q^{\ast}\cap M.$ The image of
	$X(N,\Delta_{Q})$ in $\mathbb{P}_{\mathbb{C}}^{\sharp(Q^{\ast}\cap M)-1}$ can
	be viewed as the projective variety Proj$(\textsf{S}_{Q^{\ast}}),$ where
	\[
	\textsf{S}_{Q^{\ast}}:=\mathbb{C}[C(Q^{\ast})\cap(M\times \mathbb{Z}%
	)]={\displaystyle \bigoplus \limits_{\kappa=0}^{\infty}}\left(
	{\displaystyle \bigoplus \limits_{\mathbf{m}\in Q^{\ast}\cap M}}\mathbb{C\cdot
		\,}\mathbf{e}(\mathbf{m})\xi^{\kappa}\right)  \
	\]
	$($with $C(Q^{\ast}):=\{(\lambda y_{1},\lambda y_{2},\lambda)\left \vert
	\lambda \in \mathbb{R}_{\geq0}\text{ and }(y_{1},y_{2})\in Q^{\ast}\right.  \})$
	is the semigroup algebra which is naturally graded by setting deg$(\mathbf{e}%
	(\mathbf{m})\xi^{\kappa}):=\kappa.$ (For a detailed exposition see
	\cite[Theorem 2.3.1, p. 75; Proposition 5.4.7, pp. 237-238; Theorem 5.4.8, pp.
	239-240, and Theorem 7.1.13, pp. 325-326]{CLS}.) Hyperplanes $\mathcal{H}%
	\subset \mathbb{P}_{\mathbb{C}}^{\sharp(Q^{\ast}\cap M)-1}$ give curves
	Proj$(\textsf{S}_{Q^{\ast}})\cap \mathcal{H}$ which are linearly equivalent to $-\ell
	K_{X(N,\Delta_{Q})}.$ For \textit{generic} $\mathcal{H}$'s the intersection
	$\mathcal{C}_{Q}:=$ Proj$(\textsf{S}_{Q^{\ast}})\cap \mathcal{H}$ is (by Bertini's
	Theorem) a smooth connected projective curve in the non-singular locus of
	Proj$(\textsf{S}_{Q^{\ast}})\cong X(N,\Delta_{Q}).$ The genus of $g(\mathcal{C}_{Q})$
	of $\mathcal{C}_{Q}$ is called \textit{the\ sectional genus} of $X(N,\Delta
	_{Q})$ and will be denoted simply as $g_{Q}\emph{\medskip}.$
\end{definition}

\begin{lemma}
	The\textit{\ sectional genus} of $X(N,\Delta_{Q}) (=X(N,\Sigma_{Q^{\ast}}))$ is%
	\begin{equation}
	g_{Q}=\frac{1}{2}(\ell-1)\sharp(\partial Q^{\ast}\cap M)+1.\label{SECGENUSF}%
	\end{equation}
	
	\begin{proof} By  \cite[Proposition 10.5.8, p. 509]{CLS},  
	 $g_{Q}=\sharp($int$( Q^{\ast})\cap M).$ So it suffices to apply (\ref{PICKINTQSTAR}).
\end{proof}
	
\end{lemma}

\begin{remark}The $\mathbb{C}$-vector space of the global sections of the canonical sheaf
	over $\mathcal{C}_{Q}$ is%
	\[
	H^{0}(\mathcal{C}_{Q},\omega_{\mathcal{C}_{Q}})\cong H^{0}(X(N,\Sigma
	_{\mathbf{I}(Q^{\ast})}),\mathcal{O}_{X(N,\Sigma_{\mathbf{I}(Q^{\ast})})}(D_{\mathbf{I}%
		(Q^{\ast})}))
	\]
	and has dimension $h^{0}(\mathcal{C}_{Q},\omega_{\mathcal{C}_{Q}}):=$
	dim$_{\mathbb{C}}(H^{0}(\mathcal{C}_{Q},\omega_{\mathcal{C}_{Q}}))=$
	dim$_{\mathbb{C}}(H^{1}(\mathcal{C}_{Q},\mathcal{O}_{\mathcal{C}_{Q}}%
	))=g_{Q}$ (by adjunction). Moreover,%
	\[
	\mathcal{C}_{Q}^{2}=(-\ell K_{X(N,\Delta_{Q})})^{2}=\ell^{2}%
	K_{X(N,\Delta_{Q})}^{2}\underset{\text{(\ref{K2FORMULA})}}{=}\ell \, \sharp(\partial
	Q^{\ast}\cap M).
	\]
\end{remark}

\begin{definition}\label{DEFMIRPAR}
	Let $(Q,N),$ $(Q^{\ast},M)$ be $\ell$-reflexive pairs (with $M:=$
	Hom$_{\mathbb{Z}}(N,\mathbb{Z})$). We shall say that $(Q,N)$ \textit{has the
		topological mirror property} if for any general member $\mathcal{C}_{Q}$ of
	the linear system	$\left \vert -\ell K_{X(N,\Delta_{Q})}\right \vert $
	and any general member $\mathcal{C}_{Q^{\ast}}$ of the linear system $\left \vert -\ell K_{X(M,\Delta_{Q^{\ast}})}\right \vert $
 we have
	\[
	g_{Q}:=g(\mathcal{C}_{Q})=g(\mathcal{C}_{Q^{\ast}})=:g_{Q^{\ast}}.
	\]
	In this case, we shall say that $(\mathcal{C}_{Q},\mathcal{C}_{Q^{\ast}})$ is a \textit{combinatorial mirror pair} and we may think of $\mathcal{C}_{Q}$ as \textit{combinatorial mirror partner }of
	$\mathcal{C}_{Q^{\ast}}$ and vice versa.
\end{definition}

\begin{note}
	If $\ell >1,$ then by the Twelve-Point Theorem \ref{G12PTTHM} and by (\ref{SECGENUSF}) the equality $g_{Q}=g_{Q^{\ast}}$ implies%
	\begin{equation}
	\left.
	\begin{array}
	[c]{l}%
	\sharp(\partial Q\cap N)=\sharp(\partial Q^{\ast}\cap M)\medskip \\
	\sharp(\partial Q\cap N)+\sharp(\partial Q^{\ast}\cap M)=12
	\end{array}
	\right \}  \Rightarrow \sharp(\partial Q\cap N)=\sharp(\partial Q^{\ast}\cap
	M)=6. \label{SECHSPACK}
	\end{equation}
	And conversely, from (\ref{SECHSPACK}) we get obviously $g_{Q}=g_{Q^{\ast}}.$ 
\end{note}

\begin{proposition}\label{LISTSWITHTMS}
	Let $\ell$ be an odd integer $\geq3.$ Then the families of $\ell$-reflexive
	pairs $(Q,\mathbb{Z}^{2})$ constructed by the $\mathbb{Z}^{2}$-polygons $Q$ of
	the following tables have the topological mirror property.
	\setlength\extrarowheight{5pt}
	{\small
\[
\begin{tabular}
[c]{|c||c|c|}\hline
\emph{No.} & \emph{The} $\mathbb{Z}^{2}$\emph{-triangles} & \emph{under the
	restrictions}\\ \hline \hline
\emph{(i)} & $%
\begin{array}
[c]{c}%
\emph{conv}\{ \binom{0}{1},\binom{2\ell}{3},\binom{-3\ell}{-5}\} \smallskip
\end{array}
$ & $\ell \geq7,\ 3\nmid \ell \ $\emph{and }$5\nmid \ell$\\ \hline
\emph{(ii)} & $%
\begin{array}
[c]{c}%
\emph{conv}\{ \binom{0}{1},\binom{2\ell}{5},\binom{-3\ell}{-8}\} \smallskip
\end{array}
$ & $\ell \geq7,\ 5\nmid \ell \  \emph{and}\ 13\nmid \ell$\\ \hline
\emph{(iii)} & $%
\begin{array}
[c]{c}%
\emph{conv}\{ \binom{0}{1},\binom{2\ell}{7},\binom{-3\ell}{-11}\} \smallskip
\end{array}
$ & $\ell \geq5\  \emph{and}\ j\nmid \ell,\  \forall j\in \{3,7,11\}$\\ \hline
\emph{(iv)} & $%
\begin{array}
[c]{c}%
\emph{conv}\{ \binom{0}{1},\binom{2\ell}{9},\binom{-3\ell}{-14}\} \smallskip
\end{array}
$ & $\ell \geq11\  \emph{and}\ j\nmid \ell,\  \forall j\in \{3,5,7\}$\\ \hline
\emph{(v)} & $%
\begin{array}
[c]{c}%
\emph{conv}\{ \binom{0}{1},\binom{2\ell}{11},\binom{-3\ell}{-17}\} \smallskip
\end{array}
$ & $\ell \geq13\  \emph{and}\ j\nmid \ell,\  \forall j\in \{3,5,7,11,17\}$\\ \hline
\emph{(vi)} & $%
\begin{array}
[c]{c}%
\emph{conv}\{ \binom{0}{1},\binom{2\ell}{13},\binom{-3\ell}{-20}\} \smallskip
\end{array}
$ & $\ell \geq17\  \emph{and}\ j\nmid \ell,\  \forall j\in \{3,5,7,11,13\}$\\ \hline
\emph{(vii)} & $%
\begin{array}
[c]{c}%
\emph{conv}\{ \binom{0}{1},\binom{2\ell}{15},\binom{-3\ell}{-23}\} \smallskip
\end{array}
$ & $\ell \geq11\  \emph{and}\ j\nmid \ell,\  \forall j\in \{3,5,7,23\}$\\ \hline
\emph{(viii)} & $%
\begin{array}
[c]{c}%
\emph{conv}\{ \binom{0}{1},\binom{2\ell}{17},\binom{-3\ell}{-26}\} \smallskip
\end{array}
$ & $\ell \geq11\  \emph{and}\ j\nmid \ell,\  \forall j\in \{3,5,13\}$\\ \hline
$\cdots$ & $%
\begin{array}
[c]{c}%
\cdots \cdots \smallskip
\end{array}
$ & $\cdots \cdots$\\ \hline
\end{tabular}
\]}
	\setlength\extrarowheight{-5pt}
	\vspace{0.5cm}
	\setlength\extrarowheight{5pt}
{\small
	\[
	\begin{tabular}
		[c]{|c||c|c|}\hline
		\emph{No.} & \emph{The} $\mathbb{Z}^{2}$\emph{-quadrilaterals} & \emph{under
			the restrictions}\\ \hline \hline
		\emph{(i)} & $%
		\begin{array}
		[c]{c}%
		\emph{conv}\{ \binom{0}{-1},\binom{\ell}{2},\binom{0}{1},\binom{-2\ell}%
		{-3}\} \smallskip
		\end{array}
		$ & $\ell \geq5\ $\emph{and} $3\nmid \ell$\\ \hline
		\emph{(ii)} & $%
		\begin{array}
		[c]{c}%
		\emph{conv}\{ \binom{0}{-1},\binom{\ell}{3},\binom{0}{1},\binom{-2\ell}%
		{-5}\} \smallskip
		\end{array}
		$ & $\ell \geq7,\ 3\nmid \ell \ $\emph{and} $5\nmid \ell$\\ \hline
		\emph{(iii)} & $%
		\begin{array}
		[c]{c}%
		\emph{conv}\{ \binom{0}{-1},\binom{\ell}{4},\binom{0}{1},\binom{-2\ell}%
		{-7}\} \smallskip
		\end{array}
		$ & $\ell \geq11\  \emph{and}\ j\nmid \ell,\  \forall j\in \{3,5,7\}$\\ \hline
		\emph{(iv)} & $%
		\begin{array}
		[c]{c}%
		\emph{conv}\{ \binom{0}{-1},\binom{\ell}{5},\binom{0}{1},\binom{-2\ell}%
		{-9}\} \smallskip
		\end{array}
		$ & $\ell \geq7,\ 3\nmid \ell \ $\emph{and} $5\nmid \ell$\\ \hline
		\emph{(v)} & $%
		\begin{array}
		[c]{c}%
		\emph{conv}\{ \binom{0}{-1},\binom{\ell}{6},\binom{0}{1},\binom{-2\ell}{-11}\}\smallskip
		\end{array}
		$ & $\ell \geq13\  \emph{and}\ j\nmid \ell,\  \forall j\in \{3,5,7,11\}$\\ \hline
		$\cdots$ & $%
		\begin{array}
		[c]{c}%
		\cdots \cdots \smallskip
		\end{array}
		$ & $\cdots \cdots$\\ \hline
	\end{tabular}
	\]}
	\setlength\extrarowheight{-5pt}
	\newpage
{\small
	\setlength\extrarowheight{5pt}
	\[
	\begin{tabular}
	[c]{|c||c|c|}\hline
	\emph{No.} & \emph{The} $\mathbb{Z}^{2}$\emph{-pentagons} & \emph{under the
		restrictions}\\ \hline \hline
	\emph{(i)} & $%
	\begin{array}
	[c]{c}%
	\emph{conv}\{ \binom{0}{-1},\binom{\ell}{1},\binom{\ell}{3},\binom{0}{1}%
	,\binom{-\ell}{-2}\} \smallskip
	\end{array}
	$ & $\ell \geq5\ $\emph{and} $3\nmid \ell$\\ \hline
	\emph{(ii)} & $%
	\begin{array}
	[c]{c}%
	\emph{conv}\{ \binom{0}{-1},\binom{\ell}{2},\binom{\ell}{4},\binom{0}{1}%
	,\binom{-\ell}{-3}\} \smallskip
	\end{array}
	$ & $\ell \geq7,\ 3\nmid \ell \ $\emph{and} $5\nmid \ell$\\ \hline
	\emph{(iii)} & $%
	\begin{array}
	[c]{c}%
	\emph{conv}\{ \binom{0}{-1},\binom{\ell}{3},\binom{\ell}{5},\binom{0}{1}%
	,\binom{-\ell}{-4}\}\smallskip
	\end{array}
	$ & $\ell \geq11\  \emph{and}\ j\nmid \ell,\  \forall j\in \{3,5,7\}$\\ \hline
	\emph{(iv)} & $%
	\begin{array}
	[c]{c}%
	\emph{conv}\{ \binom{0}{-1},\binom{\ell}{4},\binom{\ell}{6},\binom{0}{1}%
	,\binom{-\ell}{-5}\} \smallskip
	\end{array}
	$ & $\ell \geq11\  \emph{and}\ j\nmid \ell,\  \forall j\in \{3,5,7\}$\\ \hline
	\emph{(v)} & $%
	\begin{array}
	[c]{c}%
	\emph{conv}\{ \binom{0}{-1},\binom{\ell}{5},\binom{\ell}{7},\binom{0}{1}%
	,\binom{-\ell}{-6}\}\smallskip
	\end{array}
	$ & $\ell \geq13\  \emph{and}\ j\nmid \ell,\  \forall j\in \{3,5,7,11\}$\\ \hline
	$\cdots$ & $%
	\begin{array}
	[c]{c}%
	\cdots \cdots \smallskip
	\end{array}
	$ & $\cdots \cdots$\\ \hline
	\end{tabular}
	\]}
	\setlength\extrarowheight{-5pt}
	\vspace{0.4cm}
	\setlength\extrarowheight{5pt}
{\small
	\[
	\begin{tabular}
	[c]{|c||c|c|}\hline
	\emph{No.} & \emph{The} $\mathbb{Z}^{2}$\emph{-hexagons} & \emph{under the
		restrictions}\\ \hline \hline
	\emph{(i)} & $%
	\begin{array}
	[c]{c}%
	\emph{conv}\{ \binom{0}{-1},\binom{\ell}{1},\binom{\ell}{2},\binom{0}{1}%
	,\binom{-\ell}{-1},\binom{-\ell}{-2}\} \smallskip
	\end{array}
	$ & $-\! \! \!-$\\ \hline
	\emph{(ii)} & $%
	\begin{array}
	[c]{c}%
	\emph{conv}\{ \binom{0}{-1},\binom{\ell}{2},\binom{\ell}{3},\binom{0}{1}%
	,\binom{-\ell}{-2},\binom{-\ell}{-3}\} \smallskip
	\end{array}
	$ & $\ell \geq7\ $\emph{and} $3\nmid \ell$\\ \hline
	\emph{(iii)} & $%
	\begin{array}
	[c]{c}%
	\emph{conv}\{ \binom{0}{-1},\binom{\ell}{3},\binom{\ell}{4},\binom{0}{1}%
	,\binom{-\ell}{-3},\binom{-\ell}{-4}\}\smallskip
	\end{array}
	$ & $\ell \geq13\ $\emph{and} $3\nmid \ell$\\ \hline
	\emph{(iv)} & $%
	\begin{array}
	[c]{c}%
	\emph{conv}\{ \binom{0}{-1},\binom{\ell}{4},\binom{\ell}{5},\binom{0}{1}%
	,\binom{-\ell}{-4},\binom{-\ell}{-5}\} \smallskip
	\end{array}
	$ & $\ell \geq21\ $\emph{and} $5\nmid \ell$\\ \hline
	\emph{(v)} & $%
	\begin{array}
	[c]{c}%
	\emph{conv}\{ \binom{0}{-1},\binom{\ell}{5},\binom{\ell}{6},\binom{0}{1}%
	,\binom{-\ell}{-5},\binom{-\ell}{-6}\}\smallskip
	\end{array}
	$ & $\  \ell \geq31,\ 3\nmid \ell \ $\emph{and} $5\nmid \ell$\\ \hline
	$\cdots$ & $%
	\begin{array}
	[c]{c}%
	\cdots \cdots \smallskip
	\end{array}
	$ & $\cdots \cdots$\\ \hline
	\end{tabular}
	\]}
\setlength\extrarowheight{-5pt}
\vspace{0.2cm} 
\newline
\noindent \emph{(The sectional genus equals 3$\ell-2.$ The tables are to be continued by following the
	same pattern: One increases gradually the ordinates of the corresponding vertices, as well as the lower bounds for $\ell$, and excludes suitable primes from being divisors of $\ell$.)}
\end{proposition}

\begin{proof} It is straightforward to check that the number of lattice points lying on the boundary of each of these $\mathbb{Z}^{2}$-polygons equals $6$. \end{proof}

\begin{note}
(i) There are lots of examples of $\ell$-reflexive pairs $(Q,\mathbb{Z}^{2})$ which have the topological mirror property but they are not self-dual. For instance, for the $5$-reflexive $\mathbb{Z}^{2}$-triangle (from the third
row of the first table in Proposition \ref{LISTSWITHTMS})
\[
Q:=\text{conv}\left \{  \tbinom{0}{1},\tbinom{10}{7},\tbinom{-15}{-11}\right \}
\text{ \ with \ }Q^{\ast}=\text{conv}\left \{  \tbinom{3}{-5},\tbinom{-18}%
{25},\tbinom{4}{-5}\right \}  ,\text{ }%
\]
we have $[Q]_{\mathbb{Z}^{2}}=[Q^{\ast}]_{\mathbb{Z}^{2}}.$ On the other hand,
for the $11$-reflexive $\mathbb{Z}^{2}$-triangle (from the fourth row of the
first table in  Proposition \ref{LISTSWITHTMS})%
\[
Q:=\text{conv}\left \{  \tbinom{0}{1},\tbinom{22}{9},\tbinom{-33}{-14}\right \}
\text{ \ with \ }Q^{\ast}=\text{conv}\left \{  \tbinom{4}{-11},\tbinom{-23}%
{55},\tbinom{5}{-11}\right \}  ,\text{ }\smallskip
\]
we have $[Q]_{\mathbb{Z}^{2}}\neq \lbrack Q^{\ast}]_{\mathbb{Z}^{2}}%
=[Q^{\ast^{\prime}}]_{\mathbb{Z}^{2}},$ where $Q^{\ast^{\prime}}:=$
conv$\left \{  \tbinom{0}{1},\tbinom{22}{17},\tbinom{-33}{-26}\right \}  $ (from
the eighth row of the first table in  Proposition \ref{LISTSWITHTMS}).
\newline (ii) Setting $
\text{RP}_{\nu}(\ell;N)_{\text{t.m.p.}}:=\left \{  \left.  \left[  Q\right]  _{N}\in
\text{RP}(\ell;N)_{\text{t.m.p.}}\right \vert \sharp(\text{Vert}(Q))=\nu \right \},
$ for $\nu \in \{3,4,5,6\}$, where
\[ \text{RP}(\ell;N)_{\text{t.m.p.}}:=\left \{  \left.
\left[  Q\right]  _{N}\in \text{RP}(\ell;N)\right \vert (Q,N) \ \text{has the topological mirror property}\right \},\]
we find via the database \cite{Br-Kas} that the number $\sharp($RP$(\ell;N)_{\text{t.m.p.}})$ is by no means negligible:
\setlength\extrarowheight{2pt}
\[
\begin{tabular}
[c]{|c||c|c|c|c|c|c|c|c|c|c|c|c|c|c|c|c|c|c|}\hline
$\ell$ & $1$ & $3$ & $5$ & $7$ & $9$ & $11$ & $13$ & $15$ & $17$ & $19$ & $21$
& $23$ & $25$ & $27$ & $29$ & $31$ & $33$ & $35$\\ \hline \hline
$\sharp($RP$_{3}(\ell;N)_{\text{t.m.p.}})$ & $5$ & $0$ & $1$ & $2$ & $0$ & $6$
& $8$ & $0$ & $12$ & $14$ & $0$ & 18 & $5$ & $0$ & $24$ & $26$ & $0$ &
$2$\\ \hline
$\sharp($RP$_{4}(\ell;N)_{\text{t.m.p.}})$ & $7$ & $0$ & $1$ & $3$ & $0$ & $7$
& $9$ & $0$ & $13$ & $15$ & $0$ & 19 & $5$ & $0$ & $25$ & $27$ & $0$ &
$3$\\ \hline
$\sharp($RP$_{5}(\ell;N)_{\text{t.m.p.}})$ & $3$ & $0$ & $1$ & $2$ & $0$ & $4$
& $5$ & $0$ & $7$ & $8$ & $0$ & 10 & $5$ & $0$ & $13$ & $14$ & $0$ &
$4$\\ \hline
$\sharp($RP$_{6}(\ell;N)_{\text{t.m.p.}})$ & $1$ & $1$ & $1$ & $2$ & $1$ & $2$
& $3$ & $1$ & $3$ & $4$ & $2$ & 4 & $3$ & $2$ & $5$ & $6$ & $2$ & $3$\\ \hline
$\sharp($RP$(\ell;N)_{\text{t.m.p.}})$ & $16$ & $1$ & $4$ & $9$ & $1$ & $19$ &
$25$ & $1$ & $35$ & $41$ & $2$ & 51 & $18$ & $2$ & $67$ & $73$ & $2$ &
$12$\\ \hline
sectional genus & $1$ & $7$ & $13$ & $19$ & $25$ & $31$ & $37$ & $43$ & $49$ &
$55$ & $61$ & $67$ & $73$ & $79$ & $85$ & $91$ & $97$ & $103$\\ \hline
\end{tabular}
\  \  \
\]
\setlength\extrarowheight{-2pt}
In fact, $\sharp($RP$(\ell;N)_{\text{t.m.p.}})$ can take relative high values,
as we see from the following table for the biggest 10 values of $\ell<200$ with
$j\nmid \ell,\  \forall j\in \{3,5,7,11,13\}.$
\setlength\extrarowheight{2pt}
\[
\begin{tabular}
[c]{|c||c|c|c|c|c|c|c|c|c|c|}\hline
$\ell$ & $157$ & $163$ & $167$ & $173$ & $179$ & $181$ & $191$ & $193$ & $197$
& $199$\\ \hline
$\sharp($RP$(\ell;N)_{\text{t.m.p.}})$ & $409$ & $425$ & $435$ & $451$ & $467$
& $473$ & $499$ & $505$ & $515$ & $521$\\ \hline
sectional genus & $469$ & $487$ & $499$ & $517$ & $535$ & $541$ & $571$ &
$577$ & $589$ & $595$\\ \hline
\end{tabular}
\]
\setlength\extrarowheight{-2pt}
\end{note}
\section{Concluding remarks and questions}\label{FINALE}
\noindent \textbf{(i)} About the role of\,\footnote{$\mathbf{I}(Q)$ is often called\textit{ the
		adjoint polygon of} $Q.$} $\mathbf{I}(Q)$ (i.e., of the convex hull of
interior lattice points of an \textit{arbitrary} lattice polygon $Q$) for the
description of geometric properties of curves on the toric compact surface
defined by $Q$ the reader is referred to Koelman \cite[Chapters 2-4]{Koelman}, Schicho \cite[\S 3]{Schicho},
Castryck \cite[ \S 2-\S 3]{Castryck}, and Castryck \& Cools \cite{Castryck-Cools1},  \cite{Castryck-Cools2}. In our case, we can assume that the curves
$\mathcal{C}_{Q}$ are nothing but Zariski closures $\overline{Z_{\mathfrak{f}%
}}$ of affine hypersurfaces $Z_{\mathfrak{f}}\subset \mathbb{T}_{N}$ for
Laurent polynomials $\mathfrak{f}$ having $Q^{\ast}$ as their Newton polygon. It would
be interesting, for reflexive $\ell$-polygons $Q,$ to investigate
if (beyond the topological equivalence) there is a deeper relation between (e.g., certain complex structures on)
$\overline{Z_{\mathfrak{f}}}\subset X(N,\Delta_{Q})$ and $\overline
{Z_{\mathfrak{g}}}\subset X(M,\Delta_{Q^{\ast}})$ on the \textquotedblleft
other side\textquotedblright. For given combinatorial mirror partners (as
defined in \ref{DEFMIRPAR}) what would be the connection between their \textquotedblleft
strict\textquotedblright \ mirrors (which turn out to be particular 3-dimensional Landau--Ginzburg models) from the point of view
of the \textit{homological mirror symmetry} for curves of high
genus? (Cf. Efimov \cite{Efimov}.)\medskip

\noindent \textbf{(ii)} Let $Q\subset \mathbb{R}^{d}$ be a $d$-dimensional ($1$-)reflexive
lattice \textit{polytope} w.r.t. a lattice $N$ (of rank $d$), $Q^{\circ
}\subset \mathbb{R}^{d}$ its polar w.r.t. $M:=$ Hom$_{\mathbb{Z}}%
(N,\mathbb{Z}),$ and
{\small
\[
\left \{
\begin{array}
[c]{c}%
i\text{-dimensional }\\
\text{faces of }Q
\end{array}
\right \}  \ni F\longmapsto F^{\circ
}:=\{ \left.  \mathbf{x}\in Q^{\circ
}\right \vert \left \langle \mathbf{x},\mathbf{y}\right \rangle =-1,\forall
\mathbf{y}\in F\} \in \left \{
\begin{array}
[c]{c}%
(d-1-i)\text{-dimensional }\\
\text{faces of }Q^{\circ}%
\end{array}
\right \}
\]}
the bijection induced by the polarity. Furthermore, let us denote by
Vol$_{N}(F)$ the normalised volume of $F$ w.r.t. the lattice $N$ and by
Vol$_{M}(F^{\circ
})$ the normalised volume of $F^{\circ
}$ w.r.t. $M.$ The
following generalisations of the Twelve-Point formula (\ref{12PTFORMULA}) in dimensions
$d\geq3$ are known: If $d=3,$ then
\begin{equation}
\fbox{$%
	\begin{array}
	[c]{ccc}
	&
	\begin{array}
	[c]{c}%
	\\%
	{\displaystyle \sum \limits_{F\text{ edges of }Q}}
	\text{Vol}_{N}(F)\cdot \text{Vol}_{M}(F^{\circ
	})=24. \medskip\\
	\end{array}
	& \ \ \
	\end{array}
	$}\label{24FORMULA}
\end{equation}
(See, e.g., \cite[Part A, Theorem 7.2.1]{Batyrev3}, \cite[Theorem 5.1.16]{Haase-Nill-Paffenholz}, \cite[Theorem 1.1]{G-H-S}, and \cite[Corollary 5.4]{Batyrev-Schaller}.)
If $d=4,$ then%
{\small
\begin{equation}
\fbox{$\
	\begin{array}
	[c]{l}%
	\\
	12\,(\sharp \left(  \partial Q\cap N\right)  +\sharp \left(  \partial Q^{\circ
	}\cap M\right)  )
	=2(\text{Vol}_{N}(Q)+\text{Vol}_{M}(Q^{\circ}))-\sum \limits_{\substack{F\text{
				faces of }Q\\ \text{with dim}(F)\in \{1,2\} \text{ }}}\text{Vol}_{N}%
	(F)\cdot \text{Vol}_{M}(F^{\circ})\medskip\\
	\end{array}
	\ $} \label{D4FORMULA}
\end{equation}}(See \cite[Corollary 5.6]{Batyrev-Schaller}.) If $d\geq5,$ Ehr$_{N}(Q;k):=\sharp(kQ\cap N)\in \mathbb{Q}[k]$ the
\textit{Ehrhart polynomial} of $Q,$ and
{\small
\[
\mathfrak{Ehr}_{N}(Q;t):=\sum_{k=0}^{\infty}\text{Ehr}_{N}(Q;k)t^{k}=\frac
{1}{\left(  1-t\right)  ^{d+1}}\left(
{\displaystyle \sum \limits_{j=0}^{d}}
\psi_{j}(Q)t^{j}\right)
\]}
its \textit{Ehrhart series}, then the so-called \textit{stringy
	Libgober-Wood identity} (applied by Batyrev \& Schaller in \cite[Theorem 5.2]{Batyrev-Schaller}) gives
\begin{equation}
\fbox{$\
\begin{array}
[c]{c}%
\\%
{\displaystyle \sum \limits_{j=0}^{d}}
\psi_{j}(Q)\left(  2j-d\right)  ^{2}=\dfrac{1}{3}\left(  d\, \text{Vol}%
_{N}(Q)+ \displaystyle \sum \limits_{\substack{F\text{ faces of }Q \smallskip \\ \text{with dim}%
		(F)=d-2\text{ }}}2(\text{Vol}_{N}(F)\cdot \text{Vol}_{M}(F^{\circ}))\right),  \bigskip \\
\end{array}
\ $}\label{DDIMFORMULA}
\end{equation} i.e., a formula which is no longer symmetric w.r.t. to $Q$ and $Q^{\circ}$. 
In particular, if $Q$ happens to be a
\textit{smooth}\footnote{See \cite [Definition 2.4.2 (b) and Theorem 2.4.3, p.
	87]{CLS}.\smallskip} (also known as \textit{Delzant}) polytope and $d\geq3,$ we have
\begin{equation}
\fbox{$%
	\begin{array}
	[c]{ccc}
	&
	\begin{array}
	[c]{c}%
	\\%
	{\displaystyle \sum \limits_{F\text{ edges of }Q}}
	\text{Vol}_{N}(F)=12f_{2}+\left(  5-3d\right)  f_{1},\medskip\\
	\end{array}
	&
	\end{array}
	$}\label{DELZANT}\end{equation}
where $\mathbf{f}=(f_{0},f_{1},...,f_{d})$ is the $\mathbf{f}$-vector of $Q.$
(See Godinho, von Heymann \& Sabatini \cite[Theorem 1.2]{G-H-S}.)\medskip

\noindent \textbf{(iii)} Let $Q\subset \mathbb{R}^{d}$ be a $d$-dimensional $\ell$-reflexive
\textit{polytope}\footnote{This means that $Q$ has the origin in its (strict)
	interior, all the vertices of $Q$ are \textit{primitive} w.r.t. $N,$ and the
	\textit{local indices} of $Q$ w.r.t. all facets of $Q$ (defined in analogy to
	\ref{DEFLOCIND} (ii)) are equal to $\ell.$} w.r.t. a lattice $N$ (of rank
$d$), and $Q^{\circ}\subset \mathbb{R}^{d}$ its polar w.r.t. $M:=$
Hom$_{\mathbb{Z}}(N,\mathbb{Z}).$ If we assume that $\ell>1,$ is it possible
to generalise Theorems \ref{CHLATT} and \ref{GLOBALTHM}, as well as the
formulae in \textbf{(ii)} and other properties (as those described in
\S \ref{FIRSTPROOF}-\S \ref{CHARDIFFSEC} in the $d=2$ case) for $Q$ and
its\textit{ }dual $Q^{\ast}:=\ell Q^{\circ}$ \textit{whenever} $d\geq3?$ It
should be clear from the outset that there are certain particularities,
restrictions and limitations (with some of them already mentioned in
\cite[\S3]{KaNi}) which have to be taken into account in order to deal
with realistic conjectures: For instance,\smallskip \  \newline(a) in contrast
to what happens in dimension $d=2$ (see Corollary \ref{ODDNESS2}), already in
dimension $d=3$ there are $\ell$-reflexive polytopes also for every
\textit{even} integer $\ell \geq2.\smallskip$\newline(b) Theorem \ref{CHLATT}
and formula (\ref{24FORMULA}) fail (in general) to hold in dimension $d=3.$ An appropriate
modification is believed to be the following:\smallskip \newline 
\noindent\textbf{Conjecture A}. (\cite[\S 3.5]{KaNi}) \textit{Suppose that} $d=3,$
$\Lambda_{\text{Edg}(Q)}$ (resp., $\Lambda_{\text{Edg}(Q^{\ast})}$) \textit{is
	the sublattice of} $N$ (resp., \textit{of} $M$) \textit{generated by the set
}Edg$(Q)$ \textit{of the edges of} $Q$ (resp., \textit{by the set}
Edg$(Q^{\ast})$ \textit{of the edges of }$Q^{\ast}$) \textit{and that}
$(Q,\Lambda_{\text{Edg}(Q)})$ \textit{is an} $1$-\textit{reflexive pair}.
\textit{Then} $(Q^{\ast},\Lambda_{\text{Edg}(Q^{\ast})})$ (which is to be identified with $(Q^{\circ}%
,\text{Hom}_{\mathbb{Z}}(\Lambda_{\text{Edg}(Q)},\mathbb{Z}))$) \textit{is an}
$1$-\textit{reflexive pair too, and} (\ref{24FORMULA}) \textit{is true} (if one
replaces in it $F^{\circ}$ by $F^{\ast}$). \smallskip \newline
\noindent (c) The corresponding
modification of Theorem \ref{CHLATT} for $d\geq4$ gives\footnote{If $Q$ happens to be smooth, then one could use in the Cojecture B formula (\ref{DELZANT}) instead of (\ref{D4FORMULA}) and (\ref{DDIMFORMULA}).\smallskip}: \smallskip \newline
\noindent \textbf{Conjecture B}. \textit{Suppose that} $d\geq4,$ $\Lambda_{\mathcal{F}%
	_{d-2}(Q)}$ (resp., $\Lambda_{\mathcal{F}_{d-2}(Q^{\ast})}$) \textit{is the
	sublattice of} $N$ (resp., \textit{of} $M$) \textit{generated by the set
}$\mathcal{F}_{d-2}(Q)$ \textit{of the faces of} $Q$ (resp., \textit{by the
	set} $\mathcal{F}_{d-2}(Q^{\ast})$ \textit{of the faces of }$Q^{\ast}$)
\textit{of codimension} $2,$ \textit{and that} $(Q,\Lambda_{\mathcal{F}%
	_{d-2}(Q)})$ \textit{is an} $1$-\textit{reflexive pair}. \textit{Then}
$(Q^{\ast},\Lambda_{\mathcal{F}_{d-2}(Q^{\ast})})$ \textit{is an}
$1$-\textit{reflexive pair too, formula} (\ref{D4FORMULA}) \textit{is true for} $d=4$ (if one
replaces in it $F^{\circ}$ by $F^{\ast}$ and  $Q^{\circ}$ by $Q^{\ast}$), \textit{and
formula} (\ref{DDIMFORMULA}) \textit{is true for} $d\geq5$ (for both $Q$ and $Q^{\ast}$). \smallskip \newline
\noindent (d) Since the \textquotedblleft cyclic covering trick\textquotedblright \  of Theorem \ref{WAHLREID} is independent of the dimension (and is a standard tool for reducing log terminal and log canonical singularities of a $\mathbb{Q}$-Gorenstein variety, to canonical and, respectively, log canonical singularities of index $1$,  cf. \cite[Proposition 4-5-3, pp. 186-191]{Matsuki}), in order to tackle the above conjectures, one should come up with analogues of Lemma \ref{TRICKYLEMMA}, Theorem \ref{GLOBALTHM}, and Proposition \ref{KAQUADRAT}, being valid in dimension $d\ge3$. If  $d\ge3$, the singularities of $X(N,\Delta_{Q})$ are not
necessarily isolated, and one has to construct carefully a suitable
stratification of the singular locus. In addition, even the nature of
singularities may differ (as it is known that in dimensions $\geq3$ there
exist toric singularities which are not quotient singularities). Nevertheless,
toric singularities are  \textquotedblleft relatively mild\textquotedblright \ singularities and it seems to be not very difficult to deal with them. On the other hand, the analogues of (\ref{KSQUARE1}) in high dimensions should relate various (\textit{usual, orbifold or stringy}) \textit{Chern classes} of $X(N,\Delta_{Q})$ and $X(\Lambda_{\mathcal{F}_{d-2}(Q)},\Delta_{Q})$. (Furthermore, it would be desirable if one could keep all the required arguments independent of particular desingularizations of $X(N,\Delta_{Q})$.)\medskip

\noindent \textbf{(iv)} Recently, log del Pezzo surfaces have also attracted increasing interest in the framework of the so-called \textit{homological mirror symmetry for Fano varieties} in dimension $d=2$.  (See, e.g., \cite{Akhtar etal}, \cite{Cor-He} and \cite{KaNiPr}, and the references therein.) It was proposed that log del Pezzo surfaces with cyclic quotient singularities admit $\mathbb{Q}$-Gorenstein toric degenerations corresponding (under mirror symmetry) to maximally mutable Laurent polynomials in two variables, and that the quantum period of such a surface coincides with the classical period of its mirror partner. Thus, \textit{the combinatorics of mutation} and \textit{toric deformations} (which are closely related to geometric properties of LDP-polygons\footnote{In \cite{Akhtar etal}, \cite{Cor-He} and \cite{KaNiPr} the LDP-polygons are called \textit{Fano polygons}.}) play an important role in the conception of this new approach. It comes into question whether the toric log del Pezzo surfaces associated with $\ell$-\textit{reflexive} polygons (perhaps with \textit{prescribed} singularities) are of particular value for these investigations.

\end{document}